\documentclass[11pt]{article}
\usepackage[margin=1in]{geometry}
\usepackage{amsmath,amssymb,mathtools}
\usepackage{enumitem}
\usepackage{microtype}
\usepackage[hidelinks]{hyperref}
\usepackage{tikz}
\usepackage{xfrac}
\usepackage{xcolor}
\usepackage{authblk}
\usepackage{amsmath,amsfonts,amssymb,amsthm}

\usetikzlibrary{decorations.pathreplacing}

\hypersetup{
	colorlinks=true,       
	linkcolor=blue,          
	citecolor=blue,        
	filecolor=blue,      
	urlcolor=blue           
}

\definecolor{mygreen}{RGB}{60, 150, 60}
\definecolor{mypurple}{RGB}{150, 50, 209}

\newcommand{\R}{\mathbb{R}}
\newcommand{\N}{\mathbb{N}}
\newcommand{\HK}{\mathrm{HK}}

\newcommand{\Mcal}{\mathcal{M}}

\newcommand{\Pbb}{\mathbb{P}}
\newcommand{\Ebb}{\mathbb{E}}

\DeclareMathOperator{\Var}{Var}
\DeclareMathOperator{\Sen}{Sen}

\newcommand{\Pcal}{\mathcal{P}}

\newcommand{\Ncal}{\mathcal{N}}
\newcommand{\diff}{\,\textnormal{d}}

\DeclareMathOperator{\Div}{div}

\DeclareMathOperator{\Tan}{Tan}
\newcommand{\grad}{\nabla}

\newcommand{\estims}{\mathcal{T}}
\newcommand{\diagram}{\mathcal{D}}
\newcommand{\bdry}{\mathcal{B}}

\DeclareMathOperator{\MLE}{MLE}

\DeclareMathOperator{\WPE}{WPE}

\newcommand{\dH}{d_{\mathrm{H}}}
\newcommand{\dW}{d_{\mathrm{W}}}

\newcommand{\dHKlambda}{d_{\HK_\lambda}}

\DeclareMathOperator{\AC}{AC}

\newcommand{\CovSpace}{\mathcal{S}}

\usepackage{xcolor}

\theoremstyle{plain}
\newtheorem{lemma}{Lemma}[section]
\newtheorem{proposition}[lemma]{Proposition}
\newtheorem{corollary}[lemma]{Corollary}

\newtheorem{theorem}[lemma]{Theorem}

\theoremstyle{definition}
\newtheorem{definition}[lemma]{Definition}
\newtheorem{remark}[lemma]{Remark}

\newtheorem{example}[lemma]{Example}

\title{Optimal tradeoffs between variance and sensitivity, via Hellinger--Kantorovich geometry}

\date{\today}

\author[1]{Adam Quinn Jaffe\thanks{equal contribution}}
\author[2]{Congwei Yang\protect\footnotemark[1]}
\author[2]{Nicol\'as Garc\'ia Trillos\thanks{NGT gratefully acknowledges support from NSF-DMS Grant No. 2236447}}
\author[1]{Bodhisattva Sen\thanks{BS gratefully acknowledges acknowledges support from NSF (DMS-2311062 and DMS-2515520).}}

\affil[1]{Department of Statistics, Columbia University, New York, NY}
\affil[2]{Department of Statistics, University of Wisconsin, Madison, WI}

\date{\today}

\begin{document}
\maketitle

\begin{abstract}
    While the goal of classical statistical theory is to minimize the variance of unbiased estimators, an additional goal of some recent research is to minimize or constrain other notions of the instability of estimators.
    In particular, a subset of the authors recently introduced a notion of instability termed ``sensitivity'', which represents the expected amount that an estimator would change if its inputs were infinitesimally perturbed.
    In this work, we study the multiobjective problem of minimizing both the variance and sensitivity over unbiased estimators in a given statistical model, and we aim to determine Pareto optimal tradeoffs between these two objectives.
    Our main results show that the Pareto front for these tradeoffs is precisely determined by endowing the statistical model with the Hellinger--Kantorovich geometry, which interpolates between the Hellinger and Wasserstein geometries that are known to be associated, respectively, with the separate minimization of variance and sensitivity. We also develop a $Z$-estimation technique which constructs estimators that asymptotically achieve Pareto optimal pairs of variance and sensitivity.
    While computing the Pareto front and constructing the resulting estimators in practice requires solving a continuum of partial differential equations, we fully execute this program in some  parametric examples of interest where the analysis is already non-trivial.
    Our results also shed light on some classical statistical problems, for example by showing that the maximum likelihood estimator achieves the smallest asymptotic sensitivity among all estimators achieving the optimal asymptotic variance.

    \medskip
\noindent\textbf{Keywords:} {Cram\'er--Rao lower bound, geometries on spaces of probability measures, information geometry, optimal transport, Pareto optimality, stability, statistical tradeoffs, transport-source equation, Wasserstein geometry, $Z$-estimation}
\medskip

\noindent\textbf{2020 MSC:} 62B11, 62F10, 62F12, 35Q49, 49Q22
\end{abstract}


\section{Introduction}

Fix a statistical model $\Pcal=\{P_{\theta}:\theta\in\Theta\}$ on $\R^d$ with $\Theta\subseteq\R^p$, and suppose $X_1,\ldots,X_n$ are independent identically distributed (i.i.d.) samples from $P_{\theta}$ for some unknown $\theta\in\Theta$ and that the goal is to estimate $\theta$ from $X_1,\ldots, X_n$.
For notational simplicity we take $p=1$ in this Introduction, but the main body of the paper treats the general case $p\ge 1$.
Among unbiased estimators $T_n:(\R^d)^n\to\R$ of $\theta$, a classical statistical criterion of optimality (e.g., see \cite[Chapter~2, Definition~1.6]{LehmannCasella1998}) is to minimize the variance,
\begin{equation*}
    \Var_{\theta}(T_n)
    = \Ebb_{\theta}\left[
        \left(T_n(X_1,\ldots,X_n)-\theta\right)^2
      \right]
\end{equation*}
which, for unbiased estimators, coincides with the mean squared error (MSE) and hence provides a standard measure of estimation accuracy.
The classical Cram\'er--Rao theory provides fundamental lower bounds for the variance of unbiased estimators under suitable regularity conditions \cite[Section~3.1.3]{Shao2003}, and the maximum likelihood estimator (MLE) asymptotically attains this lower bound under appropriate regularity conditions (although it is typically biased for finite samples) \cite[Section~4.5.2]{Shao2003}.

Recent works have also emphasized various notions of \textit{stability} as complementary objectives in the design of estimators, in addition to accuracy; broadly speaking, the stability of an estimator quantifies the amount the estimator changes under appropriate perturbations of the data or model \cite{Yu2013}, and different notions of perturbation lead to different notions of stability.
For example, classical robust statistics \cite{Huber1964} and related recent developments \cite{LaiBayraktar2020,GiordanoMeagerBroderickI,GiordanoMeagerBroderickII} seek estimators stable under small amounts of contamination in the data or sampling distribution, bagging \cite{BreimanI,BreimanII,SoloffFoygelBarberWillettI,SoloffFoygelBarberWillettII} seeks to construct estimators stable under subsampling, and more.
More broadly, in the framework of \cite{YuKumbier}, it is argued that one should consider a variety of notions of stability, which ``includes, but is much broader than, the concept of sampling variability''.

In line with this theme, recent works \cite{WIM, WCR_achieve, WCR} study a notion of (in)stability, termed \emph{sensitivity}; for a sufficiently regular estimator $T_n$, its sensitivity is defined by
\begin{equation*}
    \Sen_{\theta}(T_n)
    := \Ebb_{\theta}\left[
        \sum_{i=1}^{n}
        \left\|\nabla_{x_i}T_n(X_1,\ldots,X_n)\right\|^2
       \right].
\end{equation*}
Roughly speaking, sensitivity measures the expected amount by which an estimator would change if its inputs were infinitesimally perturbed; smaller values indicate greater stability under such perturbations.
This notion of sensitivity is relevant in several statistical settings, including the propagation of small measurement errors \cite{CarollStefanksi}, adversarial robustness \cite{LaiBayraktar2020}, additive-noise mechanisms for local differential privacy \cite{DuchiJordanWainwright2018}, and distributionally robust optimization over $W_2$ balls \cite{BartlDrapeauOblojWiesel2021}.

Previous work \cite{WCR} by a subset of the authors developed a theory of sensitivity for unbiased estimation which parallels the classical Cram\'er--Rao theory of variance; since it is based on the Wasserstein geometry in the same way that the classical theory is based on the Hellinger (equivalently, Fisher--Rao) geometry, it was deemed the Wasserstein--Cram\'er--Rao theory.
Specifically, it establishes fundamental lower bounds on the sensitivity, characterizes exact attainment in certain models called transport families, and proves asymptotic efficiency of Wasserstein projection estimators (WPE) under suitable conditions.
These comprehensive results allow for analyzing many concrete statistical models of interest, sometimes leading to estimators which look and behave very differently from their classical counterparts.

The goal of this work is to develop a similarly comprehensive theory of the tradeoffs between accuracy and stability.
Specifically, we focus on the variance and sensitivity of unbiased estimators, and we aim to solve optimization problems like
\begin{equation}\label{eqn:Intro-Pareto}
    \inf_{\substack{T_n\ \text{unbiased}\\
                    \Sen_\theta(T_n)\le s_0}}
        \Var_\theta(T_n)
    \qquad \quad \text{and}\qquad \quad
    \inf_{\substack{T_n\ \text{unbiased}\\
                    \Var_\theta(T_n)\le v_0}}
        \Sen_\theta(T_n).
\end{equation}
for a given ``sensitivity budget'' $s_0>0$ or ``variance budget'' $v_0>0$; a solution $T_n$ to either problem is called a \textit{Pareto efficient} estimator, since it has the property that no other unbiased estimator can reduce either the variance or the sensitivity without increasing the other.
(For $p>1$, an appropriate version of~\eqref{eqn:Intro-Pareto} is studied in Corollary~\ref{cor:Trace}.)
As is common in the study of Pareto efficiency for multiobjective optimization (see \cite[Chapter~3]{Ehrgott}) and as we illustrate in Figure~\ref{fig:laplace-sensitivity-budget} for the Laplace location family, such problems can be reduced to studying the convex combination of the two objectives, i.e.,
\begin{equation}\label{eq:Cvx-Comb}
    \lambda\Var_{\theta}(T_n)+(1-\lambda)\Sen_{\theta}(T_n),
    \qquad 0\le\lambda\le1.
\end{equation}
Our central question is therefore:
\begin{center}
\hspace{3em}\textit{What are the fundamental lower bounds on the convex combinations of variance and sensitivity over all sufficiently regular unbiased estimators, and how can we construct consistent estimators that asymptotically attain these bounds?}\hspace{3em}
\end{center}

As discussed above, the classical Cram\'er--Rao theory of variance is associated with Hellinger (equivalently, Fisher--Rao) geometry \cite{Amari}, while the Wasserstein--Cram\'er--Rao theory is associated with Wasserstein geometry \cite{WCR}.
This suggests that lower bounds on convex combinations of variance and sensitivity should be associated with some combination of the Hellinger and Wasserstein geometries, but it is not immediately clear what is the appropriate way to combine these two geometries.
Interestingly, the na\"ive approach of directly combining the \textit{distances} does not work; instead, the appropriate construction is to consider a suitable combination of the \textit{Riemannian metric tensors}, which leads naturally to the Hellinger--Kantorovich geometry \cite{Liero_2016, LieroMielkeSavare,ChizatPeyreSchmitzerVialard}.
This geometric insight allows us to derive a comprehensive theory describing Pareto optimal tradeoffs between variance and sensitivity and constructing Pareto optimal estimators in a variety of statistical models of interest.

\begin{figure}[t]
    \centering
    \begin{tikzpicture}[font=\footnotesize]
      \def\PaneloneS{4.6}
      \begin{scope}[scale=\PaneloneS]
        \draw[thick, -stealth] (0,0) -- (1.2,0);
        \draw[thick, -stealth] (0,0) -- (0,1.2);
        \draw[thick, dashed] (0,0.2) -- (1,0.2);
        \draw[thick, dashed] (0.2,0) -- (0.2,1);
        \node at (0.5, -0.25) {variance};
        \node[rotate=90] at (-0.25, 0.75) {sensitivity};
        \draw[thick] (-0.05,0)   node[left]{$0$}      -- (0.05,0);
        \draw[thick] (-0.05,0.2) node[left]{$1$}      -- (0.05,0.2);
        \draw[thick] (-0.05,1)   node[left]{$\infty$} -- (0.05,1);
        \draw[thick] (0,-0.05)   node[below]{$0$}       -- (0,0.05);
        \draw[thick] (0.2,-0.05) node[below right]{$1$} -- (0.2,0.05);
        \draw[thick] (0.7,-0.05) node[below]{$2$}       -- (0.7,0.05);
        \path [fill=mygreen, thick, opacity=0.3, domain=0:1, samples=100]
          plot ({0.2+0.5*\x^2}, {0.2+0.8*(\x-1)^2}) -- (1,0.2) -- (1,1) -- cycle;
        \draw [color=mygreen, ultra thick, domain=0:1, samples=100]
          plot ({0.2+0.5*\x^2}, {0.2+0.8*(\x-1)^2}) -- (1,0.2);
        \node[color=mygreen,rotate=-90] at (0.1, 0.65) {Pareto front};
        \pgfmathsetmacro{\budgetX}{0.45}
        \pgfmathsetmacro{\budgetY}{0.2+0.8*(sqrt(0.5)-1)^2}
        \draw[red,thick]
          (0,\budgetY) -- (1,\budgetY);
        \node[red,anchor=south] at (0.74,\budgetY)
          {sensitivity budget};
        \draw[red,thick]
          (\budgetX,0) -- (\budgetX,\budgetY);
        \filldraw[blue]
          (\budgetX,\budgetY) circle (0.5pt);
        \draw[red,thick]
          (\budgetX,-0.05) -- (\budgetX,0.05);
        \draw[red,thick]
          (-0.05,\budgetY) node[left] {$s_0$}
          -- (0.05,\budgetY);
        \filldraw[color=blue] (0.2, 1) circle (0.5pt) node[above right] {sample median};
        \filldraw[color=blue] (0.7, 0.2) circle (0.5pt) node[below right] {sample mean};
      \end{scope}
      \begin{scope}[xshift=8.2cm,scale=\PaneloneS]
        \draw[thick, -stealth] (0,0) -- (1.2,0);
        \draw[thick, -stealth] (0,0) -- (0,1.2);
        \draw[thick, dashed] (0,0.2) -- (1,0.2);
        \draw[thick, dashed] (0.2,0) -- (0.2,1);
        \node at (0.5, -0.25) {variance};
        \node[rotate=90] at (-0.25, 0.75) {sensitivity};
        \draw[thick] (-0.05,0)   node[left]{$0$}      -- (0.05,0);
        \draw[thick] (-0.05,0.2) node[left]{$1$}      -- (0.05,0.2);
        \draw[thick] (-0.05,1)   node[left]{$\infty$} -- (0.05,1);
        \draw[thick] (0,-0.05)   node[below]{$0$}       -- (0,0.05);
        \draw[thick] (0.2,-0.05) node[below right]{$1$} -- (0.2,0.05);
        \draw[thick] (0.7,-0.05) node[below]{$2$}       -- (0.7,0.05);
        \path [fill=mygreen, thick, opacity=0.3, domain=0:1, samples=100]
          plot ({0.2+0.5*\x^2}, {0.2+0.8*(\x-1)^2}) -- (1,0.2) -- (1,1) -- cycle;
        \draw [color=mygreen, ultra thick, domain=0:1, samples=100]
          plot ({0.2+0.5*\x^2}, {0.2+0.8*(\x-1)^2}) -- (1,0.2);
        \node[color=mygreen,rotate=-90] at (0.1, 0.65) {Pareto front};
        \pgfmathsetmacro{\varianceBudgetX}{0.2+0.5/3}
        \pgfmathsetmacro{\varianceBudgetY}{0.2+0.8*(sqrt(1/3)-1)^2}
        \draw[red,thick]
          (\varianceBudgetX,0) -- (\varianceBudgetX,1);
        \node[red,anchor=west] at (\varianceBudgetX,0.85)
          {variance budget};
        \draw[red,thick]
          (0,\varianceBudgetY) -- (\varianceBudgetX,\varianceBudgetY);
        \filldraw[blue]
          (\varianceBudgetX,\varianceBudgetY) circle (0.5pt);
        \draw[red,thick]
          (\varianceBudgetX,-0.05) node[below] {$v_0$}
          -- (\varianceBudgetX,0.05);
        \draw[red,thick]
          (-0.05,\varianceBudgetY) -- (0.05,\varianceBudgetY);
        \filldraw[color=blue] (0.2, 1) circle (0.5pt) node[above right] {sample median};
        \filldraw[color=blue] (0.7, 0.2) circle (0.5pt) node[below right] {sample mean};
      \end{scope}
    \end{tikzpicture}
    \caption{Visualizing the tradeoffs between variance and sensitivity for unbiased estimators in the Laplace location family, where the sample median and the sample mean separately minimize the variance and sensitivity, respectively.
    We show the $(v,s)$ plane, the set of variance-sensitivity pairs achievable by all unbiased estimators (green), and a few estimators of interest (blue).    
    For a given sensitivity budget $s_0>0$, the minimal achievable variance is given by intersecting the horizontal line $s=s_0$ with the Pareto front (left); for a given variance budget $v_0>0$, the minimal achievable sensitivity is given by intersecting the vertical line $v=v_0$ with the Pareto front (right).
    Moreover, every point on the Pareto front can be realized as the asymptotic variance and sensitivity of a suitable $\HK_{\lambda}$ $Z$-estimator (HKZE).}
    \label{fig:laplace-sensitivity-budget}
\end{figure}
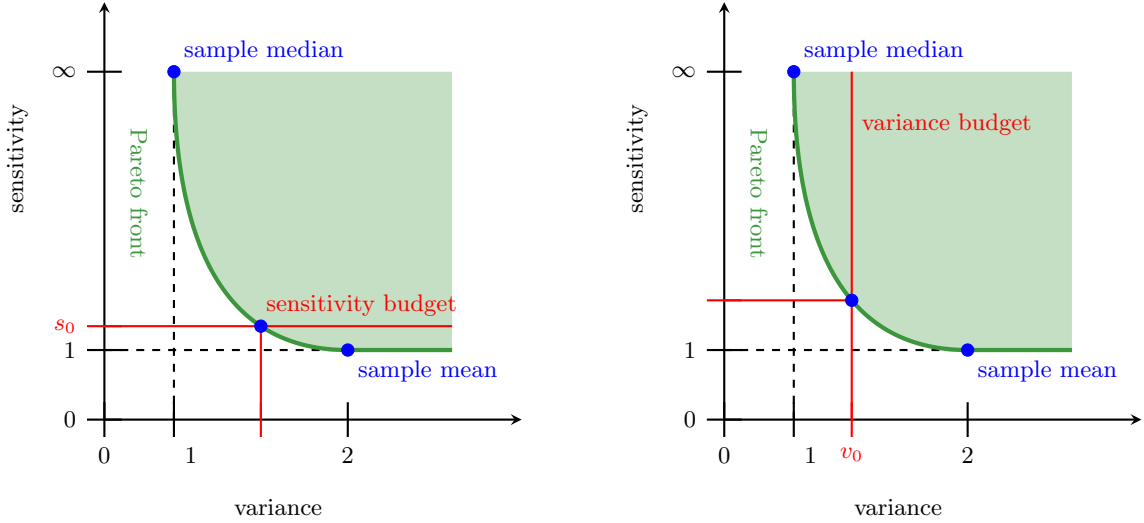

    \subsection{Summary of Results}

    All of our results begin by viewing our statistical models as submanifolds of the space of probability measures endowed with the \textit{Hellinger--Kantorovich (HK)} geometries, initially studied in \cite{LieroMielkeSavare, Liero_2016,ChizatPeyreSchmitzerVialard} with motivations from partial differential equations (PDE).
    This family of geometries describes the evolution of measures as a combination of two mechanisms: a change in mass governed by a scalar field (as in the Hellinger geometry), and movement of mass governed by a vector field (as in the Wasserstein geometry).
    As such, and as discussed above, the $\HK$ geometries have a parameter $0\le \lambda\le 1$ describing the relative strength of the mechanisms, and which interpolate between the Wasserstein ($\lambda=0$) and Hellinger ($\lambda=1$) geometries.

    To state our results, we introduce a notion of $\HK_\lambda$ regularity for the parametric statistical model $\Pcal=\{P_\theta:\theta\in\Theta\}$.
    Roughly speaking, this notion (Definition~\ref{def:HK_regular}) requires that for each $0\le \lambda \le 1$ the probability measures $\theta \mapsto P_{\theta}$ satisfy the following PDE called the \textit{transport-source equation}
    \begin{equation}\label{eq:TS-eq}
        \partial_\theta P_\theta+(1-\lambda)\Div\!\big(\nabla\phi_{\theta,\lambda}\,P_\theta\big)
        =\lambda\,\phi_{\theta,\lambda}\,P_\theta,
    \end{equation}
    for some scalar function $\phi_{\theta,\lambda}:\R^d\to\R$ called the \emph{$\HK_\lambda$ score function} of the model.
    We emphasize that~\eqref{eq:TS-eq} is understood in the weak sense (i.e., by integrating against infinitely differentiable compactly supported test functions), so $P_{\theta}$ and $\phi_{\theta,\lambda}$ are not required to be differentiable per se.
    Note also that $\phi_{\theta,1}$ coincides with the score function from classical statistics, since~\eqref{eq:TS-eq} reduces to $\partial_\theta P_\theta=\phi_{\theta,1}\,P_\theta$ when $\lambda=1$.
    But when $0<\lambda<1$, the $\HK_\lambda$ score function involves both a scalar field $\phi_{\theta,\lambda}$ and a vector field $\nabla\phi_{\theta,\lambda}$, governing the creation/destruction and the spatial transport of mass, respectively.
    Both fields are generated by the same function $\phi_{\theta,\lambda}$ through the energy minimization underlying $\HK_\lambda$ geometry.
    More precisely, among all scalar and vector fields producing the same infinitesimal change in $P_\theta$, the pair $(\phi_{\theta,\lambda},\nabla\phi_{\theta,\lambda})$ minimizes the weighted quadratic energy of mass change and transport.

    The $\HK_{\lambda}$ score function $\phi_{\theta,\lambda}$ is the fundamental object of interest in the theory of Pareto optimal tradeoffs between variance and sensitivity, and all of our lower bounds and estimation techniques are based on it in some form.
    However, computing the 
    $\HK$ score functions is non-trivial even in simple models, since it requires solving the PDE~\eqref{eq:TS-eq}; this contrasts the classical theory where the score function can be found simply by differentiating the log density.
    Nonetheless, we can compute the $\HK$ score functions in various models of interest (Laplace location family, uniform scale family, various Gaussian models) which we revisit as running examples throughout the paper.
    Additionally, we establish that~\eqref{eq:TS-eq} admits at most one solution under very mild conditions (Proposition~\ref{prop:HK-reg-suff}), so our results apply to a wide class of models and the only bottleneck is explicitly solving~\eqref{eq:TS-eq}. 
While, in general, closed-form solutions of $\phi_{\theta, \lambda}$ to \eqref{eq:TS-eq} may not be available, it is still possible to rely on the vast literature of classical methods in the numerical analysis of linear elliptic PDEs \cite{Brenner2008,LeVeque2007,Lapeyre2003} and on recent advances in methods used for high-dimensional PDEs \cite{RAISSI2019686,Aimsworth,Lee2025} to obtain numerical solutions to \eqref{eq:TS-eq}. We believe that studying the connections between fundamental statistical lower bounds, different geometries in the space of probability measures, and PDEs arising from suitable geometric representations of statistical models is a very promising direction that this paper aims to help pave.

    Our first main result is a Hellinger--Kantorovich Cram\'er--Rao lower bound for linear combinations of variance and sensitivity, in a $\HK$ regular model.
    To state this, we introduce a new notion of \emph{$\lambda$-Hellinger--Kantorovich (HK$_\lambda$) information} of a model $\Pcal$, provided it is $\HK_{\lambda}$ regular with $\HK_{\lambda}$ score function $\phi_{\theta,\lambda}$, via
    \begin{equation*}
        I_\lambda(\theta)
        :=\lambda \, \Ebb_\theta\left[|\phi_{\theta,\lambda}(X)|^2\right]
        +(1-\lambda) \, \Ebb_\theta\left[\|\nabla\phi_{\theta,\lambda}(X)\|^2\right].
    \end{equation*}
    In fact, the $\HK_{\lambda}$ information interpolates between the Wasserstein information of \cite{WCR} ($\lambda=0$) and the classical Fisher information ($\lambda=1$).
    Moreover, every sufficiently regular unbiased estimator $T_n$ of $\theta$ must satisfy (Theorem~\ref{thm:HK-CR-bound}) the bound
    \begin{equation}\label{eqn:intro-LB}
        \lambda\Var_\theta(T_n)+(1-\lambda)\Sen_\theta(T_n)
        \ge \frac{1}{nI_\lambda(\theta)}
    \end{equation}
    for each $0\leq\lambda\leq1$ with $I_\lambda(\theta)>0$.
    To understand the implications of this lower bound for our problem of interest, we introduce the quantities
    \begin{equation*}
        V_\lambda(\theta) :=\frac{\Ebb_\theta\left[|\phi_{\theta,\lambda}(X)|^2\right]}{(I_\lambda(\theta))^2}
        \qquad\textnormal{and}\qquad
        S_\lambda(\theta) :=\frac{\Ebb_\theta\left[\|\nabla\phi_{\theta,\lambda}(X)\|^2\right]}{(I_\lambda(\theta))^2},
    \end{equation*}
    and we note that~\eqref{eqn:intro-LB} implies:
    \begin{itemize}
        \item[(i)] if $\Sen_\theta(T_n)\le S_\lambda(\theta)/n$, then $\Var_\theta(T_n)\ge V_\lambda(\theta)/n$;
        \item[(ii)] if $\Var_\theta(T_n)\le V_\lambda(\theta)/n$, then $\Sen_\theta(T_n)\ge S_\lambda(\theta)/n$.
    \end{itemize}
    That is, the pair $(V_\lambda(\theta),S_\lambda(\theta))$ can be interpreted as a Pareto optimal variance-sensitivity pair.

    We also provide extensions of this lower bound, to go beyond the above setting.
    For example, we provide an analogous lower bound for unbiased estimation of a continuously differentiable function $\chi(\theta)$ of the parameter $\theta$; see Corollary~\ref{cor:HK-CR-function}.
    In another example, we use this result (by setting $\chi(\theta):=\Ebb_\theta[T_n]$) to obtain an asymptotic lower bound for certain biased estimators of $\theta$; see Corollary~\ref{cor:HK-CR-asymptotic}.
    In particular, this asymptotic bound applies to regular asymptotically linear (RAL) estimators satisfying the assumptions therein.

    In our second main result, we identify a way to construct estimators that asymptotically attain these Pareto optimal variance-sensitivity pairs.
    More precisely, we define the $\HK_\lambda$ \emph{$Z$-estimator} (HK$_{\lambda}$ZE), denoted $T_{n,\lambda}$, to be any solution $\theta\in\Theta$ to the following $Z$-equation
    \begin{equation*}
        \frac{1}{n}\sum_{i=1}^n\phi_{\theta,\lambda}(X_i)=0.
    \end{equation*}
    Then, under suitable regularity assumptions, we show (Theorem~\ref{thm:HK-Z-attainment}) that the HK$_{\lambda}$ZE satisfies
    \begin{equation*}
        \sqrt n(T_{n,\lambda}-\theta)\overset{\mathrm d}{\to}\Ncal(0,V_\lambda(\theta))\qquad \textnormal{and}\qquad
        n\sum_{i=1}^n\|\nabla_{x_i}T_{n,\lambda}\|^2\to S_\lambda(\theta)
        \quad\textnormal{in probability}.
    \end{equation*}
    Since $n\sum_{i=1}^n\|\nabla_{x_i}T_{n,\lambda}\|^2$ is the sample analogue of rescaled sensitivity (note that its expectation is $n\Sen_\theta(T_{n,\lambda})$), this shows that the HK$_\lambda$ZE asymptotically attains the Pareto optimal variance-sensitivity pair $(V_\lambda(\theta),S_\lambda(\theta))$, and it is therefore Pareto efficient.
    When specialized to the case $\lambda=0$, our theory yields a new $Z$-estimation technique for constructing asymptotically sensitivity-efficient estimators; this technique is much simpler than the Wasserstein projection estimator (WPE) technique of \cite{WCR} which requires very strict technical assumptions.
    A key ingredient of these proofs is a new result on the asymptotic sensitivity of general $Z$-estimators (Theorem~\ref{thm:z-est}), which may also be of independent interest.

    We gain further insight into our results by studying the geometry of the lower bounds discussed above.
    Indeed, the path $(V_\lambda(\theta),S_\lambda(\theta))$ for $0<\lambda<1$ traces a curve in the variance-sensitivity plane, and we prove (Corollary~\ref{cor:Pareto-optim-VSD}) that this curve is indeed a Pareto front for the set of all pairs of points that are realizable as the asymptotic variance and sensitivity of some sequence of unbiased estimators; in fact, we show (Lemma~\ref{lem:bdry-slope}) that this curve is the graph of a non-increasing convex function whose slope at $(V_\lambda(\theta),S_\lambda(\theta))$ is $-\lambda/(1-\lambda)$.
    See Figure~\ref{fig:VSD-cartoon} for an illustration of these concepts.

We also highlight a few consequences of our theory for the unconstrained problems of separately minimizing the variance or sensitivity. 
First, we obtain a new optimality property (Corollary~\ref{cor:extremal-efficiency}) for many classical estimators by setting $\lambda=1$ in our main results: the MLE  achieves the optimal asymptotic sensitivity among all estimators which achieve the optimal asymptotic variance.
Second, by sending $\lambda\to 0$ with a suitable limiting argument, we establish an impossibility result (Proposition~\ref{prop:Pareto-extremal}) that, in some models, there cannot exist any unbiased estimator which achieves optimal asymptotic variance and which achieves sensitivity vanishing at rate $n^{-1}$.

    Lastly, we emphasize a technical but fundamental aspect of the paper, namely that our notion of $\HK$ regularity is based on a notion of absolute continuity with respect to the $\HK$ metric, whereas the conditions of DQM from classical theory and DWS from \cite{WCR} are forms of differentiability with respect to the Hellinger and Wasserstein metrics, respectively. (Recall that, in general, neither differentiability nor absolute continuity implies the other.)
    Indeed, absolute continuity with respect to the $\HK_{\lambda}$ metric turns out to be closely related to solutions of the transport-source equation~\eqref{eq:TS-eq}; see \cite{Liero_2016,LieroMielkeSavare}.
    To be explicit, we note that there are three reasons that we use absolute continuity rather than differentiability in this work.
    First, although the HK distance admits a static formulation (see Appendix~\ref{app:HK-static-cone}) which can be used to define a notion of differentiability in the $\HK$ sense (denoted DHKS), solving equation~\eqref{eq:TS-eq} is typically easier than computing the static formulation; this contrasts the Hellinger and Wasserstein cases, where the static formulations can be easily computed via the score functions and the optimal transport maps, respectively.
    Second, we show (Proposition~\ref{prop:abs-cont-HK}) that absolute continuity with respect to at least one of Hellinger or Wasserstein metric automatically implies absolute continuity with respect to $\HK_{\lambda}$ for all $0<\lambda<1$, hence all statistical models studied in the classical theory can be analyzed within our framework.
    Third, we show (Theorem~\ref{thm:loc-Lip-test-funcs}) that         $\theta\mapsto\Ebb_{\theta}[T_n]$ can be differentiated for a broad class of estimators $T_n$, whenever $\theta\mapsto P_{\theta}$ is absolutely continuous with respect to $\HK_{\lambda}$; this both simplifies and generalizes the result \cite[Proposition~2.2]{WCR} from the Wasserstein case, and may be of independent interest. 

    \subsection{Extended Example: Laplace Location Family}\label{subsec:illustrative-example}
    
    To illustrate the implications of our main results, we include a brief discussion of one of the consequences for one of our many running examples, namely the Laplace location family.
    More precisely, define for $\theta\in\Theta:=\R$ the probability measure $P_{\theta}$ on $\R$ whose density with respect to the Lebesgue measure is $p_{\theta}(x)=\frac{1}{2}e^{-|x-\theta|}$ for all $x\in\R$, and consider the model $\Pcal=\{P_\theta:\theta\in\Theta\}$.
    In this model, the problems of separately minimizing the variance ($\lambda=1$) and sensitivity ($\lambda=0$) are already well-understood: the variance lower bound is $1/n$ and is asymptotically achieved by the sample median (which is the MLE); the sensitivity lower bound is $1/n$ and is asymptotically achieved by the sample mean (which is the WPE) \cite[Example~5.7]{WCR}.
    Yet, for the genuinely multiobjective version of the problem ($0<\lambda<1$), it is not clear how to understand such tradeoffs, let alone how to construct estimators achieving them.
    
    In this case, the transport-source equation~\eqref{eq:TS-eq} reduces to the following ordinary differential equation (ODE), for $0<\lambda<1$ and $x\neq\theta$:
    \begin{equation*}
        \lambda\phi_{\theta,\lambda}(x)
        -(1-\lambda)\left(
            \phi''_{\theta,\lambda}(x)
            -\operatorname{sgn}(x-\theta)\phi'_{\theta,\lambda}(x)
        \right)
        =\operatorname{sgn}(x-\theta).
    \end{equation*}
    We can explicitly solve this second-order linear ODE (Example~\ref{ex:Laplace-reg}) to determine the $\HK_{\lambda}$ score function $\phi_{\theta,\lambda}$ and hence compute $I_{\lambda}(\theta)$ as well as $(V_{\lambda}(\theta),S_{\lambda}(\theta))$.
    Then applying~\eqref{eqn:intro-LB} gives a sharp answer to the tradeoff question, namely, for any sufficiently regular unbiased estimator $T_n$ of $\theta$, we have (Example~\ref{laplace_lb}):
    \begin{equation}\label{eqn:VS-tradeoff-laplace}
        \lambda\Var_\theta(T_n)+(1-\lambda)\Sen_\theta(T_n)
        \ge\frac{1+\lambda+\sqrt{(1-\lambda)(1+3\lambda)}}{2n},
        \qquad \qquad\textnormal{for all}\qquad0\le\lambda\le1.
    \end{equation}
    This bound on~\eqref{eq:Cvx-Comb} allows us to solve problems of the form~\eqref{eqn:Intro-Pareto} as follows.
    If one is given a sensitivity budget $s_0>0$ slightly larger than the optimal sensitivity of $1/n$, then the minimal achievable variance can be determined by setting $\Sen_{\theta}(T_n)=s_0$ in~\eqref{eqn:VS-tradeoff-laplace} and then finding which $0\le \lambda\le 1$ maximizes the resulting lower bound on $\Var_{\theta}(T_n)$.
    (Conversely, one can set $\Var_{\theta}(T_n)=v_0$ for some $v_0>0$ and then maximize the resulting lower bound on $\Sen_{\theta}(T_n)$ over $0\le \lambda\le 1$.)
    In Figure~\ref{fig:laplace-sensitivity-budget} we illustrate these bounds visually.

    We can also use HKZE to construct estimators which asymptotically achieve any prescribed Pareto optimal pair of variance and sensitivity (Example~\ref{ex:Laplace-HK-Z}).
    As expected, the HK$_{\lambda}$ZE interpolates between the sample mean and the sample median as $\lambda$ ranges from 0 to 1, and each intermediate HK$_{\lambda}$ZE is an $M$-estimator with a strictly convex loss function hence it is easy to compute.
    It is also natural to compare the HK$_{\lambda}$ZE with Huber's $M$-estimators from robust statistics \cite{Huber1964}, which provide a ubiquitous family estimators interpolating between the sample median and the sample mean; it turns out that these families of estimators are distinct, but the HK$_{\lambda}$ZE indeed has a ``Huber-like'' loss function (i.e., with quadratic growth when $|x-\theta|\to 0$ and linear growth when $|x-\theta|\to \infty$).
    As we discuss in Appendix~\ref{app:Huber}, it is also possible to define a notion of $\lambda$ relative efficiency for the multiobjective problem of minimizing the variance and sensitivity, and it happens that Huber's $M$-estimators are extremely close to being optimal; their worst-case $\lambda$ relative efficiency, over $0\le \lambda\le 1$, is approximately $0.954$.

    \subsection{Related Literature}

    As discussed above, our work contributes to a growing theory of ``stability'' of statistical estimation, which includes stability under various perturbations of the model or data.
    Classical examples include stability under model contamination \cite{Huber1964,LaiBayraktar2020,GiordanoMeagerBroderickI,GiordanoMeagerBroderickII}, subsampling \cite{BreimanI,BreimanII,SoloffFoygelBarberWillettI,SoloffFoygelBarberWillettII}, and more, and in this work we specifically focus on the notion of sensitivity from \cite{WIM, WCR_achieve, WCR} which is connected to various statistical applications \cite{CarollStefanksi,LaiBayraktar2020,DuchiJordanWainwright2018,BartlDrapeauOblojWiesel2021}.
    We note also that, in some instances, stability is used only as a tool to study accuracy, e.g., in some applications of algorithmic stability~\cite{BousquetElisseeff2002}; in this work, we take the perspective that stability itself is a goal.
    We thus view our results to also contribute to the theory of the PCS framework of veridical data science \cite{Yu2013,BarterYu}, which advocates that various notions of stability should inform the design of estimators and statistical procedures more broadly.

    Another literature in which our results can be situated is the study of tradeoffs between statistical accuracy and other desiderata, which has become a fundamental aspect of modern research in statistics and machine learning.
    For instance, much is known about computational-statistical tradeoffs \cite{BerthetRigollet,ZhangWainwrightJordan, MaWu}, fairness-accuracy tradeoffs \cite{LiangLuMuOkumura,LiuMolinari,FallahJordanUlichney},     tradeoffs between statistical inference and distributed communication \cite{ZhangDuchiJordanWainwright, GargMaNguyen,BravermanGargMaNguyenWoodruff}, and more.
    In this vein, our results are most similar to the recent work \cite{ChakrabortyLuoBarber} which studies stability-accuracy tradeoffs (although for a different notion of stability); yet, we emphasize that \cite{ChakrabortyLuoBarber} studies tradeoffs between the rates of convergence of stability and accuracy, whereas we study the tradeoffs in the optimal pre-factor for models with parametric rates of convergence.

    From the geometric point of view, our results further develop the philosophy from \cite{WCR} that different geometries on the space of probability measures allow one to control different notions of instability.
    This is an extension of the perspective of information geometry (see \cite{Amari} and the references therein), whereby the Cram\'er--Rao lower bound admits a geometric interpretation in terms of the Hellinger and Fisher--Rao metrics, and was initially applied to the Wasserstein geometry in \cite{WIM,AmariMatsuda,WCR_achieve,WCR}.
    Our work may be interpreted as an instance of this paradigm with respect to the Hellinger--Kantorovich geometry which was introduced in \cite{Liero_2016,LieroMielkeSavare,ChizatPeyreSchmitzerVialard} and studied in many subsequent works (e.g., \cite{CaiChengSchmitzerThorpe2022,LieroMielkeSavare2023}).
    The Hellinger--Kantorovich geometry is also closely related to the concept of unbalanced optimal transport, which has been used in various problems in statistics and machine learning \cite{ChapelAlayaGasso,FatrasSejourneFlamaryCourty,YanWangRigollet}.

\section{Background on Hellinger--Kantorovich Geometry}\label{sec:HK-background}

In this section we introduce the geometric framework used throughout the paper, namely the family of Hellinger--Kantorovich ($\HK$) geometries; these geometries describe the evolution of probability
measures through both spatial displacement of mass (called the transport mechanism from the Wasserstein geometry) and the creation and destruction
of mass (called the source mechanism from the Hellinger geometry), and hence they interpolate between the Wasserstein and Hellinger geometries.

We review the dynamic formulation of the $\HK$ metric, its relation to the pure transport and pure source cases, and the notion of absolute continuity of a trajectory needed in the sequel. 
Some statements below are standard facts from the $\HK$ literature, while others are technical extensions tailored to the statistical applications in this paper. 
We refer the reader to \cite{LieroMielkeSavare,Fleissner,ChizatPeyreSchmitzerVialard,Liero_2016} for further background.

\begin{remark}[Hellinger--Kantorovich versus Wasserstein--Fisher--Rao geometries]\label{rem:WFR}
    Another metric which interpolates between the Hellinger and Wasserstein metrics is the \textit{Wasserstein--Fisher--Rao (WFR) metric} of \cite{ChizatPeyreSchmitzerVialard}, which is closely related to, but slightly different from the Hellinger--Kantorovich metric.
    Indeed, while the use of ``Kantorovich'' versus ``Wasserstein'' is only cosmetic \cite[Chapter~6, Bibliographic Notes]{Villani}, the terms ``Hellinger'' and ``Fisher--Rao'' refer to different geometries, but which have the same infinitesimal structure for probability measures and hence lead to the same lower bounds; see \cite[Subsection~2.2]{WCR} for a discussion.
    As such, the results in this paper could equivalently be written in terms of WFR, rather than HK, geometry. 
\end{remark}

We follow standard notation throughout the paper.
For example, $\Pcal(\R^m)$ denotes the space of Borel probability measures on $\R^m$.
Also, $\Pcal_2(\R^m)\subseteq \Pcal(\R^m)$ denotes the subset of probability measures with finite variance and $\Pcal_{2,\textnormal{ac}}(\R^m)\subseteq \Pcal_2(\R^m)$ denotes the subset with finite variance and which are absolutely continuous with respect to the Lebesgue measure of $\R^m$.
We also write $\Mcal_{+}(\R^m)$ for the space of non-negative finite Borel measures on $\R^m$.
For function spaces, we write $C_c^{\infty}=C_c^{\infty}(\R^m)$ for the space of infinitely differentiable, compactly supported functions from $\R^m\to\R$, and we write $C^1$ for the space of continuously differentiable functions from $\R^m\to\R$ and occasionally from $\R^m\to\R^k$ by slight abuse of notation.

\subsection{Transport-Source Equation}

The starting point of the $\HK$ geometry is a dynamic description of how a measure changes over time, which we formulate through a certain partial differential equation (PDE).

To set this up,
fix $0\le \lambda\le 1$ and $m\in\N$.
Let us say that a path $\{\mu_t\}_{0\le t\le 1}$ of finite nonnegative measures on $\R^m$, a path $\{\xi_t\}_{0\le t\le 1}$ of functions $\R^m\to\R$, and a path $\{v_t\}_{0\le t\le 1}$ of functions $\R^m\to\R^m$ \emph{satisfy the transport-source equation} if we have
\begin{equation*}\label{eqn:HK-differential}
\partial_t \mu_t+(1-\lambda)\Div\!\big(v_{t}\,\mu_t\big)
=\lambda\,\xi_{t}\,\mu_t\tag{$\HK_{\lambda}$}
\end{equation*}
in the weak sense; that is, for every $f\in C_{c}^{\infty}(\R^m)$ (i.e., infinitely differentiable and compactly supported functions), the map $t\mapsto\int_{\R^m}f\diff\mu_t$ is absolutely continuous on $[0,1]$, and we have
\begin{equation}\label{eqn:HK-integral}
\frac{\diff}{\diff t}\int_{\R^m}f\diff \mu_{t} = (1-\lambda)\int_{\R^m}(\nabla f)^{\top}v_{t}\diff \mu_t + \lambda\int_{\R^m}f\,\xi_{t}\diff \mu_t
\end{equation}
for Lebesgue a.e. $t \in [0,1]$.

The interpretation of this equation is that the deformation of the measure $\mu_t$ is the result of two mechanisms: existing mass moves spatially according to the vector field $v_t$, and mass is created or destroyed according to the scalar field $\xi_t$. The parameter $0\le \lambda\le 1$ controls the relative weight given to the transport and source mechanisms. As such, other authors may refer to this PDE using terms like the \textit{transport--reaction equation}, \textit{continuity equation with reaction}, or similar. 
Formally, using smooth, compactly supported approximations of the constant function $f\equiv1$ as test functions in \eqref{eqn:HK-integral} yields
\begin{equation}\label{eqn:mass-change}
    \frac{\diff}{\diff t}\mu_t(\R^m)
=
\lambda\int_{\R^m}\xi_t\diff\mu_t,
\end{equation}
so $\{\xi_t\}_{0\le t\le 1}$ controls the evolution of the total mass of $\{\mu_t\}_{0\le t\le 1}$. 
In particular, note that, in principle, $\mu_t$ for $0<t<1$ need not be a probability measure, even though we will always assume $\mu_0$ and $\mu_1$ are probability measures.

We also give names to the extremal cases (i.e., $\lambda\in\{0,1\}$) of \eqref{eqn:HK-differential}.
That is, we say that a pair $\{\mu_t\}_{0\le t\le 1}$ and $\{v_t\}_{0\le t\le 1}$ \emph{satisfies the continuity equation} if we have
\begin{equation*}\label{eqn:CE-differential}
\partial_t \mu_t+\Div\!\big(v_{t}\,\mu_t\big)
=0 \tag{$\HK_{0}$}
\end{equation*}
and we say that a pair $\{\mu_t\}_{0\le t\le 1}$ and $\{\xi_t\}_{0\le t\le 1}$  \emph{satisfies the source equation} if we have
\begin{equation*}\label{eqn:RE-differential}
\partial_t \mu_t=\xi_{t}\,\mu_t\tag{$\HK_{1}$}
\end{equation*}
both in the weak sense; see \cite[Section~8.1]{AmbrosioGigliSavare} and \cite[Section~2.2]{GallouetMonsaingeon2017}, respectively.
These extremal cases isolate the two components of the $\HK$ geometry:
at \(\lambda=0\), the only mechanism is transport, hence corresponding to the Wasserstein geometry; at \(\lambda=1\), the only mechanism is the creation and destruction of mass, hence corresponding to the Hellinger geometry.

\subsection{Metric Space Structure}

The transport-source equation (\ref{eqn:HK-differential}) describes admissible paths of measures in the $\HK$ geometry. It turns out that one can define a corresponding $\HK$ metric by assigning an energy to such paths and then minimizing this energy over all paths connecting two given endpoint measures. 

To state this, fix $0\le \lambda\le 1$ and introduce the following quantity for a locally Lipschitz function $\phi:\R^m\to\R$ and an absolutely continuous finite nonnegative measure $\mu$ on $\R^m$, or for $\phi\in C_c^\infty(\R^m)$ and an arbitrary finite nonnegative measure $\mu$:
\begin{equation}\label{eqn:def-Sobolev}
    \|(\phi,\grad\phi)\|^2_{\mathcal{H}_{\lambda}^1(\mu)}:=\int_{\R^m}(\lambda|\phi|^2 + (1-\lambda)\|\nabla\phi\|^2)\diff \mu;
\end{equation}
in later discussion, we may abuse the notation and write \(\|\phi\|_{\mathcal{H}_{\lambda}^1(\mu)} = \|(\phi,\grad\phi)\|_{\mathcal{H}_{\lambda}^1(\mu)}\).
We note that this is a norm for $0<\lambda\le 1$, but only a seminorm for $\lambda=0$.
Indeed, it fails to be positive definite for $\lambda=0$ since $\phi$ appears only through $\nabla \phi$.
Then, for $0\le \lambda\le 1$, we can define the \textit{$\HK_{\lambda}$ tangent space} at $\mu\in\mathcal M_+(\R^m)$ by
\begin{equation}\label{eqn:def-HK-tangent-space}
    \Tan_{\mu}^{\HK_{\lambda}}(\mathcal M_+(\R^m)) := \overline{\{(\phi,\nabla\phi): \phi\in C_c^{\infty}(\R^m)\}}^{\|\cdot\|_{\mathcal{H}_{\lambda}^1(\mu)}}.
\end{equation}
Lastly, note that the $\HK_{\lambda}$ tangent space coincides with the Wasserstein and Hellinger tangent spaces when $\lambda=0$ and $\lambda=1$, respectively; see Lemma~\ref{lem:HK-tangent-identification} in the Appendix for further detail.
Then:

\begin{definition}\label{def:dynamic-cstr}
    For $0\le \lambda \le 1$ and $\nu_0,\nu_1 \in \mathcal{P}_2(\mathbb{R}^m)$,
 we define $\dHKlambda$ via
    \begin{equation}\label{eqn:HK-dynamic}
        \dHKlambda^2(\nu_0,\nu_1):= \inf\left\{\int_{0}^{1}\int_{\mathbb{R}^m}\left(\lambda|\xi_t|^2+(1-\lambda)\|v_t\|^2\right)\diff \mu_t\diff t: \begin{matrix}
            (\xi_t,v_t)\in\Tan_{\mu_t}^{\HK_{\lambda}}(\mathcal M_+(\R^m)) \textnormal{ for a.e. } t,\\
            \{(\mu_t,\xi_t,v_t)\}_{0\le t\le 1}\textnormal{ solves \eqref{eqn:HK-differential},} \\ 
            \textnormal{and }\mu_0=\nu_0,\mu_1=\nu_1
        \end{matrix}\right\},
    \end{equation}
    called the \emph{Hellinger--Kantorovich metric} on  $\Pcal_2(\R^m)$.
\end{definition}

The formulation above explicitly imposes a constraint on the scalar field and vector field solving the transport-source equation (\ref{eqn:HK-differential}); roughly speaking, the condition $(\xi_t,v_t)\in\Tan_{\mu_t}^{\HK_{\lambda}}(\mathcal M_+(\R^m))$ states that there exists a potential
$\phi_t:\R^m\to\R$ such that $\xi_t =\phi_t$ and $v_t=\nabla \phi_t$ hold $\mu_t$ a.e., i.e., the vector field governing the transport of mass equals the gradient of the scalar field governing the change in mass.
(Interestingly, the value of the metric is unchanged if we remove this constraint and allow for general $v_t$ and $\xi_t$ subject to some regularity; see Lemma~\ref{lem:dynamic-free} for details.)
However, rigorously establishing this constraint can be difficult since $\Tan_{\mu}^{\HK_{\lambda}}(\mathcal M_+(\R^m))$ is implicitly defined via a closure operation; it turns out that often it is easier to work with the following function space:

\begin{definition}\label{def:Sobolev}
    For $\mu\in\Pcal_{2, \textnormal{ac}}(\R^m)$, define the function space
    \begin{equation}\label{eqn:def-Sobolev-space}
        \mathcal{H}^1_{\lambda}(\mu):=\left\{\phi:\text{supp}(\mu)\to\R:
        \begin{aligned}
            &\text{there exists }\widetilde \phi\in\operatorname{Lip}_{\mathrm{loc}}(\mathbb R^m)
            \text{ such that }\tilde \phi=\phi\ \mu\text{-a.e.},\\
            &\textnormal{and }\int_{\R^m}(\lambda|\tilde{\phi}|^2 + (1-\lambda)\|\nabla\tilde{\phi}\|^2)\diff \mu<\infty
        \end{aligned}\right\},
    \end{equation}
    for all $0\le \lambda\le 1$.
    (As we explain in Remark~\ref{rem:Sobolev-gradient}, the integral in~\eqref{eqn:def-Sobolev-space} is indeed well-defined.)
\end{definition}

Indeed, the next result shows that $\mathcal{H}_\lambda^1(\mu)$ provides a simple sufficient condition for the pair $(\phi,\nabla\phi)$ to lie in $\Tan_\mu^{\HK_\lambda}\bigl(\mathcal M_+(\mathbb R^m)\bigr)$; we will use it throughout the remainder of the paper.

\begin{proposition}\label{prop:Sobolev}
    For $\mu\in\Pcal_{2,\textnormal{ac}}(\R^m)$, we have \(\left\{(\phi,\nabla \phi):\phi\in\mathcal H_{\lambda}^1(\mu)\right\}\subseteq\Tan_\mu^{\HK_\lambda}\bigl(\mathcal M_+(\mathbb R^m)\bigr)\).
\end{proposition}

Lastly, it is useful to compare the metric $\dHKlambda$ with other metrics, specifically the Hellinger and Wasserstein metrics, denoted $\dH$ and $\dW$ respectively.
The next result shows a bound between all three metrics; we believe this result is well-known to experts, but we give a precise statement and a proof (see Appendix~\ref{app:ProofsSection2}) for completeness.

\begin{lemma}\label{lem:metric-bounds}
    We have $d_{\HK_{0}} =\dW$ and $d_{\HK_1} = \dH$, as well as
    \begin{equation*}
        d_{\HK_{\lambda}}(\nu_0,\nu_1) \le \min\left\{\frac{1}{\sqrt{\lambda}} \dH(\nu_0,\nu_1),\frac{1}{\sqrt{1-\lambda}}\dW(\nu_0,\nu_1)\right\}
    \end{equation*}
    for all $\nu_0,\nu_1$ and $0< \lambda< 1$.
\end{lemma}

\subsection{Absolutely Continuous Curves}

The dynamic formulation of the $\HK$ metric in Definition~\ref{def:dynamic-cstr} naturally leads to a notion of absolutely continuous curves.
Such curves are important for this work because the statistical models considered later will be viewed as finite-dimensional submanifolds in the space of measures endowed with the $\HK$ metric, whose requisite regularity will involve absolute continuity of suitable one-dimensional submanifolds (i.e., trajectories).

Recall that a curve $\gamma:[0,1]\to \mathfrak{M}$ in a metric space $(\mathfrak{M},d)$ is called \textit{$p$-absolutely continuous} if there exists a function $g\in L^p[0,1]$ such that we have $d(\gamma(s),\gamma(t))\le \int_{s}^{t}g(u)\diff u$ for all $0 \le s\le t\le 1$.
We write $\AC_{\lambda}^p:=\AC^p(\Pcal_2(\R^m),\dHKlambda)$ for the set of all $p$-absolutely continuous curves with respect to the $\HK$ metric.
Throughout the paper, we will mainly work with the space $\AC_{\lambda}^2$.

One important aspect of \(\AC_\lambda^2\) curves is that they can be characterized precisely as the curves which solve the transport-source equation with transport and source terms that are restricted to lie in the $\HK_\lambda$ tangent space (\cite[Theorems~8.16 and~8.17]{LieroMielkeSavare}, \cite[Proposition~3.1]{Fleissner}).
For the cases $\lambda=1$ and $\lambda=0$, see \cite[Section~2.2]{GallouetMonsaingeon2017} and \cite[Sections~8.3 and~8.4]{AmbrosioGigliSavare}, respectively.

\begin{proposition}\label{prop:ans-cont-unique-solution}
    For $0\le \lambda\le 1$, a curve $\{\mu_t\}_{0\le t\le 1}\subseteq\Pcal_{2,\mathrm{ac}}(\R^m)$ lies in $\AC_{\lambda}^2$ if and only if there exists a jointly measurable pair $(t,x)\mapsto(\xi_t(x),v_t(x))$ with
    \begin{equation}\label{eqn:AC2-finite-energy}
        \int_{0}^{1}\|(\xi_t,v_t)\|^2_{\mathcal{H}_{\lambda}^1(\mu_t)}\diff t=\int_{0}^{1}\int_{\R^m}\left(\lambda|\xi_t|^2+(1-\lambda)\|v_t\|^2\right)\diff\mu_t\diff t<\infty,
    \end{equation}
    such that $(\xi_t,v_t)\in \Tan_{\mu_t}^{\HK_\lambda}(\mathcal M_+(\R^m))$ for Lebesgue a.e. $0\le t\le 1$ and $\{(\mu_t,\xi_t,v_t)\}_{0\le t\le 1}$ solves~\eqref{eqn:HK-differential}.
    Moreover, $\{\xi_t\}_{0\le t\le 1}$ is unique up to $\diff\mu_t\diff t$ null sets if $0< \lambda\le 1$ and $\{v_t\}_{0\le t\le 1}$ is unique up to $\diff\mu_t\diff t$ null sets if $0\le \lambda< 1$.
\end{proposition}

A second important aspect of $\AC_{\lambda}^2$ curves is that any curve which is absolutely continuous for either the pure transport (i.e., Wasserstein) geometry or the pure source (i.e., Hellinger) geometry is automatically absolutely continuous for every intermediate $\HK$ geometry;
this follows easily from the metric comparison in Lemma~\ref{lem:metric-bounds}.

\begin{proposition}\label{prop:abs-cont-HK}
    If $\{\mu_t\}_{0\le t\le1}$ is $2$-absolutely continuous
with respect to at least one of $\dH$ or $\dW$, then it is
$2$-absolutely continuous with respect to $\dHKlambda$
for every $0<\lambda<1$.
In other words, $\AC_0^2\cup \AC_1^2\subseteq\AC_\lambda^2$, for $0<\lambda<1.$
\end{proposition}

In practice, we typically use the results of this subsection to make the following conclusions about a path $\{\mu_t\}_{0\le t\le 1} \subseteq\Pcal_{2,\mathrm{ac}}(\R^m)$.
First, we show that $\{\mu_t\}_{0\le t\le 1}$ is 2-absolutely continuous with respect to at least one of $\dH$ or $\dW$, hence  2-absolutely continuous with respect to $\dHKlambda$ for $0<\lambda<1$ by Proposition~\ref{prop:abs-cont-HK}.
Second, we explicitly find a potential $\phi_t\in \mathcal{H}_{\lambda}^1(\mu_t)$ for each $0\le t\le 1$ such that $\{(\mu_t,\phi_t,\nabla\phi_t)\}_{0\le t\le 1}$ solves the transport-source equation~\eqref{eqn:HK-differential}; by Proposition~\ref{prop:Sobolev} we have $(\phi_t,\nabla\phi_t)\in \Tan_{\mu_t}^{\HK_\lambda}(\mathcal M_+(\R^m))$ for all $0\le t\le 1$, hence Proposition~\ref{prop:ans-cont-unique-solution} guarantees uniqueness.

\subsection{Beyond $C_c^{\infty}$ Test Functions}

The weak formulation of the transport-source equation~\eqref{eqn:HK-integral} is stated using test functions in \(C_c^\infty(\mathbb R^m)\), i.e., infinitely-differentiable and compactly-supported functions. 
However, for the statistical applications to follow, we need to apply the transport-source equation to test functions which are neither once-differentiable nor compactly supported.
Thus, the purpose of this subsection is to record a technical extension of the weak formulation to such test functions, under suitable space-time integrability conditions.

\begin{theorem}\label{thm:loc-Lip-test-funcs}
    Fix $0\le\lambda\le1$, and suppose $\{\mu_t\}_{0\le t\le 1}$ is in $\AC_{\lambda}^2$ where $\{(\xi_t,v_t)\}_{0\le t\le 1}$ is the solution to the transport-source equation~\eqref{eqn:HK-integral} guaranteed by Proposition~\ref{prop:ans-cont-unique-solution}.
    If \(\mu_t\) is absolutely continuous with respect to Lebesgue measure for all $0\le t\le 1$, and
    if \(f:\mathbb R^m\to\mathbb R\) is locally Lipschitz and satisfies
    \begin{itemize}
        \item[(i)] $\int_{\mathbb R^m}|f(x)|\diff \mu_t(x)<\infty$ for all \(t\in[0,1]\),
        \item[(ii)] $\int_0^1\int_{\mathbb R^m}|f(x)|^2\diff \mu_t(x)\diff t<\infty$, and
        \item[(iii)] $\int_0^1\int_{\mathbb R^m}\|\nabla f(x)\|^2\diff \mu_t(x)\diff t<\infty$,

    \end{itemize}
    then the map $t\mapsto \int_{\mathbb R^m} f(x)\diff \mu_t(x)$ is absolutely continuous, and we have
    \[
    \frac{\diff}{\diff t}\int_{\mathbb R^m} f(x)\diff \mu_t(x)
    =
    (1-\lambda)\int_{\mathbb R^m}\nabla f(x)\cdot v_t(x)\diff \mu_t(x)
    +
    \lambda\int_{\mathbb R^m} f(x)\xi_t(x)\diff \mu_t(x)
    \]
    for Lebesgue a.e.~$t \in (0,1)$.
\end{theorem}

This result generalizes and simplifies \cite[Proposition~2.2]{WCR}, which focused only on the pure transport case and whose proof relied on a deep result on probabilistic representations of solutions to the continuity equation \cite[Theorem~8.2.1]{AmbrosioGigliSavare}.
Instead, in the proof of Theorem~\ref{thm:loc-Lip-test-funcs} we consider a simpler approach using a mollification argument.

\section{Pareto Optimal Tradeoffs between Variance and Sensitivity}\label{sec:LB}

In this section we describe the fundamental lower bounds on the pair of variance and sensitivity that can be achieved by any unbiased estimator in a sufficiently regular statistical model.
Throughout, we consider a statistical model $\Pcal=\{P_{\theta}:\theta\in\Theta\}\subseteq\Pcal_{2,\textnormal{ac}}(\R^d)$ where $\Theta\subseteq\R^p$ is an open set and $p\geq 1$.
We let $X_1,\ldots, X_n$ be i.i.d.~samples from $P_{\theta}$ where the goal is to estimate the unknown $\theta\in\Theta$.
More precisely, we define a particular notion of regularity that we will require from our statistical models (Subsection~\ref{subsec:model-regularity}), we state a simultaneous lower bound on the variance and sensitivity of any unbiased estimator (Subsection~\ref{subsec:LB}), and we study various aspects of the finer structure of this lower bound (Subsection~\ref{subsec:VSD}).

The results of this section rely on the preliminaries from Section~\ref{sec:HK-background} on the HK geometry, but the i.i.d.~structure leads us to consider product measures on $(\R^d)^n$, i.e., we will consider probability measures on $\R^m$ for $m=dn$.

\subsection{Regularity Assumptions on the Model}\label{subsec:model-regularity}

We begin by giving a particular notion of a model being regular with respect to the HK geometry; verifying this condition takes work in each example of interest, and at the end of this subsection we discuss this point further.

\begin{definition}\label{def:HK_regular}
    Consider a parametric family $\Pcal=\{P_{\theta}:\theta\in\Theta\}\subseteq\Pcal_{2,\textnormal{ac}}(\R^d)$ where $\Theta\subseteq\R^p$ is an open set.
    For $0\le \lambda\le 1$, we say $\Pcal$ is \emph{$\HK_{\lambda}$ regular} if there exist functions $\phi_{\theta,\lambda,1},\ldots,\phi_{\theta,\lambda,p}:\R^d\to\R$ in $\mathcal{H}_{\lambda}^1(P_{\theta})$, satisfying the following properties for every $\theta\in\Theta$ and $h\in\R^p$ such that $\theta+th\in\Theta$ for all $0\le t\le 1$:
    \begin{itemize}
        \item[(i)] For each $\theta\in\Theta$, $h\in \R^p$, and $1\le j\le p$, the map $t\mapsto \|\phi_{\theta+th,\lambda,j}\|^2_{\mathcal{H}_{\lambda}^1(P_{\theta+th})}$ is locally integrable,
        \item[(ii)] For each $h = (h_1,\ldots, h_p)\in\R^p$, there is a unique weak-sense solution to the $\HK_{\lambda}$ equation 
        \begin{equation}\label{HK-cont-eq}
            \partial_t P_{\theta+th}+(1-\lambda)\Div\!\big(\nabla\phi_{t}\,P_{\theta+th}\big)=\lambda\,\phi_{t}\,P_{\theta+th}
        \end{equation}
       (where uniqueness means the difference between any two solutions has zero $\mathcal{H}_{\lambda}^1(P_{\theta+th})$ norm for almost every $0\le t\le 1$),
        \item[(iii)] For each $h = (h_1,\ldots, h_p)\in\R^p$, the unique solution $\{\phi_t\}_{0\le t\le 1}$ to \eqref{HK-cont-eq} has the form $\phi_t = h_1\phi_{\theta+th,\lambda,1}+\cdots+h_p\phi_{\theta+th,\lambda,p}$. 
    \end{itemize}
    The function $\phi_{\theta,\lambda} = (\phi_{\theta,\lambda,1},\ldots, \phi_{\theta,\lambda,p}):\R^d\to\R^p$ is called the \emph{$\HK_{\lambda}$ score function} of $\Pcal$.
    We furthermore say that $\Pcal$ is \emph{$\HK$ regular} if it is $\HK_{\lambda}$ regular for all $0< \lambda< 1$.
\end{definition}

We now make some remarks about the notion of $\HK$ regularity.
First, we use the term \textit{score function} for $\phi_{\theta,\lambda}$ in the case of general $\lambda$, which agrees with the classical case of $\lambda=1$ where $\phi_{\theta,1}$ coincides with the score function $\dot{\ell}_{\theta}$ but differs from the term \textit{potential} used in \cite{WCR} for the case $\lambda=0$.
Second, regarding notation, we note that, if $p=1$ then we simply write $\phi_{\theta,\lambda}$ in place of $\phi_{\theta,\lambda,1}$, and we will sometimes use parameter names, rather than integers, as the third index (e.g., $\phi_{\theta,\lambda,\mu}$ and $\phi_{\theta,\lambda,\sigma^2}$ in the Gaussian $\Ncal(\mu,\sigma^2)$ Example~\ref{ex:Gaussian-loc-scale} below) when it is useful in context.
Third, we emphasize that the notion of $\HK$ regularity involves the interpolating parameter $\lambda$ ranging over the open interval $(0,1)$, not the closed interval $[0,1]$; this choice is closely related to Proposition~\ref{prop:abs-cont-HK}, which we will use below. Finally, we remark that condition (ii) can be interpreted as a form of G\^ateaux differentiability, while (iii) resembles a Fréchet differentiability assumption. Although this interpretation provides some basic intuition, our notion of regularity is closer to an absolute continuity assumption. A more appropriate definition of differentiability in the Hellinger--Kantorovich sense (DHKS) requires a different approach, which we discuss in Appendix~\ref{app:HK-static-cone}. We prefer to work with Definition~\ref{def:HK_regular} in the sequel, as finding the optimal transport-growth pairs needed in the definition of DHKS is in general difficult, even for the representative statistical models explored here; see the discussion right after Definition~\ref{def:DHKS}. 

\begin{remark}[Centering conventions for $\HK$ score functions]\label{remark:center-conv}
    While the $\HK_\lambda$ score function $\phi_{\theta,\lambda}$ for $0<\lambda\le 1$ satisfies $\Ebb_\theta[\phi_{\theta,\lambda}(X)]=0$ for almost every \(\theta\), the $\HK_0$ potential is determined only up to addition of a function whose gradient vanishes $P_\theta$-almost everywhere (since $\phi_{\theta,0}$ only appears through its gradient $\nabla\phi_{\theta,0}$ in equation~\eqref{HK-cont-eq}).
    Thus, for $\lambda=0$, we always assume that $\phi_{\theta,\lambda}$ has been chosen to satisfy $\Ebb_\theta[\phi_{\theta,0}(X)]=0$, which allows for the interpretation that $\phi_{\theta,\lambda}\to \phi_{\theta,0}$ as $\lambda\to 0$.
\end{remark}

Verifying $\HK$ regularity involves finding the unique solution to the family of PDEs in~\eqref{HK-cont-eq}, which involves both an existence and uniqueness statement.
For $p=1$, we give the following result guaranteeing uniqueness under very mild assumptions; indeed, recall that the notions of differentiability in quadratic mean (DQM) \cite{vanDerVaart} and differentiability in the Wasserstein sense (DWS) \cite{WCR} are easy to verify.

\begin{proposition}\label{prop:HK-reg-suff}
    Suppose a model $\Pcal=\{P_{\theta}:\theta\in\Theta\}\subseteq\Pcal_{2,\textnormal{ac}}(\R^d)$ with $\Theta\subseteq\R$ satisfies at least one of the following conditions:
    \begin{itemize}
        \item[(i)] $\Pcal$ is DQM with score function $\dot{\ell}_{\theta}$, and the map $\theta\mapsto \|\dot{\ell}_{\theta}\|^2_{L^2(P_{\theta})}=\|\dot{\ell}_{\theta}\|^2_{\mathcal{H}_{1}^1(P_{\theta})}$ is locally integrable.
        \item[(ii)] $\Pcal$ is DWS, with potential $\phi_{\theta}$, and the map $\theta\mapsto \|\nabla\phi_{\theta}\|^2_{L^2(P_{\theta})}=\|\phi_{\theta}\|^2_{\mathcal{H}_{0}^1(P_{\theta})}$ is locally integrable.
    \end{itemize}
    Then, for each $0<\lambda<1$ and $h\in\R$, there is at most one weak-sense solution to~\eqref{HK-cont-eq}.
\end{proposition}

\begin{remark}
    In the case of a multivariate parameter ($p\ge 2$), one can still conclude that $\{P_{\theta+th}\}_{0\le t\le 1}$ is absolutely continuous with respect to $d_{\HK_{\lambda}}$ for all $0<\lambda<1$, but the potential $\{\phi_t\}_{0\le t\le 1}$ need not be expressible as a linear combination of suitable coordinatewise potentials.
    In other words, $\Pcal$ can fail to be $\HK$ regular since condition (iii) in Definition~\ref{def:HK_regular} can fail.
\end{remark}

For existence of solutions to the PDE~\eqref{HK-cont-eq}, we typically require further work in each example of interest.
Following the standard approach in PDE literature, we first come up with some Ansatz for a solution and then check rigorously that it has the desired properties; Proposition~\ref{prop:HK-reg-suff} guarantees that this approach is valid.
However, coming up with such an Ansatz is non-trivial and requires various ODE and PDE techniques, together with reductions related to symmetries of the model \cite{OlverRosenau1986}.
We illustrate this strategy in the following simple examples, which we revisit throughout the paper.

\begin{example}[Laplace location]\label{ex:Laplace-reg}
    Suppose $P_{\theta}$ is the probability measure on $\R$ whose density with respect to Lebesgue measure is $p_\theta(x)=\frac12 e^{-|x-\theta|}$, for all $x \in \R$, where $\theta\in\R$.
    It is classical that this model is DQM and that its score function is $\phi_{\theta,1}(x)=\dot{\ell}_{\theta}(x) = \textnormal{sgn}(x-\theta)$, where $\textnormal{sgn}(\cdot)$ is the sign function.
   By Proposition~\ref{prop:HK-reg-suff} we know that, for each $0<\lambda<1$, the $\HK_{\lambda}$ equation~\eqref{HK-cont-eq} for $\{P_{\theta+th}\}_{0\le t\le 1}$ has at most one solution for every $\theta,h\in\R$.
    Solving this explicitly leads to
    \begin{equation}\label{eq:phi_laplace}
    \phi_{\theta,\lambda}(x)
    =
    \frac{1}{\lambda}\,\operatorname{sgn}(x-\theta)
    \Bigl(1-e^{-\alpha_\lambda |x-\theta|}\Bigr)\qquad\textnormal{ where }
    \quad\alpha_\lambda
    :=
    \frac{1}{2}\left(\sqrt{\frac{1+3\lambda}{1-\lambda}}-1\right).
    \end{equation}
    The functions in~\eqref{eq:phi_laplace} satisfy the conditions of Definition~\ref{def:HK_regular}, so the model is $\HK$ regular.
    The model is also $\HK_0$ regular, with centered potential $\phi_{\theta,0}(x)=x-\theta$; see Appendix \ref{app:laplace-verification}.
\end{example}

\begin{remark}[Laplace location at $\lambda=1$]
    The classical score function $\phi_{\theta,1}(x)=\operatorname{sgn}(x-\theta)$ belongs to $L^2(P_\theta)$ but not to $\mathcal H_1^1(P_\theta)$, since it has no locally Lipschitz extension. Thus the Laplace location model is not $\HK_1$ regular in the sense of Definition~\ref{def:HK_regular}.
    The score nevertheless satisfies the pure source equation, as verified in Appendix~\ref{app:laplace-verification}, and the MLE for \(\theta\) is given by the sample median.
\end{remark}

\begin{example}[Uniform scale]\label{ex:unif-reg}
    Suppose $P_{\theta}$ is the uniform distribution on the interval $[0,\theta]$ for all $\theta>0$.
    As a scale family (see \cite[Example~5.8]{WCR}), this model is DWS and its $\HK_0$ potential is quadratic, $\phi_{\theta,0}(x)=(x^2-\theta^2/3)/(2\theta)$.
    Consequently, this model is $\HK_0$ regular and $\HK$ regular, since for each $0\le\lambda<1$ the $\HK_{\lambda}$ equation for $\{P_{\theta+th}\}_{0\le t\le 1}$ has a unique solution for every $\theta>0$ and $h\in\R$ satisfying $\theta+h>0$;
    for $0<\lambda<1$, we can directly compute that the solution must be
    \begin{equation}\label{eq:phi_uniform_interior}
    \phi_{\theta,\lambda}(x)
    =
    \frac{\cosh(\beta_\lambda x)}
    {\sqrt{\lambda(1-\lambda)}\,\sinh(\beta_\lambda\theta)}
    -
    \frac{1}{\lambda\theta}
    \qquad\textnormal{ where }
    \quad \beta_\lambda:=\sqrt{\frac{\lambda}{1-\lambda}}.
    \end{equation}
    The model is however not $\HK_1$ regular; see Appendix~\ref{app:uniform-verification}.
\end{example}

\begin{example}[Gaussian correlation family]\label{ex:Gaussian-corr-reg}
    Suppose $P_{\theta} = \Ncal(0,\Sigma_{\theta})\in\Pcal_2(\R^2)$ where
    \begin{equation}\label{eqn:Gaussian-corr-matrix-def}
        \Sigma_\theta := \begin{pmatrix}
            1 & \theta \\ \theta & 1
        \end{pmatrix}
    \end{equation}
    and $\theta\in \Theta:=(-1,1)$.
    In other words, $P_{\theta}$ is a centered bivariate Gaussian distribution, where both coordinates have marginal variance equal to 1 and the correlation between the coordinates is $\theta$.
    As shown in \cite[Example~5.13]{WCR}, this model is DWS with quadratic potential. It is also $\HK_\lambda$ regular for all $0\le\lambda\le1$; see Appendix~\ref{app:gaussian-correlation-verification}.
    Solving the $\HK_{\lambda}$ equation leads to
    \begin{equation}\label{eq:phi_gauss_corr_uv}
    \phi_{\theta,\lambda}(x,y)
    =
    \frac{u^2-(1+\theta)}{2(1+\theta)\bigl(2-\lambda(1-\theta)\bigr)}
    -
    \frac{v^2-(1-\theta)}{2(1-\theta)\bigl(2-\lambda(1+\theta)\bigr)},
    \end{equation}
    where $u=(x+y)/\sqrt 2$ and $v=(x-y)/\sqrt 2$.
    We note that this change of coordinates $(x,y)\mapsto (u,v)$ has a statistical interpretation: If $(X,Y)\sim P_{\theta}$, then $(U,V)$ are independent Gaussian random variables with variances equal to $1+\theta$ and $1-\theta$ respectively; thus, the problem of estimating the correlation between $X$ and $Y$ can be equivalently restated as the problem of estimating a parameter of the marginal variances of $U$ and $V$.
\end{example}

\begin{example}[Pareto index]\label{ex:Pareto-reg}
    Suppose $P_{\theta}$ is the probability measure on $\R$ whose density with respect to Lebesgue measure is $p_\theta(x)=\theta x^{-(\theta+1)}\mathbf 1_{\{x\ge 1\}}$ for $x\in\R$ and $\theta\in\Theta:=(2,\infty)$.
    It is well-known that this model is DQM, and it is shown to be DWS in \cite[Example~5.9]{WCR} whenever $\theta>2$. As verified in Appendix~\ref{app:pareto-verification}, this model is $\HK_\lambda$ regular for all $0\le\lambda\le1$. While existence and uniqueness of weak-sense solutions to~\eqref{HK-cont-eq} for $0<\lambda<1$ follow from Proposition~\ref{prop:HK-reg-suff}, it turns out that there is no elementary closed-form expression for these solutions. More precisely, the transport-source equation reduces to an inhomogeneous second-order ODE whose associated homogeneous equation can be transformed into a modified Bessel equation. At the extremal cases $\lambda\in\{0,1\}$, we have
    \begin{equation}\label{eq:phi_pareto_potentials}
    \phi_{\theta,0}(x)
    =
    -\frac{x^2\log x}{2\theta}
    +
    \frac{x^2}{4\theta}
    +
    \frac{4-\theta}{4(\theta-2)^2}\qquad \textnormal{and}\qquad\phi_{\theta,1}(x)=\frac1\theta-\log x
    \end{equation}
    for all $x\ge 1$.
\end{example}

In the case of a multivariate parameter ($p\ge 2$), we cannot directly use Proposition~\ref{prop:HK-reg-suff} to show $\HK$ regularity.
However, we can sometimes directly verify the conditions of Definition~\ref{def:HK_regular} with slightly more work, as in the following.

\begin{example}[Gaussian mean-variance]\label{ex:Gaussian-loc-scale}
    Suppose $\Theta := \R\times (0,\infty)$ with $\theta = (\mu, \sigma^2)$ and that $P_{\theta}$ is the univariate Gaussian distribution with mean $\mu$ and variance $\sigma^2$.
    We claim that this model is $\HK_\lambda$ regular for all $0\le\lambda\le1$, and that its \(\HK_\lambda\) score function is given by
    \begin{equation}\label{eq:phi_gauss_ls_vector}
    \phi_{\theta,\lambda}(x)
    =
    \begin{pmatrix}
    \phi_{\theta,\lambda,\mu}(x)\\[1mm]
    \phi_{\theta,\lambda,\sigma^2}(x)
    \end{pmatrix}
    =
    \begin{pmatrix}
    \dfrac{x-\mu}{(1-\lambda)+\lambda\sigma^2}\\[3mm]
    \dfrac{(x-\mu)^2-\sigma^2}
    {2\sigma^2\bigl(2(1-\lambda)+\lambda\sigma^2\bigr)}
    \end{pmatrix}
    \end{equation}
    for $0\le \lambda\le 1$.
    For each fixed $0\le\lambda\le1$, the quadratic form and smooth parameter dependence of $\phi_{\theta,\lambda}$, together with the finiteness of Gaussian moments, give the local integrability required in condition (i) of Definition~\ref{def:HK_regular}.
    For each $\theta\in\Theta$ and sufficiently small $h=(h_1,h_2)\in\R^2$, the scalar function $\phi_t=h_1\phi_{\theta+th,\lambda,\mu}+h_2\phi_{\theta+th,\lambda,\sigma^2}$ is the unique weak-sense solution to the $\HK_{\lambda}$ equation for $\{P_{\theta+th}\}_{0\le t\le 1}$, verifying conditions (ii) and (iii); the calculation is given in Appendix~\ref{app:gaussian-location-scale-verification}.
\end{example}

For more complex models, establishing $\HK$ regularity and computing the $\HK$ score functions can be substantially more difficult.
In such cases, one might instead seek numerical approximations using computational solvers for the transport-source equation, still relying on Proposition~\ref{prop:HK-reg-suff} (when $p=1$) to guarantee that a unique solution is unique if exists.

\subsection{The Hellinger--Kantorovich Cram\'er--Rao Lower Bound}\label{subsec:LB}

Next, we establish fundamental lower bounds on the pair of variance and sensitivity of unbiased estimators.
The main result in this subsection is an adaptation of the Cram\'er--Rao lower bound to the Hellinger--Kantorovich geometry (Theorem~\ref{thm:HK-CR-bound}).
It reduces to the classical Cram\'er--Rao lower bound and the Wasserstein--Cram\'er--Rao lower bounds at the extremal cases $\lambda=1$ and $\lambda=0$, respectively, and, for intermediate values $0<\lambda<1$, it allows us to describe optimal tradeoffs between variance and sensitivity.

We begin by defining a concept of information in the Hellinger--Kantorovich geometry, which can be separated into a convex combination of a source term and a transport term according to the interpolation parameter appearing in the $\HK_\lambda$ geometry.

\begin{definition}\label{def:HK-info}
    Fix $0\leq \lambda\leq 1$. Suppose
    $\Pcal=\{P_{\theta}:\theta\in\Theta\}\subseteq\Pcal_{2,\mathrm{ac}}(\R^d)$
    is $\HK_\lambda$ regular. We use the convention that
    \[
        D\phi_{\theta,\lambda}(x)\in\R^{p\times d},
        \qquad
        \big(D\phi_{\theta,\lambda}(x)\big)_{\ell,\cdot}
        =
        \nabla_x\phi_{\theta,\lambda,\ell}(x)^\top .
    \]
    Then, we define its \emph{$\HK_{\lambda}$ information matrix} via
    \begin{equation*}
        I_{\lambda}(\theta)
        :=
        \lambda\Ebb_{\theta}
        \left[
            \phi_{\theta,\lambda}(X)
            (\phi_{\theta,\lambda}(X))^{\top}
        \right]
        +
        (1-\lambda)\Ebb_{\theta}
        \left[
            D\phi_{\theta,\lambda}(X)
            (D\phi_{\theta,\lambda}(X))^{\top}
        \right].
    \end{equation*}
    Whenever $I_\lambda(\theta)$ is invertible, define $V_{\lambda}(\theta)$ and $S_{\lambda}(\theta)$ by the following formulas, each provided that its corresponding expectation is finite:
    \begin{align*}
        V_{\lambda}(\theta)
        &:=
        (I_{\lambda}(\theta))^{-1}
        \Ebb_{\theta}
        \left[
            \phi_{\theta,\lambda}(X)
            (\phi_{\theta,\lambda}(X))^{\top}
        \right]
        (I_{\lambda}(\theta))^{-1},\\
        S_{\lambda}(\theta)
        &:=
        (I_{\lambda}(\theta))^{-1}
        \Ebb_{\theta}
        \left[
            D\phi_{\theta,\lambda}(X)
            (D\phi_{\theta,\lambda}(X))^{\top}
        \right]
        (I_{\lambda}(\theta))^{-1}.
    \end{align*}
    By construction, \(\lambda V_\lambda(\theta)+(1-\lambda)S_\lambda(\theta) = (I_\lambda(\theta))^{-1}\).
    For a locally Lipschitz estimator $T_n:(\R^d)^n\to\R^p$ of $\theta$, we define its variance and sensitivity matrices by
    \begin{equation*}
        \Var_\theta(T_n):=\Ebb_\theta\left[(T_n-\theta)(T_n-\theta)^\top\right]\qquad\textnormal{and}\qquad
        \Sen_\theta(T_n):=\Ebb_\theta\left[\sum_{i=1}^n D_{x_i}T_n(D_{x_i}T_n)^\top\right],
    \end{equation*}
    where $D_{x_i}T_n\in\R^{p\times d}$ is the Jacobian of $T_n$ with respect to $x_i$.
\end{definition}

We also introduce the following class of estimators.

\begin{definition}\label{def:estims-class}

    Fix a model $\Pcal=\{P_{\theta}:\theta\in\Theta\}\subseteq\Pcal_{2,\textnormal{ac}}(\R^d)$, and consider the class of estimators
    \begin{equation*}
        \estims_n := \left\{T_n:(\R^d)^n\to\R^p:\begin{matrix}
             T_n\textnormal{ is locally Lipschitz and unbiased for $\theta$, and}\\
             \theta\mapsto \Var_{\theta}(T_n) \textnormal{ and }\theta\mapsto \Sen_{\theta}(T_n) \textnormal{ are continuous}
        \end{matrix}\right\}.
    \end{equation*}
    We emphasize that unbiasedness is required to hold for all $\theta\in\Theta$.
\end{definition}

Our main result involves the $\HK_{\lambda}$ information and is given as follows.

\begin{theorem}[Hellinger--Kantorovich Cram\'er--Rao lower bound]
\label{thm:HK-CR-bound}
    Suppose $0\le\lambda\le1$,
    $\Pcal=\{P_\theta:\theta\in\Theta\}
    \subseteq\Pcal_{2,\textnormal{ac}}(\R^d)$
    is $\HK_\lambda$ regular with $\Theta\subseteq\R^p$ open, and that the maps $\vartheta\mapsto \Ebb_\vartheta[\phi_{\vartheta,\lambda}(X)\phi_{\vartheta,\lambda}(X)^\top]$ and $\vartheta\mapsto\Ebb_\vartheta[D\phi_{\vartheta,\lambda}(X)D\phi_{\vartheta,\lambda}(X)^\top
        ]$ are continuous on $\Theta$.
        For any $T_n\in\estims_n$ and $\theta\in\Theta$ such that $I_\lambda(\theta)$ is invertible, we have
    \begin{equation}\label{HK_CRLB}
        \lambda\Var_\theta(T_n)
        +(1-\lambda)\Sen_\theta(T_n)
        \succeq
        \frac{1}{n}(I_\lambda(\theta))^{-1}.
    \end{equation}
    In particular, for $0 < \lambda < 1$, the following holds:
    \begin{equation*}
        \textnormal{if}\quad \Var_\theta(T_n)\preceq \frac{1}{n}V_\lambda(\theta),
        \quad\textnormal{then}\quad
        \Sen_\theta(T_n)\succeq \frac{1}{n}S_\lambda(\theta);
    \end{equation*}
    \begin{equation*}
        \textnormal{if}\quad \Sen_\theta(T_n)\preceq \frac{1}{n}S_\lambda(\theta),
        \quad\textnormal{then}\quad \Var_\theta(T_n)\succeq \frac{1}{n}V_\lambda(\theta).
    \end{equation*}
\end{theorem}

One can easily compute $V_{\lambda}(\theta)$ and $S_{\lambda}(\theta)$ in our running examples, and hence obtain precise tradeoffs for the variance and sensitivity of sufficiently regular estimators.
While we will visualize the resulting lower bounds in the following Subsection~\ref{subsec:VSD}, we defer the explicit formulas until Subsection~\ref{subsec:HK-score} when we construct estimators achieving these tradeoffs.

The next result provides lower bounds for an appropriate multidimensional ($p\ge 1$) version of problem~\eqref{eqn:Intro-Pareto} from the Introduction.
Note for an unbiased estimator $T_n$ that $\mathrm{tr}(\Var_\theta(T_n))$ equals its MSE, so this result thus provides a fundamental lower bound for the minimization of the MSE within a class of regular and unbiased estimators satisfying a sensitivity upper bound constraint.
It follows directly from Theorem~\ref{thm:HK-CR-bound} via the monotonicity of the trace with respect to the Loewner partial order.

\begin{corollary}
\label{cor:Trace}
Fix $0<\lambda<1$. Under the same assumptions as in Theorem~\ref{thm:HK-CR-bound} we have
 \[  {v_\lambda(\theta)} \leq   \inf_{\substack{T_n  \in \estims_n \\
                    \mathrm{tr}(n \Sen_\theta(T_n))\le s_\lambda(\theta)}}
        \mathrm{tr}(n \Var_\theta(T_n)), \quad \quad  s_\lambda(\theta) \leq   \inf_{\substack{T_n  \in \estims_n \\
                    \mathrm{tr}(n \Var_\theta(T_n))\le v_\lambda(\theta)}}
        \mathrm{tr}(n\Sen_\theta(T_n)), \]
        where $v_\lambda(\theta):= \mathrm{tr}(V_\lambda(\theta))$ and $s_\lambda(\theta):= \mathrm{tr}(S_\lambda(\theta))$.
\end{corollary}

We can also obtain lower bounds for unbiased estimation of a function of the parameter.

\begin{corollary}[Estimation of a function of the parameter]
\label{cor:HK-CR-function}
    Fix $0\le\lambda\le1$. Let $\chi:\Theta\to\R^q$ be $C^1$, and let $T_n:(\R^d)^n\to\R^q$ be an unbiased estimator of $\chi(\theta)$.
    Suppose that $T_n$, $\Pcal$, and its $\HK_\lambda$ score function satisfy the regularity assumptions of Theorem~\ref{thm:HK-CR-bound}.
    Then, for every $\theta\in\Theta$ such that $I_\lambda(\theta)$ is invertible,
    \begin{equation}\label{eq:HK-CR-function}
        \lambda\Var_\theta(T_n)+(1-\lambda)\Sen_\theta(T_n)
        \succeq
        \frac1n D\chi(\theta)(I_\lambda(\theta))^{-1}(D\chi(\theta))^\top,
    \end{equation}
    where $D\chi(\theta)\in\R^{q\times p}$ is the Jacobian of $\chi$.
\end{corollary}

\begin{corollary}[Asymptotic lower bound]
\label{cor:HK-CR-asymptotic}
    Fix $0\le\lambda\le1$, and let $\chi:\Theta\to\R^q$ be $C^1$.
    For each $n$, let $T_n:(\R^d)^n\to\R^q$ and set $\chi_n(\theta):=\Ebb_\theta[T_n]$.
    Suppose that the assumptions of Corollary~\ref{cor:HK-CR-function} hold with $\chi$ replaced by $\chi_n$, and that $D\chi_n(\theta)\to D\chi(\theta)$ as $n\to\infty$ for every $\theta\in\Theta$.
    Then, for every $\theta\in\Theta$ such that $I_\lambda(\theta)$ is invertible,
    \begin{equation}\label{eq:HK-CR-asymptotic}
        \lim_{n\to\infty}
        n\left[\lambda\Var_\theta(T_n)+(1-\lambda)\Sen_\theta(T_n)\right]
        \succeq
        D\chi(\theta)I_\lambda(\theta)^{-1}D\chi(\theta)^\top,
    \end{equation}
    whenever the limit exists.
\end{corollary}

Corollary~\ref{cor:HK-CR-function} can in particular be interpreted as a lower bound for a biased estimator of $\theta$, if we view $\chi(\theta)$ as the expectation of $T_n$. Moreover, under the stated assumptions, Corollary~\ref{cor:HK-CR-asymptotic} shows that biased estimators of $\theta$ may satisfy the same lower bounds of unbiased estimators of $\theta$ asymptotically. For the latter, the Jacobian of the expectation of $T_n$ must converge to the identity matrix. 

\subsection{Structure of Variance--Sensitivity Diagram}\label{subsec:VSD}

In this subsection, we further investigate the lower bounds of Theorem~\ref{thm:HK-CR-bound} by considering the space of pairs $(v,s)\in\R^2$ of asymptotic variance and sensitivity that are achievable by sequences of unbiased estimators in a given model; we revisit Figure~\ref{fig:VSD-cartoon} throughout the subsection to illustrate the concepts herein.
For the remainder of this subsection we focus on the case of a univariate parameter ($p=1$) although we make some remarks about the multivariate case toward the end.

We begin with some notions from multiobjective optimization.
Define the partial order $\preceq$ on $\R^2$ so that we have $(v,s)\preceq (v',s')$ if we have both $v\le v'$ and $s\le s'$. We say that $(v,s)$ \textit{dominates} $(v',s')$ if $(v,s)\preceq (v',s')$ and at least one of the two inequalities is strict.
For sets $\bdry,\diagram\subseteq\R_{\ge 0}^{2}$, we say that $\bdry$ \textit{is a Pareto front for} $\diagram$ if no element of $\bdry\cup \diagram$ dominates any element of $\bdry$.
Notice that we write ``a Pareto front'' rather than ``the Pareto boundary''; this is because our notion does not require $\bdry\subseteq\diagram$ and hence is not uniquely determined from $\diagram$, so we need to distinguish it from existing notions (e.g., in \cite{Ehrgott}).

We will study Pareto fronts for the following sets; recall the class $\estims_n$ defined in Definition~\ref{def:estims-class}.

\begin{definition}\label{def:VS-diagram}

    For a model $\Pcal=\{P_{\theta}:\theta\in\Theta\}\subseteq\Pcal_{2,\textnormal{ac}}(\R^d)$ and each $\theta\in\Theta$, we define
    \begin{equation*}
        \diagram_n(\theta) := \left\{\Big(\Var_{\theta}(T_n),\Sen_{\theta}(T_n)\Big)\in\R^2:T_n\in\estims_n\right\} \qquad \textnormal{and}\qquad \diagram(\theta) := \bigcap_{m\in\N}\overline{\bigcup_{n\ge m}n\diagram_n(\theta)},
    \end{equation*}
    called the \textit{variance-sensitivity diagram} and the \textit{asymptotic variance-sensitivity diagram}, respectively.
\end{definition}

Plainly, the variance-sensitivity diagram represents the set of all variance-sensitivity coordinates that can be achieved by a suitable unbiased estimator in the model $\Pcal$, while the asymptotic variance-sensitivity diagram is a certain set-valued limit of the rescaled variance-sensitivity diagrams; this notion of set-valued limit is the 
\textit{Kuratowski outer limit}, which is a common notion of convergence of unbounded sets in mathematical optimization (e.g., \cite[Chapter~4A]{RockafellarWets}).
See Figure~\ref{fig:VSD-cartoon} for a depiction of the asymptotic variance-sensitivity diagram $\diagram(\theta)$ as the dark green region.
We note that the definition only requires $\Pcal\subseteq \Pcal_{2,\textnormal{ac}}(\R^d)$ in order for the sensitivity to be well-defined, but that more will be deduced under further hypotheses; for example, if $\Pcal$ is DQM then the $v$ coordinate is bounded below by $(I(\theta))^{-1}$ and if $\Pcal$ is DWS then the $s$ coordinate is bounded below by $(J(\theta))^{-1}$, which we display with vertical and horizontal dashed lines, respectively.

In general it is hard to precisely understand the geometry of the set $\diagram(\theta)$ of achievable variance-sensitivity pairs; so we will often instead visualize the set $\tilde{\diagram}(\theta):=\diagram(\theta)+\R_{\ge 0}^2$ which is the result of extending $\diagram(\theta)$ in the up-right direction.
To gain some more understanding of the difference between $\diagram(\theta)$ and $\tilde{\diagram}(\theta)$, we consider models $\Pcal$ in which $P_{\theta}$ satisfies a dimension-free Poincar\'e inequality for all $\theta\in\Theta$, i.e., there exists a constant $C_{\theta}>0$ such that we have $\Var_{\theta}(T_n)\le C_{\theta}\Sen_{\theta}(T_n)$ for all $n\in\N$ and all locally Lipschitz unbiased estimators $T_n$;
this condition holds for all of our running examples  (because they are all log-concave; see \cite[Chapter~4.8]{BakryGentilLedoux}) except the Pareto index family (since it does not have finite exponential moments).
In this setting, it follows easily (Proposition~\ref{prop:Poincare-diagram}) that $\diagram(\theta)$ lies entirely above the line $s=v/C_{\theta}$ in the $(v,s)$ plane, while $\tilde{\diagram}(\theta)$ cannot satisfy any such bound since it always contains infinite horizontal rays.

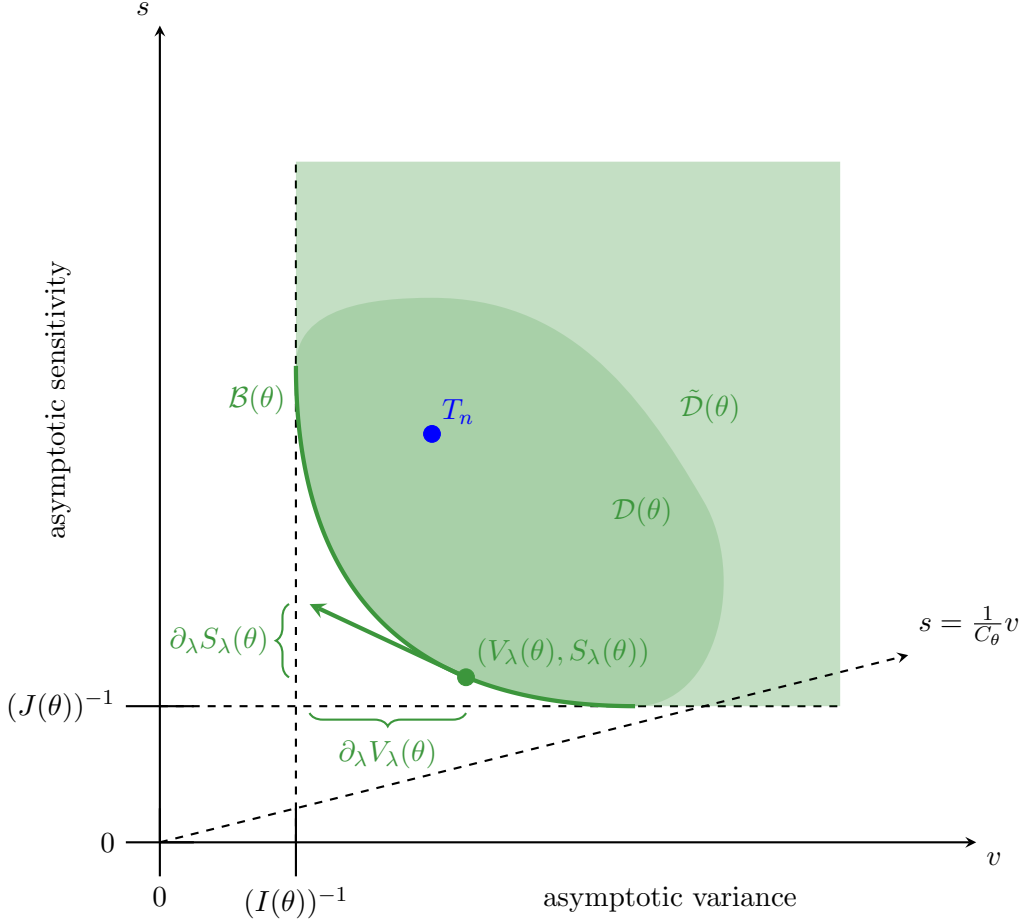
\begin{figure}[h!]
    \centering
    \begin{tikzpicture}[scale=9.]
			\draw[thick, -stealth] (0,0) --  (1.2,0);
            \draw[] (1.2,0) node[below right] {$v$};
			\draw[thick, -stealth] (0,0) -- (0,1.2);
            \draw[] (0,1.2) node[above left] {$s$};
            \draw[thick] (-0.05,0) node[left]{$0$} -- (0.05,0);
			\draw[thick] (0,-0.05) node[below]{$0$} -- (0,0.05);
			\node[below] at (0.75, -0.05) {asymptotic variance};
			\node[rotate=90] at (-0.15, 0.65) {asymptotic sensitivity};
            
            \draw[thick, dashed] (0,0.2) -- (1,0.2);
			\draw[thick, dashed] (0.2,0) -- (0.2,1);
			\draw[thick] (-0.05,0.2) node[left]{$(J(\theta))^{-1}$} -- (0.05,0.2);
			\draw[thick] (0.2,-0.05) node[below]{$(I(\theta))^{-1}$} -- (0.2,0.05);

			\path [fill=mygreen, thick,  opacity=0.3, domain=0:1, samples=100] plot ({0.2+0.5*\x^2}, {0.2+0.5*(\x-1)^2}) -- (1,0.2) -- (1,1) -- (0.2, 1) -- cycle;		
			\draw [color=mygreen, ultra thick,  domain=0:1, samples=100] plot ({0.2+0.5*\x^2}, {0.2+0.5*(\x-1)^2});		
            \draw[color=mygreen] (0.75, 0.6) node[above right] {$\tilde{\diagram}(\theta)$};
            
            \path [fill=mygreen, thick,  opacity=0.2, domain=0:1, samples=100] plot ({0.2+0.5*\x^2}, {0.2+0.5*(\x-1)^2}) to[out=0,in=-60] (0.8,0.5) to[out=120,in=0] (0.4, 0.8) to[out=180,in=90] cycle;
            \filldraw[color=blue] (0.4, 0.6) circle (0.35 pt) node[above right] {$T_n$};
            \draw[color=mygreen] (0.65, 0.45) node[above right] {$\diagram(\theta)$};

            \draw[color=mygreen] (0.2, 0.65) node[left] {$\mathcal{B}(\theta)$};

            \draw[thick, -stealth, color=black, dashed] (0,0) -- (1.1,0.275);
            \draw[color=black] (1.1,0.275) node[above right] {$s=\frac{1}{C_{\theta}}v$};

            \draw[ultra thick, -stealth, color=mygreen] (0.45,0.2425) --  (0.22,0.35);
            \filldraw[color=mygreen] (0.45,0.2425) circle (0.35 pt) node[above right] {$(V_{\lambda}(\theta),S_{\lambda}(\theta))$};
            \draw[thick, color=mygreen,decorate,decoration={brace,amplitude=5pt}] (0.45,0.19) --  (0.22,0.19) node[midway,yshift=-1.4em]{$\partial_{\lambda}V_{\lambda}(\theta)$};
            \draw[thick, color=mygreen,decorate,decoration={brace,amplitude=5pt}] (0.19,0.2425) --  (0.19,0.35) node[midway,xshift=-2.4em]{$\partial_{\lambda}S_{\lambda}(\theta)$};

		\end{tikzpicture}
    \caption{A cartoon of the asymptotic variance-sensitivity diagram and related objects in a $\HK$ regular model.
    We show the asymptotic variance-sensitivity diagram $\diagram(\theta)$ for regular unbiased estimators (dark green region), the extended diagram $\tilde{\diagram}(\theta)$ (light green region), and the curve $\bdry(\theta)$ from the $\HK$ lower bound (dark green curve).
    All regions are bounded by $(I(\theta))^{-1}$ in the $v$ coordinate via the Cram\'er--Rao lower bound (vertical dashed line) and bounded by $(J(\theta))^{-1}$ in the $s$ coordinate via the Wasserstein--Cram\'er--Rao lower bound (horizontal dashed line); the region $\diagram(\theta)$ also lies above the line $s=v/C_{\theta}$ by the Poincar\'e inequality (sloped dashed line), but $\tilde{\diagram}(\theta)$ does not.
    }
    \label{fig:VSD-cartoon}
\end{figure}

Next we connect the structure of the asymptotic variance-sensitivity diagrams $\diagram(\theta)$ and $\tilde{\diagram}(\theta)$ with the lower bounds from Theorem~\ref{thm:HK-CR-bound}.
To do this, for a model $\Pcal=\{P_{\theta}:\theta\in\Theta\}\subseteq\Pcal_{2,\textnormal{ac}}(\R^d)$ that is $\HK$ regular, we define
    \begin{equation*}
        \bdry(\theta):=\Big\{(V_{\lambda}(\theta),S_{\lambda}(\theta)): 0<\lambda<1\Big\}
    \end{equation*}
    for all $\theta\in\Theta$ such that $I_\lambda(\theta)>0$ for all $0<\lambda<1$, where $V_{\lambda}(\theta)$ and $S_{\lambda}(\theta)$ were introduced in Definition~\ref{def:HK-info}; see Figure~\ref{fig:VSD-cartoon} for a depiction of the curve $\bdry(\theta)$ in dark green.
    Then we have the following, which is immediate from Theorem~\ref{thm:HK-CR-bound}.
    
\begin{corollary}\label{cor:Pareto-optim-VSD}
    If a model $\Pcal=\{P_{\theta}:\theta\in\Theta\}\subseteq\Pcal_{2,\textnormal{ac}}(\R^d)$ is $\HK$ regular and for each $0<\lambda<1$ the maps $\theta\mapsto \Ebb_{\theta}[|\phi_{\theta,\lambda}(X)|^2]$ and $\theta\mapsto \Ebb_{\theta}[\|\nabla\phi_{\theta,\lambda}(X)\|^2]$ are continuous on $\Theta$, then $\bdry(\theta)$ is a Pareto front for $\diagram(\theta)$ and for $\tilde{\diagram}(\theta)$, for all $\theta\in\Theta$ such that $I_\lambda(\theta)>0$ for all $0<\lambda<1$.
\end{corollary}

Our analysis does not imply that $\bdry(\theta)$ is the Pareto boundary of $\diagram(\theta)$ in the standard sense, but we highlight that the only obstruction for proving this is to check that $(V_{\lambda}(\theta),S_{\lambda}(\theta))\in\diagram(\theta)$; we emphasize that this problem is also seen in the pure source case \cite{vanDerVaart} and in the pure transport case \cite{WCR}, where one typically ``achieves'' the lower bound by constructing a consistent sequence of estimators rather than a sequence of unbiased estimators.
In Section~\ref{sec:achieve} we will adopt this perspective and construct a consistent (but possibly biased) sequence of estimators which asymptotically achieves the variance-sensitivity pairs on the Pareto front.
We also note that Figure~\ref{fig:VSD-cartoon} utilizes this abuse of notation, where we plot the asymptotic variance-sensitivity coordinates, in blue, for an estimator $T_n$, considered as $n\to\infty$.

Our final goal is to understand some finer structure of the set $\bdry(\theta)$.
The main ingredient for this is the following result which states that, although $V_{\lambda}(\theta),S_{\lambda}(\theta)$, and $I_{\lambda}(\theta)$ were initially defined explicitly in terms of the $\HK_{\lambda}$ score functions $\phi_{\theta,\lambda}$, one can write $V_{\lambda}(\theta)$ and $S_{\lambda}(\theta)$ explicitly in terms of $I_{\lambda}(\theta)$.

\begin{lemma}\label{lem:Info-cvx}
    For $p=1$, suppose $\Pcal=\{P_{\theta}:\theta\in\Theta\}\subseteq\Pcal_{2,\textnormal{ac}}(\R^d)$ is $\HK$ regular and that, for each $0<\lambda<1$, the maps $\theta\mapsto \Ebb_{\theta}[|\phi_{\theta,\lambda}(X)|^2]$ and $\theta\mapsto \Ebb_{\theta}[\|\nabla\phi_{\theta,\lambda}(X)\|^2]$ are continuous on $\Theta$.
    Then, for each $\theta\in\Theta$ such that $I_\lambda(\theta)>0$ for all $0<\lambda<1$, the following hold.
    First, the map $(0,1)\ni\lambda\mapsto I_\lambda(\theta)$ is convex and the map $(0,1)\ni\lambda\mapsto(I_\lambda(\theta))^{-1}$ is concave; hence both are differentiable at all but at most countably many $0<\lambda<1$.
    Second, we have
    \begin{equation}\label{eqn:VS-from-I}
        V_{\lambda}(\theta) = (I_{\lambda}(\theta))^{-1}-(1-\lambda)\frac{\partial_{\lambda}I_{\lambda}(\theta)}{(I_{\lambda}(\theta))^2}
        \qquad\textnormal{and}\qquad
        S_{\lambda}(\theta) = (I_{\lambda}(\theta))^{-1}+\lambda\frac{\partial_{\lambda}I_{\lambda}(\theta)}{(I_{\lambda}(\theta))^2}.
    \end{equation}
    for every $0<\lambda<1$ at which $\lambda\mapsto I_\lambda(\theta)$ is differentiable.
\end{lemma}

The preceding result gives a geometric interpretation of the quantities $V_{\lambda}(\theta)$ and $S_{\lambda}(\theta)$ appearing in our lower bounds, which we illustrate in Figure~\ref{fig:VS-from-I}.
Indeed, it shows that they are the values at $1$ and at $0$, respectively, of the line tangent to the concave function $\lambda\mapsto(I_{\lambda}(\theta))^{-1}$ at $\lambda$.
We emphasize that this perspective only arises when studying the whole family of $\HK_{\lambda}$ geometries at once, whereas most existing works in the literature of optimal transport fix the relative weights of transport and source (e.g., \cite{LieroMielkeSavare,ChizatPeyreSchmitzerVialard}), and indeed changing $\lambda$ only amounts to a dilation of the underlying space (see Appendix~\ref{app:HK-static-cone}).

\begin{figure}[t]
    \centering
\begin{tikzpicture}[font=\footnotesize]
  \def\Ma{0.25}\def\Mb{1.1}\def\Mc{0.9}
  \def\lamA{0.3}    
  \def\lamB{0.75}   
  \def\lamMin{0.12} 
  \def\PanelS{4}    
  \def\PanelIIx{8.2}
  \pgfmathsetmacro{\MA}{\Ma+\Mb*\lamA-\Mc*\lamA^2}
  \pgfmathsetmacro{\MB}{\Ma+\Mb*\lamB-\Mc*\lamB^2}
  \pgfmathsetmacro{\VA}{\Ma+\Mb-2*\Mc*\lamA+\Mc*\lamA^2}
  \pgfmathsetmacro{\SA}{\Ma+\Mc*\lamA^2}
  \pgfmathsetmacro{\VB}{\Ma+\Mb-2*\Mc*\lamB+\Mc*\lamB^2}
  \pgfmathsetmacro{\SB}{\Ma+\Mc*\lamB^2}

  \begin{scope}[scale=\PanelS]
    \draw[thick, -stealth] (0,0) -- (1.2,0);
    \draw[thick, -stealth] (0,0) -- (0,1.2);
    \draw[thick, dashed] (1,0) -- (1,1);
    \draw[thick] (-0.05,0) node[left]{$0$}  -- (0.05,0);
    \draw[thick] (0,-0.05) node[below]{$0$} -- (0,0.05);
    \draw[thick] (1,-0.05) node[below]{$1$} -- (1,0.05);
    \draw[thick] (\lamA,-0.05) node[below]{$\lambda$}   -- (\lamA,0.05);
    \draw[thick] (\lamB,-0.05) node[below]{$\lambda'$}  -- (\lamB,0.05);

    \draw[mygreen, thick] (0,\SA) -- (1,\VA);
    \draw[mygreen, thick] (0,\SB) -- (1,\VB);

    \draw[mypurple, ultra thick, domain=0:1, samples=100]
      plot ({\x},{\Ma+\Mb*\x-\Mc*\x*\x});
    \node[mypurple] at (0.62,0.45) {$(I_{\lambda}(\theta))^{-1}$};

    \filldraw[mygreen] (\lamA,\MA) circle (0.5pt);
    \filldraw[mygreen] (\lamB,\MB) circle (0.5pt);
    \filldraw[mygreen] (0,\SA) circle (0.5pt) node[left]  {$S_{\lambda}(\theta)$};
    \filldraw[mygreen] (0,\SB) circle (0.5pt) node[left]  {$S_{\lambda'}(\theta)$};
    \filldraw[mygreen] (1,\VA) circle (0.5pt) node[right] {$V_{\lambda}(\theta)$};
    \filldraw[mygreen] (1,\VB) circle (0.5pt) node[right] {$V_{\lambda'}(\theta)$};
  \end{scope}

  \begin{scope}[shift={(\PanelIIx,0)}, scale=\PanelS]
    \draw[thick, -stealth] (0,0) -- (1.2,0);
    \draw[thick, -stealth] (0,0) -- (0,1.2);
    \draw[thick] (-0.05,0) node[left]{$0$}  -- (0.05,0);
    \draw[thick] (0,-0.05) node[below]{$0$} -- (0,0.05);
    \node at (0.6,-0.22) {asymptotic variance};
    \node[rotate=90] at (-0.36,0.6) {asymptotic sensitivity};

    \draw[mygreen, ultra thick, domain=\lamMin:1, samples=100]
      plot ({\Ma+\Mb-2*\Mc*\x+\Mc*\x*\x},{\Ma+\Mc*\x*\x});
    \node[mygreen] at (0.8,0.58) {$\bdry(\theta)$};

    \draw[mygreen, dashed] (\VA,0) -- (\VA,\SA) -- (0,\SA);
    \draw[mygreen, dashed] (\VB,0) -- (\VB,\SB) -- (0,\SB);
    \filldraw[mygreen] (\VA,\SA) circle (0.5pt);
    \filldraw[mygreen] (\VB,\SB) circle (0.5pt);

    \draw[mygreen, thick] (\VA,-0.05) node[below]{$V_{\lambda}(\theta)$}  -- (\VA,0.05);
    \draw[mygreen, thick] (\VB,-0.05) node[below]{$V_{\lambda'}(\theta)$} -- (\VB,0.05);
    \draw[mygreen, thick] (-0.05,\SA) node[left]{$S_{\lambda}(\theta)$}   -- (0.05,\SA);
    \draw[mygreen, thick] (-0.05,\SB) node[left]{$S_{\lambda'}(\theta)$}  -- (0.05,\SB);
  \end{scope}
\end{tikzpicture}
    \caption{Geometric representation of the Pareto optimal asymptotic variance and sensitivity, in terms of the inverse $\HK$ information.
    We show the inverse $\HK$ information $\lambda\mapsto (I_{\lambda}(\theta))^{-1}$ and its tangent lines at two values of the interpolation parameter $0<\lambda<1$ (left), and we show the Pareto front $\mathcal{B}(\theta)$ whose $(v,s)$  coordinates are the values of the tangent lines at $1$ and $0$ respectively (right).}
    \label{fig:VS-from-I}
\end{figure}
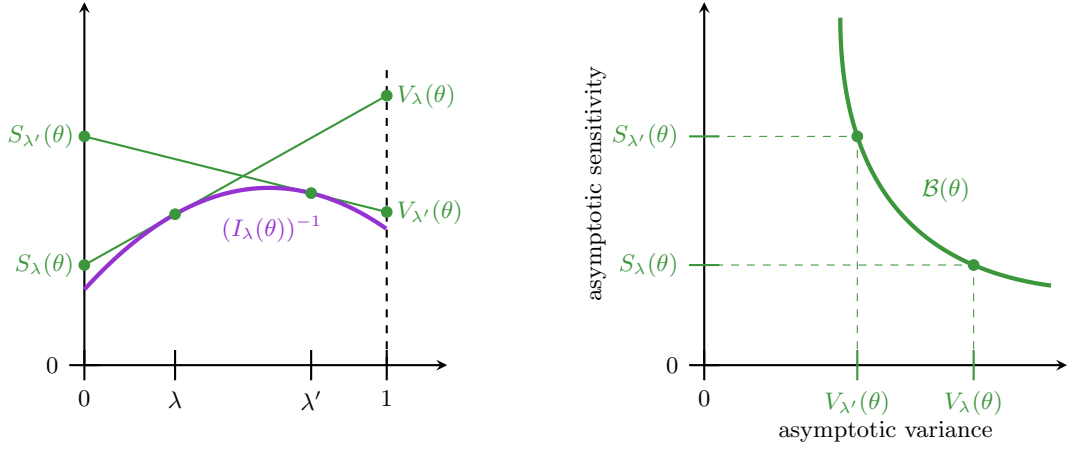

An easy consequence of this result is the following calculation of the slope of the tangent lines to $\bdry(\theta)$, which quantitatively describe the instantaneous tradeoffs between variance and sensitivity; for instance, this slope represents how much the variance must be increased in order to decrease the sensitivity by a certain amount.
See Figure~\ref{fig:VSD-cartoon} for an illustration of this tangent line and its slope.
While these slopes can be directly calculated for the running examples, we also have the following general result:

\begin{lemma}\label{lem:bdry-slope}
    For $p=1$, suppose $\Pcal=\{P_{\theta}:\theta\in\Theta\}\subseteq\Pcal_{2,\textnormal{ac}}(\R^d)$ is $\HK$ regular and that, for each $0<\lambda<1$, the maps $\theta\mapsto \Ebb_{\theta}[|\phi_{\theta,\lambda}(X)|^2]$ and $\theta\mapsto \Ebb_{\theta}[\|\nabla\phi_{\theta,\lambda}(X)\|^2]$ are continuous on $\Theta$.
    If $\theta\in\Theta$ satisfies $I_\lambda(\theta)>0$ for all $0<\lambda<1$, then there is a non-increasing convex function
    $h_\theta$ defined on the interval $\{V_\lambda(\theta):0<\lambda<1\}\subseteq\R$ such that
    we have $S_\lambda(\theta)=h_\theta(V_\lambda(\theta))$ for all $0<\lambda<1$.
    Moreover, we have $h_\theta'(V_\lambda(\theta))=-\lambda/(1-\lambda)$ whenever $V_\lambda(\theta)$ is an interior point of the domain of $h_\theta$ at which $h_\theta$ is differentiable.
\end{lemma}

Finally, we can revisit our running examples and display their asymptotic variance-sensitivity diagrams; that is, we give Figure~\ref{fig:all-VSD}, which shows a simplified version of Figure~\ref{fig:VSD-cartoon} for the Gaussian location, Laplace location, uniform scale, and Gaussian correlation families. We also show the asymptotic variance-sensitivity coordinates achieved by various estimators of interest (blue points), some of which we will construct in the following section.

\begin{figure}[t]
    \centering

    	\begin{center}
        \begin{tikzpicture}[scale=4.]
            \node at (0.65, 1.15) {Gaussian location family};

			\draw[thick, -stealth] (0,0) -- (1.2,0);
			\draw[thick, -stealth] (0,0) -- (0,1.2);
            \draw[thick, dashed] (0,0.2) -- (1,0.2);
			\draw[thick, dashed] (0.2,0) -- (0.2,1);
			
			\draw[thick] (-0.05,0) node[left]{$0$} -- (0.05,0);
			\draw[thick] (-0.05,0.2) node[left]{$1$} -- (0.05,0.2);
			
			\draw[thick] (0,-0.05) node[below]{$0$} -- (0,0.05);
            \draw[thick] (0.2,-0.05) node[below right]{$1$} -- (0.2,0.05);

			\path [fill=mygreen, thick,  opacity=0.3] (0.2,1) -- (0.2,0.2) -- (1,0.2) -- (1,1) -- cycle;		

			\filldraw[color=blue] (0.2, 0.2) circle (0.5 pt) node[above right] {sample mean};
			
		\end{tikzpicture} \begin{tikzpicture}[scale=4.]
            \node at (0.65, 1.15) {Laplace location family};
            
			\draw[thick, -stealth] (0,0) -- (1.2,0);
			\draw[thick, -stealth] (0,0) -- (0,1.2);
            \draw[thick, dashed] (0,0.2) -- (1,0.2);
			\draw[thick, dashed] (0.2,0) -- (0.2,1);
			
			\draw[thick] (-0.05,0) node[left]{$0$} -- (0.05,0);
			\draw[thick] (-0.05,0.2) node[left]{$1$} -- (0.05,0.2);
            \draw[thick] (-0.05,1) node[left]{$\infty$} -- (0.05,1);
			
			\draw[thick] (0,-0.05) node[below]{$0$} -- (0,0.05);
            \draw[thick] (0.2,-0.05) node[below right]{$1$} -- (0.2,0.05);
			\draw[thick] (0.7,-0.05) node[below]{$2$} -- (0.7,0.05);

			\path [fill=mygreen, thick,  opacity=0.3, domain=0:1, samples=100] plot ({0.2+0.5*\x^2}, {0.2+0.8*(\x-1)^2}) -- (1,0.2) -- (1,1) -- cycle;		
			\draw [color=mygreen, ultra thick,  domain=0:1, samples=100] plot ({0.2+0.5*\x^2}, {0.2+0.8*(\x-1)^2});

						
			\filldraw[color=blue] (0.2, 1) circle (0.5 pt) node[below right] {sample median};
			\filldraw[color=blue] (0.7, 0.2) circle (0.5 pt) node[above right,align=left] {sample\\mean};

            \filldraw[color=white] (-0.65, 0.2) circle (0.5 pt);

		\end{tikzpicture}\\
        \begin{tikzpicture}[scale=4.]
            \node at (0.65, 1.15) {Uniform scale family};
            
			\draw[thick, -stealth] (0,0) -- (1.2,0);
			\draw[thick, -stealth] (0,0) -- (0,1.2);
            \draw[thick, dashed] (0,0.2) -- (1,0.2);
			
			\draw[thick] (-0.05,0) node[left]{$0$} -- (0.05,0);
			\draw[thick] (-0.05,0.2) node[left]{$3$} -- (0.05,0.2);
            \draw[thick] (-0.05,1) node[left]{$\infty$} -- (0.05,1);
			
			\draw[thick] (0,-0.05) node[below]{$0$} -- (0,0.05);
            \draw[thick] (0.7,-0.05) node[below]{$\frac{1}{5}\theta^2$} -- (0.7,0.05);

			\path [fill=mygreen, thick,  opacity=0.3, domain=0:1, samples=100] plot ({0.7*\x^2}, {0.2+0.8*(\x-1)^2}) -- (1,0.2) -- (1,1) -- cycle;		
			\draw [color=mygreen, ultra thick,  domain=0:1, samples=100] plot ({0.7*\x^2}, {0.2+0.8*(\x-1)^2});

			\filldraw[color=blue] (0, 1) circle (0.5 pt) node[below right] {$X_{(n)}$};
			\filldraw[color=blue] (0.7, 0.2) circle (0.5 pt) node[above right] {$\sqrt{3\overline{X^2}}$};			
		\end{tikzpicture}\begin{tikzpicture}[scale=4.]
            \node at (0.65, 1.15) {Gaussian correlation family};
            
			\draw[thick, -stealth] (0,0) -- (1.2,0);
			\draw[thick, -stealth] (0,0) -- (0,1.2);
            \draw[thick, dashed] (0,0.2) -- (1,0.2);
			\draw[thick, dashed] (0.2,0) -- (0.2,1);
			
			\draw[thick] (-0.05,0) node[left]{$0$} -- (0.05,0);
			\draw[thick] (-0.05,0.2) node[left]{$2(1-\theta^2)$} -- (0.05,0.2);
            \draw[thick] (-0.05,0.7) node[left]{$\frac{2\left(1-\theta^{2}\right)\left(1+3\theta^{2}\right)}{\left(1+\theta^{2}\right)^{2}}$} -- (0.05,0.7);
            \draw[thick] (-0.05,0.85) node[left]{$2$} -- (0.05,0.85);
			
			\draw[thick] (0,-0.05) node[below]{$0$} -- (0,0.05);
            \draw[thick] (0.2,-0.05) node[below]{$\frac{(1-\theta^2)^2}{1+\theta^2}$} -- (0.2,0.05);
			\draw[thick] (0.7,-0.05) -- (0.7,0.05);

            \draw[] (0.6,-0.05) node[below] {$(1-\theta^{2})^{2}$};

            \draw[thick] (0.85,-0.05) node[below right]{$1+\theta^2$} -- (0.85,0.05);

			\path [fill=mygreen, thick,  opacity=0.3, domain=0:1, samples=100] plot ({0.2+0.5*\x^2}, {0.2+0.5*(\x-1)^2}) -- (1,0.2) -- (1,1) -- (0.2, 1) -- cycle;		
			\draw [color=mygreen, ultra thick,  domain=0:1, samples=100] plot ({0.2+0.5*\x^2}, {0.2+0.5*(\x-1)^2});						
            \filldraw[color=blue] (0.2, 0.7) circle (0.5 pt) node[above right] {MLE};
			\filldraw[color=blue] (0.7, 0.2) circle (0.5 pt) node[above,align=left] {Pearson\\correlation};

            \filldraw[color=blue] (0.85, 0.85) circle (0.5 pt) node[above, align=left] {$\frac1n\sum_{i=1}^n X_iY_i$};

            \filldraw[color=white] (-0.62, 0.2) circle (0.5 pt);

		\end{tikzpicture}
	\end{center}
    \caption{The asymptotic variance-sensitivity diagrams for the Gaussian location family (top left), the Laplace location family (top right), the uniform scale family (bottom left), and the Gaussian correlation family (bottom right).
    In each plot, we show the asymptotic variance-sensitivity diagram $\tilde{\diagram}(\theta)$ (light green region), the curve $\bdry(\theta)$ (dark green curve), the Cram\'er--Rao and Wasserstein--Cram\'er--Rao lower bounds if they hold (dashed lines), and the asymptotic variance-sensitivity coordinates achieved by various estimators of interest (blue).}
    \label{fig:all-VSD}
\end{figure}
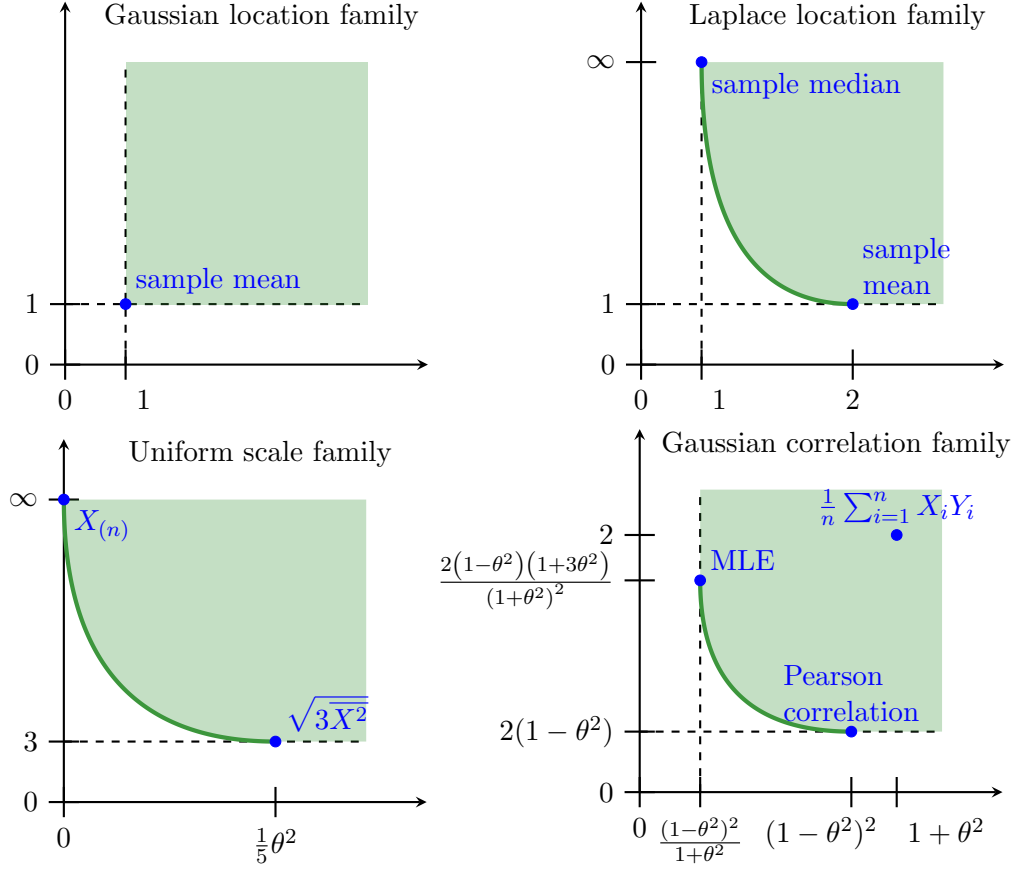

To conclude this subsection, we note that this theory may be extended to the case of a multivariate parameter ($p\ge 1$), although visualizing the resulting sets becomes more complicated.
That is, write $\CovSpace(p)$ for the cone of $p\times p$ positive semi-definite matrices, and observe by Theorem~\ref{thm:HK-CR-bound} that the multivariate variance-sensitivity diagrams are naturally viewed as subsets of $\CovSpace(p)\times\CovSpace(p)$.
When $p=1$ the cone $\CovSpace(p)$ has dimension 1 (and is isomorphic to $\R_{\ge 0}$) so we may visualize these diagrams in the plane $\R\times\R=\R^2$, but already for $p=2$ the cone $\CovSpace(2)$ has dimension 3 so visualizing the resulting diagrams involves $\R^{3}\times\R^{3}=\R^{6}$.

\section{Constructing Asymptotically Efficient Estimators}\label{sec:achieve}

For each sufficiently regular model $\Pcal=\{P_{\theta}:\theta\in\Theta\}$, the lower bounds in Section~\ref{subsec:LB} describe a tradeoff between the asymptotic variance and asymptotic sensitivity of sufficiently regular unbiased estimators of $\theta$. The purpose of this section is to construct estimators whose asymptotic variance and sensitivity achieve the Pareto optimal pairs $(V_\lambda(\theta),S_\lambda(\theta))$ indexed by $\lambda$.
We do this by studying some general properties of the asymptotic sensitivity of $Z$-estimators (Subsection~\ref{subsec:Z-estim}), and then by specializing to the case of a $Z$-estimator based on the $\HK$ score functions (Subsection~\ref{subsec:HK-score}); we also study asymptotic efficiency at the extremal cases $\lambda\in\{0,1\}$, which slightly differs from the interior cases $0<\lambda<1$.

We begin with a precise notion of (Pareto) efficiency.

\begin{definition}\label{def:Pareto-eff}
    Fix $0\le\lambda\le1$. Let $\Pcal=\{P_\theta:\theta\in\Theta\}
    \subseteq\Pcal_{2,\textnormal{ac}}(\R^d)$
    be an $\HK_\lambda$ regular model with $\Theta\subseteq\R^p$.
    Assume that $V_\lambda(\theta)$ and $S_\lambda(\theta)$, as in Definition \ref{def:HK-info}, are well defined and finite for all $\theta\in\Theta$.
    A consistent, locally Lipschitz estimator $T_n:(\R^d)^n\to\R^p$ is called \textit{$\lambda$ (Pareto) efficient} if we have
    \begin{equation*}
        \begin{aligned}
            \sqrt n(T_n-\theta_0)
            &\overset{\mathrm d}{\to}
            \mathcal N\bigl(0,V_\lambda(\theta_0)\bigr),\\
            n\sum_{i=1}^nD_{x_i}T_n(D_{x_i}T_n)^\top
            &\overset{P_{\theta_0}}{\to}
            S_\lambda(\theta_0).
        \end{aligned}
    \end{equation*}
    for all $\theta_0\in\Theta$, where
    \(D_{x_i}T_n\in\R^{p\times d}\) has rows corresponding
    to the components of \(T_n\).
    We say that $T_n$ is \textit{(Pareto) efficient} if it is $\lambda$ (Pareto) efficient for some $0\le \lambda\le 1$.
\end{definition}

\begin{remark}\label{rem:achieve-eff}
    The notion of Pareto efficiency in Definition~\ref{def:Pareto-eff} slightly differs from the lower bounds in Theorem~\ref{thm:HK-CR-bound}.
    That is, the lower bounds involve the limits of the real sequences $n\Var_{\theta}(T_n)$ and $n\Sen_{\theta}(T_n)$, while the notion of efficiency involves the variance of the limiting distribution of $T_n$ and its sample sensitivity.
    This gap between the lower bounds and how to achieve them is standard in the literature; see, e.g., \cite[Chapter~8.1]{vanDerVaart}.
\end{remark}

\subsection{Asymptotic Sensitivity of $Z$-Estimators}\label{subsec:Z-estim}

We begin with a general result which calculates the asymptotic sensitivity for a wide class of $Z$-estimators.
This complements existing results which calculate the asymptotic variance for $Z$-estimators (e.g., \cite[Theorem~5.21]{vanDerVaart}). Our later construction of asymptotically Pareto-efficient estimators will indeed use both of these calculations.

For this subsection, we let \(X_1,\ldots,X_n\) be i.i.d. samples from a distribution \(P\), which we do not assume comes from any particular parametric model, and we write $\Ebb_P$ for the corresponding expectation.
Then, for an open set $\Theta\subseteq\R^p$ and a fixed function \(s:\Theta\times\R^d\to\R^p\), we define the $Z$-estimator $T_n^s:(\R^d)^n\to\Theta$ to be any solution for $\theta\in\Theta$ to the equation
\begin{equation}\label{eqn:Z}
    \frac{1}{n}\sum_{i=1}^{n}s(\theta,X_i)=0.
\end{equation}
There are a few technical subtleties to this definition (e.g., that $T_n^s$ may fail to be well-defined on all of $(\R^d)^n$, and that it need not even be measurable when it is well-defined) but we direct the reader to standard sources like \cite[Chapter~5]{vanDerVaart} for further detail; see also our statement of Theorem~\ref{thm:z-est} below and the discussion right after it.

We write \(D_t s(t,x)\in\R^{p\times p}\) and
\(D_xs(t,x)\in\R^{p\times d}\), with rows corresponding
to the components of \(s\), and use the same convention for
\(D_{x_i}T_n^s\).
For \(\HK\) score functions,
\(D_x\phi_{\theta,\lambda}=D\phi_{\theta,\lambda}\),
as in Definition~\ref{def:HK-info}.
We write \(\|\cdot\|_{\mathrm F}\) for the Frobenius norm and
\(M^{-\top}:=(M^{-1})^\top\) for an invertible matrix \(M\).
Matrix ULLNs and convergence of matrix-valued random variables are understood entrywise.

Our main result in this subsection is the following, which states that, under mild regularity conditions, the asymptotic sensitivity of $T_n^s$ can be written in terms of expectations with respect to $P$ of certain derivatives of $s$. Note that this result requires certain function classes to satisfy a uniform law of large numbers (ULLN), sufficient conditions for which can be found in \cite[Theorem~19.4]{vanDerVaart}.

\begin{theorem}\label{thm:z-est}
    Fix \(\theta_0\in\Theta\), and let \(X_1,X_2,\ldots\) be
    i.i.d. from \(P \in \mathcal{P}_{2, \mathrm{ac} }(\R^d)\).
    Let \(s:\Theta\times\R^d\to\R^p\) be a (jointly) \(C^1\) function, and suppose the following conditions are satisfied:
    \begin{enumerate}
        \item[(i)] (Existence of a solution)
        With probability one, for all sufficiently large \(n\)
        the equation
        \begin{equation}
          \frac1n\sum_{i=1}^n s(\theta,x_i)=0
            \label{eq:EstiamtingEqn2}
        \end{equation}
        admits a solution \( \theta = T_n^s(x_1, \dots, x_n)\) for all $(x_1, \dots, x_n)$ within an open neighborhood of \((X_1,\ldots,X_n)\). Moreover, the solution is assumed to vary continuously in its input within this neighborhood.

        \item[(ii)] (Consistency)
        The selected estimator $T_n^s= T_n^s(X_1, \dots, X_n)$ satisfies
        \(T_n^s\to\theta_0\) in probability as \(n\to\infty\).

        \item[(iii)] (Invertibility of Jacobians)
        For some \(\varepsilon>0\) with
        \(B(\theta_0,\varepsilon)\subset\Theta\),
        \(\Ebb_{P}[D_\theta s(\theta,X)]\) is invertible
        for all $\theta$ satisfying \(\|\theta-\theta_0\|<\varepsilon\).

        \item[(iv)] (ULLNs)
        The classes
        \(\{D_\theta s(\theta,\cdot):
            \|\theta-\theta_0\|<\varepsilon\}\) and
        \(\{D_xs(\theta,\cdot)D_xs(\theta,\cdot)^\top:
            \|\theta-\theta_0\|<\varepsilon\}\)
        satisfy an entrywise ULLN under \(P\).
        Moreover, \(\Ebb_{P}[D_\theta s(\theta,X)]\)
        and \(\Ebb_{P}[D_xs(\theta,X)D_xs(\theta,X)^\top]\)
        have finite entries for all $\theta$ satisfying \(\|\theta-\theta_0\|<\varepsilon\).
    \end{enumerate}
    Then, as \(n\to\infty\),
    \[
        n\sum_{i=1}^nD_{x_i}T_n^s(D_{x_i}T_n^s)^\top
        \longrightarrow
        \bigl(\Ebb_{P}[D_\theta s(\theta_0,X)]\bigr)^{-1}
        \Ebb_{P}[D_xs(\theta_0,X)D_xs(\theta_0,X)^\top]
        \bigl(\Ebb_{P}[D_\theta s(\theta_0,X)]\bigr)^{-\top}
    \]
    in probability.
\end{theorem}

The proof of Theorem~\ref{thm:z-est} is based on the implicit function theorem, which requires the technical conditions stated above.
For instance, in order to guarantee that \(T_n^s\) is a differentiable function of $X_1,\ldots, X_n$, we require the empirical Jacobian
\(n^{-1}\sum_{i=1}^nD_\theta s(T_n^s,X_i)\) to be invertible for all large enough $n$, and this requires the invertibility of the population Jacobian
\(\Ebb_{P}[D_\theta s(\theta_0,X)]\) along with the stated ULLNs.
See Appendix~\ref{app:proofs-achieve} for details.

It is instructive to compare this result with the classical result
on the asymptotic covariance of \(Z\)-estimators
(e.g., \cite[Theorem~5.21]{vanDerVaart}).
Under suitable regularity conditions, that result states
\[
    \sqrt n(T_n^s-\theta)
    \overset{\mathrm d}{\to}
    \mathcal N\left(
        0,\,
        \bigl(\Ebb_P[D_\theta s(\theta,X)]\bigr)^{-1}
        \Ebb_P[s(\theta,X)s(\theta,X)^\top]
        \bigl(\Ebb_P[D_\theta s(\theta,X)]\bigr)^{-\top}
    \right).
\]
In particular, the asymptotic covariance and the asymptotic
sensitivity matrix have the same left and right normalization factors, and the middle matrices differ only in the appearance of a Jacobian; this is most visible in the special case $p=1$, where the asymptotic variance and sensitivity are
\begin{equation}\label{eqn:Z-est-VS}
        \frac{\Ebb_P\left[|s(\theta,X)|^2\right]}{\bigl(\Ebb_P[\partial_\theta s(\theta,X)]\bigr)^{2}}\qquad \textnormal{and}\qquad\frac{\Ebb_P[|\nabla_x s(\theta,X)|^2]}{\bigl(\Ebb_P[\partial_\theta s(\theta,X)]\bigr)^{2}},
\end{equation}
whose denominators are the same and whose numerators differ only in the appearance of a gradient.
This structure plays an important role in the
next subsection when we take \(s\) to be an \(\HK\) score function.

Additional examples of interest, beyond those that will arise from $\HK$ score functions, include Huber's $M$-estimators from robust statistics \cite{HuberRonchetti}.
The regularity conditions in Theorem~\ref{thm:z-est} do not hold in this setting, but we nonetheless obtain the same formulas for the asymptotic variance and asymptotic sensitivity via a direct computation.

\begin{example}[Huber's $M$ estimators]\label{ex:huber}
       
        For $\delta> 0$, define the Huber loss function via
		\begin{equation*}
			L_{\delta}(r) = \begin{cases}
				\frac{1}{2}r^2 &\textnormal{ if } |r|\le \delta \\
				\delta\left(|r|-\frac{1}{2}\delta\right) &\textnormal{ if } |r|\ge \delta, \\
			\end{cases}
		\end{equation*}
        and then consider the estimator
        \begin{equation}\label{eqn:Huber-def}
            T_n^{\delta}(X_1,\ldots,X_n):=\underset{\theta\in\R}{\arg\min}\,\frac{1}{n}\sum_{i=1}^{n}L_{\delta}(X_i-\theta).
        \end{equation}
        If the minimizer is not unique, we take the midpoint of the minimizing interval (note that Huber's loss function is convex and, as such, its set of minimizers must be an interval). This estimator interpolates between the sample median ($\delta\to 0$) and the sample mean ($\delta\to\infty$).
        By computing the first-order optimality condition for the convex optimization problem~\eqref{eqn:Huber-def} (note that Huber's loss is continuously differentiable), one can see that $T_n^{\delta}$ is precisely the $Z$-estimator for the function $s_{\delta}:\R^2\to\R$ defined via
        \begin{equation*}
            s_{\delta}(\theta,x)=\begin{cases}
				-\delta &\textnormal{ if } x\le \theta-\delta \\
                x -\theta &\textnormal{ if } \theta-\delta \le x\le \theta+\delta \\
				\delta &\textnormal{ if } x\ge \theta+\delta.
			\end{cases}
        \end{equation*}
        While Theorem \ref{thm:z-est} cannot be applied directly because $s_\delta$ is not $C^1$, it is still possible to obtain formulas for the asymptotic variance and sensitivity of $T_n^\delta$ when $P$ is absolutely continuous.
        Let \(\theta\in\R\) satisfy \(\Ebb_P[s_\delta(\theta,X)]=0\), and assume \(\Pbb_P(|X-\theta|<\delta)>0\).
    The condition \(\Ebb_P[s_\delta(\theta,X)]=0\) specifies the population target.
    It holds, in particular, when \(P\) is symmetric about \(\theta\), because \(r\mapsto s_\delta(\theta,\theta+r)\) is bounded and odd.
        Then, as \(n\to\infty\),
    \begin{equation*}
        \sqrt n(T_n^\delta-\theta)
        \overset{\mathrm d}{\to}
        \mathcal N\left(
            0,\frac{\Ebb_P[s_\delta(\theta,X)^2]}
            {\Pbb_P(|\theta-X|\le\delta)^2}
        \right), \qquad n\sum_{i=1}^{n}\|\nabla_{x_i}T_n^\delta\|^2
            \overset{P}{\to}\frac{1}{\Pbb_{P}(|\theta-X|\le\delta)};
    \end{equation*}
    a detailed proof of these facts is given in Appendix~\ref{app:proofs-achieve}. Note that, as $\delta\to \infty$, the estimator $T_n^{\delta}$ approaches the sample mean and the asymptotic sensitivity converges to 1, while as $\delta\to 0$ the estimator $T_n^{\delta}$ approaches the sample median and the asymptotic sensitivity blows up.
	\end{example}

In the next subsection we apply Theorem~\ref{thm:z-est} to construct asymptotically Pareto-efficient estimators in $\HK$ regular parametric models.
As suggested by \eqref{eqn:Z-est-VS}, specializing to the case of $s(\theta,x) = \phi_{\theta,\lambda}(x)$ will allow us to recover the optimal asymptotic variance and sensitivity pairs from Section~\ref{subsec:LB}.

\subsection{Hellinger--Kantorovich $Z$-Estimators}\label{subsec:HK-score}

We now specialize the previous subsection to the case of the
\(\HK_\lambda\) score functions $\phi_{\theta,\lambda}$ of an \(\HK_\lambda\) regular model $\Pcal$.
We will show, under suitable regularity conditions, that the resulting $Z$-estimator is asymptotically efficient.

\begin{definition}\label{def:HK-Z-est}
    If \(\Pcal\) is \(\HK_{\lambda}\) regular and
    \(\phi_{\theta,\lambda}\) denotes its \(\HK_{\lambda}\) score function, then
    \begin{equation*}
        \frac{1}{n}\sum_{i=1}^{n}\phi_{\theta,\lambda}(X_i) = 0
    \end{equation*}
    is called the
    \textit{$\HK_\lambda$ score equation}. Any solution to this equation will be referred to as a \textit{\(\HK_\lambda\) \(Z\)-estimator} (HKZE), and will be denoted by
    \(T_{n,\lambda}\).
\end{definition}

\begin{remark}[Maximum likelihood estimation]\label{rem:HK1-MLE}
    The \(\HK_1\) \(Z\)-equation is exactly the likelihood equation \(\sum_{i=1}^n\nabla_\theta\log p_\theta(X_i)=\sum_{i=1}^n\phi_{\theta,1}(X_i) =0\).
    Thus an interior MLE is an HK$_1$ZE.
\end{remark}

The $\HK_0$ equation could reasonably be called the ``potential equation'', and, under the assumptions of Theorem~\ref{thm:HK-Z-attainment}, the resulting $Z$-estimator is asymptotically sensitivity-efficient.
In fact, if the $\HK_0$ potential is known in closed form, then the HK$_0$ZE is a much simpler estimator than the WPE proposed in \cite{WCR}; it is easier to compute in practice, and its asymptotic sensitivity-efficiency requires many fewer assumptions.

The Pareto efficiency of HKZE will follow by combining Theorem~\ref{thm:z-est} with the following intermediate result that identifies the normalization in the asymptotic covariance and
sensitivity formulas~\eqref{eqn:Z-est-VS} as precisely the \(\HK_\lambda\) information.
By considering the special case $\lambda=1$ below, one can see that this is a generalization of the representation of the Fisher information as the negative expected Hessian of the log-density with respect to the parameter.

\begin{lemma}[\(\HK_\lambda\) information identity]
\label{lem:HK-information-identity}
    Fix \(0\leq\lambda\leq1\). Suppose
    \(\Pcal=\{P_\theta:\theta\in\Theta\}\subseteq
    \Pcal_{2,\mathrm{ac}}(\R^d)\) is \(\HK_\lambda\) regular,
    where \(\Theta\subseteq\R^p\) is open.
    Let \(\phi_{\theta,\lambda}\) denote its score function.
    Fix \(\theta_0\in\Theta\), and choose \(\varepsilon>0\) such that
    \(U:=B(\theta_0,\varepsilon)\) satisfies
    \(\overline U\subset\Theta\).
    Set \(s(t,x)=\phi_{t,\lambda}(x)\) and assume \(s\) is jointly
    \(C^1\) and satisfies condition (iv) of Theorem~\ref{thm:z-est}.
    In addition, assume that:
    \begin{enumerate}
        \item[(i)]
        \(\sup_{r,t\in U}\Ebb_r\left[\|\phi_{t,\lambda}(X)\|^2+\|D\phi_{t,\lambda}(X)\|_{\mathrm F}^2\right]<\infty\).

        \item[(ii)]
        If \(\lambda>0\), the map
        \((r,t)\mapsto\Ebb_r[\phi_{t,\lambda}(X)\phi_{r,\lambda}(X)^\top]\)
        is continuous at \((\theta_0,\theta_0)\).
        If \(\lambda<1\), the map
        \((r,t)\mapsto\Ebb_r[D\phi_{t,\lambda}(X)D\phi_{r,\lambda}(X)^\top]\)
        is continuous at \((\theta_0,\theta_0)\).
    \end{enumerate}
    Then
    \begin{equation}\label{eq:HK-information-identity}
        \Ebb_{\theta_0}\left[
            D_\theta\phi_{\theta_0,\lambda}(X)
        \right]
        =
        -I_\lambda(\theta_0).
    \end{equation}
\end{lemma}

Now we give our main result on constructing Pareto efficient estimators.

\begin{theorem}[Asymptotic efficiency of HKZE]
\label{thm:HK-Z-attainment}
    Fix \(0\leq\lambda\leq1\). Suppose
    \(\Pcal=\{P_\theta:\theta\in\Theta\}\subseteq
    \Pcal_{2,\mathrm{ac}}(\R^d)\) is \(\HK_\lambda\) regular with
    \(\Theta\subseteq\R^p\) open. Fix \(\theta_0\in\Theta\), and let
    \(X_1,\ldots,X_n\stackrel{\mathrm{iid}}{\sim}P_{\theta_0}\).
    Suppose that the conditions of Lemma~\ref{lem:HK-information-identity}
    and the assumptions of Theorem~\ref{thm:z-est} are satisfied by the score function \(s(t,x)=\phi_{t,\lambda}(x)\), and let \(T_{n,\lambda}\) be a consistent HKZE. If \(I_\lambda(\theta_0)\) is invertible, then as $n\to\infty$, 
    \[
        \sqrt n(T_{n,\lambda}-\theta_0)
        \overset{\mathrm d}{\to}
        \mathcal N\bigl(0,V_\lambda(\theta_0)\bigr)\qquad\textnormal{and}\qquad n\sum_{i=1}^n
        D_{x_i}T_{n,\lambda}(D_{x_i}T_{n,\lambda})^\top
        \overset{P_{\theta_0}}{\to}
        S_\lambda(\theta_0).
    \] 
\end{theorem}

Lastly, we may revisit our running examples and construct specific Pareto efficient estimators.

\begin{example}[Laplace location]\label{ex:Laplace-HK-Z}\label{laplace_lb}
    In the Laplace location family of Example~\ref{ex:Laplace-reg}, we can use the known form of the $\HK$ score functions to explicitly write $V_\lambda(\theta)$ and $S_\lambda(\theta)$ as
    \begin{equation*}
        V_{\lambda}(\theta)=
        1+
        \sqrt{\frac{1-\lambda}{1+3\lambda}}\qquad \textnormal{ and }\qquad S_{\lambda}(\theta) =
        \frac{
            \left(
                \sqrt{\frac{1+3\lambda}{1-\lambda}}+1
            \right)^2
        }{
            4\sqrt{\frac{1+3\lambda}{1-\lambda}}
        },
    \end{equation*}
    for all $0\le \lambda<1$, while in the pure transport case we have $(V_0(\theta),S_0(\theta))=(2,1)$ and in the pure source case we have  $\lim_{\lambda\uparrow1}(V_\lambda(\theta),S_\lambda(\theta))=(1,\infty)$.
    Moreover, the \(\HK_\lambda\) score function satisfies the conditions of Theorem~\ref{thm:HK-Z-attainment} for \(0\le\lambda<1\), so the estimator given by
    \[
        T_{n,\lambda}
        =
        \begin{cases}
            \frac1n\sum_{i=1}^nX_i,
            & \lambda=0,\\
            \textnormal{unique solution of }
            \sum_{i=1}^n
            \operatorname{sgn}(X_i-\theta)
            \left(
                1-e^{-\alpha_\lambda|X_i-\theta|}
            \right)=0,
            & 0<\lambda<1,\\
            \textnormal{median of } X_1,\ldots,X_n,
            & \lambda=1,
        \end{cases}
    \]
    coincides with the HKZE and is hence a Pareto-efficient estimator for each \(0\le\lambda<1\).
    One can also check directly (recall \cite[Example~1.3]{WCR}) that asymptotic variance and sensitivity of the sample median are $(1,\infty)$, hence $T_{n,1}$ is also Pareto efficient for $\lambda=1$.
\end{example}

Since Example~\ref{ex:Laplace-HK-Z} identifies a family of estimators which interpolate between the sample mean and the sample median, it is natural to investigate the relationship between Huber's $M$-estimators (previously discussed in Example~\ref{ex:huber}) and Pareto-efficient estimation in the Laplace location family.
It turns out that Huber's $M$-estimators are not Pareto-efficient estimators of $\theta$ in this setting, but they are quite close to being so; in Proposition~\ref{prop:Laplace-Huber-rel-eff} we introduce a notion of worst-case relative Pareto efficiency, and we show that Huber's $M$-estimators achieve a value of 0.954 (where $1$ would correspond to efficiency). One may additionally wonder whether a single parametric model can have Huber estimating functions as its $\HK_\lambda$ score functions for all $0<\lambda<1$. Proposition~\ref{prop:huber-nonexistence} shows that this is impossible under its stated assumptions, even when the clipping threshold depends on $\lambda$.
Nonetheless, the HKZE of Example~\ref{ex:Laplace-HK-Z} shares many qualitative features of Huber's $M$-estimator, since it can be written as
\begin{equation*}
    T_{n,\lambda}:=\underset{\theta\in\Theta}{\arg\min\,}\frac{1}{n}\sum_{i=1}^{n}L_{\lambda}(X_i-\theta)\qquad\textnormal{where}\qquad L_{\lambda}(r) = \frac{1}{\alpha_{\lambda}}\left(|r|-\frac{1-e^{-\alpha_{\lambda}|r|}}{\alpha_{\lambda}}\right).
\end{equation*}
We note that the loss function $L_{\lambda}$ is even, locally quadratic near $r= 0$, and asymptotically linear as $r\to\infty$;
as such, $T_{n, \lambda}$ possesses many desirable properties from robust statistics.
The primary difference is that $L_{\lambda}$ smoothly transitions between its regimes whereas Huber's loss function has a sharp transition, so we may regard it as an additional example of a ``Huber-like'' loss function such as those constructed from the functions $\log(\cosh(r))$ and $\sqrt{r^2+1}-1$.
For these reasons we believe that the loss function $L_{\lambda}$ may be useful for ``Huberizing'' the residuals in many location-type models (e.g., linear regression) if it is also desired to achieve small sensitivity.

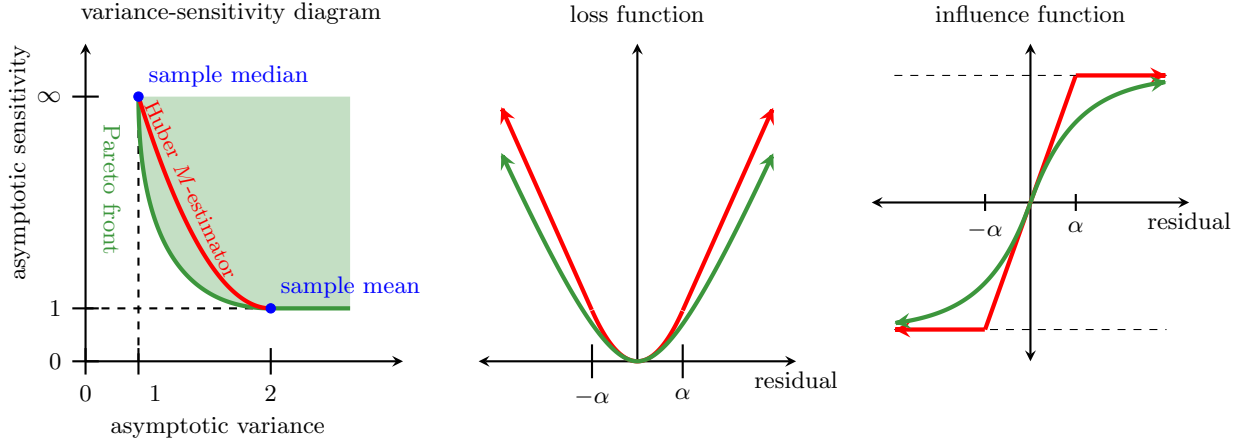
\begin{figure}
	\centering
	\begin{tikzpicture}[font=\footnotesize]

      \def\PaneloneS{3.5}   
      \def\baseS{3}         
      \def\panelIIIx{12.5}  
      \def\yscaleL{6}\def\yscaleD{3}\def\del{0.2}\def\xm{0.6}\def\ax{0.7}
      \pgfmathsetmacro{\halfW}{\ax*\baseS}                       
      \pgfmathsetmacro{\Rone}{\PaneloneS*1.2}                    
      \pgfmathsetmacro{\panelIIx}{(\panelIIIx+\Rone-\halfW)/2}   
      \pgfmathsetmacro{\yscII}{\PaneloneS*1.6}
      \pgfmathsetmacro{\yscIII}{\PaneloneS*0.8}
      \pgfmathsetmacro{\panelIIIyc}{\PaneloneS*0.6}
      \pgfmathsetmacro{\titleY}{\PaneloneS*1.2+0.4}
      \pgfmathsetmacro{\titleyI}{\titleY/\PaneloneS}
      \pgfmathsetmacro{\titleyII}{\titleY/\yscII}
      \pgfmathsetmacro{\titleyIII}{(\titleY-\panelIIIyc)/\yscIII}
    
      \begin{scope}[scale=\PaneloneS]
        \node at (0.55,\titleyI) {variance-sensitivity diagram};
        \draw[thick, -stealth] (0,0) -- (1.2,0);
        \draw[thick, -stealth] (0,0) -- (0,1.2);
        \draw[thick, dashed] (0,0.2) -- (1,0.2);
        \draw[thick, dashed] (0.2,0) -- (0.2,1);
        \node at (0.5, -0.25) {asymptotic variance};
        \node[rotate=90] at (-0.25, 0.75) {asymptotic sensitivity};
        \draw[thick] (-0.05,0)   node[left]{$0$}      -- (0.05,0);
        \draw[thick] (-0.05,0.2) node[left]{$1$}      -- (0.05,0.2);
        \draw[thick] (-0.05,1)   node[left]{$\infty$} -- (0.05,1);
        \draw[thick] (0,-0.05)   node[below]{$0$}       -- (0,0.05);
        \draw[thick] (0.2,-0.05) node[below right]{$1$} -- (0.2,0.05);
        \draw[thick] (0.7,-0.05) node[below]{$2$}       -- (0.7,0.05);
        \path [fill=mygreen, thick, opacity=0.3, domain=0:1, samples=100]
          plot ({0.2+0.5*\x^2}, {0.2+0.8*(\x-1)^2}) -- (1,0.2) -- (1,1) -- cycle;
        \draw [color=mygreen, ultra thick, domain=0:1, samples=100]
          plot ({0.2+0.5*\x^2}, {0.2+0.8*(\x-1)^2}) -- (1,0.2);
        \node[color=mygreen,rotate=-90] at (0.1, 0.65) {Pareto front};
        \draw [color=red, ultra thick, domain=0:1, samples=100]
          plot ({0.2+0.5*\x}, {0.2+0.8*(\x-1)^2});
        \node[color=red,rotate=-65,font=\scriptsize] at (0.4, 0.65) {Huber $M$-estimator};
        \filldraw[color=blue] (0.2, 1) circle (0.5pt) node[above right] {sample median};
        \filldraw[color=blue] (0.7, 0.2) circle (0.5pt) node[above right] {sample mean};
      \end{scope}
    
      \begin{scope}[shift={(\panelIIx,0)}]
       \begin{scope}[xscale=\baseS, yscale=\yscII]
        \node at (0,\titleyII) {loss function};
        \draw[thick, stealth-stealth] (-\ax,0) -- (\ax,0);
        \draw[thick, -stealth] (0,0) -- (0,0.75);
        \draw[thick] (-\del,-0.04) node[below]{$-\alpha$} -- (-\del,0.04);
        \draw[thick] ( \del,-0.04) node[below]{$\alpha$}  -- ( \del,0.04);
        \draw[red, ultra thick, stealth-, domain=-\xm:-\del, samples=100] plot ({\x},{\yscaleL*\del*(-\x-\del*0.5)});
        \draw[red, ultra thick, domain=-\del:\del, samples=100] plot ({\x},{\yscaleL*0.5*\x*\x});
        \draw[red, ultra thick, -stealth, domain=\del:\xm,   samples=100] plot ({\x},{\yscaleL*\del*(\x-\del*0.5)});
        \draw[mygreen, ultra thick, stealth-stealth, domain=-\xm:\xm, samples=200]
          plot ({\x},{\yscaleL*(\del*(abs(\x) - \del*(1-exp(-abs(\x)/\del))))});
        \node[below] at (\ax,0) {residual};
       \end{scope}
      \end{scope}
    
      \begin{scope}[shift={(\panelIIIx,\panelIIIyc)}]
       \begin{scope}[xscale=\baseS, yscale=\yscIII]
        \node at (0,\titleyIII) {influence function};
        \draw[thick, stealth-stealth] (-\ax,0) -- (\ax,0);
        \draw[thick, stealth-stealth] (0,-0.75) -- (0,0.75);
        \draw[thick] (-\del,-0.04) node[below]{$-\alpha$} -- (-\del,0.04);
        \draw[thick] ( \del,-0.04) node[below]{$\alpha$}  -- ( \del,0.04);
        \draw[dashed] (-\xm,{ \yscaleD*\del}) -- (\xm,{ \yscaleD*\del});
        \draw[dashed] (-\xm,{-\yscaleD*\del}) -- (\xm,{-\yscaleD*\del});
        \draw[red, ultra thick, stealth-, domain=-\xm:-\del, samples=100] plot ({\x},{-\yscaleD*\del});
        \draw[red, ultra thick, domain=-\del:\del, samples=100] plot ({\x},{ \yscaleD*\x});
        \draw[red, ultra thick, -stealth, domain=\del:\xm,   samples=100] plot ({\x},{ \yscaleD*\del});
        \draw[mygreen, ultra thick, stealth-, domain=-\xm:0, samples=100] plot ({\x},{\yscaleD*\del*(exp(\x/\del)-1)});
        \draw[mygreen, ultra thick, -stealth, domain=0:\xm,  samples=100] plot ({\x},{\yscaleD*\del*(1-exp(-\x/\del))});
        \node[below] at (\ax,0) {residual};
       \end{scope}
      \end{scope}
    \end{tikzpicture}
	\caption{The $M$-estimation interpretation of the HKZE in the Laplace location family.  
    We show the variance-sensitivity diagram, and the curve depicting the variance-sensitivity coordinates achieved by Huber's $M$-estimators (left).
    Then, for both the HKZE and Huber estimators, we show their loss functions (center) and influence functions (right), which share many qualitative properties.}
	\label{fig:laplace_diagram}
\end{figure}

\begin{example}[Uniform scale]\label{ex:Uniform-HK-Z}\label{uniform_lb}
    Consider the uniform scale family from Example~\ref{ex:unif-reg}, for which we can calculate
    \begin{equation*}
        V_{\lambda}(\theta)=
        \frac{
            \theta^4 \beta_\lambda^2
            \left[
                \frac12+\frac{\sinh(2\theta\beta_\lambda)}{4\theta\beta_\lambda}
                -
                \left(\frac{\sinh \theta\beta_\lambda}{\theta\beta_\lambda}\right)^2
            \right]
        }{
            \left(\theta\beta_\lambda\cosh (\theta\beta_\lambda)-\sinh (\theta\beta_\lambda)\right)^2
        }\qquad\textnormal{ and }\qquad S_{\lambda}(\theta)=
        \frac{
            \theta^4\beta_\lambda^4
            \left[
                \frac{\sinh(2\theta\beta_\lambda)}{4\theta\beta_\lambda}
                -
                \frac12
            \right]
        }{
            \left(\theta\beta_\lambda\cosh (\theta\beta_\lambda)-\sinh (\theta\beta_\lambda)\right)^2
        },
    \end{equation*}
    for all $0<\lambda<1$, where \(\beta_\lambda:=\sqrt{\lambda/(1-\lambda)}\).
    Also note that at the pure transport and pure source endpoints we have $(V_0(\theta),S_0(\theta))=(\theta^2/5,3)$ and $\lim_{\lambda\uparrow1}(V_\lambda(\theta),S_\lambda(\theta))=(0,\infty)$.
    Since the $\HK$ score functions satisfy the conditions of Theorem~\ref{thm:HK-Z-attainment} for \(0\le\lambda<1\), the estimator defined as
    \[
        T_{n,\lambda}
        =
        \begin{cases}
            \sqrt{\frac3n\sum_{i=1}^nX_i^2},
            & \lambda=0,\\
            \textnormal{the unique positive solution of }
            \frac{
                \sinh(\beta_\lambda\theta)
            }{
                \beta_\lambda\theta
            }
            =
            \frac1n\sum_{i=1}^n
            \cosh(\beta_\lambda X_i),
            & 0<\lambda<1,\\
            X_{(n)},
            & \lambda=1,
        \end{cases}
    \]
    corresponds to the HKZE for \(0\le\lambda<1\) and is therefore Pareto efficient.
    Although this model is not $\HK_1$ regular, it is easy to show $n(T_{n,1}-\theta_0)
        \to
        -\textnormal{Exp}(\theta_0)$ in distribution (where \(\textnormal{Exp}(\theta_0)\) has mean \(\theta_0\)) and $n\sum_{i=1}^n\|\nabla_{x_i}T_{n,1}\|^2=n\to\infty$, hence $T_{n,1}$ is also Pareto efficient for $\lambda=1$.
\end{example}

As in Example~\ref{ex:Laplace-HK-Z}, the HKZE from Example~\ref{ex:Uniform-HK-Z} can also be written as a suitable $M$-estimator, namely as
\begin{equation*}
    \begin{gathered}
        T_{n,\lambda}:=\underset{\theta\in\Theta}{\arg\min\,}\frac{1}{n}\sum_{i=1}^{n}L_{\lambda}(\theta,X_i)\\
        \textnormal{where}\qquad L_{\lambda}(\theta,x) = -\frac{1}{\lambda}\left(\cosh(\beta_{\lambda}x)\log\left|\tanh\left(\frac{\beta_{\lambda}\theta}{2}\right)\right|-\log \theta\right),
    \end{gathered}
\end{equation*}
which we plot in Figure~\ref{fig:uniform-scale-m-est}.
To understand the qualitative properties of this loss function, note that we have $L_{\lambda}(\theta,x)\sim \lambda^{-1}\log \theta$ as $\theta\to \infty$, and $L_{\lambda}(\theta,x)\sim \lambda^{-1}(\cosh(\beta_{\lambda}x)-1)\log (1/\theta)$ as $\theta\to 0$, with a smooth interpolation in between.
In fact, these qualitative properties are shared across the $M$-estimators representing HKZE in many scale families (e.g., Gaussian, exponential);
we have $L_{\lambda}(\theta,x)\sim \lambda^{-1}\log \theta$ as $\theta\to \infty$, and $L_{\lambda}(\theta,x)\sim \lambda^{-1}(f_{\lambda}(x)-1)\log (1/\theta)$ as $\theta\to 0$, for a suitable function $f_{\lambda}:\R^d\to\R$.
We are not aware if this $M$-estimator has been studied before, but we note that the empirical loss for the uniform scale family is quasiconvex in $\theta$ and hence can be easily numerically minimized.
Indeed, for $0<\lambda<1$ and $\theta>0$,
\[
    \frac{\diff}{\diff\theta}\left[\frac1n\sum_{i=1}^nL_\lambda(\theta,X_i)\right]
    =
    \frac{\beta_\lambda}{\lambda\sinh(\beta_\lambda\theta)}
    \left[
        \frac{\sinh(\beta_\lambda\theta)}{\beta_\lambda\theta}
        -\frac1n\sum_{i=1}^n\cosh(\beta_\lambda X_i)
    \right].
\]
Since $\theta\mapsto\sinh(\beta_\lambda\theta)/(\beta_\lambda\theta)$ increases strictly from $1$ to $\infty$ and $n^{-1}\sum_{i=1}^n\cosh(\beta_\lambda X_i)>1$ almost surely, the empirical loss decreases for $0<\theta<T_{n,\lambda}$ and increases for $\theta>T_{n,\lambda}$.

Lastly, we note that it is not true in general that an arbitrary $Z$-estimator can be written as an $M$-estimator.
When $p=1$, a continuous estimating function can be integrated with respect to $\theta$ to obtain a loss whose negative gradient is the estimating function; for $p\ge 2$, such a representation requires the estimating function to be a \textit{conservative} vector field in $\theta$ for each $x$.
A solution of the estimating equation is then a stationary point of the empirical loss; its minimizing property requires separate verification.
An interesting avenue for future work would be to understand when $\HK$ score functions correspond to conservative vector fields, and to study the resulting class of $M$-estimators.

\begin{figure}
  \centering
  \begin{tikzpicture}[font=\footnotesize]
    \def\xs{0.30}      
    \def\ys{3.2}       
    \def\bet{1}        
    \def\Acoef{2}      
    \def\Bcoef{1}      
    \def\Lsc{0.1714}   
    \def\Dsc{0.2617}   
    \def\tmax{12}      
    \def\panelIIx{4.6} 
    \pgfmathsetmacro{\cy}{0.66*\ys}              
    \pgfmathsetmacro{\titleY}{0.25*\ys+\cy+0.25} 
    \pgfmathsetmacro{\titleone}{\titleY/\ys}
    \pgfmathsetmacro{\titletwo}{(\titleY-\cy)/\ys}

    \begin{scope}[xscale=\xs, yscale=\ys]
      \node at (6,\titleone) {loss function};
      \draw[thick, -stealth] (0,0) -- (\tmax*1.06,0);      
      \draw[thick, -stealth] (0,0) -- (0,0.85);            
      \draw[mygreen, ultra thick, stealth-stealth,
            domain=0.08:\tmax, samples=220, smooth]
        plot ({\x},{\Lsc*(-\Acoef*ln((1-exp(-\bet*\x))/(1+exp(-\bet*\x)))
                          + \Bcoef*ln(\x))});
      \node[below] at (\tmax*0.98,0) {scale};
    \end{scope}

    \begin{scope}[shift={(\panelIIx,\cy)}]
     \begin{scope}[xscale=\xs, yscale=\ys]
      \node at (6,\titletwo) {influence function};
      \draw[thick, -stealth] (0,0) -- (\tmax*1.06,0);      
      \draw[thick, stealth-stealth] (0,-0.72) -- (0,0.25); 
      \draw[mygreen, ultra thick, stealth-,
            domain=0.4:\tmax, samples=220, smooth]
        plot ({\x},{\Dsc*(-2*\Acoef*\bet*exp(-\bet*\x)/(1-exp(-2*\bet*\x))
                          + \Bcoef/\x)});
      \node[below] at (\tmax*0.98,0) {scale};
     \end{scope}
    \end{scope}
  \end{tikzpicture}
    \caption{The $M$-estimation interpretation of the HKZE in the uniform scale family.
    We show the loss function $L_{\lambda}(\theta,x)$ as a function of the scale $\theta>0$ (left), and its influence function $\phi_{\theta,\lambda}(x)$ (right).}
    \label{fig:uniform-scale-m-est}
\end{figure}
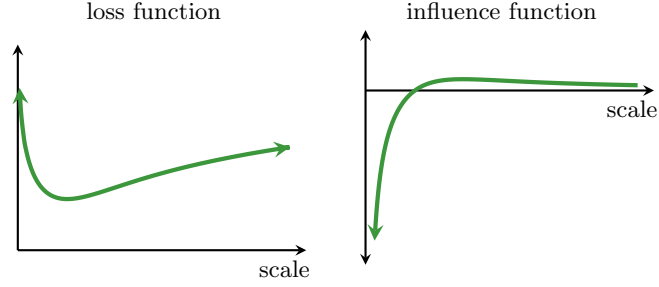

\begin{example}[Gaussian mean-variance]\label{rem:Gaussian-LS-HK-Z}\label{Gaussian_loc_lb}
    Consider the Gaussian mean-variance family of Example~\ref{ex:Gaussian-loc-scale}, i.e., $P_{\theta}=\Ncal(\mu,\sigma^2)$, where
    $\theta=(\mu,\sigma^2)\in\Theta:=\R\times(0,\infty)$.
    Now $V_\lambda(\theta)$ and $S_\lambda(\theta)$ are matrices, and we can directly compute them as
    \begin{equation*}
        V_\lambda(\theta)
        =
        \begin{pmatrix}
            \sigma^2 & 0\\
            0 & 2\sigma^4
        \end{pmatrix}\qquad \textnormal{ and }\qquad S_\lambda(\theta)
        =
        \begin{pmatrix}
            1 & 0\\
            0 & 4\sigma^2
        \end{pmatrix},
    \end{equation*}
    respectively, for all $0\le \lambda\le 1$.
    Note that $V_{\lambda}(\theta)$ and $S_{\lambda}(\theta)$ do not depend on $0\le \lambda\le 1$; this follows directly from the fact that the $\HK_{\lambda}$ score functions involve $\lambda$ only in the normalizing constant.
    The conditions of Theorem~\ref{thm:HK-Z-attainment} are satisfied for all \(0\le\lambda\le1\), and the \(\HK_\lambda\) score functions for \(\mu\) and \(\sigma^2\) are nonvanishing scalar multiples of \(x-\mu\) and \((x-\mu)^2-\sigma^2\), respectively; thus, the \(\HK_\lambda\) score equation has the same root for every \(0\leq\lambda\leq1\),
    \[
        \hat\mu=\overline X,
        \qquad
        \hat\sigma^2
        =
        \frac1n\sum_{i=1}^n(X_i-\overline X)^2.
    \]
    In other words, $T_{n,\lambda}:=(\bar{X},\hat{\sigma}^2)^{\top}$ is $\lambda$-Pareto efficient for all $0\le \lambda\le 1$.
\end{example}

\begin{example}[Gaussian correlation]\label{ex:Gaussian-corr-HK-Z}\label{gaussian_cor_lb}
    Consider the Gaussian correlation family from Example~\ref{ex:Gaussian-corr-reg}, i.e., $P_{\theta} = \Ncal(0,\Sigma_{\theta})$, where $\Sigma_{\theta}$ is given by~\eqref{eqn:Gaussian-corr-matrix-def}, for all $-1<\theta<1$.
    Then we can compute
    \begin{equation*}
        V_{\lambda}(\theta)=
        \frac{
            (1-\theta^2)^2
            \left((2-\lambda)^2+\lambda^2\theta^2\right)
        }{
            \left(2-\lambda(1-\theta^2)\right)^2
        }\qquad \textnormal{ and }\qquad S_{\lambda}(\theta)=
        \frac{
            2(1-\theta^2)
            \left((2-\lambda)^2+\lambda(4-\lambda)\theta^2\right)
        }{
            \left(2-\lambda(1-\theta^2)\right)^2
        },
    \end{equation*}
    for all $0\le \lambda\le 1$. 
    In order to describe the HKZE, let us write
    \begin{equation}
        \hat{\sigma}_X^2=\frac1n\sum_{i=1}^nX_i^2,\qquad \hat{\sigma}_Y^2=\frac1n\sum_{i=1}^nY_i^2, \qquad\textnormal{and}\qquad \hat\sigma_{XY}:=\frac1n\sum_{i=1}^nX_iY_i 
    \end{equation}
    At \(\lambda=0\), the HKZE is \(T_{n,0}=2\hat\sigma_{XY}/(\hat\sigma_X^2+\hat\sigma_Y^2)\) almost surely.
    For \(0<\lambda\le1\), we may choose \(T_{n,\lambda}\) to be the unique root \(t\in(-1,1)\) of \(\lambda t^3-\lambda\hat\sigma_{XY}t^2+(\hat\sigma_X^2+\hat\sigma_Y^2-\lambda)t+(\lambda-2)\hat\sigma_{XY}=0\) with the same sign as \(\hat\sigma_{XY}\). The selected \(T_{n,\lambda}\) satisfies the regularity conditions of Theorem~\ref{thm:HK-Z-attainment} and therefore is Pareto efficient; see Appendix~\ref{app:gaussian-correlation-verification}.
    Also, note that this estimator coincides with the MLE at \(\lambda=1\).
\end{example}

\begin{remark}[Asymptotic sensitivity-efficiency of Pearson correlation in Gaussian correlation family]
    By massaging the HK$_0$ZE in the previous example, we can find other asymptotically sensitivity-efficient estimators.
    Specifically, since the maps $(\sigma_{XY},\sigma_X^2,\sigma_Y^2)\mapsto 2\sigma_{XY}/(\sigma^2_X+\sigma_Y^2)$ and $(\sigma_{XY},\sigma_X^2,\sigma_Y^2)\mapsto\sigma_{XY}/(\sigma_X^2\sigma_Y^2)^{1/2}$ have the same value and gradient
    at $(\sigma_{XY},\sigma_X^2,\sigma_Y^2)=(\theta_0,1,1)$, it follows that
    \begin{equation*}
        \tilde{T}_{n,0}:= \frac{\hat{\sigma}_{XY}}{\sqrt{\hat{\sigma}_X^2\hat{\sigma}_Y^2}}
    \end{equation*}
    is consistent and has the same asymptotic variance and sensitivity as $T_{n,0}$, which is asymptotically sensitivity-efficient.
    Interestingly, note that $ \tilde{T}_{n,0}$ is the Pearson correlation in the model where the marginal variances are not known; although our model fixes the marginal variances, the asymptotic sensitivity somehow improves by weakening the identifiability assumptions.
\end{remark}

\begin{example}[Pareto index]\label{pareto_lb}
    Consider the Pareto index family with density $p_\theta(x)=\theta x^{-(\theta+1)}\mathbf 1_{\{x\ge 1\}}$ for $\theta\in\Theta:=(4,\infty)$.
    The stronger restriction $\theta>4$ ensures that $V_0(\theta)$ is finite, since $\Ebb_\theta[\phi_{\theta,0}(X)^2]<\infty$ if and only if $\theta>4$.
    As we saw before, the weak-sense solution to the $\HK_\lambda$ equation for $0<\lambda<1$ does not admit an elementary closed form. Nonetheless, we can compute
    \begin{equation*}
        V_{0}(\theta)=
        \frac{
            \theta(\theta-2)^2(5\theta^2-28\theta+40)
        }{
            4(\theta-4)^3
        }\qquad \textnormal{ and }\qquad S_0(\theta)=\frac{\theta(\theta-2)^3}{2},
    \end{equation*}
    and $V_1(\theta)=
        \theta^2$ and $S_1(\theta)=
        \theta^5/(\theta+2)$ for all $\theta>4$.

    For the broader parameter range \(\theta>2\), the lower-bound coordinates \((V_\lambda(\theta),S_\lambda(\theta))\) remain well defined and finite for \(0<\lambda<1\), as follows from the construction in Appendix~\ref{app:pareto-verification}.
\end{example}

\subsection{Efficiency at the Extremal Cases}\label{subsec:selection}

Our theory of lower bounds (and how to achieve them) applies at each $0\le\lambda\le1$ for which the required hypotheses hold.
These tools allow us to also say something precise about the extremal cases $\lambda\in\{0,1\}$, i.e., about the problem of minimizing purely the asymptotic variance or purely the asymptotic sensitivity. We explore a few such results in this subsection, which is limited to the case of $p=1$ for simplicity.

To set things up, we recall the class $\estims_n$ of locally Lipschitz unbiased estimators of $\theta$ with continuous variance and sensitivity from Definition~\ref{def:estims-class}, and we consider the optimization problems
\begin{equation}\label{eqn:selection-opt-source}
    \begin{cases}
        \textnormal{minimize} & \displaystyle\liminf_{n\to\infty} n\Sen_\theta(T_n) \\[0.3em]
        \textnormal{over} & \{T_n\}_{n\in\N}\textnormal{ with each }T_n\in\estims_n \\[0.3em]
        \textnormal{with} & \displaystyle\lim_{n\to\infty} n\Var_\theta(T_n)=\inf_{0<\lambda<1}V_{\lambda}(\theta)
    \end{cases}
\end{equation}
and
\begin{equation}\label{eqn:selection-opt-transport}
    \begin{cases}
        \textnormal{minimize} & \displaystyle\liminf_{n\to\infty} n\Var_\theta(T_n) \\[0.3em]
        \textnormal{over} & \{T_n\}_{n\in\N}\textnormal{ with each }T_n\in\estims_n \\[0.3em]
        \textnormal{with} & \displaystyle\lim_{n\to\infty} n\Sen_\theta(T_n)=\inf_{0<\lambda<1}S_{\lambda}(\theta).
    \end{cases}
\end{equation}
In words, problem~\eqref{eqn:selection-opt-source} aims to minimize the asymptotic sensitivity among all estimators with the optimal asymptotic variance, and problem~\eqref{eqn:selection-opt-transport} aims to minimize the asymptotic variance among all estimators with the optimal asymptotic sensitivity.
The main result in this subsection is a version of our lower bounds (Theorem~\ref{thm:HK-CR-bound}) for these two problems.

\begin{proposition}\label{prop:Pareto-extremal}

    Suppose $\Pcal=\{P_\theta:\theta\in\Theta\}
    \subseteq\Pcal_{2,\textnormal{ac}}(\R^d)$
    is $\HK$ regular with $\Theta\subseteq\R$, that 
    $\vartheta\mapsto
        \Ebb_\vartheta[\|\phi_{\vartheta,\lambda}(X)\|^2]$ and $\vartheta\mapsto
        \Ebb_\vartheta[\|\nabla\phi_{\vartheta,\lambda}(X)\|^2]$ are continuous on $\Theta$ for each $0<\lambda<1$, and that $T_n\in\estims_n$ for each $n\in\N$.
        Then, for every $\theta\in\Theta$ such that $I_\lambda(\theta)>0$ for all $0<\lambda<1$, the following holds:
        \begin{equation}\label{eqn:Pareto-extremal}
        \begin{split}
            \textnormal{if } \lim_{n\to\infty}n\Var_{\theta}(T_n)= \inf_{0<\lambda<1}V_\lambda(\theta), &\textnormal{ then }\liminf_{n\to\infty}n\Sen_{\theta}(T_n)\ge \sup_{0<\lambda<1}S_{\lambda}(\theta); \\
            \textnormal{if } \lim_{n\to\infty}n\Sen_{\theta}(T_n)= \inf_{0<\lambda<1}S_\lambda(\theta), &\textnormal{ then }\liminf_{n\to\infty}n\Var_{\theta}(T_n)\ge \sup_{0<\lambda<1}V_{\lambda}(\theta).
        \end{split}
    \end{equation}
    Thus, the value of~\eqref{eqn:selection-opt-source} is at least $\sup_{0<\lambda<1}S_{\lambda}(\theta)$ and the value of~\eqref{eqn:selection-opt-transport} is at least $\sup_{0<\lambda<1}V_{\lambda}(\theta)$.
\end{proposition}

One important consequence of this result is that, by combining it with the developments of Subsection~\ref{subsec:HK-score}, we can construct estimators which solve problems~\eqref{eqn:selection-opt-source} and~\eqref{eqn:selection-opt-transport} in a suitable sense.
As before, we follow the classical perspective whereby, although the problem is formulated in terms of a sequence of unbiased estimators and their rescaled variance and sensitivity, we construct a solution as a consistent estimator with the appropriate asymptotic variance and sensitivity.

\begin{example}[Uniform scale]
    The work \cite{WCR} constructs two different consistent estimators which achieve the optimal asymptotic sensitivity in the uniform scale family of Example~\ref{ex:unif-reg}. Namely,
    \begin{equation*}
        T_n^{\WPE} = \frac{3}{n^2}\sum_{i=1}^{n}iX_{(i)},\qquad \textnormal{and}\qquad T_{n,0} = \sqrt{\frac{3}{n}\sum_{i=1}^{n}X_i^2}.
    \end{equation*}
    The first is the Wasserstein projection estimator (WPE) and the second estimator coincides with our HK$_0$ZE from Example~\ref{ex:Uniform-HK-Z}, although \cite{WCR} found the latter by different means.
    One can directly compute that both estimators have asymptotic variance equal to $\theta^2/5$ under $P_{\theta}=\textnormal{U}[0,\theta]$, but this raises the question of whether this asymptotic variance is optimal (among all estimators with optimal asymptotic sensitivity).
    We can then compute $\sup_{0<\lambda<1}V_{\lambda}(\theta) = \theta^2/5$, so Proposition~\ref{prop:Pareto-extremal} shows indeed that both are solutions to problem~\eqref{eqn:selection-opt-transport}.
\end{example}

In fact, this example is an instance of the following more general phenomenon.
The upshot of this result is that, by combining with Proposition~\ref{prop:Pareto-extremal}, it shows under very mild conditions that the MLE solves~\eqref{eqn:selection-opt-source} (and analogously for $T_{n,0}$ solving problem~\eqref{eqn:selection-opt-transport}).

\begin{corollary}\label{cor:extremal-efficiency}
    Let $\Pcal=\{P_\theta:\theta\in\Theta\}\subseteq\Pcal_{2,\textnormal{ac}}(\R^d)$
    be $\HK$ regular with $\Theta\subseteq\R$ open, and fix $\theta\in\Theta$ with
    $I_\lambda(\theta)>0$ for all $0<\lambda<1$.
    \begin{itemize}
        \item[(i)] Suppose that $\Pcal$ is $\HK_1$ regular, that
        $\inf_{0<\lambda<1}V_\lambda(\theta)=V_1(\theta)$, and that
        $\sup_{0<\lambda<1}S_\lambda(\theta)=S_1(\theta)$.
        If $T_n^{\MLE}$ satisfies the assumptions of
        Theorem~\ref{thm:HK-Z-attainment} with $\lambda=1$ and $\theta_0=\theta$, then
        \[
            \sqrt n\,(T_n^{\MLE}-\theta)\overset{\mathrm d}{\to}
            \Ncal\Bigl(0,\inf_{0<\lambda<1}V_\lambda(\theta)\Bigr)
            \qquad\textnormal{and}\qquad
            n\sum_{i=1}^n\|\nabla_{x_i}T_n^{\MLE}\|^2\overset{P_\theta}{\to}
            \sup_{0<\lambda<1}S_\lambda(\theta).
        \]
        \item[(ii)] Suppose that $\Pcal$ is $\HK_0$ regular, that
        $\sup_{0<\lambda<1}V_\lambda(\theta)=V_0(\theta)$, and that
        $\inf_{0<\lambda<1}S_\lambda(\theta)=S_0(\theta)$.
        If $T_{n,0}$ satisfies the assumptions of
        Theorem~\ref{thm:HK-Z-attainment} with $\lambda=0$ and $\theta_0=\theta$, then
        \[
            \sqrt n\,(T_{n,0}-\theta)\overset{\mathrm d}{\to}
            \Ncal\Bigl(0,\sup_{0<\lambda<1}V_\lambda(\theta)\Bigr)
            \qquad\textnormal{and}\qquad
            n\sum_{i=1}^n\|\nabla_{x_i}T_{n,0}\|^2\overset{P_\theta}{\to}
            \inf_{0<\lambda<1}S_\lambda(\theta).
        \]
    \end{itemize}
\end{corollary}

Another consequence of Proposition~\ref{prop:Pareto-extremal} exploits the fact that it does not require $\HK$ regularity at the endpoints $\lambda\in\{0,1\}$ but only on $0<\lambda<1$.
As such, we can determine some surprising results on lower bounds even when the model is only assumed $\HK_{\lambda}$ regular for $0<\lambda<1$.
This is best illustrated in some examples.

\begin{example}[Laplace location]\label{ex:Laplace-extreme}
    It is straightforward (see \cite[Example~1.3]{WCR}) to show that, in the Laplace location family of Example~\ref{ex:Laplace-reg}, the sample median has optimal asymptotic variance but non-vanishing sensitivity.
    With Proposition~\ref{prop:Pareto-extremal} we can show much more, simply by computing $\inf_{0<\lambda<1}V_{\lambda}(\theta) = (I(\theta))^{-1} =1$ and $\sup_{0<\lambda<1}S_{\lambda}(\theta)=\infty$.
    Indeed, we conclude the stronger statement that any unbiased estimator with optimal asymptotic variance must have sensitivity decaying strictly slower than $n^{-1}$.
\end{example}

\begin{example}[Uniform scale]\label{ex:Uniform-extreme}
    It is easy to compute (see \cite[Example~1.2]{WCR}) that, in the uniform scale family of Example~\ref{ex:unif-reg}, the sample maximum has variance decaying like $n^{-2}$ and has constant (hence non-vanishing) sensitivity.
    As above, Proposition~\ref{prop:Pareto-extremal} gives a stronger conclusion by computing $\inf_{0<\lambda<1}V_{\lambda}(\theta) = 0$ and $\sup_{0<\lambda<1}S_{\lambda}(\theta)=\infty$:
    Any unbiased estimator with variance decaying strictly faster than $n^{-1}$ must have sensitivity decaying strictly slower than $n^{-1}$.
\end{example}

As the preceding examples show, the importance of Proposition~\ref{prop:Pareto-extremal} is that we do not require any conditions of the form $\inf_{0<\lambda<1}V_{\lambda}(\theta)=(I(\theta))^{-1}$ or $\inf_{0<\lambda<1}S_{\lambda}(\theta)=(J(\theta))^{-1}$, nor do we require $\sup_{0<\lambda<1}V_{\lambda}(\theta)<\infty$ or $\sup_{0<\lambda<1}S_{\lambda}(\theta)<\infty$.
Example~\ref{ex:Uniform-extreme} on the uniform scale family is particularly interesting, since we obtain a very precise conclusion about fundamental lower bounds on the asymptotic variance of unbiased estimators, even though this model is not DQM.

\bibliographystyle{alpha}
\bibliography{biblio}

\newcommand{\etalchar}[1]{$^{#1}$}
\begin{thebibliography}{BDOW21}

\bibitem[AD21]{Aimsworth}
Mark Ainsworth and Justin Dong.
\newblock {G}alerkin neural networks: a framework for approximating variational
  equations with error control.
\newblock {\em SIAM J. Sci. Comput.}, 43(4):A2474--A2501, 2021.

\bibitem[AGS08]{AmbrosioGigliSavare}
Luigi Ambrosio, Nicola Gigli, and Giuseppe Savar{\'e}.
\newblock {\em Gradient Flows in Metric Spaces and in the Space of Probability
  Measures}.
\newblock Lectures in Mathematics ETH Z{\"u}rich. Birkh{\"a}user, Basel, second
  edition, 2008.

\bibitem[AM24]{AmariMatsuda}
Shun-ichi Amari and Takeru Matsuda.
\newblock Information geometry of {W}asserstein statistics on shapes and affine
  deformations.
\newblock {\em Inf. Geom.}, 7(2):285--309, 2024.

\bibitem[Ama85]{Amari}
Shun-ichi Amari.
\newblock {\em Differential-Geometrical Methods in Statistics}, volume~28 of
  {\em Lecture Notes in Statistics}.
\newblock Springer, New York, 1985.

\bibitem[BDOW21]{BartlDrapeauOblojWiesel2021}
Daniel Bartl, Samuel Drapeau, Jan Ob{\l}{\'o}j, and Johannes Wiesel.
\newblock Sensitivity analysis of {W}asserstein distributionally robust
  optimization problems.
\newblock {\em Proc. R. Soc. A}, 477(2256):20210176, 2021.

\bibitem[BE02]{BousquetElisseeff2002}
Olivier Bousquet and Andr{\'e} Elisseeff.
\newblock Stability and generalization.
\newblock {\em J. Mach. Learn. Res.}, 2:499--526, 2002.

\bibitem[BGL14]{BakryGentilLedoux}
Dominique Bakry, Ivan Gentil, and Michel Ledoux.
\newblock {\em Analysis and Geometry of {M}arkov Diffusion Operators}, volume
  348 of {\em Grundlehren der mathematischen Wissenschaften}.
\newblock Springer, Cham, 2014.

\bibitem[BGM{\etalchar{+}}16]{BravermanGargMaNguyenWoodruff}
Mark Braverman, Ankit Garg, Tengyu Ma, Huy~L. Nguyen, and David~P. Woodruff.
\newblock Communication lower bounds for statistical estimation problems via a
  distributed data processing inequality.
\newblock In {\em STOC}, 2016.

\bibitem[BR13]{BerthetRigollet}
Quentin Berthet and Philippe Rigollet.
\newblock Complexity theoretic lower bounds for sparse principal component
  detection.
\newblock In {\em COLT}, 2013.

\bibitem[Bre96a]{BreimanI}
Leo Breiman.
\newblock Bagging predictors.
\newblock {\em Mach. Learn.}, 24(2):123--140, 1996.

\bibitem[Bre96b]{BreimanII}
Leo Breiman.
\newblock Heuristics of instability and stabilization in model selection.
\newblock {\em Ann. Statist.}, 24(6):2350--2383, 1996.

\bibitem[BS08]{Brenner2008}
Susanne~C. Brenner and L.~Ridgway Scott.
\newblock {\em The Mathematical Theory of Finite Element Methods}.
\newblock Texts in Applied Mathematics. Springer, New York, 2008.

\bibitem[CAG20]{ChapelAlayaGasso}
Laetitia Chapel, Mokhtar~Z. Alaya, and Gilles Gasso.
\newblock Partial optimal transport with applications on positive-unlabeled
  learning.
\newblock In {\em NeurIPS}, 2020.

\bibitem[CCST22]{CaiChengSchmitzerThorpe2022}
Tianji Cai, Junyi Cheng, Bernhard Schmitzer, and Matthew Thorpe.
\newblock The linearized {H}ellinger--{K}antorovich distance.
\newblock {\em SIAM J. Imaging Sci.}, 15(1):45--83, 2022.

\bibitem[CLFB26]{ChakrabortyLuoBarber}
Abhinav Chakraborty, Yuetian Luo, and Rina Foygel~Barber.
\newblock Stability and accuracy trade-offs in statistical estimation.
\newblock arXiv preprint arXiv:2601.11701, 2026.

\bibitem[CPSV18]{ChizatPeyreSchmitzerVialard}
L{\'e}na{\"\i}c Chizat, Gabriel Peyr{\'e}, Bernhard Schmitzer, and Fran{\c
  c}ois-Xavier Vialard.
\newblock An interpolating distance between optimal transport and
  {F}isher--{R}ao metrics.
\newblock {\em Found. Comput. Math.}, 18(1):1--44, 2018.

\bibitem[CS90]{CarollStefanksi}
Raymond~J. Carroll and Leonard~A. Stefanski.
\newblock Approximate quasi-likelihood estimation in models with surrogate
  predictors.
\newblock {\em J. Amer. Statist. Assoc.}, 85(411):652--663, 1990.

\bibitem[DJW18]{DuchiJordanWainwright2018}
John~C. Duchi, Michael~I. Jordan, and Martin~J. Wainwright.
\newblock Minimax optimal procedures for locally private estimation.
\newblock {\em J. Amer. Statist. Assoc.}, 113(521):182--201, 2018.

\bibitem[Ehr05]{Ehrgott}
Matthias Ehrgott.
\newblock {\em Multicriteria Optimization}.
\newblock Springer, Berlin, second edition, 2005.

\bibitem[FJU25]{FallahJordanUlichney}
Alireza Fallah, Michael~I. Jordan, and Annie Ulichney.
\newblock The statistical fairness-accuracy frontier.
\newblock arXiv preprint arXiv:2508.17622, 2025.

\bibitem[Fle23]{Fleissner}
Florentine~Catharina Flei{\ss}ner.
\newblock A note on the differentiability of the {H}ellinger--{K}antorovich
  distances.
\newblock {\em Appl. Math. Optim.}, 87(3):37, 2023.

\bibitem[FSFC21]{FatrasSejourneFlamaryCourty}
Kilian Fatras, Thibault S{\'e}journ{\'e}, R{\'e}mi Flamary, and Nicolas Courty.
\newblock Unbalanced minibatch optimal transport; applications to domain
  adaptation.
\newblock In {\em ICML}, 2021.

\bibitem[GM17]{GallouetMonsaingeon2017}
Thomas~O. Gallou{\"e}t and L{\'e}onard Monsaingeon.
\newblock A {JKO} splitting scheme for {K}antorovich--{F}isher--{R}ao gradient
  flows.
\newblock {\em SIAM J. Math. Anal.}, 49(2):1100--1130, 2017.

\bibitem[GMB26a]{GiordanoMeagerBroderickI}
Ryan Giordano, Rachael Meager, and Tamara Broderick.
\newblock An automatic finite-sample robustness metric: when can dropping a
  little data change conclusions? {P}art {I}: definitions and experiments.
\newblock {\em Philos. Trans. R. Soc. A}, 384(2321):20250001, 2026.

\bibitem[GMB26b]{GiordanoMeagerBroderickII}
Ryan Giordano, Rachael Meager, and Tamara Broderick.
\newblock An automatic finite-sample robustness metric: when can dropping a
  little data change conclusions? {P}art {II}: theory and intuition.
\newblock {\em Philos. Trans. R. Soc. A}, 384(2321):20240614, 2026.

\bibitem[GMN14]{GargMaNguyen}
Ankit Garg, Tengyu Ma, and Huy~L. Nguyen.
\newblock On communication cost of distributed statistical estimation and
  dimensionality.
\newblock In {\em NeurIPS}, 2014.

\bibitem[GTJS25]{WCR}
Nicol{\'a}s Garc{\'i}a~Trillos, Adam~Quinn Jaffe, and Bodhisattva Sen.
\newblock {W}asserstein--{C}ram{\'e}r--{R}ao theory of unbiased estimation.
\newblock arXiv preprint arXiv:2511.07414, 2025.

\bibitem[HR09]{HuberRonchetti}
Peter~J. Huber and Elvezio~M. Ronchetti.
\newblock {\em Robust Statistics}.
\newblock Wiley Series in Probability and Statistics. John Wiley \& Sons,
  Hoboken, NJ, second edition, 2009.

\bibitem[Hub64]{Huber1964}
Peter~J. Huber.
\newblock Robust estimation of a location parameter.
\newblock {\em Ann. Math. Statist.}, 35(1):73--101, 1964.

\bibitem[LB20]{LaiBayraktar2020}
Lifeng Lai and Erhan Bayraktar.
\newblock On the adversarial robustness of robust estimators.
\newblock {\em IEEE Trans. Inform. Theory}, 66(8):5097--5109, 2020.

\bibitem[LC98]{LehmannCasella1998}
E.~L. Lehmann and George Casella.
\newblock {\em Theory of Point Estimation}.
\newblock Springer Texts in Statistics. Springer, New York, second edition,
  1998.

\bibitem[Leo17]{Leoni2017}
Giovanni Leoni.
\newblock {\em A First Course in {S}obolev Spaces}, volume 181 of {\em Graduate
  Studies in Mathematics}.
\newblock American Mathematical Society, Providence, RI, second edition, 2017.

\bibitem[LeV07]{LeVeque2007}
Randall~J. LeVeque.
\newblock {\em Finite Difference Methods for Ordinary and Partial Differential
  Equations: Steady-State and Time-Dependent Problems}.
\newblock Society for Industrial and Applied Mathematics, Philadelphia, PA,
  2007.

\bibitem[LKH25]{Lee2025}
Jae~Yong Lee, Seungchan Ko, and Youngjoon Hong.
\newblock Finite element operator network for solving elliptic-type parametric
  {PDE}s.
\newblock {\em SIAM J. Sci. Comput.}, 47(2):C501--C528, 2025.

\bibitem[LLMO26]{LiangLuMuOkumura}
Annie Liang, Jay Lu, Xiaosheng Mu, and Kyohei Okumura.
\newblock Algorithm design: a fairness-accuracy frontier.
\newblock {\em J. Polit. Econ.}, 134(5):1401--1467, 2026.

\bibitem[LM24]{LiuMolinari}
Yiqi Liu and Francesca Molinari.
\newblock Inference for an algorithmic fairness-accuracy frontier.
\newblock In {\em EC}, 2024.

\bibitem[LMS16]{Liero_2016}
Matthias Liero, Alexander Mielke, and Giuseppe Savar{\'e}.
\newblock Optimal transport in competition with reaction: the
  {H}ellinger--{K}antorovich distance and geodesic curves.
\newblock {\em SIAM J. Math. Anal.}, 48(4):2869--2911, 2016.

\bibitem[LMS18]{LieroMielkeSavare}
Matthias Liero, Alexander Mielke, and Giuseppe Savar{\'e}.
\newblock Optimal entropy-transport problems and a new
  {H}ellinger--{K}antorovich distance between positive measures.
\newblock {\em Invent. Math.}, 211(3):969--1117, 2018.

\bibitem[LMS23]{LieroMielkeSavare2023}
Matthias Liero, Alexander Mielke, and Giuseppe Savar{\'e}.
\newblock Fine properties of geodesics and geodesic {$\lambda$}-convexity for
  the {H}ellinger--{K}antorovich distance.
\newblock {\em Arch. Ration. Mech. Anal.}, 247(6):112, 2023.

\bibitem[LPS03]{Lapeyre2003}
B.~Lapeyre, {\'E}.~Pardoux, and R.~Sentis.
\newblock {\em Introduction to {M}onte-{C}arlo Methods for Transport and
  Diffusion Equations}.
\newblock Oxford University Press, Oxford, 2003.

\bibitem[LZ23]{WIM}
Wuchen Li and Jiaxi Zhao.
\newblock {W}asserstein information matrix.
\newblock {\em Inf. Geom.}, 6(1):203--255, 2023.

\bibitem[MW15]{MaWu}
Zongming Ma and Yihong Wu.
\newblock Computational barriers in minimax submatrix detection.
\newblock {\em Ann. Statist.}, 43(3):1089--1116, 2015.

\bibitem[MZ25]{MielkeZhu2025HKGradientFlows}
Alexander Mielke and Jia-Jie Zhu.
\newblock {H}ellinger--{K}antorovich gradient flows: global exponential decay
  of entropy functionals.
\newblock arXiv preprint arXiv:2501.17049, 2025.

\bibitem[NM25]{WCR_achieve}
Hayato Nishimori and Takeru Matsuda.
\newblock On the attainment of the {W}asserstein--{C}ram{\'e}r--{R}ao lower
  bound.
\newblock {\em Inf. Geom.}, 2025.
\newblock Published online.

\bibitem[OR86]{OlverRosenau1986}
Peter~J. Olver and Philip Rosenau.
\newblock The construction of special solutions to partial differential
  equations.
\newblock {\em Phys. Lett. A}, 114(3):107--112, 1986.

\bibitem[RPK19]{RAISSI2019686}
M.~Raissi, P.~Perdikaris, and G.~E. Karniadakis.
\newblock Physics-informed neural networks: a deep learning framework for
  solving forward and inverse problems involving nonlinear partial differential
  equations.
\newblock {\em J. Comput. Phys.}, 378:686--707, 2019.

\bibitem[RW98]{RockafellarWets}
R.~Tyrrell Rockafellar and Roger J.-B. Wets.
\newblock {\em Variational Analysis}, volume 317 of {\em Grundlehren der
  mathematischen Wissenschaften}.
\newblock Springer, Berlin, 1998.

\bibitem[SFBW24]{SoloffFoygelBarberWillettI}
Jake~A. Soloff, Rina Foygel~Barber, and Rebecca Willett.
\newblock Bagging provides assumption-free stability.
\newblock {\em J. Mach. Learn. Res.}, 25(131):1--35, 2024.

\bibitem[SFBW26]{SoloffFoygelBarberWillettII}
Jake~A. Soloff, Rina Foygel~Barber, and Rebecca Willett.
\newblock Stability via resampling: statistical problems beyond the real line.
\newblock {\em Electron. J. Stat.}, 20(2):4179--4207, 2026.

\bibitem[Sha03]{Shao2003}
Jun Shao.
\newblock {\em Mathematical Statistics}.
\newblock Springer Texts in Statistics. Springer, New York, second edition,
  2003.

\bibitem[Soo20]{Sood2020}
Sunny Sood.
\newblock Implicit function theorems for {L}ipschitz functions.
\newblock MA4K9 dissertation, University of Warwick, 2020.

\bibitem[vdV98]{vanDerVaart}
A.~W. van~der Vaart.
\newblock {\em Asymptotic Statistics}, volume~3 of {\em Cambridge Series in
  Statistical and Probabilistic Mathematics}.
\newblock Cambridge University Press, Cambridge, 1998.

\bibitem[Vil09]{Villani}
C{\'e}dric Villani.
\newblock {\em Optimal Transport: Old and New}, volume 338 of {\em Grundlehren
  der mathematischen Wissenschaften}.
\newblock Springer, Berlin, 2009.

\bibitem[YB20]{BarterYu}
Bin Yu and Rebecca Barter.
\newblock The data science process: one culture.
\newblock {\em J. Amer. Statist. Assoc.}, 115(530):672--674, 2020.

\bibitem[YK20]{YuKumbier}
Bin Yu and Karl Kumbier.
\newblock Veridical data science.
\newblock {\em Proc. Natl. Acad. Sci. USA}, 117(8):3920--3929, 2020.

\bibitem[Yu13]{Yu2013}
Bin Yu.
\newblock Stability.
\newblock {\em Bernoulli}, 19(4):1484--1500, 2013.

\bibitem[YWR24]{YanWangRigollet}
Yuling Yan, Kaizheng Wang, and Philippe Rigollet.
\newblock Learning {G}aussian mixtures using the {W}asserstein--{F}isher--{R}ao
  gradient flow.
\newblock {\em Ann. Statist.}, 52(4):1774--1795, 2024.

\bibitem[ZDJW13]{ZhangDuchiJordanWainwright}
Yuchen Zhang, John~C. Duchi, Michael~I. Jordan, and Martin~J. Wainwright.
\newblock Information-theoretic lower bounds for distributed statistical
  estimation with communication constraints.
\newblock In {\em NeurIPS}, 2013.

\bibitem[ZWJ14]{ZhangWainwrightJordan}
Yuchen Zhang, Martin~J. Wainwright, and Michael~I. Jordan.
\newblock Lower bounds on the performance of polynomial-time algorithms for
  sparse linear regression.
\newblock In {\em COLT}, 2014.

\end{thebibliography}

\appendix

\section{Differentiability in the Hellinger--Kantorovich sense}\label{app:HK-static-cone}

In this appendix, we introduce a notion of differentiability with respect to the Hellinger--Kantorovich geometry. This notion is analogous to the notions of differentiability in quadratic mean (DQM) (see \cite{vanDerVaart}) and differentiability in the Wasserstein sense (DWS) (see \cite{WCR}) of statistical models. To introduce this notion of differentiability, we first discuss a static formulation of the Hellinger--Kantorovich distance that is reminiscent of the static version in terms of transport maps of the Wasserstein distance. We remark that the form of this static formulation differs from its Wasserstein counterpart. For the Wasserstein geometry, Brenier's theorem identifies conditions under which the optimal Kantorovich plan is induced by a deterministic transport map and in particular provides sufficient conditions for the equivalence between Kantorovich's and Monge's formulations of the Wasserstein distance (see \cite{Villani}). In contrast, in the Hellinger--Kantorovich geometry mass may also be created or destroyed, and therefore the corresponding Monge target must encode mechanisms for both spatial displacement and mass creation/destruction. We first recall the static and cone formulations of the Hellinger--Kantorovich distance, then introduce optimal transport--growth pairs through a Brenier-like theorem, and finally introduce the notion of differentiability of a statistical model in the Hellinger--Kantorovich sense, which we write DHKS for short. 

Throughout this appendix, assume \(0<\lambda<1\) and work under the normalization in equation \eqref{eqn:HK-differential}. 

\begin{definition}[Static formulation of \(\HK_\lambda\)]\label{def:HK-static}
    For \(\mu_0,\mu_1\in\mathcal M_+(\mathbb R^m)\), the static formulation of Hellinger--Kantorovich distance is
    \[
        d_{\HK_\lambda}^2(\mu_0,\mu_1)
        =
        \frac{4}{\lambda}
        \inf_{\eta\in\mathcal M(\mathbb R^m\times\mathbb R^m)}
        \left\{
            \mathrm{KL}((\pi_0)_\#\eta\|\mu_0)
            +
            \mathrm{KL}((\pi_1)_\#\eta\|\mu_1)
            +
            \int_{\mathbb R^m\times\mathbb R^m}
                c_\lambda(x,y)\diff \eta(x,y)
        \right\},
    \]
    where
    \[
        c_\lambda(x,y)
        :=
        \begin{cases}
            -\log\!\left[
                \cos^2\!\left(
                    \sqrt{\frac{\lambda}{4(1-\lambda)}}\,|x-y|
                \right)
            \right],
            & |x-y|<R_\lambda,\\[0.5em]
            +\infty,
            & |x-y|\ge R_\lambda.
        \end{cases}
        ,\qquad 
        R_\lambda
        :=
        \pi\sqrt{\frac{1-\lambda}{\lambda}}.
    \]
\end{definition}

This static formulation indeed agrees with the dynamic formulation \eqref{eqn:HK-dynamic} by \cite[Theorem~2.10]{LieroMielkeSavare2023}, after rescaling to the normalization in \eqref{eqn:HK-differential}. It is a Kantorovich-type variational problem but with soft marginal constraints. Indeed, the measure \(\eta\) records how mass at one location is matched with mass at another location, but its marginals are not required to agree exactly with \(\mu_0\) and \(\mu_1\). Instead, the two KL terms penalize the discrepancy between the marginals of \(\eta\) and the endpoint measures $\mu_0$ and $\mu_1$. The cost \(c_\lambda\) encodes the finite transport range of the \(\HK_\lambda\) geometry: only \(x, y\) pairs satisfying \(|x-y|<R_\lambda\) may be matched. Beyond this range, the optimal mechanism to change the mass at one point is by pure creation/destruction rather than by transport.

To pass from the static, Kantorovich-type formulation of the Hellinger--Kantorovich distance in Definition~\ref{def:HK-static} to a Monge-type formulation, it is useful to first lift the variational problem to an appropriate cone. Let
\(C:=(\mathbb R^m\times[0,\infty))/\sim\) be the cone over \(\mathbb R^m\), where all points of radius zero are identified with the cone vertex, and write \([x,r]\in C\). The cone distance associated with the normalization in \eqref{eqn:HK-differential} is
\[
    d_{\lambda,C}^2([x,r],[y,s])
    :=
    \frac{4}{\lambda}
    \left[
        r^2+s^2
        -
        2rs
        \cos_{\pi/2}
        \left(
            \sqrt{\frac{\lambda}{4(1-\lambda)}}\,|x-y|
        \right)
    \right],
\]
where \(\cos_{\pi/2}(u):=\cos(\min\{u,\pi/2\})\). For \(\beta\in\mathcal M_+(C)\) with \(\int_C r^2\diff\beta([x,r])<\infty\), define its homogeneous projection \(h\beta\in\mathcal M_+(\mathbb R^m)\) by
\(\int_{\mathbb R^m} f(x)\diff (h\beta)(x) = \int_C r^2 f(x)\diff \beta([x,r])\). Then the cone formulation of \(\HK_\lambda\) is
\[
    d_{\HK_\lambda}^2(\mu_0,\mu_1)
    =
    \inf_{\beta}
    \int_{C\times C}
        d_{\lambda,C}^2(z_0,z_1)\diff \beta(z_0,z_1),
\]
where the infimum is taken over all
\(\beta\in\mathcal M_+(C\times C)\) satisfying \(h((\pi_i)_\#\beta)=\mu_i \text{ for } i=0,1\); the equivalence between the cone formulation and the dynamic formulation \eqref{eqn:HK-dynamic} follows from \cite[Theorem~2.5, (2.11b)]{LieroMielkeSavare2023}.

Next, for a finite nonnegative measure $\mu$ on $\R^m$, consider a \textit{transport--growth pair} $(T,q)$ \cite[Definition~2.7]{LieroMielkeSavare2023} of measurable maps $T:\R^m\to\R^m$ and $q:\R^m\to [0,\infty)$ with $q\in L^2(\mu)$. One such pair acts on $\mu$ via
\[
    (T,q)_{\ast}\mu:=T_\#(q^2\mu),
\]
meaning that for all measurable sets $A\subseteq \R^m$ we have $\int_{\R^m}\mathbf{1}_A\diff ((T,q)_{\ast}\mu):=\int_{\R^m}\mathbf{1}_{T^{-1}(A)}q^2\diff \mu$.
In particular, $(T,q)_*\mu_0=\mu_1$ implies $\int_{\R^m}q^2\diff\mu_0=\mu_1(\R^m)$, so $\int_{\R^m}q^2\diff\mu_0=1$ when $\mu_1$ is a probability measure.
Note that $T$ specifies the spatial movement of the mass (i.e., analogous to the vector fields from the Wasserstein geometry) and that $q^2$ specifies the creation/destruction of mass (i.e., analogous to the scalar fields from the Hellinger geometry).

The cone formulation is the geometric mechanism behind transport--growth pairs. A cone point \([x,r]\) represents mass \(r^2\) located at \(x\): the base point records position, while the scalar records the amount of mass. Thus, the $\HK$ analogue of a Monge map is the cone-valued map
\[
    x\longmapsto \big([x,1],[T(x),q(x)]\big),
\]
where \((T,q)\) is a transport--growth pair. This graph induces the cone plan
\[
    \beta_{T,q}
    :=
    \left(
        x\mapsto \big([x,1],[T(x),q(x)]\big)
    \right)_\#\mu_0
    \in \mathcal M_+(C\times C).
\]
Its first homogeneous marginal is \(\mu_0\), while its second homogeneous marginal is \(h((\pi_1)_\#\beta_{T,q})=(T,q)_*\mu_0=T_\#(q^2\mu_0)\).
Thus, \(\beta_{T,q}\) is admissible in the cone formulation precisely when \((T,q)_*\mu_0=\mu_1\).
Plugging the graph plan \(\beta_{T,q}\) into the cone cost gives the transport--growth energy \cite[Problem~2.9]{LieroMielkeSavare2023}
\[
    \mathcal C_\lambda(T,q;\mu_0)
    :=
    \frac{4}{\lambda}
    \int_{\mathbb R^m}
    \left[
        1+q^2(x)
        -
        2q(x)
        \cos_{\pi/2}
        \left(
            \sqrt{\frac{\lambda}{4(1-\lambda)}}\,|T(x)-x|
        \right)
    \right]\diff \mu_0(x).
\]
Consequently, every transport--growth pair satisfying \((T,q)_*\mu_0=\mu_1\)
gives the upper bound
\[
    d_{\HK_\lambda}^2(\mu_0,\mu_1)
    \le
    \mathcal C_\lambda(T,q;\mu_0).
\]
The Brenier-type problem for \(\HK_\lambda\) is whether this upper bound is sharp, or, equivalently, whether an optimal cone plan can be represented by a transport--growth pair between the measures $\mu_0$ and $\mu_1$. Theorem~\ref{thm:HK-Brenier} below gives a sufficient condition for this Monge representation. It is directly derived from \cite[Corollary~3.5]{LieroMielkeSavare2023} by the rescaling argument given below.

\begin{theorem}[Brenier-like theorem for $\HK$ {\cite[Corollary~3.5]{LieroMielkeSavare2023}}]\label{thm:HK-Brenier}
Fix $0<\lambda<1$. Let \(\mu_0,\mu_1\in\mathcal M_+(\mathbb R^m)\), and set \(S_i:=\operatorname{supp}(\mu_i)\) for $i=0,1$. Let \(R_\lambda:=\pi\sqrt{(1-\lambda)/\lambda}\). Further denote 
\(S^{R_\lambda}:=\left\{x\in\mathbb R^m:\operatorname{dist}(x,S)<R_\lambda\right\}\). Assume that 
\begin{itemize}
    \item \(\mu_1\bigl(\mathbb R^m\setminus S_0^{R_\lambda}\bigr)=0\), i.e., the target has no mass outside the \(R_\lambda\)-neighborhood of the source support. 
    \item \(\mu_0\llcorner S_1^{R_\lambda}\ll \mathcal L^m\), where \(\mu_0\llcorner S_1^{R_\lambda}\) denotes the restriction of \(\mu_0\) to \(S_1^{R_\lambda}\); i.e., the part of the source that lies within the transport range of the target is absolutely continuous.
\end{itemize}
Then there exists an optimal \(\HK_\lambda\) transport--growth pair
\[
    (T,q):\mathbb R^m\to\mathbb R^m\times[0,\infty)
\]
satisfying \((T,q)_*\mu_0=\mu_1\). With the convention \(T(x)=x\) on
\(\{q=0\}\), this optimal pair is unique up to \(\mu_0\)-negligible sets.
\end{theorem}

Before discussing the rescaling argument used to prove this result, it is worth pausing to make a few remarks. First, note that under the assumptions of Theorem~\ref{thm:HK-Brenier}, the optimal transport--growth pair \((T,q)\) satisfies
\[
    d_{\HK_\lambda}^2(\mu_0,\mu_1)
    =
    \mathcal C_\lambda(T,q;\mu_0).
\]
Second, observe that the first assumption in the theorem means that \(\mu_1\)-almost every point lies within distance
\(R_\lambda\) of \(\operatorname{supp}(\mu_0)\). Hence, no target mass is forced to
be created independently of the source. The second assumption is the source-side
absolute continuity condition, analogous to the absolute continuity assumption in
Brenier's theorem. Finally, note that if both \(\mu_0\) and \(\mu_1\) have full support in \(\mathbb R^m\), then the
transport-range assumptions are automatically satisfied. In this case, the theorem applies as soon
as $\mu_0\ll\mathcal L^m.$ As \(\lambda\to0\), one has \(R_\lambda\to\infty\), 
which is consistent with the pure transport limit, where transport has no finite
cut-off. As \(\lambda\to1\), one has \(R_\lambda\to 0\), 
which is consistent with the pure source limit, where spatial transport
disappears and mass changes only by growth or decay.

\begin{proof}[Proof of Theorem~\ref{thm:HK-Brenier}]
Write \(d_{\HK_{1,4}}\) for the distance with \((\alpha,\beta)=(1,4)\) in \cite[Definition~2.1]{LieroMielkeSavare2023}, and let \(\kappa=\sqrt{\lambda/(4(1-\lambda))}\).
For a finite-energy solution \(\{(\mu_t,\xi_t,v_t)\}_{0\le t\le1}\) of \eqref{eqn:HK-differential}, set \(\widetilde\mu_t=(x\mapsto\kappa x)_\#\mu_t\), \(\widetilde v_t(\kappa x)=\kappa(1-\lambda)v_t(x)\), and \(\widetilde\xi_t(\kappa x)=\lambda\xi_t(x)/4\).
The transformed fields satisfy \(\partial_t\widetilde\mu_t+\Div(\widetilde v_t\widetilde\mu_t)=4\widetilde\xi_t\widetilde\mu_t\), and
\begin{equation*}
    \|\widetilde v_t(\kappa x)\|^2+4|\widetilde\xi_t(\kappa x)|^2
    =\frac{\lambda}{4}\left((1-\lambda)\|v_t(x)\|^2+\lambda|\xi_t(x)|^2\right).
\end{equation*}
The change of variables is invertible, so
\begin{equation*}
    \dHKlambda^2(\mu_0,\mu_1)
    =\frac{4}{\lambda}\,d_{\HK_{1,4}}^2(\widetilde\mu_0,\widetilde\mu_1).
\end{equation*}
For \(S_i=\operatorname{supp}(\mu_i)\), \(i=0,1\), we have \(\operatorname{dist}(\kappa x,\kappa S_i)=\kappa\operatorname{dist}(x,S_i)\).
Thus, the cutoff \(\pi/2\) for the transformed measures corresponds to a cutoff
\[
    \frac{\pi}{2\kappa}
    =\pi\sqrt{\frac{1-\lambda}{\lambda}}
    =R_\lambda
\]
in the original coordinates.
This change of variables preserves absolute continuity and transforms the hypotheses of Theorem~\ref{thm:HK-Brenier} into the corresponding conditions with cutoff \(\pi/2\).
Theorem~\ref{thm:HK-Brenier} now follows from \cite[Corollary~3.5]{LieroMielkeSavare2023}.
\end{proof}

The regularity assumptions we have placed on $\Pcal$ through the paper are qualitatively different from the notions of DQM or DWS. Indeed, our regularity assumptions (see Definition~\ref{def:HK_regular}) impose a sort of absolute continuity on the model $\Pcal$ rather than differentiability. We now introduce the notion of differentiability for the Hellinger--Kantorovich geometry. This notion can be formulated through a suitable first-order approximation assumption on optimal transport--growth pairs between elements in the statistical model.

\begin{definition}\label{def:DHKS}
    Fix $0<\lambda<1$.
    We say that $\Pcal=\{P_{\theta}:\theta\in\Theta\}\subseteq\Pcal_{2,\mathrm{ac}}(\R^d)$ with $\Theta\subseteq\R^p$ open is \emph{differentiable in the $\HK_{\lambda}$ sense (DHKS)} if, for each $\theta\in\Theta$, there exists a function $\phi_{\theta,\lambda}:\R^d\to\R^p$ such that
    \[
        (\phi_{\theta,\lambda,j},\grad\phi_{\theta,\lambda,j})
        \in\Tan_{P_\theta}^{\HK_\lambda}
        (\mathcal M_+(\R^d)),
        \qquad j=1,\ldots,p,
    \]
    and if for all sufficiently small $h\in\R^p$, there is a unique optimal $\HK_\lambda$ transport--growth pair $(T_t,q_t)$ from $P_\theta$ to $P_{\theta+th}$ such that
    \begin{equation}\label{eqn:def-DHKS}
        \begin{aligned}
            \int_{\R^d}\Big(
            &\lambda\big|q_t(x)-1-\tfrac{\lambda t}{2}(\phi_{\theta,\lambda}(x))^\top h\big|^2\\
            &+(1-\lambda)\big\|T_t(x)-x-t(1-\lambda)(D\phi_{\theta,\lambda}(x))^\top h\big\|^2
            \Big)\diff P_\theta(x)=o(t^2)
        \end{aligned}
    \end{equation}
    as $t\to0$.
    Uniqueness is understood up to $P_\theta$-negligible sets, with the convention $T_t(x)=x$ on $\{q_t=0\}$.
\end{definition}

The notion of DHKS is analogous to DQM and DWS in the sense that the optimal transport--growth pair determines the corresponding logarithmic map in the Hellinger--Kantorovich geometry (see \cite{CaiChengSchmitzerThorpe2022} and \cite[Corollary~3.5 and (3.34)]{LieroMielkeSavare2023}); hence, equation \eqref{eqn:def-DHKS} states that $\phi_{\theta,\lambda}$ represents the first-order directional variation of the map $P:\Theta\to\Pcal_{2,\textnormal{ac}}(\R^d)$ when $\Pcal_{2,\textnormal{ac}}(\R^d)$ is endowed with the $\dHKlambda$ geometry.
However, while DHKS is geometrically appealing, it is not useful in practice since the optimal transport--growth pairs are typically hard to calculate; this contrasts the pure source and transport cases, where DQM and DWS are useful precisely because likelihood ratios and transport maps are easier to calculate in a variety of examples.
We also emphasize that Proposition~\ref{prop:abs-cont-HK} makes Definition~\ref{def:HK_regular} preferable to Definition~\ref{def:DHKS}: absolute continuity in $\HK$ follows from absolute continuity in either of Hellinger or Wasserstein, whereas it is unclear whether DHKS follows from DQM or DWS.

\section{Further Remarks on Huber's $M$-Estimators}\label{app:Huber}

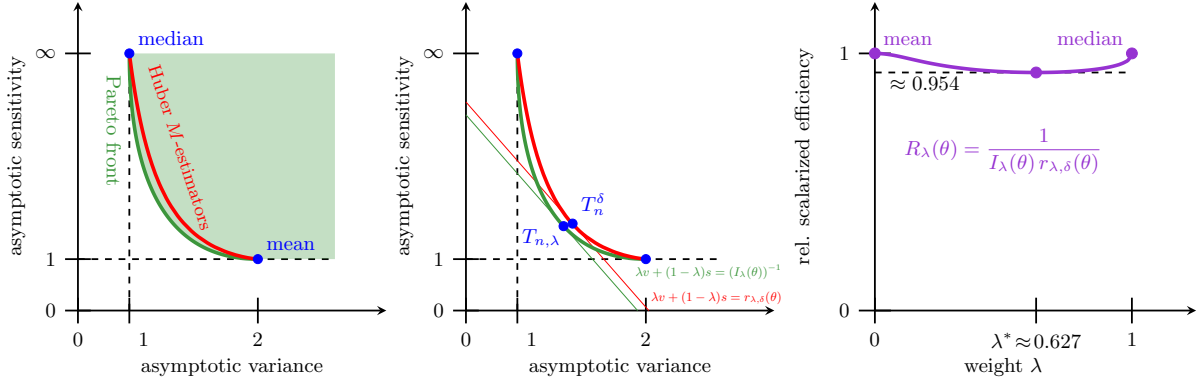
\begin{figure}
    \resizebox{0.85\width}{!}{
    \begin{tikzpicture}[scale=4.,font=\footnotesize,baseline={(0,0)}]
      \draw[thick, -stealth] (0,0) -- (1.2,0);
      \draw[thick, -stealth] (0,0) -- (0,1.2);
      \draw[thick, dashed] (0,0.2) -- (1,0.2);
      \draw[thick, dashed] (0.2,0) -- (0.2,1);
      \draw[thick] (-0.05,0)   node[left]{$0$}      -- (0.05,0);
      \draw[thick] (-0.05,0.2) node[left]{$1$}      -- (0.05,0.2);
      \draw[thick] (-0.05,1)   node[left]{$\infty$} -- (0.05,1);
      \draw[thick] (0,-0.05)   node[below]{$0$}      -- (0,0.05);
      \draw[thick] (0.2,-0.05) node[below right]{$1$} -- (0.2,0.05);
      \draw[thick] (0.7,-0.05) node[below]{$2$}      -- (0.7,0.05);
      \node[rotate=90] at (-0.24,0.6) {asymptotic sensitivity};
      \node at (0.6,-0.22) {asymptotic variance};
      \path[fill=mygreen, thick, opacity=0.3, domain=0:1, samples=100]
          plot ({0.2+0.5*\x^2}, {0.2+0.8*(\x-1)^2}) -- (1,0.2) -- (1,1) -- (0.2,1) -- cycle;
      \draw[color=mygreen, ultra thick, domain=0:1, samples=100]
          plot ({0.2+0.5*\x^2}, {0.2+0.8*(\x-1)^2});
      \node[color=mygreen, rotate=-90] at (0.14,0.7) {Pareto front};
      \node[color=red, rotate=-72, font=\scriptsize] at (0.4,0.635) {Huber $M$-estimators};
      \draw[red, ultra thick] plot[smooth] coordinates {
        (0.2017,0.9868) (0.2022,0.9826) (0.2029,0.9769) (0.2039,0.9695) (0.2052,0.9599)
        (0.2069,0.9473) (0.2092,0.9310) (0.2122,0.9101) (0.2162,0.8836) (0.2215,0.8505)
        (0.2286,0.8097) (0.2380,0.7606) (0.2504,0.7031) (0.2668,0.6379) (0.2884,0.5669)
        (0.3166,0.4932) (0.3530,0.4208) (0.3990,0.3542) (0.4547,0.2978) (0.5182,0.2546)
        (0.5831,0.2257) (0.6392,0.2096) (0.6766,0.2026) (0.6941,0.2005) (0.6998,0.2000)
      };
      \filldraw[color=blue] (0.2,1)   circle (0.5pt) node[above right] {median};
      \filldraw[color=blue] (0.7,0.2) circle (0.5pt) node[above right] {mean};
    \end{tikzpicture}%
    \begin{tikzpicture}[scale=4.,font=\footnotesize,baseline={(0,0)}]
      \draw[thick, -stealth] (0,0) -- (1.2,0);
      \draw[thick, -stealth] (0,0) -- (0,1.2);
      \draw[thick, dashed] (0,0.2) -- (1,0.2);
      \draw[thick, dashed] (0.2,0) -- (0.2,1);
      \draw[thick] (-0.05,0)   node[left]{$0$} -- (0.05,0);
      \draw[thick] (-0.05,0.2) node[left]{$1$} -- (0.05,0.2);
      \draw[thick] (-0.05,1)   node[left]{$\infty$} -- (0.05,1);
      \draw[thick] (0,-0.05)   node[below]{$0$}      -- (0,0.05);
      \draw[thick] (0.2,-0.05) node[below right]{$1$} -- (0.2,0.05);
      \draw[thick] (0.7,-0.05) node[below]{$2$}      -- (0.7,0.05);
      \node[rotate=90] at (-0.24,0.6) {asymptotic sensitivity};
      \node at (0.6,-0.22) {asymptotic variance};
    
      \draw[mygreen, thin] (0.0,0.7612) -- (0.6677,0.0);
      \draw[red,     thin] (0.0,0.8111) -- (0.7115,0.0);
    
      \draw[color=mygreen, ultra thick, domain=0:1, samples=100]
          plot ({0.2+0.5*\x^2}, {0.2+0.8*(\x-1)^2});
      \draw[red, ultra thick] plot[smooth] coordinates {
        (0.2017,0.9868) (0.2022,0.9826) (0.2029,0.9769) (0.2039,0.9695) (0.2052,0.9599)
        (0.2069,0.9473) (0.2092,0.9310) (0.2122,0.9101) (0.2162,0.8836) (0.2215,0.8505)
        (0.2286,0.8097) (0.2380,0.7606) (0.2504,0.7031) (0.2668,0.6379) (0.2884,0.5669)
        (0.3166,0.4932) (0.3530,0.4208) (0.3990,0.3542) (0.4547,0.2978) (0.5182,0.2546)
        (0.5831,0.2257) (0.6392,0.2096) (0.6766,0.2026) (0.6941,0.2005) (0.6998,0.2000)
      };
      \filldraw[color=blue] (0.3798,0.3282) circle (0.5pt) node[below left=-0.25em] {$T_{n,\lambda}$};
      \filldraw[color=blue] (0.4149,0.3381) circle (0.5pt) node[above right] {$T_n^{\delta}$};
      \filldraw[color=blue] (0.2,1)   circle (0.5pt);
      \filldraw[color=blue] (0.7,0.2) circle (0.5pt);
      \node[color=mygreen, scale=0.62, anchor=east] at (1.25,0.15) {$\lambda v+(1-\lambda)s=(I_{\lambda}(\theta))^{-1}$};
      \node[color=red,     scale=0.62, anchor=east] at (1.25,0.05) {$\lambda v+(1-\lambda)s=r_{\lambda,\delta}(\theta)$};
    \end{tikzpicture}%
    \begin{tikzpicture}[scale=4.,font=\footnotesize,baseline={(-0.0500,0)}]
      \draw[thick, -stealth] (0.2000,0) -- (1.45,0);
      \draw[thick, -stealth] (0.2000,0) -- (0.2000,1.2);
      \draw[thick] (0.2000,-0.05) node[below]{$0$} -- (0.2000,0.05);
      \draw[thick] (0.8274,-0.05) node[below]{$\lambda^{\ast}\!\approx\!0.627$} -- (0.8274,0.05);
      \draw[thick] (1.2000,-0.05) node[below]{$1$} -- (1.2000,0.05);
      \draw[thick] (0.1375,0) node[left]{$0$} -- (0.2625,0);
      \draw[thick] (0.1375,1.0) node[left]{$1$} -- (0.2625,1.0);
      \node[rotate=90] at (-0.0750,0.6) {rel. scalarized efficiency};
      \node at (0.7000,-0.22) {weight $\lambda$};
      \draw[thick, dashed] (0.2000,0.9258) -- (1.2000,0.9258);
      \node[anchor=north west, inner sep=1.5pt] at (0.2437,0.9235) {$\approx0.954$};
      \draw[mypurple, ultra thick] plot[smooth] coordinates {
        (0.2002,1.0000) (0.2243,0.9991) (0.2482,0.9966) (0.2724,0.9927) (0.2964,0.9879)
        (0.3204,0.9825) (0.3444,0.9770) (0.3685,0.9716) (0.3925,0.9664) (0.4165,0.9615)
        (0.4406,0.9570) (0.4646,0.9528) (0.4886,0.9490) (0.5128,0.9456) (0.5368,0.9425)
        (0.5607,0.9397) (0.5849,0.9371) (0.6089,0.9349) (0.6329,0.9330) (0.6569,0.9313)
        (0.6810,0.9298) (0.7050,0.9286) (0.7290,0.9276) (0.7531,0.9268) (0.7771,0.9263)
        (0.8011,0.9259) (0.8253,0.9258) (0.8493,0.9259) (0.8732,0.9262) (0.8974,0.9267)
        (0.9214,0.9275) (0.9454,0.9286) (0.9694,0.9299) (0.9935,0.9316) (1.0175,0.9336)
        (1.0415,0.9360) (1.0656,0.9390) (1.0896,0.9426) (1.1136,0.9471) (1.1378,0.9529)
        (1.1618,0.9609) (1.1857,0.9742) (1.2000,1.0000)
      };
      \filldraw[color=mypurple] (0.2000,1.0) circle (0.6pt) node[above right] {mean};
      \filldraw[color=mypurple] (1.2000,1.0) circle (0.6pt) node[above left] {median};
      \filldraw[color=mypurple] (0.8274,0.9258) circle (0.6pt);
      \node[color=mypurple] at (0.7000,0.62)
        {$R_{\lambda}(\theta)=\dfrac{1}{I_{\lambda}(\theta)\,r_{\lambda,\delta}(\theta)}$};
    \end{tikzpicture}
    }
    \caption{Relative scalarized efficiency of Huber's $M$-estimators for the Laplace location family.
    We show the variance-sensitivity diagram and the variance-sensitivity pairs achieved by Huber's $M$-estimators (left), a visualization of the relative scalarized efficiency (center), and a plot of the worst-case relative scalarized efficiency of Huber's $M$-estimators (right).}
    \label{fig:Huber-rel-eff}
\end{figure}

In this appendix, we give some further results about variance-sensitivity tradeoffs for Huber's $M$-estimators, which are a central tool in robust statistics.

First, we precisely state and prove the following result, previously discussed in Subsection~\ref{subsec:HK-score}; under its hypotheses, no single parametric model admits Huber estimating functions as its $\HK_\lambda$ score functions for all $0<\lambda<1$.

\begin{proposition}[Nonexistence of a fixed Huber-$\HK$ model]
\label{prop:huber-nonexistence}
    Let
    \(\Pcal=\{P_\theta:\theta\in\Theta\}\) be a statistical model on
    \(\R\), where each \(P_\theta\) has a strictly positive, continuous
    density \(p_\theta\). Suppose that \(\theta\mapsto\Ebb_\theta[f(X)]\)
    is differentiable on \(\Theta\) for every \(f\in C_c^\infty(\R)\).
    Write \(\psi_\delta(r):=s_\delta(0,r)\), where \(s_\delta\) is defined in Example~\ref{ex:huber}.

    If, for every \(\lambda\) in some set
    \(\Lambda\subseteq[0,1]\), there exist
    \(c_{\theta,\lambda}>0\) and \(\delta_\lambda>0\) such that
    \[
        \phi_{\theta,\lambda}(x)
        =
        c_{\theta,\lambda}
        \psi_{\delta_\lambda}(x-\theta)
    \]
    satisfies the weak transport-source equation
    \[
        \frac{\diff}{\diff\theta}\Ebb_\theta[f(X)]
        =(1-\lambda)\Ebb_\theta[f'(X)\partial_x\phi_{\theta,\lambda}(X)]
        +\lambda\Ebb_\theta[f(X)\phi_{\theta,\lambda}(X)]
    \]
    for every \(\theta\in\Theta\) and \(f\in C_c^\infty(\R)\),
    then \(\Lambda\) contains at most one
    element. In particular, no fixed model admits the Huber estimating
    function as its \(\HK_\lambda\) score for all
    \(0\leq\lambda\leq1\).
\end{proposition}

\begin{proof}
    Fix \(\theta\in\Theta\), and suppose that the transport-source equation holds
    for two distinct values \(\lambda_1,\lambda_2\in\Lambda\). Write
    \[
        c_j:=c_{\theta,\lambda_j},
        \qquad
        \delta_j:=\delta_{\lambda_j},
        \qquad
        j=1,2.
    \]
    Since
    \[
        \psi_{\delta_j}'(x-\theta)
        =
        \mathbf 1_{\{|x-\theta|<\delta_j\}}
    \]
    for Lebesgue-a.e. \(x\), the weak transport-source equations can be written,
    in the sense of distributions, as
    \[
        \partial_\theta p_\theta
        +
        (1-\lambda_j)
        \partial_x
        \left(
            c_j
            \mathbf 1_{\{|x-\theta|<\delta_j\}}
            p_\theta
        \right)
        =
        \lambda_jc_j
        \psi_{\delta_j}(x-\theta)p_\theta,
        \qquad j=1,2.
    \]
    Subtracting the two equations eliminates the unknown derivative
    \(\partial_\theta p_\theta\) and gives
    \begin{align}
        &\partial_x
        \left(
            \left[
                (1-\lambda_1)c_1
                \mathbf 1_{\{|x-\theta|<\delta_1\}}
                -
                (1-\lambda_2)c_2
                \mathbf 1_{\{|x-\theta|<\delta_2\}}
            \right]
            p_\theta(x)
        \right) \notag\\
        &\qquad=
        \left[
            \lambda_1c_1\psi_{\delta_1}(x-\theta)
            -
            \lambda_2c_2\psi_{\delta_2}(x-\theta)
        \right]
        p_\theta(x).
        \label{eq:huber-difference}
    \end{align}
    The right-hand side contains no Dirac mass. Consequently, the function inside the derivative on the left-hand side, which can be interpreted as a flux, cannot have jumps.

    We first show that \(\delta_1=\delta_2\). Suppose, without loss of
    generality, that \(\delta_1<\delta_2\). If \(\lambda_2<1\), the flux in
    \eqref{eq:huber-difference} has a jump at
    \(x=\theta+\delta_2\) of size
    \[
        (1-\lambda_2)c_2
        p_\theta(\theta+\delta_2)>0,
    \]
    producing a nonzero Dirac mass on the left-hand side. This contradicts
    the form of the right-hand side. If \(\lambda_2=1\), then
    \(\lambda_1<1\), and the same argument at
    \(x=\theta+\delta_1\) gives a jump of size
    \[
        -(1-\lambda_1)c_1
        p_\theta(\theta+\delta_1)\neq0.
    \]
    Thus \(\delta_1=\delta_2=:\delta\).

    With a common threshold, the flux in
    \eqref{eq:huber-difference} becomes
    \[
        \left[
            (1-\lambda_1)c_1
            -
            (1-\lambda_2)c_2
        \right]
        \mathbf 1_{\{|x-\theta|<\delta\}}
        p_\theta(x).
    \]
    The absence of a jump at \(x=\theta+\delta\) therefore implies
    \begin{equation}
        (1-\lambda_1)c_1
        =
        (1-\lambda_2)c_2.
        \label{eq:huber-transport-match}
    \end{equation}

    On the open tail \(x>\theta+\delta\), the flux vanishes identically, while
    \(\psi_\delta(x-\theta)=\delta\). Equation
    \eqref{eq:huber-difference} consequently gives
    \[
        \left(
            \lambda_1c_1-\lambda_2c_2
        \right)
        \delta p_\theta(x)
        =
        0.
    \]
    Since \(\delta>0\) and \(p_\theta(x)>0\),
    \begin{equation}
        \lambda_1c_1
        =
        \lambda_2c_2.
        \label{eq:huber-reaction-match}
    \end{equation}
    Combining \eqref{eq:huber-transport-match} and
    \eqref{eq:huber-reaction-match} yields
    \[
        c_1
        =
        (1-\lambda_1)c_1+\lambda_1c_1
        =
        (1-\lambda_2)c_2+\lambda_2c_2
        =
        c_2.
    \]
    Since \(c_1=c_2>0\), equation
    \eqref{eq:huber-reaction-match} implies
    \(\lambda_1=\lambda_2\), contradicting the choice of two distinct
    parameters.
\end{proof}

The proposition already allows the clipping threshold for the Huber loss to depend arbitrarily on
the parameter $\lambda$. Thus choosing
\(\delta=\delta(\lambda)\) cannot remove the main obstruction: the jump condition
first forces any two admissible thresholds to be equal, and the transport and
source matching conditions then force the corresponding values of \(\lambda\) to be
equal. Reparametrizing Huber estimators by \(\lambda\) therefore does not
make them the \(\HK_\lambda\) \(Z\)-estimators of a single fixed model.

\begin{remark}
    The model in Proposition~\ref{prop:huber-nonexistence} is not assumed to be
    a location family; only the estimating function has the Huber location form
    \(\psi_{\delta_\lambda}(x-\theta)\). In Proposition \ref{prop:huber-nonexistence} we have ruled out models consisting of distributions with full support and continuous densities. In order to rule out models with parameter-dependent support, we would require a separate analysis.
\end{remark}

Next we specialize to the case of Huber's $M$-estimators within the Laplace location family model of Example~\ref{ex:Laplace-reg}.
In this case, we can use \cite[Theorem~5.21]{vanDerVaart} and Example~\ref{ex:huber} to compute that its asymptotic variance and sensitivity are
\begin{equation*}
    V^{\delta}(\theta):=\frac{e^{\delta}-\delta-1}{\cosh\delta-1}\qquad \textnormal{ and }\qquad S^{\delta}(\theta):=\frac{1}{1-e^{-\delta}},
\end{equation*}
respectively, when $0\le \delta\le \infty$ is the cutoff, and $\theta\in\R$ is the parameter (although the resulting expression does not depend on $\theta$).
Also, recall that the Pareto optimal variance and sensitivity are given by
\begin{equation*}
        V_{\lambda}(\theta)=
        1+
        \sqrt{\frac{1-\lambda}{1+3\lambda}}\qquad \textnormal{ and }\qquad S_{\lambda}(\theta) =
        \frac{
            \left(
                \sqrt{\frac{1+3\lambda}{1-\lambda}}+1
            \right)^2
        }{
            4\sqrt{\frac{1+3\lambda}{1-\lambda}}
        },
    \end{equation*}
    for all $0\le \lambda<1$ and $\theta\in\R$. 
Qualitatively speaking, note that both of these are convex curves interpolating between $(2,1)$ and $(1,\infty)$; see Figure~\ref{fig:laplace_diagram} for an illustration.

Our last goal is to describe the relative efficiency of Huber's $M$-estimators with respect to the Pareto front.
One difficulty, however, is that there isn't a direct correspondence between the $0\le \lambda\le 1$ and $0\le \delta\le \infty$ parameters. To address this, we consider, for each $0\le \lambda\le 1$, the best possible scalarized objective achieved by any Huber $M$-estimator, i.e.
\begin{equation*}
    \min_{0\le \delta\le \infty}(\lambda V^{\delta}(\theta)+(1-\lambda)S^{\delta}(\theta)).
\end{equation*}
The solution map above provides a principled method for choosing $\delta$ as a function of $\lambda$; while there is no closed-form expression for this solution map, it is the inverse of the following explicit formula, visualized in Figure~\ref{fig:Huber-cutoff-from-wt}:
\begin{equation}\label{eqn:Huber-cutoff-from-wt}
    \lambda(\delta) = \frac{e^{-\delta}-1}{3(1-e^{-\delta})-2\delta(1+e^{-\delta})}.
\end{equation}
From this perspective, a natural notion of relative efficiency, for each $0\le \lambda\le 1$, is the following:
\begin{equation*}
    R_{\lambda}(\theta) := \frac{\lambda V_{\lambda}(\theta)+(1-\lambda)S_{\lambda}(\theta)}{\min_{0\le \delta\le \infty}(\lambda V^{\delta}(\theta)+(1-\lambda)S^{\delta}(\theta))} = \frac{1}{I_{\lambda}(\theta)\min_{0\le \delta\le \infty}(\lambda V^{\delta}(\theta)+(1-\lambda)S^{\delta}(\theta))}.
\end{equation*}
To summarize this across all $0\le \lambda\le 1$, we consider its worst-case behavior; we refer to this as the \textit{worst-case relative scalarized efficiency}.
Then we can directly calculate the following, which states that Huber's $M$-estimators are nearly Pareto optimal: 

\begin{proposition}\label{prop:Laplace-Huber-rel-eff}
    The worst-case relative scalarized efficiency of Huber's $M$-estimators in the Laplace location family is  $\min_{0\le \lambda\le 1}R_{\lambda}(\theta) \approx 0.954$, and the minimum occurs at $\lambda^{\ast} \approx 0.627$.
\end{proposition}

Figure~\ref{fig:Huber-rel-eff} visualizes this worst-case relative scalarized efficiency calculation.

\begin{figure}
    \centering
    \begin{tikzpicture}[scale=3.,font=\footnotesize,baseline={(-0.0500,0)}]
      \draw[thick, -stealth] (0.2000,0) -- (1.45,0);
      \draw[thick, -stealth] (0.2000,0) -- (0.2000,1.2);
      \draw[thick] (0.2000,-0.05) node[below]{$0$} -- (0.2000,0.05);
      \draw[thick] (1.2000,-0.05) node[below]{$1$} -- (1.2000,0.05);
      \draw[thick] (0.1375,0) node[left]{$0$}      -- (0.2625,0);
      \draw[thick] (0.1375,1.0) node[left]{$\infty$} -- (0.2625,1.0);
      \node[rotate=90] at (-0.0750,0.6) {Huber cutoff $\delta$};
      \node at (0.7000,-0.22) {weight $\lambda$};
      \draw[orange, ultra thick] plot[smooth] coordinates {
        (0.2000,1.0000) (0.2040,0.9845) (0.2049,0.9812) (0.2059,0.9772) (0.2073,0.9724) (0.2089,0.9666) (0.2109,0.9596) (0.2134,0.9512) (0.2164,0.9412) (0.2203,0.9292) (0.2249,0.9151) (0.2309,0.8985) (0.2384,0.8790) (0.2479,0.8565) (0.2602,0.8305) (0.2764,0.8009) (0.2976,0.7675) (0.3260,0.7305) (0.3636,0.6900) (0.4129,0.6463) (0.4756,0.6000) (0.4857,0.5932) (0.4986,0.5847) (0.5125,0.5758) (0.5272,0.5666) (0.5433,0.5569) (0.5603,0.5469) (0.5785,0.5363) (0.5980,0.5252) (0.6190,0.5136) (0.6414,0.5014) (0.6653,0.4886) (0.6909,0.4751) (0.7181,0.4609) (0.7470,0.4459) (0.7775,0.4300) (0.8096,0.4132) (0.8432,0.3954) (0.8782,0.3764) (0.9141,0.3562) (0.9507,0.3347) (0.9875,0.3117) (1.0239,0.2870) (1.0590,0.2605) (1.0921,0.2320) (1.1224,0.2012) (1.1487,0.1677) (1.1704,0.1314) (1.1866,0.0918) (1.1966,0.0483) (1.2000,0.0005)
      };
      \node[color=orange] at (0.8500,0.60) {$\delta=\delta(\lambda)$};
    \end{tikzpicture}
    \caption{The optimal cutoff parameter $\delta$ in Huber's $M$-estimator, in terms of the weight $\lambda$.
    There is no closed form expression for $\delta$ in terms of $\lambda$, but it can be numerically computed as the inverse of the expression for $\lambda$ in terms of $\delta$ given in~\eqref{eqn:Huber-cutoff-from-wt}.}
    \label{fig:Huber-cutoff-from-wt}
\end{figure}
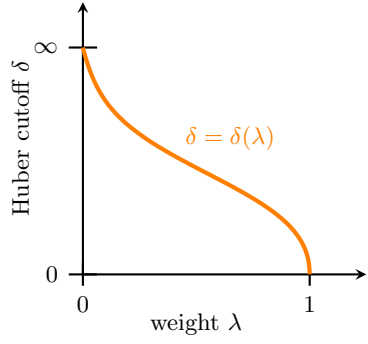

\section{Verifying $\HK$ regularity}\label{app:verify}

\subsection{Laplace location family}
\label{app:laplace-verification}

Recall the Laplace location family \(p_\theta(x)=\frac12e^{-|x-\theta|}\), \(\theta\in\R\).
For \(0<\lambda<1\), we verify the function given in~\eqref{eq:phi_laplace}:
\[
    \phi_{\theta,\lambda}(x)
    =\frac1\lambda\operatorname{sgn}(x-\theta)
    \left(1-e^{-\alpha_\lambda|x-\theta|}\right),
    \qquad
    \alpha_\lambda
    :=\frac12\left(\sqrt{\frac{1+3\lambda}{1-\lambda}}-1\right).
\]

\paragraph{Verification of the transport-source equation.}
To verify the \(\HK_\lambda\) equation for \(0\le\lambda<1\), it suffices to show that, for every \(f\in C_c^\infty(\R)\),
\[
    \partial_\theta\Ebb_\theta[f(X)]
    =(1-\lambda)\Ebb_\theta\!
    \left[f'(X)\phi'_{\theta,\lambda}(X)\right]
    +\lambda\Ebb_\theta\!
    \left[f(X)\phi_{\theta,\lambda}(X)\right].
\]
Since
\[
    \Ebb_\theta[f(X)]
    =\frac12\int_{-\infty}^{\theta}f(x)e^{x-\theta}\diff x
    +\frac12\int_{\theta}^{\infty}f(x)e^{-(x-\theta)}\diff x,
\]
direct differentiation gives
\begin{align}
    \partial_\theta\Ebb_\theta[f(X)]
    &=-\frac12\int_{-\infty}^{\theta}f(x)e^{x-\theta}\diff x
      +\frac12\int_{\theta}^{\infty}f(x)e^{-(x-\theta)}\diff x \notag\\
    &=\Ebb_\theta\!
    \left[f(X)\operatorname{sgn}(X-\theta)\right].
    \label{eq:laplace-weak-left}
\end{align}
For \(0<\lambda<1\) and \(x\neq\theta\), direct calculation gives
\[
    \phi'_{\theta,\lambda}(x)
    =\frac{\alpha_\lambda}{\lambda}e^{-\alpha_\lambda|x-\theta|},
    \qquad
    \phi''_{\theta,\lambda}(x)
    =-\frac{\alpha_\lambda^2}{\lambda}
    \operatorname{sgn}(x-\theta)e^{-\alpha_\lambda|x-\theta|},
\]
while \(p_\theta'(x)/p_\theta(x)=-\operatorname{sgn}(x-\theta)\) and \(\alpha_\lambda(\alpha_\lambda+1)=\lambda/(1-\lambda)\). Set \(J_{\theta,\lambda}(x):=\phi'_{\theta,\lambda}(x)p_\theta(x)\). It follows that, for \(x\neq\theta\),
\begin{equation}
    -(1-\lambda)J'_{\theta,\lambda}(x)
    +\lambda\phi_{\theta,\lambda}(x)p_\theta(x)
    =\operatorname{sgn}(x-\theta)p_\theta(x).
    \label{eq:laplace-flux-identity}
\end{equation}
It remains to check that no singular boundary term occurs at \(x=\theta\). Splitting at \(x=\theta\) and integrating by parts gives
\[
    \int_{\R}f'(x)J_{\theta,\lambda}(x)\diff x
    =f(\theta)\!
    \left[J_{\theta,\lambda}(\theta^-)-J_{\theta,\lambda}(\theta^+)\right]
    -\int_{\R}f(x)J'_{\theta,\lambda}(x)\diff x.
\]
The second terms at infinity vanish because \(f\) has compact support. Moreover, \(p_\theta(\theta^-)=p_\theta(\theta^+)=\frac12\) and \(\phi'_{\theta,\lambda}(\theta^-)=\phi'_{\theta,\lambda}(\theta^+)=\alpha_\lambda/\lambda\), so \(J_{\theta,\lambda}(\theta^-)=J_{\theta,\lambda}(\theta^+)\). Hence, using \eqref{eq:laplace-flux-identity},
\begin{align*}
    &(1-\lambda)\Ebb_\theta\!
    \left[f'(X)\phi'_{\theta,\lambda}(X)\right]
    +\lambda\Ebb_\theta\!
    \left[f(X)\phi_{\theta,\lambda}(X)\right]\\
    &=-(1-\lambda)\int_{\R}f(x)J'_{\theta,\lambda}(x)\diff x
      +\lambda\int_{\R}f(x)\phi_{\theta,\lambda}(x)p_\theta(x)\diff x\\
    &=\Ebb_\theta\!
    \left[f(X)\operatorname{sgn}(X-\theta)\right].
\end{align*}
Comparison with \eqref{eq:laplace-weak-left} proves the weak \(\HK_\lambda\) equation.

At \(\lambda=0\), the potential \(\phi_{\theta,0}(x)=x-\theta\) has derivative \(\phi'_{\theta,0}(x)=1\). Splitting at \(x=\theta\) and integrating by parts gives
\begin{align*}
    \Ebb_\theta\!\left[f'(X)\phi'_{\theta,0}(X)\right]
    &=\frac12\int_{-\infty}^{\theta}f'(x)e^{x-\theta}\diff x
      +\frac12\int_{\theta}^{\infty}f'(x)e^{-(x-\theta)}\diff x\\
    &=-\frac12\int_{-\infty}^{\theta}f(x)e^{x-\theta}\diff x
      +\frac12\int_{\theta}^{\infty}f(x)e^{-(x-\theta)}\diff x\\
    &=\partial_\theta\Ebb_\theta[f(X)].
\end{align*}
The boundary terms at \(\theta\) cancel, the terms at infinity vanish because \(f\) has compact support, and the last equality follows from~\eqref{eq:laplace-weak-left}.
Thus \(\phi_{\theta,0}\) solves the \(\HK_0\) continuity equation.

At \(\lambda=1\), the function \(\phi_{\theta,1}(x)=\operatorname{sgn}(x-\theta)\) satisfies
\[
    \partial_\theta\Ebb_\theta[f(X)]
    =\Ebb_\theta\!\left[f(X)\phi_{\theta,1}(X)\right]
\]
by~\eqref{eq:laplace-weak-left}, so it solves the \(\HK_1\) source equation.

\paragraph{Verification of regularity.}
Fix \(0\le\lambda<1\).
The functions \(\phi_{\theta,\lambda}\) are centered by symmetry and are continuously differentiable in \(x\).
For \(0<\lambda<1\), the bounds
\[
    |\phi_{\theta,\lambda}(x)|\le\frac1\lambda,
    \qquad
    |\phi'_{\theta,\lambda}(x)|\le\frac{\alpha_\lambda}{\lambda}
\]
give \(\phi_{\theta,\lambda}\in\mathcal H_\lambda^1(P_\theta)\) with energy bounded uniformly in \(\theta\).
At \(\lambda=0\), \(\|\phi_{\theta,0}\|_{\mathcal H_0^1(P_\theta)}^2=1\).
Thus condition (i) of Definition~\ref{def:HK_regular} holds.

For any \(\theta,h\in\R\), set \(\phi_t=h\phi_{\theta+th,\lambda}\), \(t\in[0,1]\).
The map \((t,x)\mapsto(\phi_t(x),\nabla\phi_t(x))\) is continuous and hence jointly measurable.
Applying the chain rule to the verified weak \(\HK_\lambda\) equation shows that \(\phi_t\) solves~\eqref{HK-cont-eq}, and the energy bounds give
\[
    \int_0^1
    \|\phi_t\|_{\mathcal H_\lambda^1(P_{\theta+th})}^2
    \diff t<\infty.
\]
Proposition~\ref{prop:Sobolev} and Proposition~\ref{prop:ans-cont-unique-solution} give uniqueness up to the \(\mathcal H_\lambda^1\) norm, verifying (ii).
The form \(\phi_t=h\phi_{\theta+th,\lambda}\) then verifies (iii).
Hence the model is \(\HK_0\) regular and \(\HK\) regular.

At \(\lambda=1\), the score \(\operatorname{sgn}(x-\theta)\) belongs to \(L^2(P_\theta)\), but its jump at \(x=\theta\) prevents it from having a locally Lipschitz representative.
Thus it does not belong to \(\mathcal H_1^1(P_\theta)\), and the Laplace location family is not \(\HK_1\) regular in the sense of Definition~\ref{def:HK_regular}.
However, the likelihood is proportional to \(\exp(-\sum_{i=1}^n|X_i-\theta|)\), so any sample median is an MLE for \(\theta\).

\paragraph{Remaining assumptions of Theorems~\ref{thm:HK-CR-bound} and~\ref{thm:HK-Z-attainment}.}
Fix \(\theta_0\in\R\) and \(0\le\lambda<1\).
Translation invariance shows that \(\Ebb_\theta[\phi_{\theta,\lambda}(X)^2]\) and \(\Ebb_\theta[|\phi'_{\theta,\lambda}(X)|^2]\) are finite and independent of \(\theta\), hence continuous.

For \(0<\lambda<1\), \(s(t,x)=\phi_{t,\lambda}(x)\) is jointly \(C^1\), and
\[
    \partial_t\phi_{t,\lambda}(x)
    =-\frac{\alpha_\lambda}{\lambda}e^{-\alpha_\lambda|x-t|},
    \qquad
    |\phi'_{t,\lambda}(x)|^2
    =\frac{\alpha_\lambda^2}{\lambda^2}e^{-2\alpha_\lambda|x-t|}.
\]
Let \(J\) be a compact interval containing \(\theta_0\) in its interior.
To verify the ULLNs for \(\{\partial_t\phi_{t,\lambda}:t\in J\}\) and \(\{|\phi'_{t,\lambda}|^2:t\in J\}\), let \(g_t\) denote either function.
The explicit formulas give
\[
    |g_t(x)-g_u(x)|\le C|t-u|,
    \qquad
    C:=\max\left\{
    \frac{\alpha_\lambda^2}{\lambda},
    \frac{2\alpha_\lambda^3}{\lambda^2}
    \right\},
\]
for all \(t,u\in J\) and \(x\in\R\).
For any \(\delta>0\), choose finitely many points \(t_1,\ldots,t_N\in J\) such that every \(t\in J\) is within \(\delta\) of one of them.
Approximating both the empirical average and its expectation at the nearest point gives
\begin{equation}\label{eq:laplace-derivative-ulln}
    \begin{aligned}
        &\sup_{t\in J}
        \left|
        \frac1n\sum_{i=1}^n g_t(X_i)
        -\Ebb_{\theta_0}[g_t(X)]
        \right|\\
        &\quad\le
        \max_{1\le j\le N}
        \left|
        \frac1n\sum_{i=1}^n g_{t_j}(X_i)
        -\Ebb_{\theta_0}[g_{t_j}(X)]
        \right|
        +2C\delta.
    \end{aligned}
\end{equation}
For each fixed grid, the maximum in~\eqref{eq:laplace-derivative-ulln} tends to zero in probability by the law of large numbers.
First letting \(n\to\infty\) and then \(\delta\downarrow0\) proves condition (iv) of Theorem~\ref{thm:z-est}.

Since \(|\phi_{t,\lambda}(x)|\le1/\lambda\) and \(|\phi'_{t,\lambda}(x)|\le\alpha_\lambda/\lambda\) for all \(t,x\in\R\), we have
\[
    \sup_{r,t\in J}\Ebb_r
    [\phi_{t,\lambda}(X)^2+|\phi'_{t,\lambda}(X)|^2]
    \le\frac{1+\alpha_\lambda^2}{\lambda^2}.
\]
Writing \(X=r+Z\), where \(Z\sim P_0\), dominated convergence shows that \((r,t)\mapsto\Ebb_r[\phi_{t,\lambda}(X)\phi_{r,\lambda}(X)]\) and \((r,t)\mapsto\Ebb_r[\phi'_{t,\lambda}(X)\phi'_{r,\lambda}(X)]\) are continuous.
This verifies conditions (i)--(ii) of Lemma~\ref{lem:HK-information-identity}.
Moreover, \(\Ebb_{\theta_0}[\partial_t\phi_{t,\lambda}(X)]<0\), verifying Theorem~\ref{thm:z-est}(iii).

The empirical score is strictly decreasing in \(t\) and has opposite signs at the two infinities, so its unique zero depends \(C^1\) on the observations.
The population score is also strictly decreasing and vanishes at \(\theta_0\).
For every \(\varepsilon>0\), the law of large numbers at \(\theta_0-\varepsilon\) and \(\theta_0+\varepsilon\) therefore places the empirical zero between these points with probability tending to one, proving consistency.

At \(\lambda=0\), \(s(t,x)=x-t\) gives the sample mean.
Its parameter and observation derivatives are \(-1\) and \(1\), so the ULLNs hold directly.
Since \(\Ebb_r[(X-t)^2]=2+(r-t)^2\), the uniform moment bound also holds on \(J\times J\); the gradient cross expectation equals \(1\).
Consistency follows from the law of large numbers.

\subsection{Uniform scale family}
\label{app:uniform-verification}

Recall the uniform scale family \(p_\theta(x)=\theta^{-1}\mathbf 1_{\{0<x<\theta\}}\), \(\theta>0\).

For \(0<\lambda<1\), we verify the function given in~\eqref{eq:phi_uniform_interior}:
\[
    \phi_{\theta,\lambda}(x)
    =\frac{\cosh(\beta_\lambda x)}
    {\sqrt{\lambda(1-\lambda)}\,\sinh(\beta_\lambda\theta)}
    -\frac1{\lambda\theta},
    \qquad
    \beta_\lambda:=\sqrt{\frac{\lambda}{1-\lambda}}.
\]
\paragraph{Verification of the transport-source equation.}
To verify the \(\HK_\lambda\) equation for \(0\le\lambda<1\), it suffices to show that, for every \(f\in C_c^\infty(\R)\),
\[
    \partial_\theta\Ebb_\theta[f(X)]
    =(1-\lambda)\Ebb_\theta\!
    \left[f'(X)\phi'_{\theta,\lambda}(X)\right]
    +\lambda\Ebb_\theta\!
    \left[f(X)\phi_{\theta,\lambda}(X)\right].
\]
Since \(\Ebb_\theta[f(X)]=\theta^{-1}\int_0^\theta f(x)\diff x\), direct differentiation gives
\begin{equation}
    \partial_\theta\Ebb_\theta[f(X)]
    =\frac{f(\theta)}{\theta}-\frac1{\theta^2}\int_0^\theta f(x)\diff x.
    \label{eq:uniform-weak-left}
\end{equation}
For \(0<\lambda<1\), a direct calculation gives
\[
    \phi'_{\theta,\lambda}(x)
    =\frac{\sinh(\beta_\lambda x)}
    {(1-\lambda)\sinh(\beta_\lambda\theta)},
    \qquad
    \phi''_{\theta,\lambda}(x)
    =\frac{\beta_\lambda^2\cosh(\beta_\lambda x)}
    {\sqrt{\lambda(1-\lambda)}\,\sinh(\beta_\lambda\theta)}.
\]
It follows that \(\phi'_{\theta,\lambda}(0)=0\) and \((1-\lambda)\phi'_{\theta,\lambda}(\theta)=1\). Moreover, using \((1-\lambda)\beta_\lambda^2=\lambda\), we obtain
\begin{equation}
    (1-\lambda)\phi''_{\theta,\lambda}(x)
    -\lambda\phi_{\theta,\lambda}(x)=\frac1\theta.
    \label{eq:uniform-interior-identity}
\end{equation}
Therefore, integrating by parts,
\begin{align*}
    &(1-\lambda)\Ebb_\theta\!
    \left[f'(X)\phi'_{\theta,\lambda}(X)\right]
    +\lambda\Ebb_\theta\!
    \left[f(X)\phi_{\theta,\lambda}(X)\right]\\
    &=\frac{1-\lambda}{\theta}\int_0^\theta
        f'(x)\phi'_{\theta,\lambda}(x)\diff x
      +\frac{\lambda}{\theta}\int_0^\theta
        f(x)\phi_{\theta,\lambda}(x)\diff x\\
    &=\frac{1-\lambda}{\theta}
      \left[f(x)\phi'_{\theta,\lambda}(x)\right]_{x=0}^{x=\theta}
      -\frac{1-\lambda}{\theta}\int_0^\theta
        f(x)\phi''_{\theta,\lambda}(x)\diff x
      +\frac{\lambda}{\theta}\int_0^\theta
        f(x)\phi_{\theta,\lambda}(x)\diff x\\
    &=\frac{f(\theta)}{\theta}
      -\frac1\theta\int_0^\theta f(x)
      \left[(1-\lambda)\phi''_{\theta,\lambda}(x)
      -\lambda\phi_{\theta,\lambda}(x)\right]\diff x\\
    &=\frac{f(\theta)}{\theta}-\frac1{\theta^2}\int_0^\theta f(x)\diff x,
\end{align*}
where the fourth line uses the two boundary conditions and the final line uses \eqref{eq:uniform-interior-identity}. Comparison with \eqref{eq:uniform-weak-left} proves the weak \(\HK_\lambda\) equation.

At \(\lambda=0\),
\[
    \phi_{\theta,0}(x)=\frac{x^2-\theta^2/3}{2\theta},
    \qquad
    \phi'_{\theta,0}(0)=0,
    \qquad
    \phi'_{\theta,0}(\theta)=1,
    \qquad
    \phi''_{\theta,0}(x)=\frac1\theta.
\]
The same integration-by-parts argument gives
\begin{align*}
    \Ebb_\theta\!\left[f'(X)\phi'_{\theta,0}(X)\right]
    &=\frac1{\theta^2}\int_0^\theta xf'(x)\diff x\\
    &=\frac{f(\theta)}{\theta}-\frac1{\theta^2}\int_0^\theta f(x)\diff x\\
    &=\partial_\theta\Ebb_\theta[f(X)],
\end{align*}
where the last equality follows from~\eqref{eq:uniform-weak-left}.
Thus \(\phi_{\theta,0}\) solves the \(\HK_0\) continuity equation.

\paragraph{Verification of regularity.}
Fix \(0\le\lambda<1\).
The explicit formulas define smooth functions on \(\R\).
A direct integration shows that \(\Ebb_\theta[\phi_{\theta,\lambda}(X)]=0\), and compactness of the support implies that the score function has finite \(\mathcal H^1_\lambda(P_\theta)\) energy.
Thus \(\phi_{\theta,\lambda}\in\mathcal H_\lambda^1(P_\theta)\).
At \(\lambda=0\), \(\|\phi_{\theta,0}\|_{\mathcal H_0^1(P_\theta)}^2=\Ebb_\theta[(X/\theta)^2]=1/3\).
After the change of variables \(x=\theta u\), the energy is an integral over \([0,1]\) of a function continuous in \((\theta,u)\), so it is continuous and locally bounded for \(\theta>0\).
This verifies condition (i) of Definition~\ref{def:HK_regular}.

For any \(\theta>0\) and \(h\in\R\) with \(\theta+h>0\), set \(\phi_t=h\phi_{\theta+th,\lambda}\), \(t\in[0,1]\).
The map \((t,x)\mapsto(\phi_t(x),\nabla\phi_t(x))\) is continuous and hence jointly measurable.
Applying the chain rule to the verified weak \(\HK_\lambda\) equation shows that \(\phi_t\) solves~\eqref{HK-cont-eq}, and the local energy bounds give
\[
    \int_0^1
    \|\phi_t\|_{\mathcal H_\lambda^1(P_{\theta+th})}^2
    \diff t<\infty.
\]
Proposition~\ref{prop:Sobolev} and Proposition~\ref{prop:ans-cont-unique-solution} give uniqueness up to the \(\mathcal H_\lambda^1\) norm, verifying (ii).
The form \(\phi_t=h\phi_{\theta+th,\lambda}\) then verifies (iii).
Hence the model is \(\HK_0\) regular and \(\HK\) regular.

Under the pure source case, however,
\[
    \partial_\theta p_\theta
    =-\frac1{\theta^2}\mathbf 1_{\{0<x<\theta\}}+\frac1\theta\delta_\theta.
\]
The boundary Dirac mass cannot be represented as \(\phi_{\theta,1}p_\theta\) by an ordinary function \(\phi_{\theta,1}\). Hence the uniform scale family has no classical function-valued \(\HK_1\) score function.
Using the density \(p_\theta(x)=\theta^{-1}\mathbf 1_{\{0\le x\le\theta\}}\), the likelihood is proportional to \(\theta^{-n}\mathbf 1_{\{\theta\ge X_{(n)}\}}\).
Thus the sample maximum \(X_{(n)}\) is the MLE for \(\theta\).

\paragraph{Remaining assumptions of Theorems~\ref{thm:HK-CR-bound} and~\ref{thm:HK-Z-attainment}.}
Fix \(\theta_0>0\), \(0\le\lambda<1\), and a compact interval \(J=[a,b]\subset(0,\infty)\) containing \(\theta_0\) in its interior.
For \(r\in J\), the distribution \(P_r\) is supported on \([0,b]\).
The explicit formulas are smooth in \((t,x)\) on \(J\times[0,b]\).
Consequently, for either \(g_t(x)=\partial_t\phi_{t,\lambda}(x)\) or \(g_t(x)=|\phi'_{t,\lambda}(x)|^2\), the functions \(g_t(x)\) and \(\partial_tg_t(x)\) are uniformly bounded there.
The mean value theorem gives \(|g_t(x)-g_u(x)|\le C|t-u|\) for a finite constant \(C\).
Similar to the Laplace case, we obtain the required ULLNs in Theorem~\ref{thm:z-est}(iv) from~\eqref{eq:laplace-derivative-ulln}.

For \(0<\lambda<1\), the explicit formulas give, for \(t\in[a,b]\) and \(0\le x\le b\),
\begin{equation}\label{eq:uniform-score-bounds}
    \begin{aligned}
        |\phi_{t,\lambda}(x)|
        &\le
        \frac{\cosh(\beta_\lambda b)}
        {\sqrt{\lambda(1-\lambda)}\,\sinh(\beta_\lambda a)}
        +\frac1{\lambda a},\\
        |\phi'_{t,\lambda}(x)|
        &\le
        \frac{\sinh(\beta_\lambda b)}
        {(1-\lambda)\sinh(\beta_\lambda a)}.
    \end{aligned}
\end{equation}
At \(\lambda=0\), the corresponding bounds are \(|\phi_{t,0}(x)|\le b^2/(2a)+b/6\) and \(|\phi'_{t,0}(x)|\le b/a\).
These bounds are independent of \(r,t\in J\), and therefore give \(\sup_{r,t\in J}\Ebb_r[\phi_{t,\lambda}(X)^2+|\phi'_{t,\lambda}(X)|^2]<\infty\).

To verify continuity when the sampling distribution also varies, write \(X=rZ\), where \(Z\) is uniform on \([0,1]\).
For example,
\[
    \Ebb_r[\phi_{t,\lambda}(X)^2]
    =\int_0^1\phi_{t,\lambda}(ru)^2\diff u.
\]
The integrand is continuous in \((r,t,u)\) and bounded uniformly on \(J\times J\times[0,1]\).
The same applies to the products \(|\phi'_{t,\lambda}(ru)|^2\), \(\phi_{t,\lambda}(ru)\phi_{r,\lambda}(ru)\), and \(\phi'_{t,\lambda}(ru)\phi'_{r,\lambda}(ru)\).
Dominated convergence therefore gives the continuity of \(\Ebb_\theta[\phi_{\theta,\lambda}(X)^2]\) and \(\Ebb_\theta[|\phi'_{\theta,\lambda}(X)|^2]\) in \(\theta\), and the cross-expectation continuity required in Lemma~\ref{lem:HK-information-identity}(ii).

Whenever \(I_\lambda(\theta_0)\) is invertible, the information identity gives \(\Ebb_{\theta_0}[\partial_t\phi_{\theta_0,\lambda}(X)]=-I_\lambda(\theta_0)\ne0\).
Continuity of \(t\mapsto\Ebb_{\theta_0}[\partial_t\phi_{t,\lambda}(X)]\) then verifies Theorem~\ref{thm:z-est}(iii) near \(\theta_0\).

For \(0<\lambda<1\), the estimating equation in Example~\ref{ex:Uniform-HK-Z} is
\[
    H_\lambda(T_{n,\lambda})
    =\frac1n\sum_{i=1}^n\cosh(\beta_\lambda X_i),
    \qquad
    H_\lambda(t):=\frac{\sinh(\beta_\lambda t)}{\beta_\lambda t}.
\]
For \(t>0\),
\[
    H_\lambda'(t)
    =
    \frac{\beta_\lambda t\cosh(\beta_\lambda t)-\sinh(\beta_\lambda t)}
    {\beta_\lambda t^2}>0,
\]
since \(z\cosh z-\sinh z\) is zero at \(z=0\) and has derivative \(z\sinh z>0\) for \(z>0\).
Moreover, \(H_\lambda(t)\to1\) as \(t\downarrow0\), and \(H_\lambda(t)\to\infty\) as \(t\to\infty\).
Thus \(H_\lambda\) has a \(C^1\) inverse on \((1,\infty)\).

For every sample other than the zero vector, \(n^{-1}\sum_i\cosh(\beta_\lambda X_i)>1\), so
\begin{equation}\label{eq:uniform-estimator-inverse}
    T_{n,\lambda}
    =
    H_\lambda^{-1}\left(
    \frac1n\sum_{i=1}^n\cosh(\beta_\lambda X_i)
    \right)
\end{equation}
is the unique positive solution and depends locally \(C^1\) on the observations.
Samples from \(P_{\theta_0}\) are nonzero almost surely.
Finally,
\[
    \frac1n\sum_{i=1}^n\cosh(\beta_\lambda X_i)
    \longrightarrow
    \frac1{\theta_0}\int_0^{\theta_0}\cosh(\beta_\lambda x)\diff x
    =H_\lambda(\theta_0)
\]
almost surely.
Applying the continuous function \(H_\lambda^{-1}\) in~\eqref{eq:uniform-estimator-inverse} proves \(T_{n,\lambda}\to\theta_0\) almost surely.

At \(\lambda=0\), the estimator is \(\sqrt{3n^{-1}\sum_iX_i^2}\).
It is locally smooth away from the zero sample, and the law of large numbers gives \(n^{-1}\sum_iX_i^2\to\theta_0^2/3\), proving consistency.

\subsection{Gaussian location-scale family}
\label{app:gaussian-location-scale-verification}

Recall the Gaussian location-scale family
\[
    p_\theta(x)=p_{\mu,\sigma^2}(x)
    =\frac{1}{\sqrt{2\pi\sigma^2}}
    \exp\left(-\frac{(x-\mu)^2}{2\sigma^2}\right),
    \qquad
    \theta=(\mu,\sigma^2)\in\R\times(0,\infty).
\]

For \(0\leq\lambda\leq1\), set \(a_\lambda:=(1-\lambda)+\lambda\sigma^2\) and \(b_\lambda:=2(1-\lambda)+\lambda\sigma^2\), and recall from Example~\ref{ex:Gaussian-loc-scale} that
\[
    \phi_{\theta,\lambda}(x)
    =\begin{pmatrix}
        \phi_{\theta,\lambda,\mu}(x)\\[1mm]
        \phi_{\theta,\lambda,\sigma^2}(x)
    \end{pmatrix}
    =\begin{pmatrix}
        \dfrac{x-\mu}{a_\lambda}\\[3mm]
        \dfrac{(x-\mu)^2-\sigma^2}{2\sigma^2b_\lambda}
    \end{pmatrix}.
\]
\paragraph{Verification of the transport-source equation.}

To verify the \(\HK_\lambda\) equation, it suffices to show that, for every \(f\in C_c^\infty(\R)\) and \(\eta\in\{\mu,\sigma^2\}\),
\begin{equation}\label{eq:gaussian-ls-weak}
    \partial_\eta\Ebb_\theta[f(X)]
    =(1-\lambda)\Ebb_\theta\!
    \left[f'(X)\phi'_{\theta,\lambda,\eta}(X)\right]
    +\lambda\Ebb_\theta\!
    \left[f(X)\phi_{\theta,\lambda,\eta}(X)\right].
\end{equation}
We first record
\[
    \frac{p_\theta'(x)}{p_\theta(x)}=-\frac{x-\mu}{\sigma^2},
    \qquad
    \frac{\partial_\mu p_\theta(x)}{p_\theta(x)}=\frac{x-\mu}{\sigma^2},
    \qquad
    \frac{\partial_{\sigma^2}p_\theta(x)}{p_\theta(x)}
    =\frac{(x-\mu)^2-\sigma^2}{2\sigma^4}.
\]
For the location coordinate, \(\phi'_{\theta,\lambda,\mu}(x)=1/a_\lambda\) and \(\phi''_{\theta,\lambda,\mu}(x)=0\). Consequently,
\begin{align*}
    &-(1-\lambda)
    \frac{\partial_x\!
    \left(\phi'_{\theta,\lambda,\mu}(x)p_\theta(x)\right)}{p_\theta(x)}
    +\lambda\phi_{\theta,\lambda,\mu}(x)\\
    &=\frac{x-\mu}{a_\lambda}
    \left(\frac{1-\lambda}{\sigma^2}+\lambda\right)
    =\frac{x-\mu}{\sigma^2}
    =\frac{\partial_\mu p_\theta(x)}{p_\theta(x)}.
\end{align*}
For the scale coordinate, \(\phi'_{\theta,\lambda,\sigma^2}(x)=(x-\mu)/(\sigma^2b_\lambda)\) and \(\phi''_{\theta,\lambda,\sigma^2}(x)=1/(\sigma^2b_\lambda)\). It follows that
\begin{align*}
    &-(1-\lambda)
    \frac{\partial_x\!
    \left(\phi'_{\theta,\lambda,\sigma^2}(x)p_\theta(x)\right)}{p_\theta(x)}
    +\lambda\phi_{\theta,\lambda,\sigma^2}(x)\\
    &=-(1-\lambda)
    \left[
        \frac1{\sigma^2b_\lambda}
        -\frac{(x-\mu)^2}{\sigma^4b_\lambda}
    \right]
    +\lambda\frac{(x-\mu)^2-\sigma^2}{2\sigma^2b_\lambda}\\
    &=\frac{(x-\mu)^2-\sigma^2}{\sigma^4b_\lambda}
    \left(1-\lambda+\frac{\lambda\sigma^2}{2}\right)\\
    &=\frac{(x-\mu)^2-\sigma^2}{2\sigma^4}
    =\frac{\partial_{\sigma^2}p_\theta(x)}{p_\theta(x)}.
\end{align*}
Thus, for each \(\eta\in\{\mu,\sigma^2\}\),
\[
    \partial_\eta p_\theta
    =-(1-\lambda)\partial_x\!
    \left(\phi'_{\theta,\lambda,\eta}p_\theta\right)
    +\lambda\phi_{\theta,\lambda,\eta}p_\theta.
\]
Multiplying by \(f\), integrating over \(\R\), and integrating the transport term by parts gives
\begin{align*}
    \partial_\eta\Ebb_\theta[f(X)]
    &=\int_{\R}f(x)\partial_\eta p_\theta(x)\diff x\\
    &=-(1-\lambda)\int_{\R}f(x)
      \partial_x\!\left(\phi'_{\theta,\lambda,\eta}(x)p_\theta(x)\right)\diff x\\
    &\quad+\lambda\int_{\R}f(x)
      \phi_{\theta,\lambda,\eta}(x)p_\theta(x)\diff x\\
    &=(1-\lambda)\Ebb_\theta\!
      \left[f'(X)\phi'_{\theta,\lambda,\eta}(X)\right]
      +\lambda\Ebb_\theta\!
      \left[f(X)\phi_{\theta,\lambda,\eta}(X)\right].
\end{align*}
Hence both coordinates satisfy the weak \(\HK_\lambda\) equation.

At the endpoints,
\[
    \phi_{\theta,0}(x)
    =\begin{pmatrix}
        x-\mu\\[1mm]
        \dfrac{(x-\mu)^2-\sigma^2}{4\sigma^2}
    \end{pmatrix},
    \qquad
    \phi_{\theta,1}(x)
    =\begin{pmatrix}
        \dfrac{x-\mu}{\sigma^2}\\[2mm]
        \dfrac{(x-\mu)^2-\sigma^2}{2\sigma^4}
    \end{pmatrix}.
\]
The first vector solves the two coordinatewise \(\HK_0\) continuity equations, while the second is precisely the Gaussian score vector and therefore solves the \(\HK_1\) source equation.

\paragraph{Verification of regularity.}
Fix \(0\le\lambda\le1\).
Since \(\Ebb_\theta[X-\mu]=0\) and \(\Ebb_\theta[(X-\mu)^2]=\sigma^2\), both coordinates are centered under \(P_\theta\).
The coordinate functions are smooth in \(x\).
Using \(\Ebb_\theta[(X-\mu)^4]=3\sigma^4\), we obtain
\begin{align*}
    \|\phi_{\theta,\lambda,\mu}\|_{\mathcal H_\lambda^1(P_\theta)}^2
    &=\frac{\lambda\sigma^2+(1-\lambda)}{a_\lambda^2}
    =\frac1{a_\lambda},\\
    \|\phi_{\theta,\lambda,\sigma^2}\|_{\mathcal H_\lambda^1(P_\theta)}^2
    &=\frac{\lambda}{2b_\lambda^2}
    +\frac{1-\lambda}{\sigma^2b_\lambda^2}
    =\frac1{2\sigma^2b_\lambda}.
\end{align*}
Thus both coordinate functions belong to \(\mathcal H_\lambda^1(P_\theta)\), and their energies are continuous in \(\theta\) and bounded for each \(\theta\).
This verifies condition (i) of Definition~\ref{def:HK_regular}.

For any \(\theta\in\Theta\) and \(h=(h_1,h_2)\in\R^2\) such that \(\theta+th\in\Theta\) for all \(t\in[0,1]\), set
\[
    \phi_t
    =h_1\phi_{\theta+th,\lambda,\mu}
    +h_2\phi_{\theta+th,\lambda,\sigma^2}.
\]
The map \((t,x)\mapsto(\phi_t(x),\nabla\phi_t(x))\) is continuous and hence jointly measurable.
We have verified that \(\phi_t\) solves~\eqref{HK-cont-eq}.
Since \(\{\theta+th:0\le t\le1\}\) is a compact subset of \(\Theta\), the bounded \(\mathcal H_\lambda^1\) energy at each \(\theta\) gives
\[
    \int_0^1
    \|\phi_t\|_{\mathcal H_\lambda^1(P_{\theta+th})}^2
    \diff t<\infty.
\]
Proposition~\ref{prop:Sobolev} and Proposition~\ref{prop:ans-cont-unique-solution} give uniqueness up to the \(\mathcal H_\lambda^1\) norm, verifying (ii).
The form of \(\phi_t\) then verifies (iii).
Hence the Gaussian location-scale family is \(\HK_\lambda\) regular for all \(0\le\lambda\le1\).

\paragraph{Remaining assumptions of Theorems~\ref{thm:HK-CR-bound} and~\ref{thm:HK-Z-attainment}.}
Fix \(\theta_0\in\Theta\), \(0\le\lambda\le1\), and a compact neighborhood \(K\subset\Theta\) of \(\theta_0\).
The score is jointly \(C^1\) and quadratic in the observation.
Every entry of \(D_t\phi_{t,\lambda}(x)\) and \(D\phi_{t,\lambda}(x)D\phi_{t,\lambda}(x)^\top\) is a linear combination of \(1,x,x^2\).
Its coefficients are continuous in \(t\), so each coefficient is uniformly bounded on \(K\).
Consequently, the empirical errors, uniformly over \(t\in K\), are bounded, up to a constant coefficient, by
\[
    \left|
    \frac1n\sum_{i=1}^nX_i-\Ebb_{\theta_0}[X]
    \right|
    +
    \left|
    \frac1n\sum_{i=1}^nX_i^2-\Ebb_{\theta_0}[X^2]
    \right|.
\]
The law of large numbers makes this expression tend to zero, proving the entrywise ULLNs in Theorem~\ref{thm:z-est}(iv).

Since the score is quadratic and its observation derivative is linear, there is a constant \(C_K<\infty\) such that
\[
    \|\phi_{t,\lambda}(x)\|^2
    +\|D\phi_{t,\lambda}(x)\|_{\mathrm F}^2
    \le C_K(1+|x|^4),
    \qquad t\in K.
\]
Under \(P_r\), with \(r=(\mu,\sigma^2)\), \(\Ebb_r[X^4]=\mu^4+6\mu^2\sigma^2+3\sigma^4\).
This expression is bounded for \(r\in K\).
Therefore,
\[
    \sup_{r,t\in K}\Ebb_r[
    \|\phi_{t,\lambda}(X)\|^2
    +\|D\phi_{t,\lambda}(X)\|_{\mathrm F}^2]
    \le C_K\left(1+\sup_{r\in K}\Ebb_r[X^4]\right)<\infty,
\]
verifying Lemma~\ref{lem:HK-information-identity}(i).

The products \(\phi_{t,\lambda}(x)\phi_{r,\lambda}(x)^\top\) and \(D\phi_{t,\lambda}(x)D\phi_{r,\lambda}(x)^\top\) have polynomial entries of degree at most four, with coefficients continuous in \((r,t)\).
Their expectations are finite sums of these coefficients multiplied by Gaussian moments, which are continuous in the distribution parameter \(r\).
This proves Lemma~\ref{lem:HK-information-identity}(ii).
Setting \(r=t=\theta\) gives the continuity of \(\Ebb_\theta[\phi_{\theta,\lambda}(X)\phi_{\theta,\lambda}(X)^\top]\) and \(\Ebb_\theta[D\phi_{\theta,\lambda}(X)D\phi_{\theta,\lambda}(X)^\top]\) required in Theorem~\ref{thm:HK-CR-bound}.

Whenever \(I_\lambda(\theta_0)\) is invertible, the information identity and continuity of \(\Ebb_{\theta_0}[D_t\phi_{t,\lambda}(X)]\) give Theorem~\ref{thm:z-est}(iii) near \(\theta_0\).
For \(n\ge2\), the sample variance is positive almost surely, so the explicit estimator in Example~\ref{rem:Gaussian-LS-HK-Z} lies in \(\Theta\) and is locally smooth in the observations.
Its consistency follows from the law of large numbers.

\subsection{Gaussian correlation family}
\label{app:gaussian-correlation-verification}

Recall the Gaussian correlation family
\[
    P_\theta=\Ncal(0,\Sigma_\theta),
    \qquad
    \Sigma_\theta
    =
    \begin{pmatrix}
        1 & \theta\\
        \theta & 1
    \end{pmatrix},
    \qquad
    \theta\in(-1,1).
\]
Introduce the orthogonal coordinates
\[
    u=\frac{x+y}{\sqrt2},
    \qquad
    v=\frac{x-y}{\sqrt2}.
\]
Under \(P_\theta\), the variables \(U\) and \(V\) are independent, with
\[
    U\sim\Ncal(0,1+\theta),
    \qquad
    V\sim\Ncal(0,1-\theta).
\]

Fix \(0\leq\lambda\leq1\), and recall from
Example~\ref{ex:Gaussian-corr-reg} that, in the \((u,v)\)-coordinates,
\[
    \phi_{\theta,\lambda}(u,v)
    =
    \frac{u^2-(1+\theta)}
    {2(1+\theta)\bigl(2-\lambda(1-\theta)\bigr)}
    -
    \frac{v^2-(1-\theta)}
    {2(1-\theta)\bigl(2-\lambda(1+\theta)\bigr)}.
\]
\paragraph{Verification of the transport-source equation.}

Since the change of coordinates is orthogonal, it suffices to verify the weak
transport-source equation in the \((u,v)\)-coordinates. The density in these
coordinates is
\[
    p_\theta(u,v)
    =
    \frac{1}{2\pi\sqrt{(1+\theta)(1-\theta)}}
    \exp\left(
        -\frac{u^2}{2(1+\theta)}
        -
        \frac{v^2}{2(1-\theta)}
    \right).
\]
Direct differentiation gives
\begin{equation}
    \frac{\partial_\theta p_\theta(u,v)}
    {p_\theta(u,v)}
    =
    \frac{u^2-(1+\theta)}{2(1+\theta)^2}
    -
    \frac{v^2-(1-\theta)}{2(1-\theta)^2}.
    \label{eq:gaussian-corr-score}
\end{equation}
Thus, for every \(f\in C_c^\infty(\R^2)\),
\[
    \partial_\theta\Ebb_\theta[f(U,V)]
    =
    \Ebb_\theta
    \left[
        f(U,V)
        \left(
            \frac{U^2-(1+\theta)}{2(1+\theta)^2}
            -
            \frac{V^2-(1-\theta)}{2(1-\theta)^2}
        \right)
    \right].
\]
The gradient of the score function is
\[
    \nabla\phi_{\theta,\lambda}(u,v)
    =
    \left(
        \frac{u}
        {(1+\theta)\bigl(2-\lambda(1-\theta)\bigr)},
        -\frac{v}
        {(1-\theta)\bigl(2-\lambda(1+\theta)\bigr)}
    \right),
\]
while
\[
    \nabla\log p_\theta(u,v)
    =
    \left(
        -\frac{u}{1+\theta},
        -\frac{v}{1-\theta}
    \right).
\]
Consequently,
\begin{align*}
    \frac{
        \Div\left(
            p_\theta\nabla\phi_{\theta,\lambda}
        \right)
    }{
        p_\theta
    }
    &=
    \Delta\phi_{\theta,\lambda}
    +
    \nabla\phi_{\theta,\lambda}
    \cdot
    \nabla\log p_\theta\\
    &=
    -\frac{u^2-(1+\theta)}
    {(1+\theta)^2\bigl(2-\lambda(1-\theta)\bigr)}
    +
    \frac{v^2-(1-\theta)}
    {(1-\theta)^2\bigl(2-\lambda(1+\theta)\bigr)}.
\end{align*}
Using
\[
    2-\lambda(1-\theta)
    =
    2(1-\lambda)+\lambda(1+\theta),
\]
and
\[
    2-\lambda(1+\theta)
    =
    2(1-\lambda)+\lambda(1-\theta),
\]
we obtain
\begin{align*}
    &-(1-\lambda)
    \frac{
        \Div\left(
            p_\theta\nabla\phi_{\theta,\lambda}
        \right)
    }{
        p_\theta
    }
    +
    \lambda\phi_{\theta,\lambda}\\
    &=
    \frac{u^2-(1+\theta)}
    {2(1+\theta)^2\bigl(2-\lambda(1-\theta)\bigr)}
    \left[
        2(1-\lambda)+\lambda(1+\theta)
    \right]\\
    &\quad-
    \frac{v^2-(1-\theta)}
    {2(1-\theta)^2\bigl(2-\lambda(1+\theta)\bigr)}
    \left[
        2(1-\lambda)+\lambda(1-\theta)
    \right]\\
    &=
    \frac{u^2-(1+\theta)}{2(1+\theta)^2}
    -
    \frac{v^2-(1-\theta)}{2(1-\theta)^2}.
\end{align*}
Comparison with \eqref{eq:gaussian-corr-score} gives the strong identity
\[
    \partial_\theta p_\theta
    =
    -(1-\lambda)
    \Div\left(
        p_\theta\nabla\phi_{\theta,\lambda}
    \right)
    +
    \lambda\phi_{\theta,\lambda}p_\theta.
\]
Multiplying by \(f\), integrating over \(\R^2\), and integrating the transport
term by parts gives
\begin{align*}
    \partial_\theta\Ebb_\theta[f(U,V)]
    &=
    (1-\lambda)
    \Ebb_\theta
    \left[
        \nabla f(U,V)
        \cdot
        \nabla\phi_{\theta,\lambda}(U,V)
    \right]\\
    &\quad+
    \lambda
    \Ebb_\theta
    \left[
        f(U,V)\phi_{\theta,\lambda}(U,V)
    \right].
\end{align*}
The boundary term vanishes because \(f\) has compact support. Hence
\(\phi_{\theta,\lambda}\) solves the weak
\(\HK_\lambda\) equation.

Under the pure transport case, 
\[
    \phi_{\theta,0}(u,v)
    =
    \frac{u^2-(1+\theta)}{4(1+\theta)}
    -
    \frac{v^2-(1-\theta)}{4(1-\theta)},
\]
while under the pure source case,
\[
    \phi_{\theta,1}(u,v)
    =
    \frac{u^2-(1+\theta)}{2(1+\theta)^2}
    -
    \frac{v^2-(1-\theta)}{2(1-\theta)^2},
\]
which is the Gaussian correlation score function. Thus they solve the \(\HK_0\) continuity equation and the \(\HK_1\) source equation, respectively.

\paragraph{Verification of regularity.}
Fix \(0\le\lambda\le1\).
The score function is centered because
\[
    \Ebb_\theta[U^2]=1+\theta,
    \qquad
    \Ebb_\theta[V^2]=1-\theta.
\]
It is smooth in \((u,v)\).
Independence of \(U\) and \(V\), together with
\(\Ebb_\theta[U^4]=3(1+\theta)^2\) and
\(\Ebb_\theta[V^4]=3(1-\theta)^2\), gives
\begin{align*}
    \Ebb_\theta[\phi_{\theta,\lambda}(U,V)^2]
    &=
    \frac{1}{2\bigl(2-\lambda(1-\theta)\bigr)^2}
    +\frac{1}{2\bigl(2-\lambda(1+\theta)\bigr)^2},\\
    \Ebb_\theta[\|\nabla\phi_{\theta,\lambda}(U,V)\|^2]
    &=
    \frac{1}{(1+\theta)\bigl(2-\lambda(1-\theta)\bigr)^2}
    +\frac{1}{(1-\theta)\bigl(2-\lambda(1+\theta)\bigr)^2}.
\end{align*}
Consequently,
\[
    \|\phi_{\theta,\lambda}\|_{\mathcal H_\lambda^1(P_\theta)}^2
    =
    \frac{1}{2(1+\theta)\bigl(2-\lambda(1-\theta)\bigr)}
    +\frac{1}{2(1-\theta)\bigl(2-\lambda(1+\theta)\bigr)}.
\]
Thus \(\phi_{\theta,\lambda}\) belongs to \(\mathcal H_\lambda^1(P_\theta)\), and its energy is continuous in \(\theta\) and bounded for each \(\theta\).
This verifies condition (i) of Definition~\ref{def:HK_regular}.

For any \(\theta\in(-1,1)\) and \(h\in\R\) such that \(\theta+th\in(-1,1)\) for all \(t\in[0,1]\), set \(\phi_t=h\phi_{\theta+th,\lambda}\).
The map \((t,u,v)\mapsto(\phi_t(u,v),\nabla\phi_t(u,v))\) is continuous and hence jointly measurable.
We have verified that \(\phi_t\) solves~\eqref{HK-cont-eq}.
Since \(\{\theta+th:0\le t\le1\}\) is a compact subset of \((-1,1)\), the bounded \(\mathcal H_\lambda^1\) energy at each \(\theta\) gives
\[
    \int_0^1
    \|\phi_t\|_{\mathcal H_\lambda^1(P_{\theta+th})}^2
    \diff t<\infty.
\]
Proposition~\ref{prop:Sobolev} and Proposition~\ref{prop:ans-cont-unique-solution} give uniqueness up to the \(\mathcal H_\lambda^1\) norm, verifying (ii).
The form of \(\phi_t\) then verifies (iii).
Hence the Gaussian correlation family is \(\HK_\lambda\) regular for all \(0\le\lambda\le1\).

\paragraph{Remaining assumptions of Theorems~\ref{thm:HK-CR-bound} and~\ref{thm:HK-Z-attainment}.}
Fix \(\theta_0\in(-1,1)\) and \(0\le\lambda\le1\).
The expressions computed above show that \(\theta\mapsto\Ebb_\theta[\phi_{\theta,\lambda}(U,V)^2]\) and \(\theta\mapsto\Ebb_\theta[\|\nabla\phi_{\theta,\lambda}(U,V)\|^2]\) are continuous.

The score is jointly \(C^1\).
The functions \(\partial_t\phi_{t,\lambda}(u,v)\) and \(\|\nabla\phi_{t,\lambda}(u,v)\|^2\) are linear combinations of \(1,u^2,v^2\).
On any compact parameter interval \(K\subset(-1,1)\), their coefficients are continuous, and each coefficient is uniformly bounded on \(K\).
Since
\[
    \frac1n\sum_{i=1}^nU_i^2\longrightarrow1+\theta_0,
    \qquad
    \frac1n\sum_{i=1}^nV_i^2\longrightarrow1-\theta_0
\]
almost surely, the empirical averages of these functions converge to their expectations uniformly over \(t\in K\).
This verifies Theorem~\ref{thm:z-est}(iv).

There is also a constant \(C_K<\infty\) such that
\[
    |\phi_{t,\lambda}(u,v)|^2
    +\|\nabla\phi_{t,\lambda}(u,v)\|^2
    \le C_K(1+u^4+v^4),
    \qquad t\in K.
\]
Because \(\Ebb_r[U^4]=3(1+r)^2\le12\) and \(\Ebb_r[V^4]=3(1-r)^2\le12\), the expectation of this bound is at most \(25C_K\), uniformly for \(r,t\in K\).
This proves Lemma~\ref{lem:HK-information-identity}(i).
Expanding \(\phi_{t,\lambda}\phi_{r,\lambda}\) and \(\nabla\phi_{t,\lambda}\cdot\nabla\phi_{r,\lambda}\) expresses their expectations in terms of \(1\), \(1\pm r\), \(3(1\pm r)^2\), and \(\Ebb_r[U^2V^2]=1-r^2\), with coefficients continuous in \((r,t)\).
Their expectations are therefore continuous, verifying Lemma~\ref{lem:HK-information-identity}(ii).
Whenever \(I_\lambda(\theta_0)\) is invertible, the information identity and continuity of \(\Ebb_{\theta_0}[\partial_t\phi_{t,\lambda}(U,V)]\) give Theorem~\ref{thm:z-est}(iii) near \(\theta_0\).

Almost surely, \(\hat\sigma_{XY}\ne0\) and \(\hat\sigma_X^2+\hat\sigma_Y^2>2|\hat\sigma_{XY}|\).
For \(0<\lambda\le1\), the estimating equation
\begin{equation}\label{eq:gaussian-corr-estimating-verification}
    \lambda t^3-\lambda\hat\sigma_{XY}t^2
    +(\hat\sigma_X^2+\hat\sigma_Y^2-\lambda)t
    +(\lambda-2)\hat\sigma_{XY}=0
\end{equation}
has a zero between \(0\) and \(\operatorname{sgn}(\hat\sigma_{XY})\).
At any such zero \(t\), its derivative with respect to \(t\) equals
\begin{equation}\label{eq:gaussian-corr-root-derivative}
    \begin{aligned}
        &3\lambda t^2-2\lambda\hat\sigma_{XY}t
        +\hat\sigma_X^2+\hat\sigma_Y^2-\lambda\\
        &\qquad=
        2\lambda t^2+
        \frac{\hat\sigma_{XY}}t
        (2-\lambda-\lambda t^2)>0.
    \end{aligned}
\end{equation}
By~\eqref{eq:gaussian-corr-root-derivative}, every zero with the same sign as \(\hat\sigma_{XY}\) crosses from negative to positive.
There cannot be two such zeros, since a return from positive to negative would require a zero at which the derivative is not positive.
Thus there is exactly one zero with the same sign as \(\hat\sigma_{XY}\).
The nonzero derivative allows the implicit function theorem to identify this selected zero as a locally \(C^1\) function of the observations.

At \(\lambda=0\),
\[
    T_{n,0}
    =\frac{2\hat\sigma_{XY}}{\hat\sigma_X^2+\hat\sigma_Y^2}.
\]

At \(\lambda=1\), the \(\HK_\lambda\) \(Z\)-estimating equation~\eqref{eq:gaussian-corr-estimating-verification} is equivalent to \(\ell_n'(t)=0\), where the log-likelihood, up to an additive constant, is
\[
    \ell_n(t)
    =-\frac n2\log(1-t^2)
    -\frac{n(\hat\sigma_X^2+\hat\sigma_Y^2-2t\hat\sigma_{XY})}
    {2(1-t^2)},\qquad -1<t<1.
\]
Almost surely, \(\hat\sigma_X^2+\hat\sigma_Y^2>2|\hat\sigma_{XY}|\) and \(\hat\sigma_{XY}\ne0\).
Consequently, \(\ell_n(t)\to-\infty\) as \(t\to\pm1\), so its maximum is attained in \((-1,1)\).
Moreover, \(\ell_n(t)-\ell_n(-t)=2nt\hat\sigma_{XY}/(1-t^2)\) and \(\ell_n'(0)=n\hat\sigma_{XY}\ne0\), so every maximizer has the same sign as \(\hat\sigma_{XY}\).
The uniqueness of the root with this sign therefore shows that \(T_{n,1}\) is the MLE.

To prove consistency, substitute the selected estimator into~\eqref{eq:gaussian-corr-estimating-verification} and rearrange:
\begin{equation}\label{eq:gaussian-corr-estimator-consistency}
    (T_{n,\lambda}-\hat\sigma_{XY})
    (2-\lambda+\lambda T_{n,\lambda}^2)
    =
    (2-\hat\sigma_X^2-\hat\sigma_Y^2)T_{n,\lambda}.
\end{equation}
Equation~\eqref{eq:gaussian-corr-estimator-consistency} also holds at \(\lambda=0\).
Since \(|T_{n,\lambda}|<1\) and \(2-\lambda+\lambda T_{n,\lambda}^2\ge1\),
\[
    |T_{n,\lambda}-\theta_0|
    \le
    |\hat\sigma_{XY}-\theta_0|
    +|\hat\sigma_X^2+\hat\sigma_Y^2-2|.
\]
The law of large numbers gives \(\hat\sigma_{XY}\to\theta_0\) and \(\hat\sigma_X^2+\hat\sigma_Y^2\to2\) almost surely.
Hence \(T_{n,\lambda}\to\theta_0\) almost surely for every \(0\le\lambda\le1\).

\subsection{Pareto index family}
\label{app:pareto-verification}

Recall the Pareto index family \(p_\theta(x)=\theta x^{-(\theta+1)}\mathbf1_{\{x\ge1\}}\), \(\theta>2\).
For $\theta>2$, we have $\Ebb_\theta[X^2]=\theta/(\theta-2)$, so the Pareto model belongs to $\Pcal_{2,\mathrm{ac}}(\R)$.

\paragraph{Verification of the transport-source equation.}
The endpoint functions in~\eqref{eq:phi_pareto_potentials} are centered and satisfy
\[
    \phi'_{\theta,0}(x)=-\frac{x\log x}{\theta},
    \qquad
    \phi'_{\theta,1}(x)=-\frac1x,
    \qquad x>1.
\]
Since $\phi'_{\theta,0}(1)=0$, the transport flux has no boundary Dirac mass at $1$.
Direct calculation therefore gives
\begin{equation}\label{eq:pareto-endpoint-equations}
    \partial_\theta p_\theta
    +\partial_x\bigl(p_\theta\phi'_{\theta,0}\bigr)=0,
    \qquad
    \partial_\theta p_\theta
    =p_\theta\phi_{\theta,1}
\end{equation}
in the sense of distributions on $\R$.

For \(0<\lambda<1\), we construct a finite-energy solution without requiring a closed-form expression.
Fix \(\theta>2\), and let
\[
    V_\theta=
    \left\{
    u:[1,\infty)\to\R:
    \begin{aligned}
        &u|_{[1,R]}\in H^1([1,R])\text{ for every }R>1,\\
        &\int_1^\infty(u^2+|u'|^2)p_\theta\diff x<\infty
    \end{aligned}
    \right\},
\]
where derivatives are understood weakly.
This is a Hilbert space with inner product
\[
    \langle u,v\rangle_{\theta,\lambda}
    :=
    \int_1^\infty
    [\lambda uv+(1-\lambda)u'v']p_\theta\diff x.
\]
Restrictions of functions in \(C_c^\infty(\R)\) are dense in \(V_\theta\), by cutoff and mollification.

Since \(\log X\) has the exponential distribution with rate \(\theta\) under \(P_\theta\), we have
\[
    \Ebb_\theta\!\left[
    \left(\frac1\theta-\log X\right)^2
    \right]=\frac1{\theta^2}.
\]
Thus, for every \(f\in V_\theta\),
\[
    \begin{aligned}
        \left|
        \int_1^\infty
        \left(\frac1\theta-\log x\right)f\,p_\theta\diff x
        \right|
        &\le\frac1\theta\|f\|_{L^2(P_\theta)}\\
        &\le\frac1{\theta\sqrt\lambda}
        \left(
        \int_1^\infty
        [\lambda f^2+(1-\lambda)|f'|^2]p_\theta\diff x
        \right)^{1/2}\\
        &=\frac1{\theta\sqrt\lambda}
        \langle f,f\rangle_{\theta,\lambda}^{1/2}.
    \end{aligned}
\]
The Riesz representation theorem therefore gives a unique \(\phi_{\theta,\lambda}\in V_\theta\) satisfying
\begin{equation}\label{eq:pareto-variational}
    \int_1^\infty
    [\lambda\phi_{\theta,\lambda}f
    +(1-\lambda)\phi'_{\theta,\lambda}f']p_\theta\diff x
    =
    \int_1^\infty
    \left(\frac1\theta-\log x\right)f\,p_\theta\diff x
\end{equation}
for every \(f\in V_\theta\).

On each interval \([1,R]\), \(p_\theta\) is smooth and bounded away from zero.
Equation~\eqref{eq:pareto-variational} implies \(p_\theta\phi'_{\theta,\lambda}\in H^1([1,R])\), and hence \(\phi_{\theta,\lambda}\in H^2([1,R])\).
Taking smooth test functions supported in \((1,\infty)\) gives the differential equation, while allowing \(f(1)\ne0\) and integrating by parts gives the boundary condition:
\begin{equation}\label{eq:pareto-interior-equation}
    \begin{aligned}
        (1-\lambda)\left(
        \phi''_{\theta,\lambda}(x)
        -\frac{\theta+1}{x}\phi'_{\theta,\lambda}(x)
        \right)
        &=
        \lambda\phi_{\theta,\lambda}(x)
        -\left(\frac1\theta-\log x\right),
        \qquad x>1,\\
        \phi'_{\theta,\lambda}(1)&=0.
    \end{aligned}
\end{equation}
Thus the transport flux has no boundary Dirac mass at \(1\), and \(\phi_{\theta,\lambda}\) solves the \(\HK_\lambda\) equation in the sense of distributions on \(\R\).

\paragraph{Verification of regularity.}
The endpoint functions admit locally Lipschitz extensions obtained by extending them constantly to $x<1$.
Their energies are
\begin{equation}\label{eq:pareto-endpoint-energy}
    \begin{aligned}
        \|\phi_{\theta,0}\|_{\mathcal H_0^1(P_\theta)}^2
        &=\Ebb_\theta[\phi'_{\theta,0}(X)^2]
        =\frac{2}{\theta(\theta-2)^3},\\
        \|\phi_{\theta,1}\|_{\mathcal H_1^1(P_\theta)}^2
        &=\Ebb_\theta[\phi_{\theta,1}(X)^2]
        =\frac1{\theta^2}.
    \end{aligned}
\end{equation}

For \(0<\lambda<1\), the local \(H^2\) regularity and the boundary condition in~\eqref{eq:pareto-interior-equation} give a \(C^1\) extension by setting \(\phi_{\theta,\lambda}(x)=\phi_{\theta,\lambda}(1)\) for \(x<1\).
Hence \(\phi_{\theta,\lambda}\in\mathcal H_\lambda^1(P_\theta)\).
Taking \(f=1\) in~\eqref{eq:pareto-variational} gives \(\Ebb_\theta[\phi_{\theta,\lambda}(X)]=0\).
Taking \(f=\phi_{\theta,\lambda}\) gives
\[
    \begin{aligned}
        \|\phi_{\theta,\lambda}\|_{\mathcal H_\lambda^1(P_\theta)}^2
        &=
        \Ebb_\theta\!\left[
        \left(\frac1\theta-\log X\right)\phi_{\theta,\lambda}(X)
        \right]\\
        &\le
        \frac1{\theta\sqrt\lambda}
        \|\phi_{\theta,\lambda}\|_{\mathcal H_\lambda^1(P_\theta)}.
    \end{aligned}
\]
Therefore
\begin{equation}\label{eq:pareto-interior-energy}
    \|\phi_{\theta,\lambda}\|_{\mathcal H_\lambda^1(P_\theta)}^2
    \le\frac1{\lambda\theta^2}.
\end{equation}

We next verify parameter measurability for fixed \(0<\lambda<1\).
Let \(\theta_k\to\theta>2\), with \(\theta_k>2\), and choose a compact interval \(K\subset(2,\infty)\) containing \(\theta\) and all \(\theta_k\).
For every \(R>1\), the map \((\theta,x)\mapsto p_\theta(x)\) is positive and continuous on \(K\times[1,R]\), so \(p_{\theta_k}(x)\) has a positive lower bound independent of \(k\) on \([1,R]\).
Combining this lower bound with~\eqref{eq:pareto-interior-energy} gives
\begin{equation}\label{eq:pareto-local-H1-bound}
    \sup_k\int_1^R
    \left(
    |\phi_{\theta_k,\lambda}(x)|^2
    +|\phi'_{\theta_k,\lambda}(x)|^2
    \right)\diff x<\infty.
\end{equation}
Thus the solutions are uniformly bounded in \(H^1([1,R])\).
Moreover, \eqref{eq:pareto-interior-equation} gives
\[
    \phi''_{\theta_k,\lambda}(x)
    =
    \frac{\theta_k+1}{x}\phi'_{\theta_k,\lambda}(x)
    +
    \frac{\lambda}{1-\lambda}\phi_{\theta_k,\lambda}(x)
    +
    \frac{\log x-1/\theta_k}{1-\lambda}.
\]
The coefficients \((\theta_k+1)/x\) and \(\lambda/(1-\lambda)\), together with the last term, are uniformly bounded for \(\theta_k\in K\) and \(x\in[1,R]\).
Hence~\eqref{eq:pareto-local-H1-bound} also bounds \(\phi''_{\theta_k,\lambda}\) in \(L^2([1,R])\), giving
\[
    \sup_k\|\phi_{\theta_k,\lambda}\|_{H^2([1,R])}<\infty.
\]

On each fixed interval \([1,R]\), the uniform \(H^2\) bound and the compact embedding \(H^2([1,R])\hookrightarrow C^1([1,R])\) allow us to extract a subsequence along which both the functions and their first derivatives converge uniformly.

We choose a single subsequence for which this convergence holds on every bounded interval as follows.
First extract a subsequence that converges in \(C^1([1,2])\).
From this subsequence, extract a further subsequence that converges in \(C^1([1,3])\); it still converges on \([1,2]\).
Continuing on \([1,4]\), \([1,5]\), and so on gives nested subsequences.
Choose the \(j\)-th element of the \(j\)-th nested subsequence and denote the resulting diagonal subsequence by \(\{\phi_{\theta_{k_j},\lambda}\}\).
For every fixed integer \(m\ge2\), all sufficiently late terms of this diagonal subsequence belong to the subsequence chosen for \([1,m]\).
Thus both the functions and their first derivatives converge uniformly on \([1,m]\).

The limits on overlapping intervals agree and therefore define a single function \(\psi\in C^1([1,\infty))\).
Every interval \([1,R]\) is contained in some \([1,m]\), so
\begin{equation}\label{eq:pareto-local-C1-convergence}
    \phi_{\theta_{k_j},\lambda}\longrightarrow\psi
    \quad\text{in }C^1([1,R])
    \quad\text{for every }R>1.
\end{equation}

To retain the finite-energy condition on the whole half-line, we apply Fatou's lemma and~\eqref{eq:pareto-interior-energy}:
\[
    \begin{aligned}
        &\int_1^\infty
        [\lambda\psi^2+(1-\lambda)|\psi'|^2]p_\theta\diff x\\
        &\quad\le
        \liminf_{j\to\infty}
        \int_1^\infty
        [\lambda\phi_{\theta_{k_j},\lambda}^2
        +(1-\lambda)|\phi'_{\theta_{k_j},\lambda}|^2]
        p_{\theta_{k_j}}\diff x\\
        &\quad\le\frac1{\lambda\theta^2}.
    \end{aligned}
\]
Therefore \(\psi\in V_\theta\).

For \(f\in C_c^\infty(\R)\), the integrals in~\eqref{eq:pareto-variational} involve only a bounded interval.
On that interval, \eqref{eq:pareto-local-C1-convergence} gives uniform convergence of the functions and their first derivatives, while \(p_{\theta_{k_j}}\to p_\theta\) uniformly.
We may therefore pass to the limit in~\eqref{eq:pareto-variational}, obtaining
\begin{equation}\label{eq:pareto-limit-variational}
    \int_1^\infty
    [\lambda\psi f+(1-\lambda)\psi'f']p_\theta\diff x
    =
    \int_1^\infty
    \left(\frac1\theta-\log x\right)f\,p_\theta\diff x.
\end{equation}
By density of the restrictions of \(C_c^\infty(\R)\) in \(V_\theta\), \eqref{eq:pareto-limit-variational} holds for every \(f\in V_\theta\).
The uniqueness of the variational solution in~\eqref{eq:pareto-variational} therefore gives \(\psi=\phi_{\theta,\lambda}\).

The same argument applies to every subsequence of \(\{\theta_k\}\).
Thus every subsequence of \(\{\phi_{\theta_k,\lambda}\}\) has a further subsequence converging to \(\phi_{\theta,\lambda}\), so the full sequence converges in \(C^1_{\mathrm{loc}}([1,\infty))\).
This proves continuity of \(\theta\mapsto\phi_{\theta,\lambda}\) in that space.
The constant extensions to \(x<1\), together with \(\phi'_{\theta,\lambda}(1)=0\), then give jointly continuous representatives of the functions and their spatial derivatives.
In particular, these representatives are jointly measurable.

Together with the explicit endpoint formulas, this gives jointly measurable representatives of the functions and their gradients for every fixed \(0\le\lambda\le1\).
Equations~\eqref{eq:pareto-endpoint-energy} and~\eqref{eq:pareto-interior-energy} show that the energy is locally bounded in \(\theta\), verifying condition (i) of Definition~\ref{def:HK_regular}.

For every parameter path $\theta+th$ contained in $(2,\infty)$, set $\phi_t=h\phi_{\theta+th,\lambda}$, with $0\le\lambda\le1$.
The verified weak equations and \(\partial_t p_{\theta+th}=h\partial_\theta p_{\theta+th}\) show that \(\phi_t\) solves the transport-source equation~\eqref{HK-cont-eq} along the parameter path.
The local energy bounds give
\[
    \int_0^1
    \|\phi_t\|_{\mathcal H_\lambda^1(P_{\theta+th})}^2
    \diff t<\infty.
\]
Proposition~\ref{prop:Sobolev} and Proposition~\ref{prop:ans-cont-unique-solution} give uniqueness up to the \(\mathcal H_\lambda^1\) norm, verifying (ii).
The form \(\phi_t=h\phi_{\theta+th,\lambda}\) then verifies (iii).
Hence the Pareto family is \(\HK_\lambda\) regular for all \(0\le\lambda\le1\), and in particular is \(\HK\) regular.

Finally, $\phi_{\theta,0}(x)\sim-x^2\log x/(2\theta)$ as $x\to\infty$, so $\Ebb_\theta[\phi_{\theta,0}(X)^2]<\infty$ if and only if $\theta>4$.

\paragraph{Remaining assumptions of Theorem~\ref{thm:HK-CR-bound}.}
At the endpoints, direct integration of~\eqref{eq:phi_pareto_potentials} gives
\begin{align*}
    \Ebb_\theta[\phi_{\theta,0}(X)^2]
    &=
    \frac{5\theta^2-28\theta+40}
    {\theta(\theta-2)^4(\theta-4)^3},
    &&\theta>4,\\
    \Ebb_\theta[|\phi'_{\theta,1}(X)|^2]
    &=\frac{\theta}{\theta+2},
    &&\theta>2.
\end{align*}
Together with~\eqref{eq:pareto-endpoint-energy}, these expressions verify the continuity assumptions of Theorem~\ref{thm:HK-CR-bound} on \((4,\infty)\) when \(\lambda=0\), and on \((2,\infty)\) when \(\lambda=1\).

Now fix \(0<\lambda<1\), let \(\theta_k\to\theta>2\), and choose \(2<a<b\) such that \(\theta,\theta_k\in[a,b]\) for every \(k\).
For every fixed \(R>1\), the continuity in \(\theta\) established above gives
\begin{equation}\label{eq:pareto-local-uniform-convergence}
    \begin{aligned}
        \sup_{x\in[1,R]}
        |\phi_{\theta_k,\lambda}(x)-\phi_{\theta,\lambda}(x)|
        &\longrightarrow0,\\
        \sup_{x\in[1,R]}
        |\phi'_{\theta_k,\lambda}(x)-\phi'_{\theta,\lambda}(x)|
        &\longrightarrow0
    \end{aligned}
\end{equation}
as \(k\to\infty\).
Thus both the functions and their first derivatives converge uniformly on each fixed bounded interval.
To prove continuity of their expectations over \([1,\infty)\), we first establish convergence of the total energy.

Taking \(f=\phi_{\theta_k,\lambda}\) in~\eqref{eq:pareto-variational} gives
\begin{equation}\label{eq:pareto-energy-identity-continuity}
    \|\phi_{\theta_k,\lambda}\|_{\mathcal H_\lambda^1(P_{\theta_k})}^2
    =
    \int_1^\infty
    \left(\frac1{\theta_k}-\log x\right)
    \phi_{\theta_k,\lambda}(x)p_{\theta_k}(x)\diff x.
\end{equation}
For every fixed \(R>1\), uniform convergence on \([1,R]\) gives
\[
    \begin{aligned}
        &\int_1^R
        \left(\frac1{\theta_k}-\log x\right)
        \phi_{\theta_k,\lambda}(x)p_{\theta_k}(x)\diff x\\
        &\qquad\longrightarrow
        \int_1^R
        \left(\frac1\theta-\log x\right)
        \phi_{\theta,\lambda}(x)p_\theta(x)\diff x.
    \end{aligned}
\]
It remains to control the integrals over \((R,\infty)\), uniformly in \(k\).

The energy bound~\eqref{eq:pareto-interior-energy} implies
\[
    \lambda\int_1^\infty
    \phi_{\theta_k,\lambda}(x)^2p_{\theta_k}(x)\diff x
    \le
    \|\phi_{\theta_k,\lambda}\|_{\mathcal H_\lambda^1(P_{\theta_k})}^2
    \le\frac1{\lambda a^2},
\]
so \(\|\phi_{\theta_k,\lambda}\|_{L^2(P_{\theta_k})}\le1/(\lambda a)\).
Also, \(p_{\theta_k}(x)=\theta_kx^{-\theta_k-1}\le bx^{-a-1}\) for \(x\ge1\).
Cauchy--Schwarz therefore gives
\begin{equation}\label{eq:pareto-tail-control}
    \begin{aligned}
        &\left|
        \int_R^\infty
        \left(\frac1{\theta_k}-\log x\right)
        \phi_{\theta_k,\lambda}(x)p_{\theta_k}(x)\diff x
        \right|\\
        &\quad\le
        \left(
        \int_R^\infty
        \phi_{\theta_k,\lambda}(x)^2p_{\theta_k}(x)\diff x
        \right)^{1/2}
        \left(
        \int_R^\infty
        \left(\frac1{\theta_k}-\log x\right)^2
        p_{\theta_k}(x)\diff x
        \right)^{1/2}\\
        &\quad\le
        \frac1{\lambda a}
        \left[
        b\int_R^\infty
        \left(\frac1a+\log x\right)^2x^{-a-1}\diff x
        \right]^{1/2}.
    \end{aligned}
\end{equation}
The bound in~\eqref{eq:pareto-tail-control} is independent of \(k\) and tends to zero as \(R\to\infty\).
Since \(\theta\in[a,b]\), the same calculation also bounds the tail integral with \(\theta_k\) replaced by \(\theta\).

Let \(\varepsilon>0\).
Choose \(R\) large enough that each of these tail integrals has absolute value less than \(\varepsilon/3\), uniformly in \(k\).
Keep this \(R\) fixed.
On \([1,R]\), \(\phi_{\theta_k,\lambda}\to\phi_{\theta,\lambda}\) and \(p_{\theta_k}\to p_\theta\) uniformly, while \(1/\theta_k\to1/\theta\).
Therefore, for all sufficiently large \(k\),
\[
    \left|
    \int_1^R
    \left[
    \left(\frac1{\theta_k}-\log x\right)
    \phi_{\theta_k,\lambda}p_{\theta_k}
    -
    \left(\frac1\theta-\log x\right)
    \phi_{\theta,\lambda}p_\theta
    \right]\diff x
    \right|<\frac{\varepsilon}{3}.
\]
Splitting the integrals in~\eqref{eq:pareto-energy-identity-continuity} at \(R\) and applying the triangle inequality gives
\begin{equation}\label{eq:pareto-energy-splitting}
    \begin{aligned}
        &\left|
        \|\phi_{\theta_k,\lambda}\|_{\mathcal H_\lambda^1(P_{\theta_k})}^2
        -
        \|\phi_{\theta,\lambda}\|_{\mathcal H_\lambda^1(P_\theta)}^2
        \right|\\
        &\quad\le
        \left|
        \int_1^R
        \left[
        \left(\frac1{\theta_k}-\log x\right)
        \phi_{\theta_k,\lambda}p_{\theta_k}
        -
        \left(\frac1\theta-\log x\right)
        \phi_{\theta,\lambda}p_\theta
        \right]\diff x
        \right|\\
        &\qquad+
        \left|
        \int_R^\infty
        \left(\frac1{\theta_k}-\log x\right)
        \phi_{\theta_k,\lambda}p_{\theta_k}\diff x
        \right|\\
        &\qquad+
        \left|
        \int_R^\infty
        \left(\frac1\theta-\log x\right)
        \phi_{\theta,\lambda}p_\theta\diff x
        \right|
        <\varepsilon.
    \end{aligned}
\end{equation}
Since \(\varepsilon\) is arbitrary, we obtain
\begin{equation}\label{eq:pareto-energy-continuity}
    \|\phi_{\theta_k,\lambda}\|_{\mathcal H_\lambda^1(P_{\theta_k})}^2
    \longrightarrow
    \|\phi_{\theta,\lambda}\|_{\mathcal H_\lambda^1(P_\theta)}^2.
\end{equation}

We now prove convergence of the function and derivative expectations separately.
The uniform convergence in~\eqref{eq:pareto-local-uniform-convergence}, together with \(p_{\theta_k}\to p_\theta\), gives pointwise convergence of the nonnegative functions \(\phi_{\theta_k,\lambda}^2p_{\theta_k}\) and \(|\phi'_{\theta_k,\lambda}|^2p_{\theta_k}\).
Applying Fatou's lemma with respect to Lebesgue measure on \([1,\infty)\) gives
\begin{equation}\label{eq:pareto-moment-liminf}
    \begin{aligned}
        \liminf_k\Ebb_{\theta_k}[\phi_{\theta_k,\lambda}(X)^2]
        &\ge\Ebb_\theta[\phi_{\theta,\lambda}(X)^2],\\
        \liminf_k\Ebb_{\theta_k}[|\phi'_{\theta_k,\lambda}(X)|^2]
        &\ge\Ebb_\theta[|\phi'_{\theta,\lambda}(X)|^2].
    \end{aligned}
\end{equation}
These inequalities establish lower semicontinuity of \(\theta\mapsto\Ebb_\theta[\phi_{\theta,\lambda}(X)^2]\) and \(\theta\mapsto\Ebb_\theta[|\phi'_{\theta,\lambda}(X)|^2]\).
The energy decomposes as
\[
    \begin{aligned}
        \|\phi_{\theta_k,\lambda}\|_{\mathcal H_\lambda^1(P_{\theta_k})}^2
        &=
        \lambda\Ebb_{\theta_k}[\phi_{\theta_k,\lambda}(X)^2]\\
        &\quad+(1-\lambda)
        \Ebb_{\theta_k}[|\phi'_{\theta_k,\lambda}(X)|^2].
    \end{aligned}
\]
Solving this identity for the function contribution and using~\eqref{eq:pareto-energy-continuity} yields
\begin{equation}\label{eq:pareto-moment-limsup}
    \begin{aligned}
        \lambda\limsup_k
        \Ebb_{\theta_k}[\phi_{\theta_k,\lambda}(X)^2]
        &=
        \|\phi_{\theta,\lambda}\|_{\mathcal H_\lambda^1(P_\theta)}^2
        -(1-\lambda)\liminf_k
        \Ebb_{\theta_k}[|\phi'_{\theta_k,\lambda}(X)|^2]\\
        &\le
        \|\phi_{\theta,\lambda}\|_{\mathcal H_\lambda^1(P_\theta)}^2
        -(1-\lambda)\Ebb_\theta[|\phi'_{\theta,\lambda}(X)|^2]\\
        &=\lambda\Ebb_\theta[\phi_{\theta,\lambda}(X)^2].
    \end{aligned}
\end{equation}
The same calculation for the derivative contribution gives
\[
    \begin{aligned}
        (1-\lambda)\limsup_k
        \Ebb_{\theta_k}[|\phi'_{\theta_k,\lambda}(X)|^2]
        &=
        \|\phi_{\theta,\lambda}\|_{\mathcal H_\lambda^1(P_\theta)}^2
        -\lambda\liminf_k
        \Ebb_{\theta_k}[\phi_{\theta_k,\lambda}(X)^2]\\
        &\le
        \|\phi_{\theta,\lambda}\|_{\mathcal H_\lambda^1(P_\theta)}^2
        -\lambda\Ebb_\theta[\phi_{\theta,\lambda}(X)^2]\\
        &=(1-\lambda)\Ebb_\theta[|\phi'_{\theta,\lambda}(X)|^2].
    \end{aligned}
\]
Since \(\lambda>0\) and \(1-\lambda>0\), dividing the limit-superior estimates by these constants proves upper semicontinuity of \(\theta\mapsto\Ebb_\theta[\phi_{\theta,\lambda}(X)^2]\) and \(\theta\mapsto\Ebb_\theta[|\phi'_{\theta,\lambda}(X)|^2]\).
Together with the lower semicontinuity established in~\eqref{eq:pareto-moment-liminf}, this proves continuity.
In particular,
\[
    \begin{aligned}
        \Ebb_{\theta_k}[\phi_{\theta_k,\lambda}(X)^2]
        &\longrightarrow\Ebb_\theta[\phi_{\theta,\lambda}(X)^2],\\
        \Ebb_{\theta_k}[|\phi'_{\theta_k,\lambda}(X)|^2]
        &\longrightarrow\Ebb_\theta[|\phi'_{\theta,\lambda}(X)|^2].
    \end{aligned}
\]
Thus the continuity assumptions of Theorem~\ref{thm:HK-CR-bound} hold for each fixed \(0<\lambda<1\).

\section{Proofs of Main Results}

\subsection{Proofs from Section~\ref{sec:HK-background}}
\label{app:ProofsSection2}

Recall in Definition~\ref{def:dynamic-cstr} that our dynamical formulation of the metric $\dHKlambda$ explicitly required $(\xi_t,v_t)$ to lie in the $\HK_{\lambda}$ tangent space of $\mu_t$ for Lebesgue a.e. $0\le t\le 1$.
As we state next, the value of the metric stays the same even if we remove this constraint; this unconstrained representation is useful in some of our proofs below.
For $0<\lambda<1$, the result follows from \cite[Proposition~3.1]{Fleissner} and \cite[Theorem~8.17]{LieroMielkeSavare}.
For $\lambda=0$, it follows from \cite[Theorem~8.3.1 and Lemma~8.4.2]{AmbrosioGigliSavare}.
For $\lambda=1$, finiteness of the integral gives $\xi_t\in L^2(\mu_t)$ for almost every $t$, so the tangent space constraint is automatic by the density of $C_c^\infty(\R^m)$ in $L^2(\mu_t)$.

\begin{lemma}\label{lem:dynamic-free}
    We have
    \begin{equation*}
        \dHKlambda^2(\nu_0,\nu_1):= \inf\left\{\int_{0}^{1}\int_{\mathbb{R}^m}\left(\lambda|\xi_t|^2 +(1-\lambda)\|v_t\|^2\right)\diff \mu_t\diff t: \begin{matrix}
            \{(\mu_t,\xi_t,v_t)\}_{0\le t\le 1}\textnormal{ solves \eqref{eqn:HK-differential},} \\ \textnormal{and }\mu_0=\nu_0,\mu_1=\nu_1
        \end{matrix}\right\}
    \end{equation*}
    for all $\nu_0,\nu_1$.
\end{lemma}

\begin{proof}[Proof of Lemma~\ref{lem:metric-bounds}]
    The first claim is precisely the Benamou--Brenier theorem, and the second claim follows from the dynamic formulation of Hellinger metric \cite[Equation (2.11)]{MielkeZhu2025HKGradientFlows}. 
    For the last claim, consider setting $v_t \equiv 0$ in the definition of $\dHKlambda$ to get:
    \begin{align*}
        \dHKlambda^2(\nu_0,\nu_1)&= \inf\left\{\int_{0}^{1}\int_{\mathbb{R}^d}\left(\lambda|\xi_t|^2+(1-\lambda)\|v_t\|^2\right)\diff\mu_t\diff t: \begin{matrix}
            \{(\mu_t,\xi_t,v_t)\}_{0\le t\le 1}\textnormal{ solves \eqref{eqn:HK-differential},} \\ \textnormal{and }\mu_0=\nu_0,\mu_1=\nu_1
        \end{matrix}\right\} \\
        &\le \inf\left\{\lambda\int_{0}^{1}\int_{\mathbb{R}^d}|\xi_t|^2\diff \mu_t\diff t: \begin{matrix}
            \{(\mu_t,\xi_t,0)\}_{0\le t\le 1}\textnormal{ solves \eqref{eqn:HK-differential},} \\ \textnormal{and }\mu_0=\nu_0,\mu_1=\nu_1
        \end{matrix}\right\} \\
        &= \inf\left\{\lambda\int_{0}^{1}\int_{\mathbb{R}^d}|\xi_t|^2\diff \mu_t\diff t: \begin{matrix}
            \{(\mu_t,\lambda\xi_t,0)\}_{0\le t\le 1}\textnormal{ solves \eqref{eqn:RE-differential}} \\ \textnormal{and }\mu_0=\nu_0,\mu_1=\nu_1
        \end{matrix}\right\} \\
        &= \frac{1}{\lambda}\inf\left\{\int_{0}^{1}\int_{\mathbb{R}^d}|\lambda\xi_t|^2\diff \mu_t\diff t: \begin{matrix}
            \{(\mu_t,\lambda\xi_t,0)\}_{0\le t\le 1}\textnormal{ solves \eqref{eqn:RE-differential}} \\ \textnormal{and }\mu_0=\nu_0,\mu_1=\nu_1
        \end{matrix}\right\} \\
        &= \frac{1}{\lambda}\dH^2(\nu_0,\nu_1),
    \end{align*}
    where in the last line we used $d_{\HK_1} = \dH$.
    Similarly, consider setting $\xi_t \equiv 0$ to get
    \begin{align*}
        \dHKlambda^2(\nu_0,\nu_1)&= \inf\left\{\int_{0}^{1}\int_{\mathbb{R}^d}\left(\lambda|\xi_t|^2+(1-\lambda)\|v_t\|^2\right)\diff \mu_t\diff t: \begin{matrix}
            \{(\mu_t,\xi_t,v_t)\}_{0\le t\le 1}\textnormal{ solves \eqref{eqn:HK-differential},} \\ \textnormal{and }\mu_0=\nu_0,\mu_1=\nu_1
        \end{matrix}\right\} \\
        &\le \inf\left\{(1-\lambda)\int_{0}^{1}\int_{\mathbb{R}^d}\|v_t\|^2\diff \mu_t\diff t: \begin{matrix}
            \{(\mu_t,0,v_t)\}_{0\le t\le 1}\textnormal{ solves \eqref{eqn:HK-differential},} \\ \textnormal{and }\mu_0=\nu_0,\mu_1=\nu_1
        \end{matrix}\right\} \\
        &= \inf\left\{(1-\lambda)\int_{0}^{1}\int_{\mathbb{R}^d}\|v_t\|^2\diff \mu_t \diff t: \begin{matrix}
            \{(\mu_t,0,(1-\lambda) v_t)\}_{0\le t\le 1}\textnormal{ solves \eqref{eqn:CE-differential}} \\ \textnormal{and }\mu_0=\nu_0,\mu_1=\nu_1
        \end{matrix}\right\} \\
        &= \frac{1}{1-\lambda}\inf\left\{\int_{0}^{1}\int_{\mathbb{R}^d}\|(1-\lambda)v_t\|^2\diff \mu_t\diff t: \begin{matrix}
            \{(\mu_t,(1-\lambda)v_t)\}_{0\le t\le 1}\textnormal{ solves \eqref{eqn:CE-differential}} \\ \textnormal{and }\mu_0=\nu_0,\mu_1=\nu_1
        \end{matrix}\right\} \\
        &= \frac{1}{1-\lambda}\dW^2(\nu_0,\nu_1)
    \end{align*}
    where in the last line we used $d_{\HK_0} = \dW$.
\end{proof}

\begin{proof}[Proof of Proposition~\ref{prop:abs-cont-HK}]
    Suppose that $\{\mu_t\}_{0\le t\le 1}$ is absolutely continuous with respect to $\dH$.
    By definition there exists a function $g\in L^2[0,1]$ such that we have
    \begin{equation*}
        \dH(\mu_s,\mu_t) \le \int_{s}^{t}g(u)\diff u
    \end{equation*}
    for all $0\le s\le t\le 1$.
    Now use Lemma~\ref{lem:metric-bounds} to get $\dHKlambda \le \frac{1}{\sqrt{\lambda}}\dH$ so that the above displays becomes
    \begin{equation*}
        \dHKlambda(\mu_s,\mu_t) \le \int_{s}^{t}\frac{g(u)}{\sqrt{\lambda}}\diff u.
    \end{equation*}
    Since \(g/\sqrt{\lambda}\in L^2[0,1]\), this proves that the curve is \(2\)-absolutely continuous with respect to \(d_{\HK_\lambda}\).
    The proof for $\dW$ is identical, using
    \[
    d_{\HK_\lambda}\le \frac{1}{\sqrt{1-\lambda}}\,\dW.
    \]
\end{proof}

\subsection{Proofs from Section~\ref{sec:LB}}

\begin{proof}[Proof of Proposition~\ref{prop:HK-reg-suff}]
    We focus on DQM, and the proof is identical for DWS.
    So, suppose $\Pcal$ is DQM and that its score function $\dot{\ell}_{\theta}$ is such that $\theta\mapsto \|\dot{\ell}_{\theta}\|^2_{L^2(P_{\theta})}$ is locally integrable.
    Then, Proposition~\ref{prop:ans-cont-unique-solution} shows that $\{P_{\theta+th}\}_{0\le t\le 1}$ is absolutely continuous with respect to $d_{\HK_1}=\dH$, hence Proposition~\ref{prop:abs-cont-HK} shows that it is also absolutely continuous with respect to $d_{\HK_{\lambda}}$ for all $0<\lambda<1$.
    Thus, applying Proposition~\ref{prop:ans-cont-unique-solution} in the reverse direction implies the uniqueness.
\end{proof}

\begin{proof}[Proof of Theorem~\ref{thm:HK-CR-bound}]
    Fix $\theta\in\Theta$ such that $I_\lambda(\theta)$ is invertible,
    and let $a,b\in\R^p$ be arbitrary. Since $\Theta$ is open, there
    exists $\varepsilon>0$ such that
    $\theta_t:=\theta+tb\in\Theta$ for all $0\leq t\leq\varepsilon$.

    By $\HK_\lambda$ regularity, the continuity assumptions, and tensorization, the product curve
    $t\mapsto P_{\theta_t}^{\otimes n}$ is in $\AC_\lambda^2$ and has
    product potential
    \[
        (x_1,\ldots,x_n)
        \longmapsto
        \sum_{i=1}^n
        b^\top\phi_{\theta_t,\lambda}(x_i).
    \]
    By the continuity assumptions, $t\mapsto\|b^\top\phi_{\theta_t,\lambda}\|_{\mathcal H_\lambda^1(P_{\theta_t})}^2$ is locally integrable. These assumptions also ensure that the locally Lipschitz function $a^\top T_n$ satisfies the integrability conditions of Theorem~\ref{thm:loc-Lip-test-funcs}.

    Writing $X_1,\ldots,X_n\stackrel{\mathrm{iid}}{\sim}P_{\theta_t}$,
    unbiasedness gives
    $\Ebb_{\theta_t}[T_n]=\theta_t$, while the centering convention from Remark~\ref{remark:center-conv} gives $\Ebb_{\theta_t}[\phi_{\theta_t,\lambda}(X)]=0$. Hence, for almost
    every $t\in(0,\varepsilon)$,
    \begin{equation*}
        a^\top b
        =
        \lambda
        \Ebb_{\theta_t}
        \left[
            a^\top(T_n-\theta_t)
            \sum_{i=1}^n
            b^\top\phi_{\theta_t,\lambda}(X_i)
        \right] +
        (1-\lambda)
        \Ebb_{\theta_t}
        \left[
            \sum_{i=1}^n
            a^\top D_{x_i}T_n
            \left(D\phi_{\theta_t,\lambda}(X_i)\right)^\top b
        \right].
    \end{equation*}
    By weighted Cauchy--Schwarz, using only the gradient terms when $\lambda=0$ and only the function terms when $\lambda=1$,
    \begin{align*}
        (a^\top b)^2
        &\leq
        \left\{
            \lambda
            \Ebb_{\theta_t}
            \left[
                \big(a^\top(T_n-\theta_t)\big)^2
            \right]
            +
            (1-\lambda)
            \Ebb_{\theta_t}
            \left[
                \sum_{i=1}^n
                \|a^\top D_{x_i}T_n\|^2
            \right]
        \right\} \\
        &\quad\times
        \left\{
            \lambda
            \Ebb_{\theta_t}
            \left[
                \left(
                    \sum_{i=1}^n
                    b^\top\phi_{\theta_t,\lambda}(X_i)
                \right)^2
            \right]
            +
            (1-\lambda)
            \Ebb_{\theta_t}
            \left[
                \sum_{i=1}^n
                \left\|
                    \left(D\phi_{\theta_t,\lambda}(X_i)\right)^\top b
                \right\|^2
            \right]
        \right\}.
    \end{align*}
    The first bracket is
    \[
        a^\top
        \left[
            \lambda\Var_{\theta_t}(T_n)
            +(1-\lambda)\Sen_{\theta_t}(T_n)
        \right]a.
    \]
    By independence of the $X_i$'s and the centering of
    $\phi_{\theta_t,\lambda}$, the second bracket is
    $n\,b^\top I_\lambda(\theta_t)b$. Therefore, for Lebesgue almost every $t\in(0,\varepsilon)$,
    \[
        (a^\top b)^2
        \leq
        a^\top
        \left[
            \lambda\Var_{\theta_t}(T_n)
            +(1-\lambda)\Sen_{\theta_t}(T_n)
        \right]a
        \cdot
        n\,b^\top I_\lambda(\theta_t)b.
    \]

    Choose a sequence $t_k\downarrow0$ from the full Lebesgue measure set on which the preceding inequality holds. By the continuity assumptions,
    letting $k\to\infty$ gives
    \[
        (a^\top b)^2
        \leq
        a^\top
        \left[
            \lambda\Var_\theta(T_n)
            +(1-\lambda)\Sen_\theta(T_n)
        \right]a
        \cdot
        n\,b^\top I_\lambda(\theta)b.
    \]
    Taking $b=I_\lambda(\theta)^{-1}a$ yields
    \[
        \left(
            a^\top I_\lambda(\theta)^{-1}a
        \right)^2
        \leq
        a^\top
        \left[
            \lambda\Var_\theta(T_n)
            +(1-\lambda)\Sen_\theta(T_n)
        \right]a
        \cdot
        n\,a^\top I_\lambda(\theta)^{-1}a.
    \]
    For $a\neq0$, the invertibility of $I_\lambda(\theta)$ implies
    $a^\top I_\lambda(\theta)^{-1}a>0$, while the case $a=0$ is immediate.
    Thus, for every $a\in\R^p$,
    \[
        a^\top
        \left[
            \lambda\Var_\theta(T_n)
            +(1-\lambda)\Sen_\theta(T_n)
        \right]a
        \geq
        \frac1n
        a^\top I_\lambda(\theta)^{-1}a.
    \]
    Hence
    \[
        \lambda\Var_\theta(T_n)
        +(1-\lambda)\Sen_\theta(T_n)
        \succeq
        \frac1n I_\lambda(\theta)^{-1}.
    \]

    Finally, by the definitions of $I_\lambda(\theta)$,
    $V_\lambda(\theta)$, and $S_\lambda(\theta)$,
    \[
        \lambda V_\lambda(\theta)
        +(1-\lambda)S_\lambda(\theta)
        =
        I_\lambda(\theta)^{-1}.
    \]
    Therefore,
    \[
        \lambda\Var_\theta(T_n)
        +(1-\lambda)\Sen_\theta(T_n)
        \succeq
        \frac1n
        \left[
            \lambda V_\lambda(\theta)
            +(1-\lambda)S_\lambda(\theta)
        \right].
    \]
     Multiplying by $n$ and rearranging gives
    \[
        (1-\lambda)
        \left(
            n\Sen_\theta(T_n)-S_\lambda(\theta)
        \right)
        \succeq
        \lambda
        \left(
            V_\lambda(\theta)-n\Var_\theta(T_n)
        \right).
    \]
    If $n\Var_\theta(T_n)\preceq V_\lambda(\theta)$, then the right-hand side is positive semidefinite. Since $1-\lambda>0$, it follows that
    \[
        n\Sen_\theta(T_n)\succeq S_\lambda(\theta).
    \]
    The other implication is proved similarly.
\end{proof}

\begin{proof}[Proof of Corollary~\ref{cor:HK-CR-function}]
    Fix $\theta\in\Theta$ such that $I_\lambda(\theta)$ is invertible.
    For $a\in\R^q$ and $b\in\R^p$, repeat the argument in the proof of Theorem~\ref{thm:HK-CR-bound} along $\theta_t=\theta+tb$, now using $\frac{\diff}{\diff t}\Ebb_{\theta_t}[a^\top T_n]=a^\top D\chi(\theta_t)b$.
    By the continuity assumptions and $\chi\in C^1$, passing to $t=0$ gives
    \[
        \bigl(a^\top D\chi(\theta)b\bigr)^2
        \le
        n\,a^\top\left[
            \lambda\Var_\theta(T_n)+(1-\lambda)\Sen_\theta(T_n)
        \right]a\,b^\top I_\lambda(\theta)b.
    \]
    If $D\chi(\theta)^\top a\ne0$, take $b=I_\lambda(\theta)^{-1}D\chi(\theta)^\top a$ and divide by $a^\top D\chi(\theta)I_\lambda(\theta)^{-1}D\chi(\theta)^\top a>0$.
    If $D\chi(\theta)^\top a=0$, the required inequality follows from positive semidefiniteness of the variance and sensitivity matrices.
    Since $a$ is arbitrary, this proves~\eqref{eq:HK-CR-function}.
\end{proof}

\begin{proof}[Proof of Corollary~\ref{cor:HK-CR-asymptotic}]
    Since $T_n$ is unbiased for $\chi_n(\theta)$, Corollary~\ref{cor:HK-CR-function} gives
    \[
        n\left[\lambda\Var_\theta(T_n)+(1-\lambda)\Sen_\theta(T_n)\right]
        \succeq
        D\chi_n(\theta)I_\lambda(\theta)^{-1}D\chi_n(\theta)^\top.
    \]
    Taking the limit and using $D\chi_n(\theta)\to D\chi(\theta)$ proves~\eqref{eq:HK-CR-asymptotic}.
\end{proof}

\begin{proposition}\label{prop:Poincare-diagram}
    If $\Pcal=\{P_{\theta}:\theta\in\Theta\}\subseteq\Pcal_{2,\textnormal{ac}}(\R^d)$ has the property that $P_{\theta}$ satisfies the dimension-free Poincar\'e inequality with constant $C_{\theta}>0$, then we have $\diagram(\theta)\subseteq\{(v,s):s\ge v/C_{\theta}\}$.
\end{proposition}

\begin{proof}
    If $(v,s)\in \diagram(\theta)$ then there exists a subsequence $\{n_k\}_{k\in\N}$ and an unbiased estimator $T_k:(\R^d)^{n_k}\to \R$ for each $k\in\N$ such that we have $n_k\Var_{\theta}(T_k)\to v$ and $n_k\Sen_{\theta}(T_k) \to s$ as $k\to\infty$.
    Applying the Poincar\'e inequality to $T_k$ yields $n_k\Var_{\theta}(T_k)\le C_{\theta}n_k\Sen_{\theta}(T_k)$ for each $k\in\N$, so sending $k\to\infty$ yields $v\le C_{\theta}s$ as needed.
\end{proof}

\begin{proof}[Proof of Lemma~\ref{lem:Info-cvx}]
    Fix $\theta\in\Theta$, and define the linear functional
    $\ell_{\theta}:C_c^{\infty}(\R^d)\to\R$ via $\ell_{\theta}(\phi):=\frac{\diff}{\diff\theta}
    \Ebb_{\theta}[\phi(X)]$.
    Using the transport-source equation, Cauchy--Schwarz,
    and the elementary inequality $\lambda\sqrt{ab}+(1-\lambda)\sqrt{AB}\le\sqrt{(\lambda a+
    (1-\lambda)A)(\lambda b+(1-\lambda)B)}$ for $a,b,A,B\ge0$, one computes for $\phi\in
    C_c^{\infty}(\R^d)$ that
    \begin{equation}\label{eqn:op-norm-calc}
        \begin{split}
            |\ell_{\theta}(\phi)| &= \left|\frac{\diff}{\diff \theta}\Ebb_{\theta}[\phi(X)]\right| \\
        &= \left|\lambda\int_{\R^d}\phi_{\theta,\lambda}\phi\diff P_{\theta} + (1-\lambda)\int_{\R^d}(\nabla\phi_{\theta,\lambda})^{\top}\nabla\phi\diff P_{\theta}\right| \\
        &\le \lambda\int_{\R^d}|\phi_{\theta,\lambda}\phi|\diff P_{\theta} + (1-\lambda)\int_{\R^d}|(\nabla\phi_{\theta,\lambda})^{\top}\nabla\phi|\diff P_{\theta} \\
        &\le \lambda\sqrt{\int_{\R^d}|\phi_{\theta,\lambda}|^2\diff P_{\theta}\int_{\R^d}|\phi|^2\diff P_{\theta}}+ (1-\lambda)\sqrt{\int_{\R^d}\|\nabla\phi_{\theta,\lambda}\|^2\diff P_{\theta} \int_{\R^d}\|\nabla\phi\|^2\diff P_{\theta}}\\        &\le\|\phi_{\theta,\lambda}\|_{\mathcal{H}^1_{\lambda}(P_{\theta})}\|\phi\|_{\mathcal{H}^1_{\lambda}(P_{\theta})}.
        \end{split}
    \end{equation}
    This shows that $\ell_{\theta}$ is a bounded linear functional on
    $C_c^{\infty}(\R^d)$ with respect to $\mathcal{H}^1_{\lambda}(P_{\theta})$, and hence it admits a unique continuous extension to $\mathcal{H}^1_{\lambda}(P_{\theta})$, since $C_c^{\infty}(\R^d)$ is
    dense in $\mathcal{H}^1_{\lambda}(P_{\theta})$.

    We first show that $\lambda\mapsto I_{\lambda}(\theta)$ is convex.
    To do this, define $Q_{\theta,\lambda}:\mathcal{H}^1_{\lambda}(P_{\theta})
    \to\R$ as the unique continuous extension of
    \begin{equation*}
        Q_{\theta,\lambda}(\phi):=2\frac{\partial}{\partial\theta}\Ebb_{\theta}[\phi]
        -\lambda\int_{\R^d}|\phi|^2\diff P_{\theta}
        -(1-\lambda)\int_{\R^d}\|\nabla\phi\|^2\diff P_{\theta},
        \qquad \phi\in C_c^{\infty}(\R^d).
    \end{equation*}
    By the transport-source equation and continuity, for every
    $\phi\in\mathcal{H}^1_{\lambda}(P_{\theta})$ we have
    \begin{equation}\label{eqn:Q-completed-square}
        Q_{\theta,\lambda}(\phi)
        =I_\lambda(\theta)
        -\|\phi-\phi_{\theta,\lambda}\|_{\mathcal H^1_\lambda(P_\theta)}^2.
    \end{equation}
    Since $\phi_{\theta,\lambda}\in\mathcal H^1_\lambda(P_\theta)$ by $\HK$ regularity,
    equation~\eqref{eqn:Q-completed-square} shows that it is the unique maximizer and
    $Q_{\theta,\lambda}(\phi_{\theta,\lambda})=I_\lambda(\theta)$.
    By density of $C_c^\infty(\R^d)$ in $\mathcal H^1_\lambda(P_\theta)$, it follows that
    \begin{equation}\label{eqn:I-sup}
        I_{\lambda}(\theta)
        =
        \sup\left\{2\frac{\partial}{\partial\theta}\Ebb_{\theta}[\phi]
        -\lambda\int_{\R^d}|\phi|^2\diff P_{\theta}
        -(1-\lambda)\int_{\R^d}\|\nabla\phi\|^2\diff P_{\theta}:\phi\in C_c^{\infty}(\R^d)\right\}.
    \end{equation}
    The right-hand side is a supremum of functions that are affine in $\lambda$, whence convexity.

    We next show that $\lambda\mapsto(I_{\lambda}(\theta))^{-1}$ is concave; this is not automatic,
    since the reciprocal of a general convex function need not be concave, and it relies instead on a
    dual variational formula. From the calculation~\eqref{eqn:op-norm-calc} and Cauchy--Schwarz, we have
    $$\ell_{\theta}(\phi)^2\le I_{\lambda}(\theta)\|\phi\|^2_{\mathcal{H}^1_{\lambda}(P_{\theta})}$$ for
    all $\phi\in\mathcal H^1_\lambda(P_\theta)$, with equality if and only if
    $\phi=c\phi_{\theta,\lambda}$ in $\mathcal H^1_\lambda(P_\theta)$ for some $c\in\R$.
    In particular, $\phi_{\theta,\lambda}/I_\lambda(\theta)$ satisfies
    $\ell_\theta(\phi_{\theta,\lambda}/I_\lambda(\theta))=1$
    and has squared norm $(I_\lambda(\theta))^{-1}$.
    Choose $\phi_k\in C_c^\infty(\R^d)$ converging to $\phi_{\theta,\lambda}$
    in $\mathcal H^1_\lambda(P_\theta)$. Then
    $\ell_\theta(\phi_k)\to I_\lambda(\theta)>0$, so
    $\phi_k/\ell_\theta(\phi_k)$ is well defined for all sufficiently large $k$,
    satisfies the same constraint, and has squared norm converging to
    $(I_\lambda(\theta))^{-1}$.
    Together with \eqref{eqn:op-norm-calc}, this proves
    \begin{equation}\label{eqn:Minv-inf}
        (I_{\lambda}(\theta))^{-1}
        =
        \inf\Big\{\|\phi\|^2_{\mathcal{H}^1_{\lambda}(P_{\theta})}
        :\phi\in C_c^{\infty}(\R^d),\ \ell_{\theta}(\phi)=1\Big\}.
    \end{equation}
    The feasible set $\{\phi\in C_c^{\infty}(\R^d):\ell_{\theta}(\phi)=1\}$ does not depend on
    $\lambda$, while for each fixed feasible $\phi$ the value $$\|\phi\|^2_{\mathcal{H}^1_{\lambda}
    (P_{\theta})}=\lambda\int|\phi|^2\diff P_{\theta}+(1-\lambda)\int\|\nabla\phi\|^2\diff P_{\theta}$$
    is a function affine in $\lambda$. Thus the right-hand side of \eqref{eqn:Minv-inf} is an infimum of functions
    affine in $\lambda$, whence concavity. 
    The differentiability claim follows from the fact that convex and concave functions in one dimension are differentiable at all but countably many points.

    For the second part, let $\lambda \in (0,1)$ be a point of differentiability of $ I_{\lambda}(\theta)$. For any other $0<\lambda'<1$, the spaces
    $\mathcal H^1_\lambda(P_\theta)$ and $\mathcal H^1_{\lambda'}(P_\theta)$
    coincide with equivalent norms, and the continuous extensions of $\ell_\theta$ agree.
    Applying \eqref{eqn:Q-completed-square} with $\lambda'$ in place of $\lambda$ gives
    \begin{equation}\label{eqn:I-supporting-line}
        \begin{aligned}
            I_{\lambda'}(\theta)
            &\ge Q_{\theta,\lambda'}(\phi_{\theta,\lambda})\\
            &=I_\lambda(\theta)+(\lambda'-\lambda)
            \left(
                \int_{\R^d}\|\nabla\phi_{\theta,\lambda}\|^2\diff P_\theta
                -\int_{\R^d}|\phi_{\theta,\lambda}|^2\diff P_\theta
            \right).
        \end{aligned}
    \end{equation}
    Dividing \eqref{eqn:I-supporting-line} by $\lambda'-\lambda$ after subtracting
    $I_\lambda(\theta)$, and letting $\lambda'$ approach $\lambda$ from both sides, gives
    \begin{equation}\label{eqn:slope-identity-1}
        -\partial_{\lambda}I_{\lambda}(\theta)
        =
        \int_{\R^d}|\phi_{\theta,\lambda}|^2\diff P_{\theta}
        -\int_{\R^d}\|\nabla\phi_{\theta,\lambda}\|^2\diff P_{\theta},
    \end{equation}
    which after multiplying by $(I_{\lambda}(\theta))^{-2}$ and using Definition~\ref{def:HK-info}
    yields
    \begin{equation}\label{eqn:slope-identity-2}
        -\frac{\partial_{\lambda}I_{\lambda}(\theta)}{(I_{\lambda}(\theta))^2}
        =V_{\lambda}(\theta)-S_{\lambda}(\theta).
    \end{equation}
    Combining \eqref{eqn:slope-identity-2} with the identity $\lambda V_{\lambda}(\theta)+(1-\lambda)
    S_{\lambda}(\theta)=(I_{\lambda}(\theta))^{-1}$ from Definition~\ref{def:HK-info} and solving the
    resulting linear system gives \eqref{eqn:VS-from-I}.
    Finally, the geometric claim follows simply by writing
    $\partial_{\lambda}(I_{\lambda}(\theta))^{-1}=-\partial_{\lambda}I_{\lambda}(\theta)/
    (I_{\lambda}(\theta))^2$,  so that the identities \eqref{eqn:VS-from-I} become
    \begin{equation*}
        V_{\lambda}(\theta)=(I_{\lambda}(\theta))^{-1}+(1-\lambda)\partial_{\lambda}(I_{\lambda}(\theta))^{-1}
        \qquad\textnormal{and}\qquad
        S_{\lambda}(\theta)=(I_{\lambda}(\theta))^{-1}-\lambda\,\partial_{\lambda}(I_{\lambda}(\theta))^{-1}.
    \end{equation*}
    The right-hand sides are the values at $z=1$ and $z=0$, respectively, of the tangent line
    $z\mapsto(I_{\lambda}(\theta))^{-1}+(z-\lambda)\partial_{\lambda}(I_{\lambda}(\theta))^{-1}$ to
    $\lambda\mapsto(I_{\lambda}(\theta))^{-1}$ at the point $\lambda$, as claimed.
\end{proof}

\begin{proof}[Proof of Lemma~\ref{lem:bdry-slope}]
        Fix $\theta\in\Theta$ and write $M_\lambda(\theta):=(I_\lambda(\theta))^{-1}$.

    Let $\ell_\theta$ be the linear functional defined in the proof of Lemma~\ref{lem:Info-cvx}.
    For $0<\lambda,\lambda_0<1$, we have $\ell_\theta(\phi_{\theta,\lambda}/I_\lambda(\theta))=1$.
    Applying \eqref{eqn:op-norm-calc} at $\lambda_0$ to $\phi_{\theta,\lambda}/I_\lambda(\theta)$ gives
    \begin{equation}\label{eqn:bdry-supp}
        \lambda_0 V_\lambda(\theta)+(1-\lambda_0)S_\lambda(\theta)
        \ge
        M_{\lambda_0}(\theta),
    \end{equation}
    with equality at $\lambda=\lambda_0$. In other words, the line
    $\{(v,s):\lambda_0 v+(1-\lambda_0)s=M_{\lambda_0}(\theta)\}$ passes through the point
    $(V_{\lambda_0}(\theta),S_{\lambda_0}(\theta))$, has slope $-\lambda_0/(1-\lambda_0)$, and the
    entire set $\bdry(\theta)$ lies in the half-plane
    $\{(v,s):\lambda_0 v+(1-\lambda_0)s\ge M_{\lambda_0}(\theta)\}$.

    We next check continuity in $\lambda$. Since
    $\ell_\theta(\phi_{\theta,\lambda}/I_\lambda(\theta))=1$, the transport-source equation gives
    \begin{equation}\label{eqn:bdry-score-continuity}
        \begin{aligned}
            \left\|\frac{\phi_{\theta,\lambda}}{I_\lambda(\theta)}
            -\frac{\phi_{\theta,\lambda_0}}{I_{\lambda_0}(\theta)}\right\|_{\mathcal H^1_{\lambda_0}(P_\theta)}^2
            &=\lambda_0V_\lambda(\theta)+(1-\lambda_0)S_\lambda(\theta)-M_{\lambda_0}(\theta)\\
            &=M_\lambda(\theta)-M_{\lambda_0}(\theta)
            +(\lambda_0-\lambda)(V_\lambda(\theta)-S_\lambda(\theta)).
        \end{aligned}
    \end{equation}
    By Lemma~\ref{lem:Info-cvx}, $M_\lambda(\theta)$ is continuous in $\lambda$.
    The identity $\lambda V_\lambda(\theta)+(1-\lambda)S_\lambda(\theta)=M_\lambda(\theta)$
    also shows that $V_\lambda(\theta)$ and $S_\lambda(\theta)$ are locally bounded on $(0,1)$.
    Thus \eqref{eqn:bdry-score-continuity} implies that
    $\phi_{\theta,\lambda}/I_\lambda(\theta)\to\phi_{\theta,\lambda_0}/I_{\lambda_0}(\theta)$
    in $\mathcal H^1_{\lambda_0}(P_\theta)$ as $\lambda\to\lambda_0$.
    Since $0<\lambda_0<1$, both the functions and their gradients converge in $L^2(P_\theta)$,
    proving continuity of $V_\lambda(\theta)$ and $S_\lambda(\theta)$.

    We now deduce that the coordinates are monotone. Fix $\lambda_1<\lambda_2$ in $(0,1)$ and abbreviate $\Delta V=V_{\lambda_2}(\theta)-
    V_{\lambda_1}(\theta)$ and $\Delta S=S_{\lambda_2}(\theta)-S_{\lambda_1}(\theta)$. Applying
    \eqref{eqn:bdry-supp} with $(\lambda_0,\lambda)=(\lambda_1,\lambda_2)$ and with
    $(\lambda_0,\lambda)=(\lambda_2,\lambda_1)$, and subtracting the respective equality cases at
    $\lambda=\lambda_0$, yields
    \begin{equation*}
        \lambda_1\Delta V+(1-\lambda_1)\Delta S\ge0
        \qquad\textnormal{and}\qquad
        \lambda_2\Delta V+(1-\lambda_2)\Delta S\le0.
    \end{equation*}
    Eliminating $\Delta S$ between these two inequalities gives $(\lambda_1-\lambda_2)\Delta V\ge0$,
    hence $\Delta V\le0$; eliminating $\Delta V$ gives $\Delta S\ge0$. Thus
    $\lambda\mapsto V_\lambda(\theta)$ is non-increasing and $\lambda\mapsto S_\lambda(\theta)$ is
    non-decreasing on $(0,1)$. Moreover, the two
    displayed inequalities force $\Delta S=0$ whenever $\Delta V=0$, so each value of $V$ corresponds
    to a single value of $S$; hence $\bdry(\theta)$ is the graph of a non-increasing function
    $h_\theta$ on $\{V_\lambda(\theta):0<\lambda<1\}$, which is an interval by continuity of $V_\lambda(\theta)$.

    Finally, by \eqref{eqn:bdry-supp}, the graph of $h_\theta$ lies on or above the line through
    $(V_{\lambda_0}(\theta),S_{\lambda_0}(\theta))$ of slope $-\lambda_0/(1-\lambda_0)$
    for every $0<\lambda_0<1$. Thus $h_\theta$ is the supremum of these affine functions on its
    domain and is therefore convex. If $V_{\lambda_0}(\theta)$ is an interior point of the domain
    at which $h_\theta$ is differentiable, the derivative equals the supporting slope, giving
    $h_\theta'(V_{\lambda_0}(\theta))=-\lambda_0/(1-\lambda_0)$.
\end{proof}

\subsection{Proofs from Section~\ref{sec:achieve}}\label{app:proofs-achieve}

\begin{proof}[Proof of Theorem~\ref{thm:z-est}]
    For each \(n\), let \(U_n\) be a neighborhood around $(X_1, \dots, X_n)$ in which   equation~\eqref{eq:EstiamtingEqn2} admits solutions and the solution map is continuous. Also,
    consider the event
    \[
        E_n=\left\{\det\left(\frac1n\sum_{i=1}^n
                D_\theta s(T_n^s(X_1, \dots, X_n),X_i)\right)\ne0
        \right\}.
    \]
    We first show that \(\Pbb_{P}(E_n^c)\to0\).

    Choose \(0<\delta<\varepsilon\) and set
    \(J={B(\theta_0,\delta)}\).
    By \(C^1\) regularity, both empirical maps
    \(t\mapsto n^{-1}\sum_{i=1}^nD_t s(t,X_i)\) and
    \(t\mapsto n^{-1}\sum_{i=1}^nD_xs(t,X_i)D_xs(t,X_i)^\top\)
    are continuous on \(J\).
    The entrywise ULLNs give a common subsequence along which
    these empirical maps converge uniformly on \(J\), almost surely,
    to their respective population expectations.
    As uniform limits of continuous functions, the maps
    \(t\mapsto\Ebb_{P}[D_t s(t,X)]\) and
    \(t\mapsto\Ebb_{P}[D_xs(t,X)D_xs(t,X)^\top]\)
    are therefore continuous on \(J\).
    By consistency of $T_n^s(X_1, \dots, X_n)$ and continuity of $T_n^s$, we may shrink the neighborhood $U_n$ if necessary to conclude that  \(T_n^s(x_1, \dots, x_n)\in J\) for all $(x_1, \dots, x_n) \in U_n$, with probability tending to one.
    Applying the ULLNs on the event \(\{T_n^s(X_1, \dots, X_n)\in J\}\) and using continuity of the
    population expectations at \(\theta_0\), we obtain
    \begin{equation}\label{eqn:Z-empirical-limits}
        \begin{aligned}
            \frac1n\sum_{i=1}^nD_\theta s(T_n^s,X_i)
            &\longrightarrow
            \Ebb_{P}[D_\theta s(\theta_0,X)],\\
            \frac1n\sum_{i=1}^n
                D_xs(T_n^s,X_i)D_xs(T_n^s,X_i)^\top
            &\longrightarrow
            \Ebb_{P}[D_xs(\theta_0,X)D_xs(\theta_0,X)^\top]
        \end{aligned}
    \end{equation}
    in probability.
    By \eqref{eqn:Z-empirical-limits}, continuity of the determinant, and the fact that 
    \(\det\Ebb_{P}[D_\theta s(\theta_0,X)]\ne0\) we conclude that
    \(\Pbb_{P}(E_n^c)\to0\).

    Next we write \(F_n(\theta,x)=n^{-1}\sum_{i=1}^n s(\theta,x_i)\),
    where \(x=(x_1,\ldots,x_n)\).
    Since \(F_n\) is \(C^1\) and its parameter Jacobian at
    \((T_n^s(X_1, \dots, X_n),(X_1, \dots, X_n))\) is invertible in the event $E_n$, the classical implicit function
    theorem (see \cite[Corollary~3.10]{Sood2020} for a Lipschitz
    extension) gives a unique function \(g:W\to B_r( T_n^s(X_1, \dots, X_n)) \subseteq \Theta\), for $W$ some
open neighborhood of \((X_1, \dots, X_n)\) and $r>0$, such that $F_n(g(x),x)=0$ for all $x \in W$ and $g(X_1, \dots, X_n) = T_n^s(X_1, \dots, X_n)$; the theorem also guarantees that $g$ is $C^1$. On the other hand, continuity of $T_n^s$ guarantees that $T_n^s(x) \in B_r( T_n^s(X_1, \dots, X_n))  $ for all $x \in U_n$ (shrinking $U_n$ further if necessary). The uniqueness of $g$ then implies that  \(g(x)=T_n^s(x)\) for all \(x\in U_n \cap W\). This implies that $T_n^s$ is $C^1$ around $(X_1, \dots, X_n)$ in the event $E_n$.

  In this event, we can thus differentiate with respect to \(x_j\) both sides of the equation
  \[ \frac{1}{n}\sum_{i=1}^n s(T_n^s(x), x_i) =0 \] 
  to obtain
    \begin{equation}\label{eqn:Z-chain-rule}
        \left(\frac1n\sum_{i=1}^nD_\theta s(T_n^s,X_i)\right)
        D_{x_j}T_n^s+\frac1nD_xs(T_n^s,X_j)=0.
    \end{equation}
  Using the fact that the matrix in front of \(D_{x_j}T_n^s\) is invertible (because we are assuming we are in the event \(E_n\)), multiplying each data Jacobian
    by its transpose, and then summing over \(j\) yields
    \begin{equation}\label{eqn:Z-prelimit}
        \begin{aligned}
            n\sum_{i=1}^nD_{x_i}T_n^s(D_{x_i}T_n^s)^\top
            &=
            \left(\frac1n\sum_{i=1}^n
                D_\theta s(T_n^s,X_i)\right)^{-1}\\
            &\quad{}\times
            \left(\frac1n\sum_{i=1}^n
                D_xs(T_n^s,X_i)D_xs(T_n^s,X_i)^\top\right)
            \left[\left(\frac1n\sum_{i=1}^n
                D_\theta s(T_n^s,X_i)\right)^{-1}\right]^\top.
        \end{aligned}
    \end{equation}
    Since \(\Ebb_{P}[D_\theta s(\theta_0,X)]\) is invertible,
    continuity of matrix inversion, \eqref{eqn:Z-prelimit},
    and \eqref{eqn:Z-empirical-limits} give the claimed limit on \(E_n\). Since $\mathbb{P}_{P}(E_n^c) \rightarrow 0$, this completes the proof.
    
\end{proof}

\begin{proof}[Proof of Example~\ref{ex:huber}]
    For every point \(x=(x_1, \dots, x_n)\), the empirical Huber loss is convex
    and coercive, so its minimizers form a nonempty compact interval.
    Since \(s_\delta(t,x_i)\) is nondecreasing in \(x_i\) and
    nonincreasing in \(t\), both endpoints of this interval, and
    hence its midpoint, are nondecreasing in each observation.
    Shifting every observation and the parameter by the same
    \(h\in\R\) leaves the residuals unchanged.
    The minimizing interval therefore shifts by \(h\), so
    \(T_n^\delta(x+h\mathbf{1})=T_n^\delta(x)+h\),
    where \(\mathbf{1}=(1,\ldots,1)\).

    For two points \(x,y \in \R^n\), set \(r:=\max_i|x_i-y_i|\).
    Then \(x-r\mathbf{1}\le y\le x+r\mathbf{1}\) coordinatewise,
    and monotonicity gives
    \[
        T_n^\delta(x)-r
        =T_n^\delta(x-r\mathbf{1})
        \le T_n^\delta(y)
        \le T_n^\delta(x+r\mathbf{1})
        =T_n^\delta(x)+r;
    \]
    in the above, $\mathbf{1}$ represents the vector with all entries equal to one. Hence \(|T_n^\delta(x)-T_n^\delta(y)|
    \le r=\max_i|x_i-y_i|\), so \(T_n^\delta\) is globally Lipschitz.
    Rademacher's theorem and the absolute continuity of
    \(P^{\otimes n}\) imply that its gradient exists almost surely.

    Write \(q:=\Pbb_P(|X-\theta|<\delta)>0\).
    For every \(\varepsilon>0\) and \(x\in\R\),
    \(s_\delta(\theta-\varepsilon,x)\ge s_\delta(\theta,x)
    \ge s_\delta(\theta+\varepsilon,x)\), with strict inequalities when \(|x-\theta|<\delta\).
    Since \(\Ebb_P[s_\delta(\theta,X)]=0\) and \(q>0\), it follows that
    \(\Ebb_P[s_\delta(\theta-\varepsilon,X)]>0\) and
    \(\Ebb_P[s_\delta(\theta+\varepsilon,X)]<0\).
    Since the score associated to the Huber loss is bounded, the law of large numbers applied for each of the two parameters $\theta+ \varepsilon$ and $\theta-\varepsilon$ implies that, with probability tending
    to one, every empirical root lies between
    \(\theta-\varepsilon\) and \(\theta+\varepsilon\).
    In particular, \(T_n^\delta\to\theta\) in probability.

    Let \(k_n:=\sum_{i=1}^n
    \mathbf{1}\{|X_i-T_n^\delta|<\delta\}\).
    For \(0<\varepsilon<\delta\), on
    \(|T_n^\delta-\theta|<\varepsilon\) we have
    \[
        P_n((\theta-\delta+\varepsilon,\theta+\delta-\varepsilon))
        \le\frac{k_n}{n}
        \le
        P_n((\theta-\delta-\varepsilon,\theta+\delta+\varepsilon)),
    \]
    where $P_n$ denotes the empirical measure. Applying the law of large numbers on each side of the above inequality
    and then letting \(\varepsilon\downarrow0\), and then using the absolute continuity
    of \(P\) and consistency of \(T_n^\delta\) for \(\theta\)
    give \(k_n/n\to q\) in probability.
    Consequently, \(\Pbb_P(k_n=0)\to0\).

    We next exclude observations at the clipping points in the event
    \(\{k_n>0\}\).
    Set \(A:=\{i:|X_i-T_n^\delta|<\delta\}\) and define
    \[
        c:=\#\{i:X_i-T_n^\delta\ge\delta\}
        -\#\{i:X_i-T_n^\delta\le-\delta\}.
    \]
    Since \(T_n^\delta\) minimizes the empirical
    Huber loss (which is differentiable), first-order optimality conditions give
    \begin{equation}\label{eqn:Huber-root-decomposition}
        0=\sum_{i=1}^n s_\delta(T_n^\delta,X_i)
        =\sum_{i\in A}(X_i-T_n^\delta)+c\delta.
    \end{equation}
    Rearranging \eqref{eqn:Huber-root-decomposition} gives
    \(|A|T_n^\delta=\sum_{i\in A}X_i+c\delta\).
    If \(X_j-T_n^\delta=\pm\delta\), then \(j\notin A\) and
    \[
        |A|X_j-\sum_{i\in A}X_i=(c\pm|A|)\delta.
    \]
    For each choice of \(A,j,c\) and sign, the above is a nontrivial
    affine constraint on the sample.
    Since there are only finitely many such choices, absolute
    continuity of \(P^{\otimes n}\) thus implies that the probability that any observation coincides with any of the two clipping points is zero.

    On \(\{k_n>0\}\), away from this null event, continuity of
    \(T_n^\delta\) ensures that the active and clipped observations
    remain unchanged in a neighborhood of the sample.
    Differentiating \eqref{eqn:Huber-root-decomposition}
    with these sets fixed gives
    \begin{equation}\label{eqn:Huber-sample-sensitivity}
        \nabla_{x_i}T_n^\delta
        =\frac{\mathbf{1}\{|X_i-T_n^\delta|<\delta\}}{k_n},
        \qquad
        n\sum_{i=1}^n\|\nabla_{x_i}T_n^\delta\|^2
        =\frac{n}{k_n}.
    \end{equation}
    Since \(k_n/n\to q>0\) and \(\Pbb_P(k_n=0)\to0\),
    \eqref{eqn:Huber-sample-sensitivity} gives the claimed
    sensitivity limit.
    Finally, \(q=\Pbb_P(|X-\theta|\le\delta)\) by absolute
    continuity of \(P\).
    The stated central limit theorem follows from \cite[Lemma~4]{Huber1964}.
\end{proof}

\begin{proof}[Proof of Lemma~\ref{lem:HK-information-identity}]
    We first consider score increments with the sampling distribution
    fixed at \(P_{\theta_0}\).
    Take \(X_1,X_2,\ldots\) i.i.d. from \(P_{\theta_0}\).
    As shown in the proof of Theorem~\ref{thm:z-est},
    joint \(C^1\) regularity and the ULLN for the matrix of derivatives with respect to the parameter imply
    continuity of
    \(t\mapsto\Ebb_{\theta_0}[D_t\phi_{t,\lambda}(X)]\)
    within an open ball centered at \(\theta_0\) (we will refer to this neighborhood as \(U\)).
    After shrinking $U$ if necessary, this map can be assumed to be bounded on \(U\),
    and the ULLN for the matrix of derivatives with respect to the parameter gives
    \begin{equation}\label{eqn:HK-info-ULLN}
        \sup_{t\in U}
        \left\|
            \frac1n\sum_{i=1}^nD_t\phi_{t,\lambda}(X_i)
            -\Ebb_{\theta_0}[D_t\phi_{t,\lambda}(X)]
        \right\|_{\mathrm F}
        \longrightarrow0
    \end{equation}
    in probability, as \(n\to\infty\).

    Fix \(h\in\R^p\) such that \(\theta_0+h\in U\).
    Since \(U\) is a ball, \(\theta_0+uh\in U\) for \(0\le u\le1\).
    Applying the fundamental theorem of calculus to
    \(u\mapsto\phi_{\theta_0+uh,\lambda}(X_i)\) gives
    \begin{equation}\label{eqn:HK-info-empirical-FTC}
        \frac1n\sum_{i=1}^n
        \left[
            \phi_{\theta_0+h,\lambda}(X_i)
            -\phi_{\theta_0,\lambda}(X_i)
        \right]
        =
        \left(
            \int_0^1\frac1n\sum_{i=1}^n
            D_\theta\phi_{\theta_0+uh,\lambda}(X_i)\,\diff u
        \right)h.
    \end{equation}
    Replacing the empirical Jacobian in
    \eqref{eqn:HK-info-empirical-FTC} on the right hand side of the above expression by its expectation introduces an error no larger than
    \[
        \|h\|\sup_{t\in U}
        \left\|
            \frac1n\sum_{i=1}^nD_t\phi_{t,\lambda}(X_i)
            -\Ebb_{\theta_0}[D_t\phi_{t,\lambda}(X)]
        \right\|_{\mathrm F},
    \]
    which converges to zero in probability by
    \eqref{eqn:HK-info-ULLN}.
    Applying the law of large numbers to the score increment,
    using (i), then yields
    \begin{equation}\label{eqn:HK-info-population-FTC}
        \Ebb_{\theta_0}
        \left[
            \phi_{\theta_0+h,\lambda}(X)
            -\phi_{\theta_0,\lambda}(X)
        \right]
        =
        \left(
            \int_0^1
            \Ebb_{\theta_0}
            [D_\theta\phi_{\theta_0+uh,\lambda}(X)]\,\diff u
        \right)h.
    \end{equation}
    Continuity of the expected parameter Jacobian at \(\theta_0\)
    now gives
    \begin{equation}\label{eqn:HK-info-fixed-law-limit}
        \Ebb_{\theta_0}
        \left[
            \phi_{\theta_0+h,\lambda}(X)
            -\phi_{\theta_0,\lambda}(X)
        \right]
        =
        \Ebb_{\theta_0}[D_\theta\phi_{\theta_0,\lambda}(X)]h
        +o(\|h\|).
    \end{equation}

    By \(\HK_\lambda\) regularity, the path
    \(u\mapsto P_{\theta_0+uh}\) solves the transport-source equation
    with
    \[
        \xi_u=h^\top\phi_{\theta_0+uh,\lambda},
        \qquad
        v_u=(D\phi_{\theta_0+uh,\lambda})^\top h.
    \]
    Assumption (i) and Proposition~\ref{prop:Sobolev} show that
    this pair belongs to the tangent space and has bounded energy.
    Proposition~\ref{prop:ans-cont-unique-solution} therefore gives
    \(2\)-absolute continuity of this path.
    Assumption (i) also verifies the test-function conditions of
    Theorem~\ref{thm:loc-Lip-test-funcs} for every component
    of each fixed \(\phi_{t,\lambda}\), \(t\in U\).

    For \(r,t\in U\), write
    \[
        K(r,t):=
        \lambda\Ebb_r[
            \phi_{t,\lambda}(X)\phi_{r,\lambda}(X)^\top]
        +(1-\lambda)\Ebb_r[
            D\phi_{t,\lambda}(X)D\phi_{r,\lambda}(X)^\top].
    \]
    Apply 
    Theorem~\ref{thm:loc-Lip-test-funcs} componentwise to
    \(x\mapsto\phi_{\theta_0+h,\lambda}(x)\),
    keeping this function fixed while \(u\) varies along
    \(P_{\theta_0+uh}\).
    Using the centering convention in Remark~\ref{remark:center-conv},
    we obtain
    \begin{equation}\label{eqn:HK-info-transport-difference}
        \begin{aligned}
            \Ebb_{\theta_0}
            \left[
                \phi_{\theta_0+h,\lambda}(X)
                -\phi_{\theta_0,\lambda}(X)
            \right]
            &=
            -\left(
                \int_0^1K(\theta_0+uh,\theta_0+h)\,\diff u
            \right)h\\
            &=-I_\lambda(\theta_0)h+o(\|h\|).
        \end{aligned}
    \end{equation}
    The last equality follows from (ii), since
    \(K(\theta_0+uh,\theta_0+h)\) converges to
    \(K(\theta_0,\theta_0)=I_\lambda(\theta_0)\)
    uniformly for \(u\in[0,1]\) as \(h\to0\).
    Comparing \eqref{eqn:HK-info-fixed-law-limit} and
    \eqref{eqn:HK-info-transport-difference} along each coordinate
    direction gives
    \(\Ebb_{\theta_0}[D_\theta\phi_{\theta_0,\lambda}(X)]
    =-I_\lambda(\theta_0)\), as desired.
\end{proof}

\begin{proof}[Proof of Theorem \ref{thm:HK-Z-attainment}]
    We derive the asymptotic variance and sensitivity of the \(Z\)-estimator.
    Write \(\hat\theta_\lambda=T_{n,\lambda}\), and let \(U\) be the
    neighborhood used in the proof of Lemma~\ref{lem:HK-information-identity}.
    By consistency, \(\hat\theta_\lambda\in U\) with probability
    tending to one as $n \rightarrow \infty$.
    On this event, \(\theta_0+u(\hat\theta_\lambda-\theta_0)\in U\)
    for every \(u\in[0,1]\).
    Applying the fundamental theorem of calculus to
    \(u\mapsto
    \phi_{\theta_0+u(\hat\theta_\lambda-\theta_0),\lambda}(X_i)\)
    gives
    \begin{equation}\label{eqn:HK-Z-linearization}
        \begin{aligned}
            0
            &=\frac1n\sum_{i=1}^n
                \phi_{\hat\theta_\lambda,\lambda}(X_i)\\
            &=\frac1n\sum_{i=1}^n
                \phi_{\theta_0,\lambda}(X_i)
            +\left(
                \int_0^1\frac1n\sum_{i=1}^n
                D_\theta
                \phi_{\theta_0+u(\hat\theta_\lambda-\theta_0),\lambda}(X_i)
                \,\diff u
            \right)(\hat\theta_\lambda-\theta_0).
        \end{aligned}
    \end{equation}
    We first replace the empirical parameter Jacobian in
    \eqref{eqn:HK-Z-linearization} by its expectation under
    \(P_{\theta_0}\).
    In the expectations below, the parameter is held fixed when
    integrating over \(X\), and the random parameter is then substituted.
    On the event \(\{\hat\theta_\lambda\in U\}\),
    \begin{equation}\label{eqn:HK-Z-integral-ULLN}
        \begin{aligned}
            &\left\|
                \int_0^1\frac1n\sum_{i=1}^n
                D_\theta
                \phi_{\theta_0+u(\hat\theta_\lambda-\theta_0),\lambda}(X_i)
                \,\diff u
                -
                \int_0^1\Ebb_{\theta_0}\left[
                D_\theta
                \phi_{\theta_0+u(\hat\theta_\lambda-\theta_0),\lambda}(X)
                \right]\,\diff u
            \right\|_{\mathrm F}\\
            &\qquad\le
            \sup_{t\in U}
            \left\|
                \frac1n\sum_{i=1}^nD_t\phi_{t,\lambda}(X_i)
                -\Ebb_{\theta_0}[D_t\phi_{t,\lambda}(X)]
            \right\|_{\mathrm F}
            \longrightarrow0
        \end{aligned}
    \end{equation}
    in probability as \(n\to\infty\), by \eqref{eqn:HK-info-ULLN}.
    By consistency of $\hat{\theta}_\lambda$,
    \(\theta_0+u(\hat\theta_\lambda-\theta_0)\) converges to
    \(\theta_0\) uniformly for \(u\in[0,1]\), in probability.
    Together with the continuity at \(\theta_0\) of
    \(t\mapsto\Ebb_{\theta_0}[D_t\phi_{t,\lambda}(X)]\),
    which was established in the proof of Lemma~\ref{lem:HK-information-identity},
    this implies that the expectation integral in
    \eqref{eqn:HK-Z-integral-ULLN} converges in probability to
    \(\Ebb_{\theta_0}[D_\theta\phi_{\theta_0,\lambda}(X)]\).
    Combining this convergence with
    \eqref{eqn:HK-Z-integral-ULLN} and with the information identity
    \eqref{eq:HK-information-identity} gives
    \begin{equation}\label{eqn:HK-Z-derivative-limit}
        \int_0^1\frac1n\sum_{i=1}^n
        D_\theta
        \phi_{\theta_0+u(\hat\theta_\lambda-\theta_0),\lambda}(X_i)
        \,\diff u
        \longrightarrow
        \Ebb_{\theta_0}[D_\theta\phi_{\theta_0,\lambda}(X)]
        =-I_\lambda(\theta_0)
    \end{equation}
    in probability.
    Since \(I_\lambda(\theta_0)\) is invertible, the integral
    term in \eqref{eqn:HK-Z-linearization} is invertible
    with probability tending to one.

    Condition (i) of Lemma~\ref{lem:HK-information-identity}
    and the centering convention in Remark~\ref{remark:center-conv}
    give
    \(\Ebb_{\theta_0}[\|\phi_{\theta_0,\lambda}(X)\|^2]<\infty\)
    and \(\Ebb_{\theta_0}[\phi_{\theta_0,\lambda}(X)]=0\).
    Hence Chebyshev's inequality yields
    \(n^{-1/2}\sum_{i=1}^n\phi_{\theta_0,\lambda}(X_i)=O_p(1)\).
    Substituting \eqref{eqn:HK-Z-derivative-limit} into
    \eqref{eqn:HK-Z-linearization} and multiplying by \(\sqrt n\)
    gives
    \begin{equation}\label{eqn:HK-Z-substituted-equation}
        0=
        n^{-1/2}\sum_{i=1}^n\phi_{\theta_0,\lambda}(X_i)
        +\bigl(-I_\lambda(\theta_0)+o_p(1)\bigr)
        \sqrt n(\hat\theta_\lambda-\theta_0).
    \end{equation}
    Rearranging \eqref{eqn:HK-Z-substituted-equation}, we obtain
    \[
        \begin{aligned}
            \sqrt n(\hat\theta_\lambda-\theta_0)
            &=
            \bigl(I_\lambda(\theta_0)+o_p(1)\bigr)^{-1}
            n^{-1/2}\sum_{i=1}^n\phi_{\theta_0,\lambda}(X_i)\\
            &=
            I_\lambda(\theta_0)^{-1}
            n^{-1/2}\sum_{i=1}^n\phi_{\theta_0,\lambda}(X_i)
            +o_p(1),
        \end{aligned}
    \]
    using continuity of matrix inversion and the tightness of
    the normalized score sum.
    The multivariate central limit theorem, Slutsky's theorem,
    and symmetry of \(I_\lambda(\theta_0)\) now give
    \[
        \begin{aligned}
            \sqrt n(\hat\theta_\lambda-\theta_0)
            &\overset{\mathrm d}{\to}
            \mathcal N\left(
                0,\,
                I_\lambda(\theta_0)^{-1}
                \Ebb_{\theta_0}\left[
                    \phi_{\theta_0,\lambda}(X)
                    \phi_{\theta_0,\lambda}(X)^\top
                \right]
                I_\lambda(\theta_0)^{-1}
            \right)\\
            &=\mathcal N\bigl(0,V_\lambda(\theta_0)\bigr).
        \end{aligned}
    \]

    Note that the Jacobians in Theorem~\ref{thm:HK-Z-attainment} are well-defined for sufficiently large $n$ almost surely.
    Applying Theorem~\ref{thm:z-est} with
    \(s(\theta,x)=\phi_{\theta,\lambda}(x)\) and using
    \eqref{eq:HK-information-identity} yields
    \[
        \begin{aligned}
            n\sum_{i=1}^n
            D_{x_i}\hat\theta_\lambda
            (D_{x_i}\hat\theta_\lambda)^\top
            &\longrightarrow
            I_\lambda(\theta_0)^{-1}
            \Ebb_{\theta_0}\left[
                D\phi_{\theta_0,\lambda}(X)
                D\phi_{\theta_0,\lambda}(X)^\top
            \right]
            I_\lambda(\theta_0)^{-1}\\
            &=S_\lambda(\theta_0)
        \end{aligned}
    \]
    in probability, where the last equality follows from
    Definition~\ref{def:HK-info}.
\end{proof}

\begin{proof}[Proof of Proposition~\ref{prop:Pareto-extremal}]
    We prove the first implication; the second follows by interchanging the roles of the variance and the sensitivity.
    For each $0<\lambda<1$, use Theorem~\ref{thm:HK-CR-bound} to get
    \begin{equation*}
        \lambda n\Var_{\theta}(T_n)+(1-\lambda)n\Sen_{\theta}(T_n) \ge \lambda V_{\lambda}(\theta) + (1-\lambda) S_{\lambda}(\theta),
    \end{equation*}
    and the bound holds for every $n$.
    Since $n\Var_\theta(T_n)$ converges, taking the lower limit as $n\to\infty$ gives
    \begin{align*}
        \lambda\lim_{n\to\infty} &n\Var_{\theta}(T_n)+(1-\lambda)\liminf_{n\to\infty}n\Sen_{\theta}(T_n) \\
        &=\liminf_{n\to\infty}\left(\lambda n\Var_{\theta}(T_n)+(1-\lambda)n\Sen_{\theta}(T_n)\right) \\
        &\ge \lambda V_{\lambda}(\theta) + (1-\lambda) S_{\lambda}(\theta).
    \end{align*}
    By assumption, $\lim_{n\to\infty}n\Var_\theta(T_n)=\inf_{0<\lambda<1}V_\lambda(\theta)$, so $V_\lambda(\theta)-\lim_{n\to\infty}n\Var_\theta(T_n)\ge0$ for every $0<\lambda<1$.
    Rearranging and using $\lambda/(1-\lambda)>0$ therefore gives
    \begin{equation*}
        \liminf_{n\to\infty}n\Sen_{\theta}(T_n) \ge \frac{\lambda}{1-\lambda}\left(V_{\lambda}(\theta)-\lim_{n\to\infty}n\Var_{\theta}(T_n)\right) + S_{\lambda}(\theta) \ge S_{\lambda}(\theta).
    \end{equation*}
    Since $0<\lambda<1$ was arbitrary, taking the supremum over $0<\lambda<1$ finishes the proof.
\end{proof}

\begin{proof}[Proof of Corollary~\ref{cor:extremal-efficiency}]
    The argument is immediate from Theorem~\ref{thm:HK-Z-attainment} in the cases $\lambda\in\{0,1\}$, along with the assumptions $\inf_{0<\lambda<1}V_\lambda(\theta)=V_1(\theta)$, $\sup_{0<\lambda<1}S_\lambda(\theta)=S_1(\theta)$, $\sup_{0<\lambda<1}V_\lambda(\theta)=V_0(\theta)$, and $\inf_{0<\lambda<1}S_\lambda(\theta)=S_0(\theta)$.
\end{proof}

\section{Proof of Technical Results}\label{app:technical}

\begin{lemma}[Tangent-space identifications]\label{lem:HK-tangent-identification}
For $\mu\in\Pcal_{2,\textnormal{ac}}(\R^m)$, the space defined in \eqref{eqn:def-HK-tangent-space} is canonically isometrically isomorphic to
\[
    \begin{cases}
        \overline{\{\nabla\phi:\phi\in C_c^\infty(\R^m)\}}^{L^2(\mu;\R^m)},
            &\lambda=0,\\[4pt]
        \overline{\{(\phi,\nabla\phi):\phi\in C_c^\infty(\R^m)\}}^{L^2(\mu;\R^{m+1})},
            &0<\lambda<1,\\[4pt]
        L^2(\mu),&\lambda=1.
    \end{cases}
\]
\end{lemma}

The endpoint spaces carry their usual $L^2$ norms, while in the interior case the squared norm of $(\xi,v)$ is
\[
    \int_{\R^m}\left(\lambda|\xi|^2+(1-\lambda)\|v\|^2\right)\diff\mu.
\]
These are the velocity, transport-source, and source representations of the Wasserstein, Hellinger--Kantorovich, and Hellinger tangent spaces, respectively, with the normalization in \eqref{eqn:def-Sobolev}.

\begin{proof}
For $\lambda=0$, \eqref{eqn:def-Sobolev} gives $\|\phi\|_{\mathcal H_0^1(\mu)}=\|\nabla\phi\|_{L^2(\mu;\R^m)}$.
Thus the map $(\phi,\nabla\phi)\mapsto\nabla\phi$ gives the first identification.
Only the gradient component is retained; convergence of the scalar component is not required.

For $0<\lambda<1$, the weighted energy norm is equivalent to the product $L^2$ norm, so the completion identifies isometrically with the stated closure equipped with the weighted norm.
This is the tangent-field closure appearing in \cite[Proposition~3.1 and equation~(3.6)]{Fleissner}, with $\Lambda=1-\lambda$, $\Sigma=\lambda$, and the component order reversed.
In particular, its underlying set is independent of $\lambda\in(0,1)$, although its norm depends on $\lambda$.

For $\lambda=1$, \eqref{eqn:def-Sobolev} gives $\|\phi\|_{\mathcal H_1^1(\mu)}=\|\phi\|_{L^2(\mu)}$, so the relevant map is $(\phi,\nabla\phi)\mapsto\phi$.
Since $\mu$ is a finite Borel measure, $C_c(\R^m)$ is dense in $L^2(\mu)$, and mollification approximates every $C_c(\R^m)$ function uniformly by functions in $C_c^\infty(\R^m)$.
Uniform convergence implies $L^2(\mu)$ convergence because $\mu$ is finite.
Thus $C_c^\infty(\R^m)$ is dense in $L^2(\mu)$, which proves the last identification.
Only the scalar component is retained; convergence of the gradient component is not required.
\end{proof}

\begin{proof}[Proof of Proposition~\ref{prop:Sobolev}]
Fix $0\le\lambda\le1$ and $\phi\in\mathcal H_\lambda^1(\mu)$, and choose a locally Lipschitz extension $\widetilde\phi:\R^m\to\R$ as in Definition~\ref{def:Sobolev}. Before continuing with the proof, we make the following remark.
\begin{remark}\label{rem:Sobolev-gradient}
    Let $\widetilde\psi$ be another such extension, so $\widetilde\phi=\widetilde\psi$ on $\operatorname{supp}\mu$.
    At every Lebesgue density point of $\operatorname{supp}\mu$ where $\widetilde\phi-\widetilde\psi$ is differentiable, its gradient is zero. The Lebesgue density theorem and Rademacher's theorem therefore imply that the two gradients agree Lebesgue-almost everywhere on $\operatorname{supp}\mu$, and hence $\mu$-almost everywhere because $\mu\ll\mathcal L^m$.
\end{remark}
The above remark implies that there is no ambiguity when defining the gradient of $\phi$, as any Lipschitz extension of $\phi$ can be used for that purpose. Continuing with the proof, we note that, by Definition~\ref{def:Sobolev},
\[
    \int_{\R^m}\left(\lambda|\widetilde\phi|^2
       +(1-\lambda)\|\nabla\widetilde\phi\|^2\right)\diff\mu
    =\|\widetilde\phi\|_{\mathcal H_\lambda^1(\mu)}^2<\infty.
\]
We now construct $\varphi_n\in C_c^\infty(\R^m)$ such that
\[
    \|\varphi_n-\widetilde\phi\|_{\mathcal H_\lambda^1(\mu)}
    \longrightarrow0.
\]
This would prove the desired result. We split the proof into three cases.

\medskip
\noindent\textbf{Case 1: $0<\lambda<1$.}
The integrability condition gives
\begin{align*}
    \|\widetilde\phi\|_{L^2(\mu)}^2
    &\le\frac1\lambda\|\widetilde\phi\|_{\mathcal H_\lambda^1(\mu)}^2<\infty,\\
    \|\nabla\widetilde\phi\|_{L^2(\mu;\R^m)}^2
    &\le\frac1{1-\lambda}\|\widetilde\phi\|_{\mathcal H_\lambda^1(\mu)}^2<\infty.
\end{align*}
We proceed in three steps.

\medskip
\noindent\textbf{Step 1: Approximation for globally Lipschitz functions with compact support.}
Let $g:\R^m\to\R$ be globally Lipschitz and compactly supported, and let $L:=\operatorname{Lip}(g)$.
Choose a nonnegative $\eta\in C_c^\infty(\R^m)$ with $\int_{\R^m}\eta(z)\diff z=1$ and $\operatorname{supp}(\eta)\subseteq B_1$, and set
\[
    \eta_\varepsilon(x):=\varepsilon^{-m}\eta(x/\varepsilon),
    \qquad g_\varepsilon(x):=(g*\eta_\varepsilon)(x)
    =\int_{\R^m}g(x-y)\eta_\varepsilon(y)\diff y.
\]
Then $g_\varepsilon\in C_c^\infty(\R^m)$, and
\begin{align*}
    |g_\varepsilon(x)-g(x)|
    &\le\int_{\R^m}|g(x-y)-g(x)|\eta_\varepsilon(y)\diff y\\
    &=\int_{\R^m}|g(x-\varepsilon z)-g(x)|\eta(z)\diff z\\
    &\le L\varepsilon\int_{\R^m}|z|\eta(z)\diff z
    =:  L C_\eta\varepsilon.
\end{align*}
Thus $g_\varepsilon\to g$ uniformly as $\varepsilon\downarrow0$. Since $\mu$ is finite,
\[
    \|g_\varepsilon-g\|_{L^2(\mu)}^2
    \le L^2\varepsilon^2  C_\eta^2\mu(\R^m)\longrightarrow0,
\]
as $ \varepsilon\downarrow 0$. To control the gradients, note that
\begin{align*}
    |g_\varepsilon(x)-g_\varepsilon(y)|
    &\le\int_{\R^m}|g(x-\varepsilon z)-g(y-\varepsilon z)|\eta(z)\diff z\\
    &\le L|x-y|.
\end{align*}
Hence $\|\nabla g_\varepsilon\|\le L$ everywhere, while Rademacher's theorem gives $\|\nabla g\|\le L$ Lebesgue-almost everywhere.
For each $j=1,\ldots,m$, weak differentiation commutes with mollification \cite[proof of Lemma~11.25]{Leoni2017}:
\[
    \partial_jg_\varepsilon=(\partial_jg)*\eta_\varepsilon.
\]
At every Lebesgue point $x$ of $\partial_jg$, $\partial_j g_\varepsilon(x)=((\partial_jg)*\eta_\varepsilon)(x)$ converges to $\partial_jg(x)$ as $\varepsilon\downarrow0$ \cite[Theorem~C.16(ii)]{Leoni2017}.
Consequently, $\nabla g_\varepsilon\to\nabla g$ Lebesgue-a.e., and also $\mu$-a.e. because $\mu\ll\mathcal L^m$.
The bound
\[
    \|\nabla g_\varepsilon(x)-\nabla g(x)\|^2\le4L^2
    \qquad \text{for }\mu\text{-a.e. } x
\]
and the dominated convergence now yield
\[
    \|\nabla g_\varepsilon-\nabla g\|_{L^2(\mu;\R^m)}
    \longrightarrow0.
\]
Thus, $g$ and its gradient can be approximated, w.r.t. $L^2(\mu)$ norms, by $C_c^\infty(\R^m)$ functions and their gradients.

\medskip
\noindent\textbf{Step 2: Cutoff approximation.}
For $R>0$, define
\[
    \chi_R(x):=
    \begin{cases}
        1,& |x|\le R,\\
        2-|x|/R,& R<|x|<2R,\\
        0,& |x|\ge2R,
    \end{cases}
    \qquad
    \widetilde\phi_R:=\widetilde\phi\chi_R.
\]
We have $0\le\chi_R\le1$, $\operatorname{supp}(\chi_R)\subseteq\overline B_{2R}$, and $\|\nabla\chi_R\|\le R^{-1}$ Lebesgue-almost everywhere.
Since $\widetilde\phi$ is Lipschitz and bounded on $\overline B_{2R}$, the function $\widetilde\phi_R$ is globally Lipschitz and compactly supported.
Moreover,
\[
    \|\widetilde\phi_R-\widetilde\phi\|_{L^2(\mu)}^2
    \le\int_{\{|x|>R\}}|\widetilde\phi|^2\diff\mu
    \longrightarrow0
\]
as $R\to\infty$, by the square integrability of $\widetilde\phi$.
The product rule gives
\[
    \nabla\widetilde\phi_R-\nabla\widetilde\phi
    =(\chi_R-1)\nabla\widetilde\phi+\widetilde\phi\nabla\chi_R
    \qquad\mu\text{-a.e}.
\]
These two terms satisfy
\begin{align*}
    \|(\chi_R-1)\nabla\widetilde\phi\|_{L^2(\mu;\R^m)}^2
    &\le\int_{\{|x|>R\}}\|\nabla\widetilde\phi\|^2\diff\mu
    \longrightarrow0,\\
    \|\widetilde\phi\nabla\chi_R\|_{L^2(\mu;\R^m)}^2
    &\le\frac1{R^2}\int_{\{R<|x|<2R\}}|\widetilde\phi|^2\diff\mu\\
    &\le\frac1{R^2}\|\widetilde\phi\|_{L^2(\mu)}^2
    \longrightarrow0.
\end{align*}
The triangle inequality therefore gives
\[
    \|\nabla\widetilde\phi_R-\nabla\widetilde\phi\|_{L^2(\mu;\R^m)}
    \longrightarrow0.
\]

\medskip
\noindent\textbf{Step 3: Diagonal construction.}
Let $n \in \N$. By Step~2 and \eqref{eqn:def-Sobolev}, we may choose $R_n\ge n$ such that
\[
    \|\widetilde\phi_{R_n}-\widetilde\phi\|_{\mathcal H_\lambda^1(\mu)}<\frac1{2n}.
\]
Step~1 gives
$\|\widetilde\phi_{R_n}*\eta_\varepsilon-\widetilde\phi_{R_n}\|_{\mathcal H_\lambda^1(\mu)}\to0$ as $\varepsilon\downarrow0$.
Choose $0<\varepsilon_n<1/n$ so that
\[
    \|\widetilde\phi_{R_n}*\eta_{\varepsilon_n}
       -\widetilde\phi_{R_n}\|_{\mathcal H_\lambda^1(\mu)}<\frac1{2n},
\]
and set $\varphi_n:=\widetilde\phi_{R_n}*\eta_{\varepsilon_n}\in C_c^\infty(\R^m)$.
Then
\begin{align*}
    \|\varphi_n-\widetilde\phi\|_{\mathcal H_\lambda^1(\mu)}
    &\le\|\varphi_n-\widetilde\phi_{R_n}\|_{\mathcal H_\lambda^1(\mu)}
       +\|\widetilde\phi_{R_n}-\widetilde\phi\|_{\mathcal H_\lambda^1(\mu)}\\
    &<\frac1n.
\end{align*}
This proves the claim in the interior case.
We also record that Steps~1--3 apply to any globally defined locally Lipschitz function that is square integrable, as well as its gradient, with respect to $\mu$: every such function can be approximated by $C_c^\infty(\R^m)$ functions in the $\mathcal H_\lambda^1(\mu)$ norm for each fixed $0<\lambda<1$.
We will use this observation in Case~3 below.

\medskip
\noindent\textbf{Case 2: $\lambda=1$.}
Here the integrability condition gives $\widetilde\phi\in L^2(\mu)$, and in fact we  don't even require that $\tilde{\phi}$ is locally Lipschitz or that it has a weak gradient. 
Using the cutoffs from Step~2 in Case~1, set $\widetilde\phi_R:=\widetilde\phi\chi_R$.
These functions are globally Lipschitz and compactly supported, and
\[
    \|\widetilde\phi_R-\widetilde\phi\|_{L^2(\mu)}^2
    \le\int_{\{|x|>R\}}|\widetilde\phi|^2\diff\mu
    \longrightarrow0.
\]
For each fixed $R$, the function approximation in Step~1 gives
$\|\widetilde\phi_R*\eta_\varepsilon-\widetilde\phi_R\|_{L^2(\mu)}\to0$ as $\varepsilon\downarrow0$.
For $n \in \N$, choose $R_n\ge n$ and $0<\varepsilon_n<1/n$ such that
\[
    \|\widetilde\phi_{R_n}-\widetilde\phi\|_{L^2(\mu)}<\frac1{2n},
    \qquad
    \|\widetilde\phi_{R_n}*\eta_{\varepsilon_n}
       -\widetilde\phi_{R_n}\|_{L^2(\mu)}<\frac1{2n}.
\]
For $\varphi_n:=\widetilde\phi_{R_n}*\eta_{\varepsilon_n}$, the triangle inequality gives
\[
    \|\varphi_n-\widetilde\phi\|_{\mathcal H_1^1(\mu)}
    =\|\varphi_n-\widetilde\phi\|_{L^2(\mu)}
    <\frac1n.
\]
This proves the claim for $\lambda=1$.

\medskip
\noindent\textbf{Case 3: $\lambda=0$.}
In this case we only know that $\nabla\widetilde\phi\in L^2(\mu;\R^m)$.
Since $\widetilde\phi$ need not belong to $L^2(\mu)$, we first truncate its values, while approximately preserving its gradient. For $M>0$, define
\[
    \widetilde\phi_M:=M\arctan(\widetilde\phi/M).
\]
This function is locally Lipschitz and satisfies $|\widetilde\phi_M|\le\pi M/2$, so
\[
    \|\widetilde\phi_M\|_{L^2(\mu)}^2
    \le\frac{\pi^2M^2}{4}\mu(\R^m)<\infty.
\]
The chain rule gives
\[
    \nabla\widetilde\phi_M
    =\frac{\nabla\widetilde\phi}{1+(\widetilde\phi/M)^2}
    \qquad\mu\text{-a.e.}
\]
In particular, $\|\nabla\widetilde\phi_M\|\le\|\nabla\widetilde\phi\|$ pointwisely, and hence $\nabla\widetilde\phi_M\in L^2(\mu;\R^m)$.
As $M\to\infty$, the multiplier $1/(1+(\widetilde\phi/M)^2)$ converges to one pointwise.
Moreover,
\[
    \|\nabla\widetilde\phi_M-\nabla\widetilde\phi\|^2
    \le4\|\nabla\widetilde\phi\|^2
    \qquad\mu\text{-a.e}.
\]
The right-hand side is $\mu$-integrable, so dominated convergence gives
\[
    \|\nabla\widetilde\phi_M-\nabla\widetilde\phi\|_{L^2(\mu;\R^m)}
    \longrightarrow0,
\]
as $M \rightarrow \infty$. 

Next for each $n\in \N$ choose $M_n\ge n$ such that
\[
\|\nabla\widetilde\phi_{M_n}-\nabla\widetilde\phi\|_{L^2(\mu;\R^m)}
    <\frac1{2n}.
\]
Each $\widetilde\phi_{M_n}$ is globally defined and locally Lipschitz, and it is square integrable, as well as its gradient, with respect to $\mu$.
Applying the approximation established in Case~1 with parameter $1/2$, choose $\varphi_n\in C_c^\infty(\R^m)$ such that
\[
    \|\varphi_n-\widetilde\phi_{M_n}\|_{\mathcal H_{1/2}^1(\mu)}<\frac1{2\sqrt2\,n}.
\]
By \eqref{eqn:def-Sobolev},
\[
    \|\nabla\varphi_n-\nabla\widetilde\phi_{M_n}\|_{L^2(\mu;\R^m)}
    \le\sqrt2\,\|\varphi_n-\widetilde\phi_{M_n}\|_{\mathcal H_{1/2}^1(\mu)}
    <\frac1{2n}.
\]
Consequently,
\begin{align*}
    \|\varphi_n-\widetilde\phi\|_{\mathcal H_0^1(\mu)}
    &=\|\nabla\varphi_n-\nabla\widetilde\phi\|_{L^2(\mu;\R^m)}\\
    &\le\|\nabla\varphi_n-\nabla\widetilde\phi_{M_n}\|_{L^2(\mu;\R^m)}\\
    &\quad+\|\nabla\widetilde\phi_{M_n}-\nabla\widetilde\phi\|_{L^2(\mu;\R^m)}\\
    &<\frac1n.
\end{align*}
This proves the claim for $\lambda=0$.

\medskip
In all three cases, we have constructed $\varphi_n\in C_c^\infty(\R^m)$ with
$\|\varphi_n-\widetilde\phi\|_{\mathcal H_\lambda^1(\mu)}\to0$.
The triangle inequality gives
\[
    \|\varphi_n-\varphi_k\|_{\mathcal H_\lambda^1(\mu)}
    \le\|\varphi_n-\widetilde\phi\|_{\mathcal H_\lambda^1(\mu)}
       +\|\varphi_k-\widetilde\phi\|_{\mathcal H_\lambda^1(\mu)}
    \longrightarrow0
\]
as $n,k\to\infty$.
Thus the generating pairs $(\varphi_n,\nabla\varphi_n)$ are Cauchy in the energy seminorm and determine an element of the completion in \eqref{eqn:def-HK-tangent-space}.
If $\psi_n\in C_c^\infty(\R^m)$ is another sequence with the same approximation property, then
\[
    \|\varphi_n-\psi_n\|_{\mathcal H_\lambda^1(\mu)}
    \le\|\varphi_n-\widetilde\phi\|_{\mathcal H_\lambda^1(\mu)}
       +\|\psi_n-\widetilde\phi\|_{\mathcal H_\lambda^1(\mu)}
    \longrightarrow0.
\]
The element of the completion is therefore independent of the approximating sequence.
The same argument shows that functions whose difference has zero energy determine the same element.
With these identifications,
\[
    (\widetilde\phi,\nabla\widetilde\phi)\in\Tan_\mu^{\HK_\lambda}(\mathcal M_+(\R^m)).
\]
By Remark~\ref{rem:Sobolev-gradient}, this element is independent of the chosen extension of $\phi$, which proves the proposition.
\end{proof}

\subsection{Proof of Theorem \ref{thm:loc-Lip-test-funcs}. }

\begin{proof}
    \textbf{Step 1}: Let $ f$ be globally Lipschitz with compact support. Consider constructing a sequence of functions $ f_n$ using standard mollifiers. In particular, let $ f_n =  f * \eta_{\epsilon_n}(x) := \frac{1}{\epsilon_n^d}\int  f(y)\eta\left(\frac{ x - y }{\epsilon_n}\right)\diff y$ and $\epsilon_n = 1/n$. Here, $\eta \in C_c^{\infty}$ satisfies  $\eta \geq 0$, $\int \eta(x)\diff x = 1$, and $\text{supp}(\eta) \subset B(0, 1)$. Moreover, for a path \((\mu_t)_{0\le t\le 1}\), we write
    \[
    \diff \gamma(x,t):=\diff \mu_t(x)\diff t
    \]
    for the associated space-time measure on \(\mathbb R^d\times[0,1]\).
    For a scalar function \(F(x,t)\), we write \(F\in L^2(\gamma)\) if
    \[
    \|F\|_{L^2(\gamma)}^2
    :=
    \int_0^1\int_{\mathbb R^d}|F(x,t)|^2\diff \mu_t(x)\diff t
    <\infty.
    \]
    For a vector field \(G(x,t)\), we write \(G\in L^2(\gamma;\mathbb R^d)\) if
    \[
    \|G\|_{L^2(\gamma)}^2
    :=
    \int_0^1\int_{\mathbb R^d}\|G(x,t)\|^2\diff \mu_t(x)\diff t
    <\infty.
    \]

    We first note that \(\gamma\) is finite on compact sets. Since
    \(t\mapsto\mu_t\) is assumed to be absolutely continuous with respect to
    \(d_{\HK_\lambda}\), it is continuous on \([0,1]\) with respect to the topology induced by $\dHKlambda$. For \(\lambda>0\),
    \[
        d_{\HK_\lambda}^2(\mu_t,0)
        =
        \frac4\lambda\,\mu_t(\R^d),
    \]
    where $0$ here denotes the zero measure (recall that $\dHKlambda$ can be defined for all non-negative measures). Therefore, 
    \[
        \sup_{t\in[0,1]}\mu_t(\R^d)<\infty.
    \]
    At \(\lambda=0\), the Wasserstein metric is restricted to a fixed total mass, so \(t\mapsto\mu_t(\R^d)\) is constant. Hence, in all cases,
    \[
        \gamma(K\times[0,1])
        =
        \int_0^1\mu_t(K)\diff t
        \leq
        \sup_{t\in[0,1]}\mu_t(\R^d)
        <\infty
    \]
    for every compact set \(K\subset\R^d\).

    Also note that, by the finite-action assumption, when \(0<\lambda<1\),
    \[
        v\in L^2(\gamma;\R^d),
        \qquad
        \xi\in L^2(\gamma).
    \]
    Under the pure transport case with \(\lambda=0\), the source field does not enter the continuity equation and may be replaced by \(\xi\equiv0\). Similarly, under the pure source case with \(\lambda=1\), the transport field may be replaced by \(v\equiv0\). Thus, without loss of generality, we may assume throughout that
    \[
        v\in L^2(\gamma;\R^d),
        \qquad
        \xi\in L^2(\gamma).
    \]
    
    \textbf{Claim 1}: \(f_n\to f\) in \(L^2(\gamma)\).

    Let \(L:=\operatorname{Lip}(f)\). Changing variables \(y=x-\epsilon_n z\) gives
    \[
        f_n(x)=\int_{\mathbb R^d}f(x-\epsilon_n z)\eta(z)\diff z.
    \]
    Using the Lipschitz bound, we obtain
    \begin{align*}
        |f_n(x)-f(x)|
        &\le \int_{\mathbb R^d}|f(x-\epsilon_n z)-f(x)|\eta(z)\diff z\\
        &\le L\epsilon_n\int_{\mathbb R^d}|z|\eta(z)\diff z
        =: L C_\eta\epsilon_n.
    \end{align*}
     Thus \(\|f_n-f\|_\infty\le L\epsilon_n\to0\) as \(n\to\infty\).
    Since \(f\) and the mollifier have compact support, there is a compact set
    \(K\subset\mathbb R^d\) containing
    \(\operatorname{supp}(f_n)\cup\operatorname{supp}(f)\)
    for all sufficiently large \(n\).
    Using \(\gamma(K\times[0,1])<\infty\), we obtain
    \[
        \|f_n-f\|_{L^2(\gamma)}^2
        \le \|f_n-f\|_\infty^2\,\gamma(K\times[0,1])
        \longrightarrow0
        \qquad\text{as }n\to\infty.
    \]
    \textbf{Claim 2}: \(\operatorname{Lip}(f_n)\le\operatorname{Lip}(f)\)
    and \(\|\nabla f_n(x)\|\le\operatorname{Lip}(f)\)
    for every \(x\in\mathbb R^d\).

    For \(x,\widetilde{x}\in\mathbb R^d\),
    \begin{align*}
        |f_n(\widetilde{x})-f_n(x)|
        &\le \int_{B(0,1)}
        |f(\widetilde{x}-\epsilon_n z)-f(x-\epsilon_n z)|
        \eta(z)\diff z\\
        &\le \operatorname{Lip}(f)|\widetilde{x}-x|.
    \end{align*}
    It follows that \(\operatorname{Lip}(f_n)\le\operatorname{Lip}(f)\).
    Since \(f_n\in C^\infty(\mathbb R^d)\), we also have
    \[
        \|\nabla f_n\|_\infty
        \le\operatorname{Lip}(f_n)
        \le\operatorname{Lip}(f).
    \]
    \textbf{Claim 3}: \(\grad f_n \rightarrow \grad f\) in \(L^2(\gamma; \mathbb{R}^d)\). 
    
    Similar to Case~1, Step~1 of the proof of Proposition~\ref{prop:Sobolev},
    weak differentiation commutes with mollification
    \cite[proof of Lemma~11.25]{Leoni2017}, and thus
    \[
        \nabla f_n=(\nabla f)*\eta_{\epsilon_n}
        \longrightarrow\nabla f
        \qquad\text{Lebesgue-a.e.}
    \]
    Since \(\mu_t\ll\mathcal L^d\) for every \(t\in[0,1]\), this convergence
    also holds \(\gamma\)-almost everywhere. With \(L\) and \(K\) as in
    Claim~1, Claim~2 and Rademacher's theorem give, for all sufficiently large \(n\),
    \[
        \|\nabla f_n-\nabla f\|^2\le4L^2\mathbf1_K
        \qquad\gamma\text{-almost everywhere}.
    \]
    Since \(\gamma(K\times[0,1])<\infty\), dominated convergence yields
    \[
        \|\nabla f_n-\nabla f\|_{L^2(\gamma;\mathbb R^d)}
        \longrightarrow0.
    \]
    Since \(f_n\in C_c^\infty(\mathbb R^d)\), the weak formulation of the transport-source equation gives, for every \(0\le s\le t\le 1\),
    \[
    \int_{\mathbb R^d} f_n\diff \mu_t-\int_{\mathbb R^d} f_n\diff \mu_s
    =
    \int_s^t
    \left[
    (1-\lambda)\int_{\mathbb R^d} \nabla f_n\cdot v_r\diff \mu_r
    +
    \lambda\int_{\mathbb R^d} f_n\,\xi_r\diff \mu_r
    \right]\diff r.
    \]
    We now pass to the limit \(n\to\infty\). Since \(f_n\to f\) uniformly and \(\mu_s,\mu_t\) are finite measures,
    \[
    \int_{\mathbb R^d} f_n\diff \mu_t\to \int_{\mathbb R^d} f\diff \mu_t,
    \qquad
    \int_{\mathbb R^d} f_n\diff \mu_s\to \int_{\mathbb R^d} f\diff \mu_s.
    \]
    For the source term, by Cauchy--Schwarz,
    \[
    \left|
    \int_s^t\int_{\mathbb R^d} (f_n-f)\xi_r\diff \mu_r\diff r
    \right|
    \le
    \|f_n-f\|_{L^2(\gamma)}\|\xi\|_{L^2(\gamma)}
    \to 0.
    \]
    Similarly, for the transport term,
    \[
    \left|
    \int_s^t\int_{\mathbb R^d} (\nabla f_n-\nabla f)\cdot v_r\diff \mu_r\diff r
    \right|
    \le
    \|\nabla f_n-\nabla f\|_{L^2(\gamma)}\|v\|_{L^2(\gamma)}
    \to 0.
    \]
    Therefore,
    \[
    \int_{\mathbb R^d} f\diff \mu_t-\int_{\mathbb R^d} f\diff \mu_s
    =
    \int_s^t
    \left[
    (1-\lambda)\int_{\mathbb R^d} \nabla f\cdot v_r\diff \mu_r
    +
    \lambda\int_{\mathbb R^d} f\,\xi_r\diff \mu_r
    \right]\diff r.
    \]
    Thus the weak formulation extends to every globally Lipschitz, compactly supported test function \(f\).
    
    \textbf{Step 2.}
    Now let \(f\) be locally Lipschitz and satisfy the assumptions of the theorem.
    For \(R>0\), use the cutoffs \(\chi_R\) from Case~1, Step~2
    of the proof of Proposition~\ref{prop:Sobolev}, and set
    \(f_R(x):=f(x)\chi_R(x)\).
    Since \(f\) is locally Lipschitz, \(f_R\) is globally Lipschitz and compactly
    supported. Hence Step~1 applies to \(f_R\).
    
    Since \(f_R=f\) on \(B_R\),
    \[
        |f_R-f|^2
        \leq
        |f|^2\mathbf 1_{\{|x|>R\}}.
    \]
    Because \(f\in L^2(\gamma)\), it follows that
    \[
        \|f_R-f\|_{L^2(\gamma)}^2
        \leq
        \int_0^1
        \int_{\{|x|>R\}}
            |f(x)|^2\diff \mu_t(x)\diff t
        \overset{R\rightarrow \infty}\longrightarrow0.
    \]
    
    Moreover, almost everywhere,
    \[
        \nabla f_R
        =
        \chi_R\nabla f
        +
        f\nabla\chi_R,
    \]
    and hence
    \[
        \nabla f_R-\nabla f
        =
        (\chi_R-1)\nabla f
        +
        f\nabla\chi_R.
    \]
    For the first term,
    \[
        \|(\chi_R-1)\nabla f\|_{L^2(\gamma)}^2
        \leq
        \int_0^1
        \int_{\{|x|>R\}}
            \|\nabla f(x)\|^2\diff \mu_t(x)\diff t
        \overset{R\rightarrow \infty}\longrightarrow0,
    \]
    because \(\nabla f\in L^2(\gamma;\R^d)\). For the second term,
    \begin{align*}
        \|f\nabla\chi_R\|_{L^2(\gamma)}^2
        &\leq
        \frac1{R^2}
        \int_0^1
        \int_{\{R<|x|<2R\}}
            |f(x)|^2\diff \mu_t(x)\diff t\\
        &\leq
        \frac1{R^2}
        \|f\|_{L^2(\gamma)}^2
        \overset{R\rightarrow \infty}\longrightarrow0.
    \end{align*}
    Therefore,
    \[
        \nabla f_R
        \longrightarrow
        \nabla f
        \qquad
        \textnormal{in }L^2(\gamma;\R^d).
    \]
    
    Since each \(f_R\) is globally Lipschitz with compact support, Step~1 gives, for every \(0\leq s\leq t\leq1\),
    \[
        \int_{\R^d}f_R\diff \mu_t
        -
        \int_{\R^d}f_R\diff \mu_s
        =
        \int_s^t
        \left[
            (1-\lambda)
            \int_{\R^d}
                \nabla f_R\cdot v_r\diff \mu_r
            +
            \lambda
            \int_{\R^d}
                f_R\xi_r\diff \mu_r
        \right]\diff r.
    \]
    
    We now let \(R\to\infty\). Since \(f_R\to f\) pointwise and
    \(|f_R|\leq|f|\), assumption~(i) and the dominated convergence theorem give
    \[
        \int_{\R^d}f_R\diff \mu_t
        \longrightarrow
        \int_{\R^d}f\diff \mu_t,
        \qquad
        \int_{\R^d}f_R\diff \mu_s
        \longrightarrow
        \int_{\R^d}f\diff \mu_s.
    \]
    For the source term, Cauchy--Schwarz gives
    \[
        \left|
            \int_s^t
            \int_{\R^d}
                (f_R-f)\xi_r\diff \mu_r\diff r
        \right|
        \leq
        \|f_R-f\|_{L^2(\gamma)}
        \|\xi\|_{L^2(\gamma)}
        \longrightarrow0.
    \]
    Similarly, for the transport term,
    \[
        \left|
            \int_s^t
            \int_{\R^d}
                (\nabla f_R-\nabla f)\cdot v_r\diff \mu_r\diff r
        \right|
        \leq
        \|\nabla f_R-\nabla f\|_{L^2(\gamma;\R^d)}
        \|v\|_{L^2(\gamma;\R^d)}
        \longrightarrow0.
    \]
    Therefore,
    \[
        \int_{\R^d}f\diff \mu_t
        -
        \int_{\R^d}f\diff \mu_s
        =
        \int_s^t
        \left[
            (1-\lambda)
            \int_{\R^d}
                \nabla f\cdot v_r\diff \mu_r
            +
            \lambda
            \int_{\R^d}
                f\xi_r\diff \mu_r
        \right]\diff r.
    \]
    
    Finally, set
    \[
    F(t):=\int_{\mathbb R^d} f\diff \mu_t,
    \]
    and
    \[
    h(t):=
    (1-\lambda)\int_{\mathbb R^d}\nabla f\cdot v_t\diff \mu_t
    +
    \lambda\int_{\mathbb R^d}f\,\xi_t\diff \mu_t.
    \]
    By Cauchy--Schwarz,
    \[
    \int_0^1
    \left|
    \int_{\mathbb R^d}\nabla f\cdot v_t\diff \mu_t
    \right|\diff t
    \le
    \|\nabla f\|_{L^2(\gamma)}\|v\|_{L^2(\gamma)},
    \]
    and
    \[
    \int_0^1
    \left|
    \int_{\mathbb R^d}f\,\xi_t\diff \mu_t
    \right|\diff t
    \le
    \|f\|_{L^2(\gamma)}\|\xi\|_{L^2(\gamma)}.
    \]
    Hence \(h\in L^1(0,1)\). Since
    \[
    F(t)-F(s)=\int_s^t h(r)\diff r
    \qquad
    \text{for all }0\le s\le t\le1,
    \]
    it follows that \(F\) is absolutely continuous on \([0,1]\), and
    \[
    F'(t)=h(t)
    \qquad
    \text{for a.e. }t\in(0,1).
    \]
\end{proof}

\end{document}